\documentclass[11pt,reqno]{amsart}

\usepackage[utf8]{inputenc}
\usepackage[T1]{fontenc}
\usepackage{mathpazo}

\usepackage{geometry}

\usepackage{amsmath}
\usepackage{amssymb}
\usepackage{amsthm}
\usepackage{mathtools}
\usepackage{mathrsfs}
\usepackage{stmaryrd}
\usepackage{dsfont}
\usepackage{upgreek}

\usepackage{enumitem}
\usepackage{longtable}

\usepackage{xcolor}
\usepackage{tikz}
\usetikzlibrary{
    calc,
    intersections,
    patterns
}
\usepgflibrary{shadings}
\usepackage[percent]{overpic}

\usepackage{verbatim}
\usepackage{marginnote}
\usepackage{lipsum}

\usepackage[
    colorinlistoftodos,
    backgroundcolor=green!20!white,
    bordercolor=green!75!black,
    linecolor=green!75!black
]{todonotes}

\usepackage[
    colorlinks,
    citecolor=blue,
    linkcolor=blue,
    urlcolor=green
]{hyperref}
\usepackage{aliascnt}
\usepackage{cleveref}

\numberwithin{equation}{section}

\theoremstyle{plain}

\newtheorem{theorem}{Theorem}[section]

\newaliascnt{lemma}{theorem}
\newtheorem{lemma}[lemma]{Lemma}
\aliascntresetthe{lemma}

\newaliascnt{proposition}{theorem}
\newtheorem{proposition}[proposition]{Proposition}
\aliascntresetthe{proposition}

\newaliascnt{corollary}{theorem}
\newtheorem{corollary}[corollary]{Corollary}
\aliascntresetthe{corollary}

\newtheorem{claim}{Claim}

\theoremstyle{definition}

\newaliascnt{definition}{theorem}
\newtheorem{definition}[definition]{Definition}
\aliascntresetthe{definition}

\newaliascnt{remark}{theorem}
\newtheorem{remark}[remark]{Remark}
\aliascntresetthe{remark}

\crefname{theorem}{theorem}{theorems}
\Crefname{theorem}{Theorem}{Theorems}

\crefname{lemma}{lemma}{lemmas}
\Crefname{lemma}{Lemma}{Lemmas}

\crefname{proposition}{proposition}{propositions}
\Crefname{proposition}{Proposition}{Propositions}

\crefname{corollary}{corollary}{corollaries}
\Crefname{corollary}{Corollary}{Corollaries}

\crefname{definition}{definition}{definitions}
\Crefname{definition}{Definition}{Definitions}

\crefname{remark}{remark}{remarks}
\Crefname{remark}{Remark}{Remarks}

\crefname{claim}{claim}{claims}
\Crefname{claim}{Claim}{Claims}

\newcommand{\N}{\mathbf{N}}
\newcommand{\Z}{\mathbf{Z}}
\newcommand{\C}{\mathbf{C}}
\newcommand{\R}{\mathbf{R}}

\DeclareMathOperator{\spt}{spt}
\DeclareMathOperator{\dist}{dist}

\newcommand{\out}{\mathrm{out}}
\newcommand{\h}{\boldsymbol{h}}
\newcommand{\U}{\mathbf{U}}
\newcommand{\X}{\R^2_*}
\newcommand{\p}{\boldsymbol{p}}

\newcommand{\Ical}{\mathds{1}}
\renewcommand{\L}{\mathscr{L}}

\newcommand{\eps}{\varepsilon}

\definecolor{forestgreen}{rgb}{0.0,0.733,0.408}

\title[A two-dimensional Allen--Cahn theory]
{A two-dimensional Allen--Cahn theory for interfaces with boundary}

\author{Marco Badran}
\address{
Department of Mathematics,
ETH Zürich,
Rämistrasse 101,
8092 Zürich,
Switzerland
}
\email{marco.badran@math.ethz.ch}

\author{Manuel del Pino}
\address{
Department of Mathematical Sciences,
University of Bath,
4 West, Claverton Down,
Bath BA2 7AY,
United Kingdom
}
\email{m.delpino@bath.ac.uk}

\author{Marco A. M. Guaraco}
\address{
Department of Mathematics,
Imperial College London,
Huxley Building,
180 Queen's Gate,
London SW7 2AZ,
United Kingdom
}
\email{guaraco@imperial.ac.uk}



\begin{document}

\begin{abstract}
We develop a two-dimensional Allen--Cahn theory for interfaces with
boundary in the line-bundle framework introduced by Fr\"ohlich and
Struwe. The central new object is a model solution on the punctured
plane whose nodal set is a half-line. Near the boundary of an
interface, this solution plays the role of the heteroclinic profile in
the classical interior theory. However, unlike the interior setting,
where the analysis in directions normal to the interface reduces to
the ODE satisfied by the heteroclinic, the boundary model has non-flat
level sets and must be studied as a genuinely two-dimensional elliptic
solution. In the first part of this paper, we construct this model
solution and develop its stability and invertibility theory.

In the second part of this paper, we present the main application: we
construct Allen--Cahn sections whose nodal sets concentrate on any
prescribed finite collection of disjoint line segments in the plane.
The construction uses a Lyapunov--Schmidt reduction, but the boundary
introduces a new difficulty: it creates large error terms along the
interior of the interface. Controlling these terms requires a refined
ansatz and new gluing arguments beyond the standard interior
Allen--Cahn theory.
\end{abstract}

\maketitle

\tableofcontents

\section{Introduction}

The connection between the theory of phase transitions and minimal hypersurfaces
is by now well established. The basic object is the Allen--Cahn equation
\begin{equation}\label{eq:intro-allen-cahn}
    \eps^2\Delta u-W'(u)=0.
\end{equation}
This is the equation satisfied by critical points of the Allen--Cahn energy
\begin{equation}\label{eq:intro-allen-cahn-energy}
    E_\eps(u)
    \coloneqq
    \int_{\Omega}\eps\frac{|\nabla u|^2}{2}
    +\frac{W(u)}{\eps}.
\end{equation}
Here $W$ is an even double-well potential with nondegenerate wells at $\pm1$;
the standard example is $W(s)=(1-s^2)^2/4$.
For small $\eps$, solutions look like two phases separated by a thin transition
layer. Roughly speaking, as $\eps\to0$, the nodal set $\{u=0\}$ converges to a
minimal hypersurface. In recent years, this picture has been developed in several
directions, enriching both the theory of PDEs and minimal surface theory. We
recall some of these developments later in this introduction.

The point of departure for this paper is that the usual Allen--Cahn equation
is limited to modelling separating interfaces: the nodal set of a real-valued function separates
the phases $\{u>0\}$ and $\{u<0\}$. This is natural for closed minimal
hypersurfaces, but not for the important set of classical Plateau-type problems, whose limits have prescribed
boundary and are nonseparating. The idea for bypassing this
restriction was introduced by Fr\"ohlich and Struwe
\cite{Frohlich-Struwe90} and recently extended by Guaraco and Lynch \cite{Guaraco-Lynch}. In this framework,
Allen--Cahn critical sections over $\R^n\setminus\Gamma$ converge to stationary
hypersurfaces with boundary $\Gamma$, while the elliptic structure, i.e, equation
\eqref{eq:intro-allen-cahn} is retained.

The present paper takes a first step in understanding new phenomena
appearing near the prescribed boundary. Away from the boundary, a local
trivialization of the line bundle leads back to the classical interior
Allen--Cahn theory. By contrast, near the boundary, the topology of the bundle is
visible, and the natural local model becomes a branched-cover elliptic problem,
which is genuinely two-dimensional rather than the classical one-dimensional
heteroclinic ODE.

\subsection{The interior theory}

To understand this contrast better, we now recall the interior local model.
At a smooth interior point of a limiting
hypersurface, the solution has a distinguished normal profile: in the normal
direction it resembles the centered, increasing one-dimensional heteroclinic
$H$, a smooth step from $-1$ to $1$. For the standard potential
$W(s)=(1-s^2)^2/4$, one has $H(t)=\tanh(t/\sqrt2)$; see
\Cref{fig:normal-profile}.

\begin{figure}[ht]
    \centering
    \begin{tikzpicture}[x=0.55cm,y=0.85cm]
        \draw[thin,->] (-4.4,0) -- (4.4,0) node[right] {$t$};
        \draw[thin,->] (0,-1.35) -- (0,1.35) node[above] {$H(t)$};
        \draw[densely dashed,gray] (-4.2,1) -- (4.2,1);
        \draw[densely dashed,gray] (-4.2,-1) -- (4.2,-1);
        \node[left] at (0,1) {$1$};
        \node[left] at (0,-1) {$-1$};
        \draw[thick,domain=-4:4,samples=120,smooth]
            plot ({\x},{(exp(2*\x/sqrt(2))-1)/(exp(2*\x/sqrt(2))+1)});
    \end{tikzpicture}
    \hspace{0.08\textwidth}
    \begin{tikzpicture}[x=1cm,y=1cm]
        \path[use as bounding box] (-2.9,-1.55) rectangle (2.9,1.55);
        \draw[very thick]
            (-2.55,-0.25)
            .. controls (-1.75,0.55) and (-0.95,-0.55) .. (-0.10,-0.05)
            .. controls (0.75,0.45) and (1.55,-0.20) .. (2.45,0.20)
            node[right] {$\Sigma$};
        \draw[blue!70!black,->,thick] (-1.25,-0.03) -- (-1.05,0.89);
        \draw[blue!70!black,->,thick] (0,0) -- (0,1.15)
            node[above] {$\nu$};
        \draw[blue!70!black,->,thick] (1.25,0.12) -- (1.50,1.02);
        \fill (0,0) circle (1.2pt) node[below right] {$p$};
        \node[fill=white,inner sep=1pt] at (-2.25,1.05) {$u\approx 1$};
        \node[fill=white,inner sep=1pt] at (-2.25,-1.05) {$u\approx -1$};
        \node[fill=white,inner sep=1pt] at (0.90,-1.05)
            {$u(p+t\nu)\approx H(t/\eps)$};
    \end{tikzpicture}
    \caption{The heteroclinic as a normal profile. Left: the standard profile
    $H(t)=\tanh(t/\sqrt2)$. Right: near a smooth interior point of the limiting
    hypersurface $\Sigma$, the leading-order approximation varies in the normal
    coordinate $t=d(x,\Sigma)$.}
    \label{fig:normal-profile}
\end{figure}
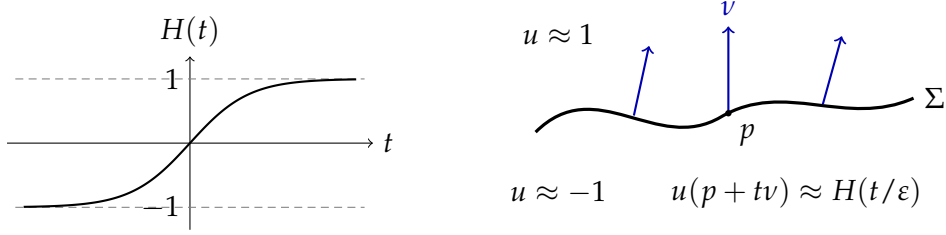

Thus, near a smooth point of an interface, a solution is expected to look to
first order like a perturbation of $H(d(x,\Sigma)/\eps)$, where $d(x,\Sigma)$ is the signed
distance to the limiting hypersurface. This local picture is one of the central
mechanisms behind the classical interior theory: it appears in gluing
constructions near prescribed minimal hypersurfaces
\cite{PacardRitore,delPino-Kowalczyk-Wei13,delPino-Kowalczyk-Wei11}, in
Allen--Cahn min-max constructions and applications to minimal hypersurfaces
\cite{Guaraco18,Gaspar-Guaraco18,Gaspar-Guaraco19,Guaraco-Marques-Neves19,Chodosh-Mantoulidis20,Chodosh-Mantoulidis23PWidths},
and in classification and regularity questions for the level sets of solutions
\cite{Ambrosio-Cabre00,Ghoussoub-Gui03,Savin09Flat,Hutchinson-Tonegawa00,Tonegawa-Wic12,Wickramasekera14,Wang17Savin,WangWei19a,WangWei19b}.

The shared analytical point is that, after linearization, the normal direction is
governed by the one-dimensional heteroclinic operator $\partial_t^2-W''(H)$,
whose kernel is generated by $H'$, while the tangential directions lead to the
Jacobi operator of the limiting hypersurface. This separation is used directly in
gluing constructions, and in regularity theory it is used in the reverse
direction: once the normal profile is identified, the separate analysis of the
normal and tangential operators gives quantitative control on the level sets.

\subsection{The Line Bundle Framework}

The line-bundle framework is perhaps best understood with the aid of
a double cover. This construction is carried out in
more detail in \Cref{subs: line bundle}, but for the purposes
of this introduction we now briefly describe it in the two-dimensional case that
concerns us in this paper.

Let $\R^2_{\p}\coloneqq\R^2\setminus\p$, where
$\p=\{p_1,\ldots,p_k\}\subset\R^2$ is a finite set of points. There is a unique
double cover
$\pi\colon\widetilde X_{\p}\to \R^2_{\p}$ characterized by the following local
picture: if $B_r(p_i)$ contains no point of $\p$ other than $p_i$, then over
$B_r(p_i)\setminus\{p_i\}$ the cover is topologically equivalent to the complex square-root cover
$\zeta\mapsto p_i+\zeta^2$. Writing $\iota$ for the deck involution of $\widetilde X_{\p}$, we define a line bundle by
\[
    \L_{\p}\coloneqq
    (\widetilde X_{\p}\times\R)/\{(\widetilde x,t)\sim(\iota\widetilde x,-t)\}
\]
Thus, sections of $\L_{\p}$ correspond to real-valued functions $\widetilde u$ on the double cover which are \emph{odd} in the sense that
$\widetilde u\circ\iota=-\widetilde u$, and conversely every function with this
oddness property descends to a section.

\subsection{The Model Boundary Solution}

Our first goal in this paper is to establish the existence and properties of the
model boundary solution on the punctured plane. This is the
solution whose nodal set is a half-line, and it is the profile expected to arise
by blowing up an Allen--Cahn section near a boundary point of the limiting
interface. It is therefore the boundary analogue of the one-dimensional
heteroclinic in the classical interior theory, and we expect it to have a
central role in perturbative and regularity results even in higher dimensions;
see \Cref{fig:model-boundary-solution}.

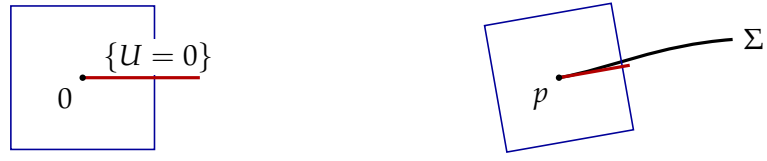
\begin{figure}[ht]
    \centering
    \begin{tikzpicture}[x=1cm,y=1cm]
        \path[use as bounding box] (-2.35,-1.45) rectangle (2.35,1.45);
        \draw[blue!60!black,semithick] (-0.95,-0.95) rectangle (0.95,0.95);
        \draw[very thick,red!70!black] (0,0) -- (1.55,0);
        \fill (0,0) circle (1.2pt) node[below left] {$0$};
        \node[fill=white,inner sep=1pt] at (1.00,0.27) {$\{U=0\}$};
        \node[fill=white,inner sep=1pt] at (0,-1.25) {};
    \end{tikzpicture}
    \hspace{0.08\textwidth}
    \begin{tikzpicture}[x=1cm,y=1cm]
        \path[use as bounding box] (-2.55,-1.45) rectangle (2.75,1.45);
        \begin{scope}[rotate=10]
            \draw[very thick]
                (0,0)
                .. controls (0.70,0.00) and (1.35,0.20) .. (2.35,0.10)
                node[right] {$\Sigma$};
            \draw[blue!60!black,semithick] (-0.85,-0.85) rectangle (0.85,0.85);
            \draw[very thick,red!70!black] (0,0) -- (0.95,0);
        \end{scope}
        \fill (0,0) circle (1.2pt) node[below left] {$p$};
    \end{tikzpicture}
    \caption{The model boundary solution. Left: on the punctured plane, the
    nodal set of the model solution is a half-line issuing from the puncture.
    Right: after blowing up near a boundary point $p$ of the limiting interface
    $\Sigma$, we expect to see the local model as an analogue to a half-line tangent to
    $\Sigma$ at $p$.}
    \label{fig:model-boundary-solution}
\end{figure}

The setting is the following: let $\p=\{(0,0)\}$ and set
$\R^2_*\coloneqq\R^2_{\p}$. We denote by $\L=\L_{\p}$ the corresponding line bundle.

\noindent\textbf{Theorem A.} \emph{There exists a section, unique up to sign,
$U$ of $\L\to\R^2_*$ satisfying
\[
    \Delta U-W'(U) =0
\]
whose nodal set is the positive half-line $\{(x,0):x>0\}$.}

A more precise statement can be found in \Cref{prop: existence} below. A key
point in the construction is the following monotonicity: if $\U$ denotes the odd
scalar lift of $U$ under the square-root cover $\xi\mapsto \xi^2$, then
$\partial_{\xi_2}\U>0$. This inequality is far from evident, given that, in contrast with the interior theory, $|U|$ and $\U$ do not have any flat level sets other than the nodal set. The inequality is obtained through
the careful construction carried out in \Cref{subs: construct U}. This is the
two-dimensional analogue of the positivity of $H'$ in the interior theory and is
also reminiscent of the monotonicity hypothesis in De Giorgi-type problems. As
such, it has a central role in the linear theory.

\subsection{The Linear Theory at the Boundary}
One of the main analytic contributions of this paper is the development of a
linear theory around the model boundary solution. In the interior Allen--Cahn theory,
the relevant linear analysis is ultimately governed by the invertibility properties
of the linearized ODE
\begin{equation}\label{eq:linearized-heteroclinic}
    \bigl(\partial_t^2-W''(H)\bigr)H'=0.
\end{equation}
In the boundary case, the corresponding operator is a two-dimensional elliptic
operator on a line bundle, and its invertibility properties are more subtle. First,
we prove that the model solution is stable and that the bounded kernel of its
linearized operator is trivial.

\noindent\textbf{Theorem B.} \emph{The model solution $U$ is stable. Moreover, if
$\phi$ is a bounded section of $\L\to\R^2_*$ satisfying
\[
    \Delta\phi-W''(U)\phi=0,
\]
then $\phi\equiv0$. }

A precise statement of Theorem B can be found in
\Cref{prop: stability of L,cor: nondegeneracy-section} below. The
nondegeneracy of Theorem B,
however, is not enough for the estimates needed later. More precisely, along the end of the
half-line, i.e. when $x\to\infty$, and under a suitable trivialisation, the section $U(x,z)$
approaches $H(z)$. In particular, the operator approaches the linearized
heteroclinic operator in \eqref{eq:linearized-heteroclinic}; since the kernel of
this limiting operator is generated by $H'$, it produces an almost-kernel at
infinity for the linearization around the boundary solution. This means that the boundary operator is nondegenerate in the bounded sense, but it is still
asymptotically degenerate. We develop a fine invertibility theory with
estimates that keep track of the asymptotic heteroclinic component.

Given a bounded section $f$ of $\L\to\R^2_*$, we define the \emph{tangential part of $f$} as
\[
    f^\top(x)\coloneqq\int_{\R} f(x,z)H'(z)\,dz
\]
where $f$ is understood in the scalar trivialization obtained by removing the
negative half-line $\{(x,0):x<0\}$. Given a decay profile \(g\) on \(\R^+\),
for example \(g(x)=(1+x)^{2-\tau}\) with \(\tau>2\), the tangential part of the
source is measured by \(A_g(f)\), the infimum of the constants \(A\geq0\) such
that \(|f^\top(x)|\leq A g''(x)\) for every \(x>0\), with
\(A_g(f)=+\infty\) if no such constant exists; see \Cref{def:Ag-norm}.

\noindent\textbf{Theorem C.} \emph{
Let $g\colon\R^+\to\R$ be an admissible decay profile as above.
Then the equation
\[
    \Delta\phi-W''(U)\phi=f
\]
has a unique bounded solution $\phi$ for every bounded H\"older section $f$ of $\L\to\R^2_*$ satisfying
\[
    A_g(f)<+\infty.
\]
Moreover, there exists a constant
$\alpha_\infty$ such that
\[
    \phi^\top(x)\longrightarrow \alpha_\infty
    \qquad\text{as }x\to+\infty.
\]
The constant $\alpha_\infty$ is given explicitly by
\begin{equation*}
    \alpha_\infty
    =\lim_{R\to\infty}\int_{(-R,R)^2}
    f(x,z)(-x\partial_zU+z\partial_xU)\,dx\,dz,
\end{equation*}
Moreover, both $\phi$ and the convergence rate of its tangential part are controlled by
\(\lVert f\rVert_{L^\infty}\), \(A_g(f)\), and \(g\).}

A more precise statement of Theorem C can be found in \Cref{lem: boundary invertibility} below.

\subsection{Line-Segment Interfaces}

The second part of the paper gives a first application of this linear theory: the
construction of Allen--Cahn interfaces concentrating on any prescribed finite
collection of disjoint line segments in the plane.

\noindent\textbf{Theorem D.} \emph{Let $I_1,\dots,I_N$ be a finite collection of
pairwise disjoint line segments in $\R^2$, and let
$\p=\bigcup_{j=1}^N\partial I_j$. Then, for all sufficiently small
$\eps>0$, there exists a section $u=u_\eps$ of
$\L_{\p}\to\R^2_{\p}$ satisfying
\[
    \eps^2\Delta u-W'(u)=0
\]
such that the nodal sets $\{u=0\}$ converge to
$I_1\cup\cdots\cup I_N$ as $\eps\to0$.}

This result should be understood as a two-dimensional version of a Plateau
problem: the prescribed boundary is the finite set $\p$, and the limiting
interface spans it by the chosen collection of line segments. The precise
statement of Theorem D, including the topology of convergence and the estimates
satisfied by the solutions, will be stated in \Cref{sec: LS}.

Working with straight segments removes the additional terms that would come from
curvature or higher dimensional geometry. This lets us isolate the new
difficulty introduced by the line-bundle framework in constructions in the
spirit of classical Allen--Cahn gluing
\cite{PacardRitore,delPino-Kowalczyk-Wei11,delPino-Kowalczyk-Wei-Yang10}. Near each endpoint the natural
profile is a rotated copy of the two-dimensional boundary solution $U$, while
along the interior of the segment the natural profile is the one-dimensional
heteroclinic $H$, perturbed by translations in the normal direction. These two
adjustments are different: the boundary ansatz is adjusted by rotations, whereas
the interior ansatz is adjusted by vertical shifts. Matching them creates both a
coordinate mismatch and a solution mismatch near the endpoints. The resulting
errors are not purely local; they propagate into the interior problem along the
segment, and some of their components are not captured by the usual Jacobi
operator governing translations of the heteroclinic. The construction therefore
requires a careful balance of the system, between the boundary operator, the inner operator
over the segments, and the outer operator on the complement.

We would like to emphasize that perturbative results of this kind are important
for two related reasons. First, they complete one direction of the
Allen--Cahn/minimal-surface correspondence: prescribed geometric objects can be
realized as limits of solutions. Second, and more importantly for the present
theory, gluing constructions reveal the analytic mechanisms that should also
appear when studying the regularity of general solutions. The choice of coordinates, the decomposition into
boundary, inner, and outer operators, the relevant almost-kernels, and the error
terms that must be balanced, are not merely artifacts of the ansatz; they are
expected to be generic features of solutions whose nodal sets have suitable
geometric control, for instance under stability hypotheses. In the classical
theory, insights of this kind have later been reverse-engineered into regularity
and existence results
\cite{WangWei19a,WangWei19b,Chodosh-Mantoulidis20}.

\subsection{Key ideas and organisation of the paper}

We close the introduction by indicating where the main ideas enter the proofs.
Theorems A--C develop the boundary model and its linear theory, while Theorem D
uses that theory in a perturbative construction of Allen--Cahn sections
concentrating on line segments.

For Theorem A, we work on the square-root double cover. A section $U$ is
represented by an odd scalar function $\U$ through the change of variables
$\pi(\xi)=\xi^2$. In these variables the puncture is removed: $\U$ is an
ordinary map $\R^2\to\R$, and the Allen--Cahn equation becomes
\begin{equation}\label{eq:lifted-weighted-equation}
    \Delta\U=4|\xi|^2W'(\U).
\end{equation}
The solution is obtained as a limit of minimizers for Dirichlet problems on
expanding rectangles. The finite-domain boundary data is chosen so that the
minimizers are monotone in the $\xi_2$ direction, and this monotonicity passes to
the limit. This is the construction behind the property
$\partial_{\xi_2}\U>0$ emphasized above. The construction and asymptotics of
$U$ are proved in \Cref{subs: construct U,subs: asymp U}.

The proof of Theorem B starts from this monotonicity. Differentiating the
lifted equation shows that $\partial_{\xi_2}\U$ is a positive supersolution for
the lifted linearized operator, and the ground-state identity gives stability.
For bounded nondegeneracy, a bounded kernel section is lifted to an odd bounded
function on the cover. After dividing by a positive Jacobi solution, logarithmic
cutoffs force the quotient to be constant; the oddness then forces this constant
to be zero. These arguments are carried out in \Cref{subs: linearised op}.

For Theorem C, as mentioned above, bounded nondegeneracy is not enough: along
the positive end the boundary operator approaches the linearized heteroclinic
operator in \eqref{eq:linearized-heteroclinic}. The proof isolates the
corresponding almost-kernel by projecting the right-hand side and the solution
onto the $H'$ direction. The projected equation has the form
$\alpha''=f^\top+E$, where $E$ is exponentially small because $U$ converges to
$H$ along the end. The convex function $g$ provides a barrier for this ODE, and
the remaining a priori estimate follows by contradiction: any loss of compactness
must drift out along the end and limit to a multiple of $H'$, which the
projected estimates rule out. This refined invertibility theory is proved in
\Cref{subs: linearised op}.

For Theorem D, the first step is the construction of the approximate solution in
\Cref{subs: approximate solution}. It is built from three pieces: rotated copies
of $U$ near the endpoints, translated heteroclinic profiles along the interiors
of the segments, and the constant states $\pm1$ away from the interface. The
nonclassical point is the endpoint matching. A Fermi-coordinate heteroclinic
ansatz would fit the interior transition well but create large errors at the
boundary points; instead we keep a rigid rotated copy of $U$ near each endpoint
and choose the transition scale so that the boundary and interior errors are
balanced.

This choice of ansatz determines the Lyapunov--Schmidt reduction carried out in
\Cref{subs: correction-splitting}. Since the heteroclinic ansatz only takes over away from the
endpoints, the Jacobi equation for the translation parameter is supported on an
interior portion of each segment and does not reach the boundary. At the same
time, the boundary correction has an asymptotic $H'$ component, and this is the
datum through which it couples to the inner equation. The estimates must
therefore treat the boundary, inner, and outer problems as a coupled system and
track the $H'$ projection; the estimates needed to close the argument are proved
in \Cref{sec: proofs}.

\subsection{Open questions and future directions}

The two-dimensional theory developed here raises several natural questions. The
first concerns classification. In the classical planar Allen--Cahn theory,
De Giorgi-type statements, stability, and finite Morse index are closely tied to
the asymptotic structure of level sets. In the line-bundle setting it is natural
to ask whether the model solution $U$ is the only stable solution on
$\L\to\R^2_*$, and whether the same conclusion holds under a
one-directional monotonicity assumption on the lift. More generally, one expects
finite-index solutions to have a finite, necessarily odd, number of ends, in
analogy with the classical results of \cite{WangWei19a}.

A second direction is regularity up to the prescribed boundary. In the classical
case, the nodal sets of stable solutions satisfy $C^{2,\theta}$ estimates
\cite{WangWei19a,WangWei19b}. It is natural to ask whether analogous second
order estimates hold uniformly up to boundary points for stable sections in the
line-bundle problem. From this point of view, the model solution $U$ should be
the blow-up profile governing the endpoint behavior of the interface.

The two-dimensional model is also intended as a starting point for higher
dimensional Plateau problems. The techniques developed here, together with the
line-bundle compactness framework of \cite{Guaraco-Lynch}, should be refined and
extended to approximate Plateau solutions by Allen--Cahn sections of suitable
line bundles. We plan to pursue this in future work.

It would also be interesting to understand how far this viewpoint extends beyond
the Allen--Cahn equation. Analogous boundary-profile phenomena may appear in
other elliptic theories, such as Ginzburg--Landau or Yang--Mills--Higgs-type
equations. In the planar setting, one could also explore whether the
line-bundle Allen--Cahn flow gives a useful diffuse-interface perspective on
geometric matching problems for finite sets of points
\cite{Vaidya89,Cook-Rohe99}.

\subsection{Notation guide}

This guide is meant only as a quick reference. The precise definitions and
normalizations are given at the indicated points in the text.

\begingroup
\small
\renewcommand{\arraystretch}{1.15}
\begin{longtable}{@{}p{0.24\textwidth}p{0.70\textwidth}@{}}
\textit{Notation} & \textit{Meaning}\\
\hline
\endfirsthead
\textit{Notation} & \textit{Meaning}\\
\hline
\endhead
\multicolumn{2}{@{}l}{\textit{Domains, covers, and sections}}\\
$\p$ &
Finite set of punctures. In the line-segment problem,
$\p=\bigcup_j\partial I_j$; see \Cref{sec: preliminaries-gluing}.\\
$\Omega_{\p}$, $\R^2_{\p}$ &
Punctured domains, $\Omega_{\p}=\Omega\setminus\p$ and
$\R^2_{\p}=\R^2\setminus\p$.\\
$\R^2_*$ &
The punctured plane with one puncture at $(0,0)$.\\
$\ell_+,\ell_-$ & The positive and negative half-lines, defined by $\{(x,0)\in \R^2: x>0\}$ and $\{(x,0)\in \R^2: x<0\}$.\\
$\pi$, $\iota$ &
The sign double cover $\pi\colon\widetilde\Omega_{\p}\to\Omega_{\p}$ and its
deck involution; see \Cref{sec: preliminaries}.\\
$\L_{\p}$ &
The associated real line bundle; scalar representatives change sign after one
turn around a puncture; see \Cref{subs: line bundle}.\\
$\L$ &
The one-puncture line bundle over $\R^2_*$.\\
$\Gamma(V,\L_{\p})$, $\Gamma(\R^2_*,\L)$ &
Spaces of smooth sections of the corresponding line bundles over the indicated
domains; see \Cref{subs: line bundle}.\\

\multicolumn{2}{@{}l}{\textit{Model sections and equations}}\\
$H$ &
The normalized one-dimensional heteroclinic; see
\Cref{subs:allen-cahn-on-sections}.\\
$U$, $\U$ &
The model boundary section on $\L\to\R^2_*$ and its odd scalar lift under
$\xi\mapsto\xi^2$; see \Cref{sec: 2d sol}.\\
$\sigma$ &
The parameter dictating the exponential rate of decay of $|U-H|$, see
\Cref{cor: decay U_1}.\\
$\Ical$ &
The outer section determined by a choice of sheet over the complement of
the line segments; see \Cref{def: outer section}.\\
$E_\eps$, $\widetilde E$ &
The Allen--Cahn energy for sections and the lifted energy on the square-root
cover; see \eqref{eq:section-allen-cahn-energy}.\\
$S(u)$ &
The error operator $S(u)=\eps^2\Delta u-W'(u)$ used in the gluing argument; see
\eqref{eq: EoA}.\\

\multicolumn{2}{@{}l}{\textit{Cut-offs and coordinates}}\\
$\delta$ &
The fixed separation scale for the line segments, chosen so that
\[
    0<\delta<\frac{1}{24}
    \min\Big\{\min_j |I_j|,
    \min_{j\ne k}\dist(I_j,I_k)\Big\},
\]
with the second minimum omitted when $N=1$; see
\Cref{sec: preliminaries-gluing}.\\
$\alpha$ &
The exponent defining the gluing scale $\eps^\alpha$, chosen with
\[
    \alpha\in\left(\frac{1}{2+\gamma},\frac12\right);
\]
see \Cref{sec: preliminaries-gluing}.\\
$\chi_{j,m}$, $\chi_{\Ical,m}$ &
Cut-offs near the segment $I_j$ and in the complement of all segments.\\
$X(q,z)$ &
Fermi coordinates along a segment; see \eqref{eq: coordinates Xj}.\\
$Y_p(x,z)$ &
Boundary coordinates near $p\in\partial I_j$; see \eqref{eq: coordinates Yp}.\\
$\zeta_m$, $\eta_m$ &
Boundary and interior cut-offs along a segment, with $\eta_m=1-\zeta_m$.\\

\multicolumn{2}{@{}l}{\textit{Norms and gluing parameters}}\\
$\lambda$, $\gamma$ &
The H\"older exponent and exponential weight. The exponent
\(\gamma\in(0,1)\) is fixed in \eqref{eq: exponential-decay-params}, and the
gluing weight is fixed in \Cref{subs: gluing-conventions}.\\
$\tau$, $R_\eps$ &
The boundary decay exponent and correction radius fixed in
\Cref{subs: gluing-conventions}; see \eqref{eq: boundary-decay-exponent-range} and
\eqref{eq: correction-radius}.\\
$\beta$ &
The smallness exponent for vertical shifts and corrections, taken with
$\beta\geq2$ in the gluing construction; see
\Cref{subs: gluing-conventions}.\\
$C_{*,\eps}^{k,\gamma}$ &
The scaled section norm adapted to the square-root cover at punctures.\\
$C_{A,\eps,\lambda}^{k,\gamma}$ &
The exponentially weighted section norm about a set $A$.\\
$\mathsf C_{\eps,\lambda}^{k,\gamma}$,
$\mathcal C_{\eps,\lambda}^{k,\gamma}$ &
The weighted product-space norms in physical variables \((x,z)\) and stretched
interior variables \((x,t)\), respectively; see
\Cref{subs: gluing-conventions}.\\
$\h=(h_1,\dots,h_N)$, $\theta_j$ &
Vertical shifts along the intervals and endpoint rotation parameters, with
$\theta_j(p)=h_j'(p)$ in the gluing construction.\\
$\omega$ &
The approximate solution built from $\Ical$, copies of $U$, and heteroclinic
profiles; see \eqref{eq: section-valued-global-ansatz}.\\
$A\lesssim B$ &
There is a constant $C>0$, independent of $\eps$, such that $A\leq CB$. The
constant may depend on the fixed geometric data and auxiliary parameters, once
these have been chosen.\\
\end{longtable}
\endgroup

\section{Line bundles}\label{sec: preliminaries}

This section fixes the basic language used throughout the paper. We first
construct the sign line bundle associated with a finite set of punctures and
explain how the Allen--Cahn energy and equation are interpreted for its
sections. We then record the square-root description near a single puncture,
which is the local model used for the boundary profile and for the section
norms. Finally, we introduce the relevant
norms used in the gluing construction.

Let $\Omega\subset\R^2$ be a simply connected open set and let
\[
    \p=\{p_1,\dots,p_n\}\subset\Omega
\]
be a finite set of points, with $n\geq1$. We write
\[
    \Omega_{\p}\coloneqq \Omega\setminus\p .
\]

The construction of the line bundle follows naturally from a double cover of
$\Omega_{\p}$, so we begin by describing this cover and its uniqueness.

\subsection{The double cover}

The goal is to construct the unique connected double cover whose monodromy is
nontrivial around each puncture.

For each $i$, let $\gamma_i \subset \Omega$ be a positively oriented simple loop around
$p_i$ which encloses no other point of $\p$. Let $D_i$ be pairwise disjoint
disks centered at the points $p_i$, chosen so that
$\overline{D_i}\subset\Omega$, and set
$A\coloneqq \bigcup_{i=1}^n D_i$ and $B\coloneqq \Omega_{\p}$.
Then $\Omega=A\cup B$, while $A$ is a disjoint union of disks and
$A\cap B$ is a disjoint union of punctured disks. In particular,
$H_1(A;\Z)=0$ and
$H_1(A\cap B;\Z)\simeq \bigoplus_{i=1}^n \Z[\gamma_i]$.
The Mayer--Vietoris sequence \cite[p.~149]{Hatcher2002}, together with
$H_2(\Omega;\Z)=H_1(\Omega;\Z)=0$ for our simply connected planar domain, gives
the exact sequence
\[
    0
    \longrightarrow H_1(A\cap B;\Z)
    \longrightarrow H_1(B;\Z)
    \longrightarrow 0.
\]
Thus
\[
    H_1(\Omega_{\p};\Z)\simeq \bigoplus_{i=1}^n \Z[\gamma_i].
\]

A connected double cover of $\Omega_{\p}$ is determined by an index-two
subgroup of $\pi_1(\Omega_{\p})$; see the classification of covering spaces in
\cite[Theorem~1.38]{Hatcher2002}. Since every index-two subgroup is normal,
\cite[Proposition~1.39]{Hatcher2002} shows that this is equivalently the
kernel of a surjective homomorphism
\[
    \rho\colon \pi_1(\Omega_{\p})\longrightarrow \Z_2.
\]
Because the target is abelian, $\rho$ factors through the abelianization
$H_1(\Omega_{\p};\Z)$. We want the cover to see every puncture, in the sense
that going once around any $p_i$ exchanges the two sheets. Thus we impose
\[
    \rho([\gamma_i])=1\in\Z_2,
    \qquad i=1,\dots,n.
\]
Since the classes $[\gamma_i]$ generate $H_1(\Omega_{\p};\Z)$, this condition
determines $\rho$ uniquely. We denote the corresponding double cover by
\[
    \pi\colon \widetilde\Omega_{\p}\longrightarrow \Omega_{\p}
\]
and write $\iota$ for its deck involution.

The construction above determines $\widetilde\Omega_{\p}$ as a topological
double cover, uniquely up to isomorphism over $\Omega_{\p}$. Whenever a metric
on $\widetilde\Omega_{\p}$ is needed, unless otherwise stated, we use the
pullback of the Euclidean metric on $\Omega_{\p}$,
\[
    \widetilde g_{\p}\coloneqq \pi^*(dx^2).
\]
With this choice, $\pi$ is a local isometry: around every point of
$\widetilde\Omega_{\p}$ there is a neighborhood on which $\pi$ is an isometry
onto its image. This metric is a convention, not part of the topological
classification of the cover. Near the punctures it is often useful to use
different local metrics or coordinates, as in the norm definitions we
introduce later on.

\subsection{The line bundle}\label{subs: line bundle}

\begin{figure}
    \centering
    \begin{overpic}[width=.5\textwidth]{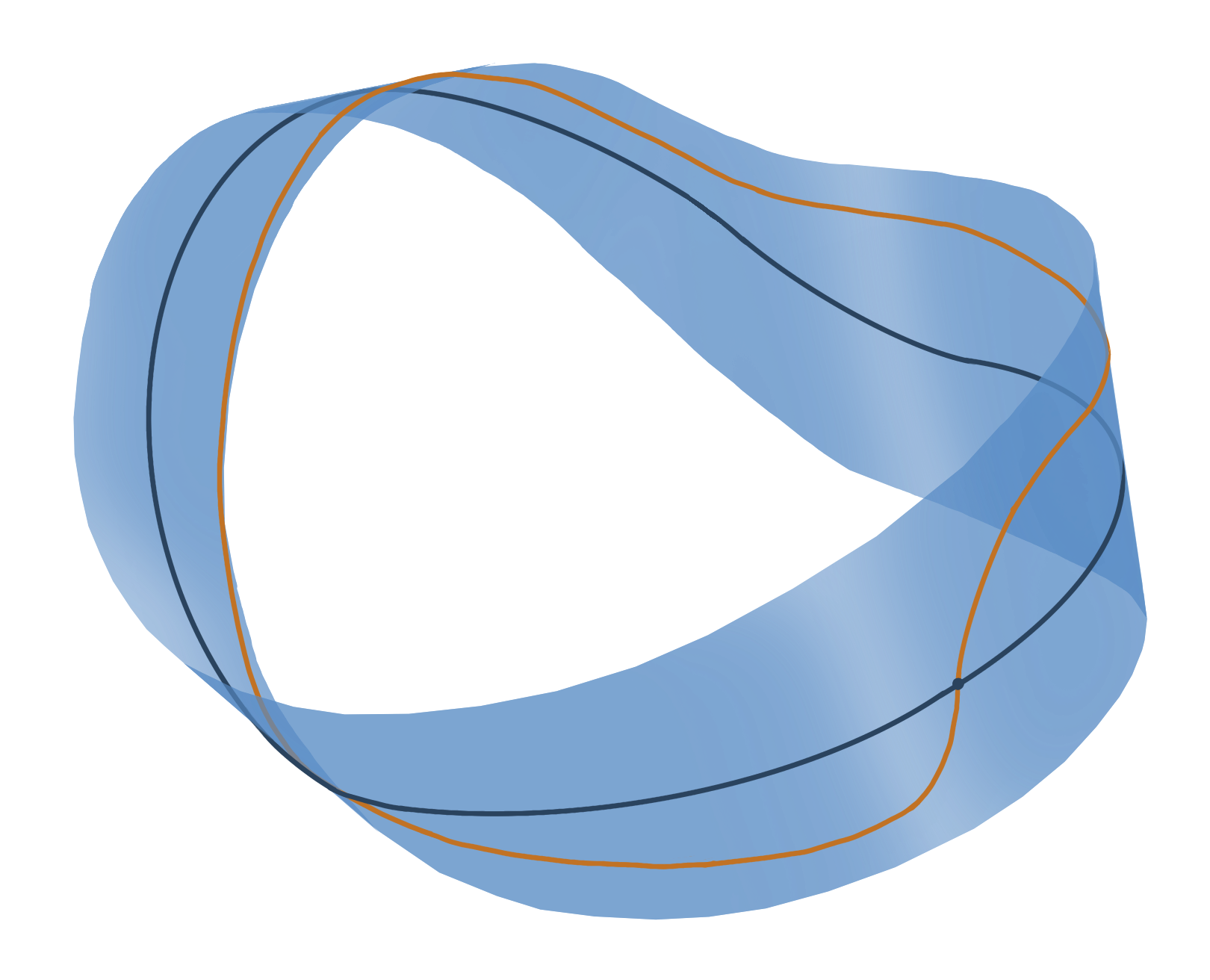}
    \end{overpic}
    \caption{The M\"obius bundle around a circle, with a continuous section crossing zero only once. The topology of the bundle allows for the nodal set of the section to be a single point, which is not separating.}
    \label{fig: mobius}
\end{figure}

The double cover $\widetilde\Omega_{\p}$ gives a natural way to build the line
bundle over $\Omega_{\p}$. We start upstairs from the trivial line bundle
$\widetilde\Omega_{\p}\times\R\to\widetilde\Omega_{\p}$. A section of this
trivial bundle is just the graph of a function
$\widetilde u\colon\widetilde\Omega_{\p}\to\R$. We now quotient this trivial
bundle to obtain a line bundle over $\Omega_{\p}$ in such a way that the graphs
which become sections downstairs are exactly the graphs of odd functions on the
cover.

The definition is as follows. We identify the two points lying over the same
point of $\Omega_{\p}$ with opposite signs in the real factor:
\[
    \L_{\p}\coloneqq
    (\widetilde\Omega_{\p}\times\R)/\sim,
    \qquad
    (\widetilde x,t)\sim(\iota(\widetilde x),-t).
\]
The projection is induced by $(\widetilde x,t)\mapsto\pi(\widetilde x)$.
Thus the fibre over $x\in\Omega_{\p}$ is a one-dimensional real vector space.
Going once around any puncture exchanges the two sheets of the cover and hence
changes the sign in this real factor.
One way to picture the twist is to restrict the bundle to a small circle
centered at one of the punctures. Over such a circle, the bundle is a M\"obius
line bundle, see \Cref{fig: mobius}: after one turn around the puncture, the two local scalar
representatives differ by a sign.

Thus a section of $\L_{\p}$ can be described by an odd function upstairs. More
precisely, if
$\widetilde u\colon\widetilde\Omega_{\p}\to\R$ satisfies
\[
    \widetilde u(\iota\widetilde x)=-\widetilde u(\widetilde x),
\]
then it defines a section of $\L_{\p}$ by
\[
    u(\pi(\widetilde x))=[\widetilde x,\widetilde u(\widetilde x)],
\]
because replacing $\widetilde x$ by $\iota\widetilde x$ changes both the sheet
and the sign. Conversely, pulling back a section to the double cover gives a
function with this oddness property. Hence sections of $\L_{\p}$ are
equivalently odd real-valued functions on $\widetilde\Omega_{\p}$.

If $V\subset\Omega_{\p}$ is open, we write
\[
    \Gamma(V,\L_{\p})
\]
for the space of smooth sections of $\L_{\p}$ over $V$. When a different
regularity or support condition is needed, it will be indicated explicitly.
In the one-puncture case $\Omega=\R^2$ and $\p=\{(0,0)\}$, we write
\[
    \R^2_*\coloneqq\R^2\setminus\{(0,0)\},
    \qquad
    \L\coloneqq\L_{\p}=\L_{\{(0,0)\}},
\]
and hence write $\Gamma(\R^2_*,\L)$ for the corresponding space of smooth
sections.

\subsection{The Allen--Cahn equation on sections of the line bundle}
\label{subs:allen-cahn-on-sections}

The previous construction applies over any simply connected planar domain with
punctures. For the analytic problems considered below, we will mostly work on
punctured planes, and we now specialize the notation accordingly.

Let $\eps>0$ and let $W\colon \R\to \R^+$ be a double-well potential satisfying
$W(\pm1)=0$ and $W>0$ elsewhere. We assume that $W$ is smooth, even, and has
nondegenerate minima at $\pm1$. We write
\[
    \kappa_W\coloneqq W''(1)=W''(-1)>0
\]
for the pure-phase linearized constant. We also assume
\begin{equation}\label{eq: assumption W}
     W'(s)<0 \quad \text{for }0<s<1,
    \qquad
    \frac{d}{ds}\left(\frac{-W'(a+s)}{s}\right)< 0
    \quad \forall a\in[0,1]\text{ and }s>0.
\end{equation}

The second monotonicity condition is used below when applying the uniqueness
result of Brezis--Oswald. The standard example is
$W(u)=\frac14(1-u^2)^2$.

We henceforth let $H\colon\R\to(-1,1)$ be the unique increasing solution of
$H''-W'(H)=0$ satisfying $H(0)=0$ and
$H(t)\to\pm1$ as $t\to\pm\infty$.

For a collection of points $\p\subset\R^2$, write
\[
    \R^2_{\p}\coloneqq\R^2\setminus\p .
\]
For $U\in\Gamma(\R^2_{\p},\L_{\p})$, the Allen--Cahn energy is
\begin{equation}\label{eq:section-allen-cahn-energy}
    E_\eps(U)
    \coloneqq
    \int_{\R^2_{\p}}
    \eps\frac{|\nabla U|^2}{2}
    +\frac{W(U)}{\eps}.
\end{equation}
This expression is well defined on sections. Indeed, after choosing a local
trivialization of $\L_{\p}$, the section is represented by a scalar function
$u$. On overlaps two such representatives differ by a sign. Since the transition
functions are constant signs, $|\nabla u|^2$ is unchanged by changing
trivialization; since $W$ is even, $W(u)$ is unchanged as well. Thus the
integrand in \eqref{eq:section-allen-cahn-energy} is independent of all local
choices.

We say that $U\in\Gamma(\R^2_{\p},\L_{\p})$ solves the Allen--Cahn equation,
and write
\begin{equation}\label{eq:section-allen-cahn}
    \eps\Delta U=\frac{1}{\eps}W'(U),
\end{equation}
if the same identity holds in any local trivialization for the scalar
representative $u$ of $U$.
Equivalently, for every relatively compact open set
$\Omega'\Subset\R^2_{\p}$, the restriction of $U$ is a critical point of
the energy \eqref{eq:section-allen-cahn-energy} restricted to $\Omega'$, with
respect to compactly supported section variations in $\Omega'$. Thus differential
expressions involving a section are always understood through local
trivializations. The convention is
unambiguous: if one representative is replaced by $-u$, then both $\Delta u$ and
$W'(u)$ change sign, because $W'$ is odd. Hence the two sides of
\eqref{eq:section-allen-cahn}
transform as the same section. This is the same as writing the equation using
the flat connection induced by the double cover, but in the present setting the
local formulation is enough.

\subsection{The one-puncture square-root model}
\label{subs: plane-minus-one-point}

In this subsection we describe how the complex square map provides a concrete local model for the line bundle near isolated punctures. From this point of view, the metric of the associated double cover is different than the pull-back of the euclidean metric introduced above. In fact, both metrics are related via a conformal change which will be central in the study of the boundary model solution.

We first describe it globally in the one-puncture case
$\Omega=\R^2$ and $\p=\{(0,0)\}$. With the notation $\R^2_*$ and $\L$ fixed
above, the double cover is represented by
\[
    \pi_{\mathrm{sq}}\colon \R^2_*\longrightarrow\R^2_*,
    \qquad
    \pi_{\mathrm{sq}}(\xi)=\xi^2,
\]
where we identify $\R^2$ with $\C$ in the usual way. A section
$U\in\Gamma(\R^2_*,\L)$ is represented on this cover by an odd function $\U$
through
\[
    U(\pi_{\mathrm{sq}}(\xi))=[\xi,\U(\xi)],
    \qquad
    \U(-\xi)=-\U(\xi).
\]

We next record the corresponding local statement near an arbitrary puncture.
Let $\p\subset\Omega$ be a finite set and let
$p_i\in\p$. Choose $r>0$ so that $\overline{B_r(p_i)}\subset\Omega$ and
$B_r(p_i)\cap\p=\{p_i\}$. We also choose an orthonormal basis
$(e,\nu)$ of $\R^2$, with no orientation condition, and use it to write
coordinates near $p_i$ by
\[
    Y_{i,e,\nu}\colon B_r(0)\longrightarrow B_r(p_i),
    \qquad
    Y_{i,e,\nu}(x,z)=p_i+xe+z\nu .
\]
The associated square-root parametrization of the punctured disk is
\[
    D_i^*=B_r(p_i)\setminus\{p_i\},
    \qquad
    \sigma_{i,e,\nu}\colon B_{\sqrt r}(0)\setminus\{0\}\longrightarrow D_i^*,
    \qquad
    \sigma_{i,e,\nu}(\xi)
    =
    Y_{i,e,\nu}\bigl(\pi_{\mathrm{sq}}(\xi)\bigr).
\]
Equivalently, if $\xi=(\xi_1,\xi_2)$, then
\[
    \sigma_{i,e,\nu}(\xi)
    =
    p_i+(\xi_1^2-\xi_2^2)e+2\xi_1\xi_2\nu .
\]
Changing the orientation of the basis reverses the orientation of the base
coordinates, but does not change the fact that a loop around the origin
corresponds to a loop around $p_i$ which exchanges the two sheets.

The restriction
\[
    \pi\colon \pi^{-1}(D_i^*)\longrightarrow D_i^*
\]
of the global double cover is isomorphic, as a double cover of $D_i^*$, to
$\sigma_{i,e,\nu}$. More precisely, there exists a diffeomorphism
\[
    \Phi_{i,e,\nu}\colon B_{\sqrt r}(0)\setminus\{0\}
    \longrightarrow \pi^{-1}(D_i^*)
\]
such that
\[
    \pi\circ\Phi_{i,e,\nu}=\sigma_{i,e,\nu}.
\]
Indeed, both covers are connected double covers of the punctured disk for which
a loop going once around $p_i$ exchanges the two sheets, and this property
characterizes the cover up to isomorphism. However, $\Phi_{i,e,\nu}$ is not
canonical: there are two choices which differ only by the deck involution.
Consequently, the induced identification between
$\L|_{B_r(0)\setminus\{0\}}$ and $\L_{\p}|_{D_i^*}$ is also determined only up
to multiplication by $-1$ on the fibres.

We also record how this identification interacts with the corresponding metrics.
If $\pi^{-1}(D_i^*)$ is equipped with the pullback Euclidean metric
$\pi^*|dx|^2$ and $B_{\sqrt r}(0)\setminus\{0\}$ is equipped with
$\sigma_{i,e,\nu}^*|dx|^2$, then $\Phi_{i,e,\nu}$ is an isometry. However, we
will most often use the Euclidean metric of the square-root coordinate $\xi$, in
which case
\[
    \sigma_{i,e,\nu}^*|dx|^2=4|\xi|^2|d\xi|^2
\]
i.e., the metrics differ by a conformal factor. In particular,  the Allen--Cahn
equation is written differently on the square-root cover variable $\xi$. More precisely, on any open set where a sheet of the cover has been chosen, if $u$ is the
corresponding scalar representative of $U$, then $\U$ is, up to a fixed sign,
the pullback of $u$. Thus the conformal change of variables gives
\[
    \Delta_\xi \U(\xi)
    =
    4|\xi|^2(\Delta_x u)(\pi_{\mathrm{sq}}(\xi)).
\]
Thus, if $U$ solves \eqref{eq:section-allen-cahn} with $\eps=1$, then its lift
$\U$ satisfies
\begin{equation*}
    \Delta\U=4|\xi|^2W'(\U)\quad \text{in }\R^2_*.
\end{equation*}
This is the analytic reason for passing to the square-root cover. The monodromy
of the bundle is absorbed by the oddness of the lift, while the branch point
appears in the scalar equation through the smooth vanishing factor $4|\xi|^2$.
In the bounded finite-energy setting considered by Fr\"ohlich and Struwe
\cite{Frohlich-Struwe90}, this is the mechanism by which the lifted solution has
a removable singularity after setting its value to be zero at the origin.
Under the same change of variables, for $\eps=1$,
\[
    E(U)=\frac12\widetilde E(\U),
    \qquad
    \widetilde E(\U)
    \coloneqq
    \int_{\R^2_*}
    \frac{|\nabla \U|^2}{2}+4|\xi|^2W(\U).
\]
The factor $\frac12$ comes from the fact that $\pi$ covers the punctured plane
twice. The positive factor relating $E$ and $\widetilde E$ is irrelevant for
criticality and stability, so we use $\widetilde E$ on the space of odd
functions when working on the cover.

\subsection{Local norms for sections}

In this subsection we define local norms for sections on two types of balls:
trivializing balls and singular balls. These norms are the building blocks for
several specialised weighted norms used throughout the paper. They depend on a
fixed parameter $R_0>1$ and on the scaling parameter $\eps>0$, which determines
both the relevant physical radii and the scaling of derivatives.

Roughly speaking, trivializing balls are positioned away from the punctures. On
such balls, after choosing one sheet of the double cover, sections are viewed as
scalar functions. Singular balls are centered at the punctures; on these balls,
we use local coordinates induced by the square-root map around the puncture to
define the norm.

From now on, we assume that $\eps>0$ is small enough so that, when $n\geq2$,
\[
    R_0\eps<\frac1{20}\min_{i\ne j}|p_i-p_j|,
\]
and so that
\[
    \overline{B_{R_0\eps}(p_i)}\subset \Omega
    \qquad\text{for every }i.
\]
The first condition is omitted when there is only one puncture.

\begin{definition}[Trivializing balls]
\label{def:trivializing-ball}
A ball $B_\eps(q)$ is called trivializing for $\L_{\p}$ over $\Omega_{\p}$ if
\[
    \overline{B_\eps(q)}\subset\Omega_{\p},
    \qquad
    \dist(q,\p)\geq R_0\eps .
\]
We denote by $\mathcal B_\eps$ the collection of all trivializing balls.
\end{definition}

\begin{definition}[Local norm on a trivializing ball]
\label{def:trivializing-ball-norm}
Let $B_\eps(q)\in\mathcal B_\eps$. The inverse image
$\pi^{-1}(B_\eps(q))$ has two components. Choosing one of them identifies
$\L_{\p}|_{B_\eps(q)}$ with the product bundle $B_\eps(q)\times\R$ and
represents a section $u$ by a scalar function
$u_{q,\eps}\colon B_\eps(q)\to\R$. We rescale this representative at scale
$\eps$ by
\[
    u_{q,\eps}^\sharp(y)\coloneqq u_{q,\eps}(q+\eps y),
    \qquad y\in B_1(0).
\]
For $k\geq0$ and $0<\gamma<1$, we define the $\eps$-scaled local norm by
\[
    \lVert u\rVert_{C_\eps^{k,\gamma}(B_\eps(q),\L_{\p})}
    \coloneqq
    \lVert u_{q,\eps}^\sharp\rVert_{C^{k,\gamma}(B_1(0))}.
\]
Choosing the other component of $\pi^{-1}(B_\eps(q))$ replaces $u_{q,\eps}$ by
$-u_{q,\eps}$, so the norm is independent of this choice.
\end{definition}

\begin{definition}[Singular balls]
\label{def:singular-ball}
For each puncture $p_i$, the corresponding singular ball is
\[
    B_{R_0\eps}(p_i).
\]
\end{definition}

\begin{definition}[Local norm on a singular ball]
\label{def:singular-ball-norm}
On the singular ball $B_{R_0\eps}(p_i)$ we work over
$B_{R_0\eps}(p_i)\setminus\{p_i\}$. We identify $B_{R_0\eps}(p_i)$ with
$B_{R_0\eps}(0)\subset\C$ by subtracting $p_i$ and then use the local double
cover
\[
    B_{\sqrt{R_0\eps}}(0)\setminus\{0\}
    \longrightarrow
    B_{R_0\eps}(p_i)\setminus\{p_i\},
    \qquad
    \zeta\longmapsto p_i+\zeta^2 .
\]
If $u$ is a section of $\L_{\p}$ over
$B_{R_0\eps}(p_i)\setminus\{p_i\}$, this chart represents it by an odd function
\[
    \widetilde u_{i,\eps}\colon
    B_{\sqrt{R_0\eps}}(0)\setminus\{0\}\longrightarrow\R,
    \qquad
    \widetilde u_{i,\eps}(-\zeta)=-\widetilde u_{i,\eps}(\zeta),
\]
unique up to an overall sign. Define its zero extension by
\[
    \widehat u_{i,\eps}(\zeta)
    =
    \begin{cases}
    \widetilde u_{i,\eps}(\zeta), & \zeta\ne0,\\
    0, & \zeta=0.
    \end{cases}
\]
Rescaling the upstairs disk at the square-root scale $\sqrt\eps$ gives
\[
    \widehat u_{i,\eps}^\sharp(\eta)
    \coloneqq
    \widehat u_{i,\eps}(\sqrt\eps\eta),
    \qquad \eta\in B_{\sqrt{R_0}}(0).
\]
Equivalently, one may first rescale the downstairs ball by writing
$x=p_i+\eps y$, and then take the square-root cover $y=\eta^2$. The resulting
map to the original ball is again
\[
    \eta\longmapsto p_i+\eps\eta^2.
\]
Thus taking the square-root cover first and then rescaling by $\sqrt\eps$ is
the same as first rescaling downstairs by $\eps$ and then passing to the
square-root cover.
For $k\geq0$ and $0<\gamma<1$, we define the $\eps$-scaled local norm by
\[
    \lVert u\rVert_{C_\eps^{k,\gamma}(B_{R_0\eps}(p_i),\L_{\p})}
    \coloneqq
    \lVert\widehat u_{i,\eps}^\sharp\rVert_{C^{k,\gamma}(B_{\sqrt{R_0}}(0))}.
\]
If the zero extension is not of class $C^{k,\gamma}$, this norm is infinite.
\end{definition}

\begin{remark}
The singular norm is measured in the square-root scale near the puncture. Thus
it is not the ordinary Euclidean H\"older norm in the original $x$-coordinate on
$B_{R_0\eps}(p_i)$, but the ordinary H\"older norm of the zero extension in the
coordinate $\eta$ with $x=p_i+\eps\eta^2$.
\end{remark}

\begin{remark}
The following is only a compatibility observation. There is a small annular
region around each puncture where both descriptions can be used. In such a
region the square-root chart stays a fixed scaled distance away from the branch
point, and therefore the singular and trivializing norms measure equivalent
scaled quantities. More precisely, if
\[
    B_{r\eps}(q)\subset B_{R_0\eps}(p_i)\setminus B_{c\eps}(p_i)
\]
for some fixed \(c>1\) and \(r>0\), then, after the downstairs rescaling
\(x=p_i+\eps y\), the map \(y=\eta^2\) restricts on each sheet over
\(B_r((q-p_i)/\eps)\) to a smooth change of coordinates with \(C^{k+1}\) bounds
for the map and its inverse depending only on \(c,R_0,r,k\), and not on
\(\eps\).
Consequently, on such a ball the ordinary scaled \(C^{k,\gamma}\) norm of a
local scalar representative is equivalent to the norm obtained by restricting
the singular lift to the corresponding sheet, with constants independent of
\(\eps\).
\end{remark}

\subsection{The global scaled section norm}

During the gluing argument we measure sections at the scale $\eps>0$ with
respect to several norms. The value of $R_0$ is part of the choice of norm, and
all constants in the estimates below are allowed to depend on it. We define the
global section norm directly from the $\eps$-scaled local norms on the singular
balls $B_{R_0\eps}(p_i)$ and on the balls in $\mathcal B_\eps$.

\begin{definition}[The $C_{*,\eps}^{k,\gamma}$ section norm]
\label{def:scaled-c-star-holder-norm}
Let $k$ be a nonnegative integer and let $0<\gamma<1$. For a section $u$ of
$\L_{\p}$ over $\Omega_{\p}$, define
\[
    \lVert u\rVert_{C_{*,\eps}^{k,\gamma}(\Omega_{\p},\L_{\p})}
    \coloneqq
    \sup_{i=1,\dots,n}
    \lVert u\rVert_{C_\eps^{k,\gamma}(B_{R_0\eps}(p_i),\L_{\p})}
    +
    \sup_{B_\eps(q)\in\mathcal B_\eps}
    \lVert u\rVert_{C_\eps^{k,\gamma}(B_\eps(q),\L_{\p})}.
\]
\end{definition}

In this definition, the first supremum is taken over the singular balls
centered at the punctures, while the second is taken over trivializing balls
whose centers are at distance at least $R_0\eps$ from the punctures.

\subsection{Weighted section norms}\label{subs: weighted-section-norms}

The exponential decay norms below depend on the parameters
\begin{equation}\label{eq: exponential-decay-params}
    \lambda>0,
    \qquad
    \gamma\in(0,1).
\end{equation}
We display $\lambda$ in the notation for the weighted norms.

\begin{definition}[Exponential decay norm relative to a set]\label{def: exp decay rel set}
Let $A\subset\R^2$ be a nonempty set, let $D\subset\R^2_{\p}$, and let
$\varphi$ be a section of $\L_{\p}$ over an open neighborhood of $D$ in
$\R^2_{\p}$.
Set
\[
    \mathcal S_\eps(D)
    \coloneqq
    \left\{i:\overline{B_{R_0\eps}(p_i)}\setminus\{p_i\}
    \subset D\right\},
    \qquad
    \mathcal B_\eps(D)
    \coloneqq
    \left\{B_\eps(q)\in\mathcal B_\eps:
    \overline{B_\eps(q)}\subset D\right\}.
\]
We define
\begin{multline*}
    \lVert\varphi\rVert_{C^{k,\gamma}_{A,\eps,\lambda}(D,\L_{\p})}
    \coloneqq
    \sup_{i\in\mathcal S_\eps(D)}
    e^{\lambda\dist(p_i,A)/\eps}
    \lVert\varphi\rVert_{C_\eps^{k,\gamma}(B_{R_0\eps}(p_i),\L_{\p})}\\
    {}+
    \sup_{B_\eps(q)\in\mathcal B_\eps(D)}
    e^{\lambda\dist(q,A)/\eps}
    \lVert\varphi\rVert_{C_\eps^{k,\gamma}(B_\eps(q),\L_{\p})}.
\end{multline*}
An empty supremum is interpreted as zero.
Equivalently, $C^{k,\gamma}_{A,\eps,\lambda}(D,\L_{\p})$ is the space of
sections defined near $D$ for which this quantity is finite. We will use this
norm mainly when \(A\) is a line segment or a half-line based at a puncture. In
the one-puncture case $\p=\{(0,0)\}$, we write $\L$ in place of $\L_{\p}$, as
before.
When $\eps=1$, we suppress it from the notation and write
\[
    C^{k,\gamma}_{A,\lambda}(D,\L_{\p})
    \coloneqq
    C^{k,\gamma}_{A,1,\lambda}(D,\L_{\p}).
\]
\end{definition}

With the same indexing sets, deleting the exponential factors defines
\(\lVert\varphi\rVert_{C_{*,\eps}^{k,\gamma}(D,\L_{\p})}\). For
\(D=\Omega_{\p}\), this agrees with
\Cref{def:scaled-c-star-holder-norm}. For a general \(D\), both quantities are
interior seminorms on the region covered by the singular and trivializing balls
contained in \(D\).

\section{The model boundary solution}\label{sec: 2d sol}

Recall the one-puncture notation introduced in
\Cref{subs: line bundle,subs: plane-minus-one-point}: we write
\[
    \R^2_*=\R^2\setminus\{(0,0)\},
    \qquad
    \L=\L_{\{(0,0)\}}\to\X .
\]
In this section we write points of $\X$ as $(x,z)$; the second coordinate is
denoted $z$ in preparation for the notation we will use for higher-dimensional
interfaces in future work. We use the global square-root cover
\[
    \pi\colon\X\longrightarrow\X,
    \qquad
    \pi(\xi)=\xi^2,
\]
and identify sections of $\L$ with odd scalar functions upstairs.

Let \(\ell_+\coloneqq\{(x,0)\in\X:x>0\}\) be the positive half-line.
The goal of this section is to construct a model section $U$ whose zero set is
$\ell_+$. Since $\ell_+$ is non-separating in $\X$, this model is naturally a
section of $\L$, or equivalently an odd function on the double cover.

\subsection{Construction of the model solution U}\label{subs: construct U}

We construct $U\colon\X\to\L$ by solving the lifted equation. From now on, we
denote the coordinates of the cover by $\xi=(\xi_1,\xi_2)$, so that
$\pi(\xi)=(\xi_1^2-\xi_2^2,2\xi_1\xi_2)$. The nontrivial element of the group
of deck transformations of the cover is $\iota(\xi)=-\xi$, and a section $U$
corresponds
to an odd function $\U\colon\X\to\R$ through the relation
\begin{equation}\label{eq: COV lift}
    U(\pi(\xi))=[\xi,\U(\xi)],
    \qquad
    \U(-\xi)=-\U(\xi).
\end{equation}

With the conventions of \Cref{subs: plane-minus-one-point}, the Allen--Cahn
equation \eqref{eq:section-allen-cahn} with $\eps=1$ becomes
\(\Delta U-W'(U)=0\).
By the square-root computation in \Cref{subs: plane-minus-one-point}, solving
this equation is equivalent, on the cover, to solving
\[
    -\Delta\U+4|\xi|^2W'(\U)=0\quad \text{in }\X .
\]
We use this lifted equation below. The corresponding lifted energy was fixed in
\Cref{subs: plane-minus-one-point}.

\begin{proposition}\label{prop: lifted-existence}
    There exists a smooth function $\U\colon\R^2\to\R$, unique up to sign among
    bounded solutions satisfying the properties below, solving
    \[
        -\Delta\U+4|\xi|^2W'(\U)=0\quad \text{in }\R^2,
    \]
    whose nodal set is $\R\times\{0\}$. Moreover, $\U$ is odd with respect to the
    origin and satisfies $\partial_{\xi_2}\U>0$.
\end{proposition}
\begin{proof}
    We construct $\U$ variationally. Let $R>0$, set $\Omega_R\coloneqq (-R,R)\times(0,R)$, define $\varphi_R(\xi)=\xi_2/R$ and consider the problem
\begin{equation}\label{eq: problem on Omega R}
	\begin{cases}
		-\Delta \U_R+4|\xi|^2W'(\U_R)=0&\text{in }\Omega_R,\\
		\U_R=\varphi_R &\text{on }\partial\Omega_R.
	\end{cases}
\end{equation}
Existence of a solution to \eqref{eq: problem on Omega R} follows by the direct method of calculus of variations, by minimising the coercive energy
\begin{equation*}
    \widetilde E(\U,\Omega_R)\coloneqq
    \int_{\Omega_R}\frac{|\nabla \U|^2}{2}+4|\xi|^2W(\U)
\end{equation*}
in the space $\varphi_R+H^1_0(\Omega_R)$.

We claim that $\U_R\geq 0$ in $\Omega_R$. Indeed, $|\U_R|$ is a minimiser as
well and elliptic regularity in the interior ensures that $\U_R$ cannot change
sign by the maximum principle. Since \(0\leq\U_R\leq1\) on
\(\partial\Omega_R\), the upper bound \(\U_R\leq1\) in \(\Omega_R\) follows
from the maximum principle as well. In particular, this means that $\U_R$ is
superharmonic in $\Omega_R$ since
\begin{equation*}
    \Delta\U_R=4|\xi|^2W'(\U_R)\leq 0
\end{equation*}
and by comparison principle $\U_R\geq \varphi_R$ in $\Omega_R$, since
$\varphi_R$ is harmonic. Our next claim is that
$\partial_{\xi_2}\U_R>0$ in $\Omega_R$. Observe that, by differentiating the
equation, we have
\begin{equation*}
	-\Delta\partial_{\xi_2}\U_R+4|\xi|^2W''(\U_R)\partial_{\xi_2}\U_R=-8\xi_2W'(\U_R)\geq 0
\end{equation*}
which follows from the fact that $W'(s)$ is negative when $s$ is positive, and
$\U_R\geq 0$. Thus, $\partial_{\xi_2}\U_R$ is a supersolution for the operator
$-\Delta+4|\xi|^2W''(\U_R)$. On the vertical sides the boundary condition gives
$\partial_{\xi_2}\U_R=1/R$, while the Hopf lemma applied to $\U_R$ on the bottom
side and to $1-\U_R$ on the top side gives $\partial_{\xi_2}\U_R>0$ there. At
the corners this boundary argument is justified by a local reflection: after
subtracting the affine boundary value $\varphi_R=\xi_2/R$, the function
$\U_R-\varphi_R$ vanishes on the two perpendicular sides meeting at the corner,
and odd reflection across both sides gives an interior equation with bounded
right-hand side. Interior $W^{2,p}$ regularity, for $p>2$, gives a continuous
gradient up to the corner, so the adjacent side traces agree. Hence
$\partial_{\xi_2}\U_R>0$ in a tiny neighbourhood of $\partial\Omega_R$ in
$\Omega_R$. This means that we can find a slightly smaller
domain $\widetilde\Omega_R\subset\Omega_R$ such that
$\partial_{\xi_2}\U_R>0$ on $\partial\widetilde\Omega_R$. Moreover, $\U_R$ is
energy minimising in $\Omega_R$ which implies that it is strictly stable in
$\widetilde\Omega_R$, meaning that there exists a positive solution
$\varphi_1>0$ to
\begin{equation*}
-\Delta\phi_1+4|\xi|^2W''(\U_R)\phi_1=\lambda_1\phi_1
\end{equation*}
for $\lambda_1>0$, i.e.~ $\phi_1$ is a positive supersolution, which in particular implies that the operator satisfies maximum principle in $\widetilde\Omega_R$, see \cite[Theorem 2.11]{HanLin2011}. In particular, $\partial_{\xi_2}\U_R$ cannot attain its nonpositive minimum in the interior, thus $\partial_{\xi_2}\U_R>0$.

By elliptic regularity we argue that the sequence $\{\U_R\}$ converges uniformly over compact subsets of $\{\xi_2\geq 0\}$ to a function $\U_+$ satisfying the same equation and such that $\U_+(\xi_1,0)=0$ for every $\xi_1\in\R$. We claim that $\U_+>0$ on $\{\xi_2>0\}$. To rule out the zero limit, choose a sufficiently large ball $B_*\Subset\{\xi_2>0\}$ such that the zero-Dirichlet problem for $\widetilde E(\cdot,B_*)$ has a positive minimiser $\Phi_*>0$; this is the Brezis--Oswald criterion, applied using \eqref{eq: assumption W}. For every large $R$, the ball $B_*$ lies in $\Omega_R$, and the usual sliding comparison with $\vartheta\Phi_*$, $\vartheta\in(0,1)$, gives $\U_R\geq\Phi_*$ on $B_*$. Passing to the limit gives $\U_+\geq\Phi_*$ on $B_*$, so $\U_+$ is not identically zero. Since $\U_+\ge0$ and solves the equation in the connected half-plane, the strong maximum principle gives $\U_+>0$ on $\{\xi_2>0\}$.

Next, we show that $\U_+$ is even in $\xi_1$. To do so, note that
$\mathbf{V}_R\coloneqq\U_R-\varphi_R$ solves
\begin{equation*}
\begin{cases}
	-\Delta \mathbf{V}_R=f(\xi,\mathbf{V}_R)&\text{in }\Omega_R\\
		\mathbf{V}_R=0 &\text{on }\partial\Omega_R.
\end{cases}
\end{equation*}
where $f(\xi,s)\coloneqq -4|\xi|^2W'(s+\varphi_R(\xi))$, which is bounded in
$\xi$ and continuous in $s$. Assumption \eqref{eq: assumption W}, with
$a=\varphi_R(\xi)$, gives $\tfrac{d}{ds}(f(\xi,s)/s)\leq 0$ for all $s>0$.
By \cite[Section 2]{BrezisOswald1986}, we have that $\mathbf{V}_R$
(and thus $\U_R$) is unique. By uniqueness and the fact that the boundary datum
is invariant under $\xi_1$ reflection we deduce that $\U_R$ is even in $\xi_1$,
and so is $\U_+$.

Next, define $\U$ as the $\xi_2$-odd extension of $\U_+$ to the entire plane. It holds
\begin{equation*}
	-\U(\xi_1,\xi_2)=\U(\xi_1,-\xi_2)=\U(-\xi_1,-\xi_2),
\end{equation*}
namely, $\U$ is odd with respect to the origin. Moreover,
$\partial_{\xi_2}\U\geq0$ everywhere, and since $\partial_{\xi_2}\U$ is a non-zero supersolution for the operator $-\Delta+4|\xi|^2W''(\U)$ we have $\partial_{\xi_2}\U>0$ by maximum principle.

Lastly, uniqueness follows from a modification of \cite[Section 2]{BrezisOswald1986}, which we sketch now. Suppose that there are two bounded functions $\U_1$ and $\U_2$ satisfying the thesis and assume up to changing sign that they are both nonnegative in $\{\xi_2>0\}$. Pick a cut-off function $\eta\in C^\infty_c(\{\xi_2>0\})$ to be chosen later and test the equations for $\U_1$ and $\U_2$ with
\begin{equation*}
    \phi_1\coloneqq \eta^2\frac{\U_1^2-\U_2^2}{\U_1},\quad \text{and}\quad \phi_2\coloneqq \eta^2\frac{\U_2^2-\U_1^2}{\U_2}
\end{equation*}
respectively. Adding the two test yields, setting $f(\xi,s)=-4|\xi|^2W'(s)$
\begin{align*}
\int_{\{\xi_2>0\}}
\eta^2(\U_1^2+\U_2^2)
\left|\nabla \log \frac{\U_1}{\U_2}\right|^2
+2\int_{\{\xi_2>0\}}
\eta(\U_1^2-\U_2^2)
\nabla\eta\cdot \nabla \log \frac{\U_1}{\U_2}\\
=\int_{\{\xi_2>0\}}
\eta^2
\left(
\frac{f(\zeta,\U_1)}{\U_1}
-
\frac{f(\zeta,\U_2)}{\U_2}
\right)
(\U_1^2-\U_2^2).
\end{align*}
By \eqref{eq: assumption W}, the right-hand side is nonpositive. Using Young's inequality, we get
\begin{equation*}
    \int_{\{\xi_2>0\}}
\eta^2(\U_1^2+\U_2^2)
\left|\nabla \log \frac{\U_1}{\U_2}\right|^2\leq C\int_{\{\xi_2>0\}}
(\U_1^2+\U_2^2)
\left|\nabla \eta\right|^2.
\end{equation*}
Picking $\eta$ to be a log-cutoff and sending its support to $+\infty$ we conclude that
\begin{equation*}
    \nabla \log \frac{\U_1}{\U_2}\equiv 0\iff \U_1=c\U_2.
\end{equation*}
and using the equation we find $c=1$.
\end{proof}

\begin{corollary}\label{prop: existence}
    The function $\U$ constructed in \Cref{prop: lifted-existence} descends to a
    section $U\colon\R^2_*\to\L$ satisfying
    \eqref{eq:section-allen-cahn} with $\eps=1$. Its nodal set is the
    positive half-line $\ell_+$, and its lift satisfies
    $\partial_{\xi_2}\U>0$.
\end{corollary}

\begin{proof}
    Since $\U$ is odd with respect to the origin, \eqref{eq: COV lift} defines a
    section $U$ of $\L$. The equivalence between the Allen--Cahn equation with
    $\eps=1$ and the lifted equation was recalled above. Finally, the projection of
    $\R\times\{0\}\setminus\{0\}$ under $\pi(\xi)=\xi^2$ is the positive
    half-line $\ell_+$.
\end{proof}

\subsection{Trivialisations of the model solution}\label{subs: trivialisations-model-solution}

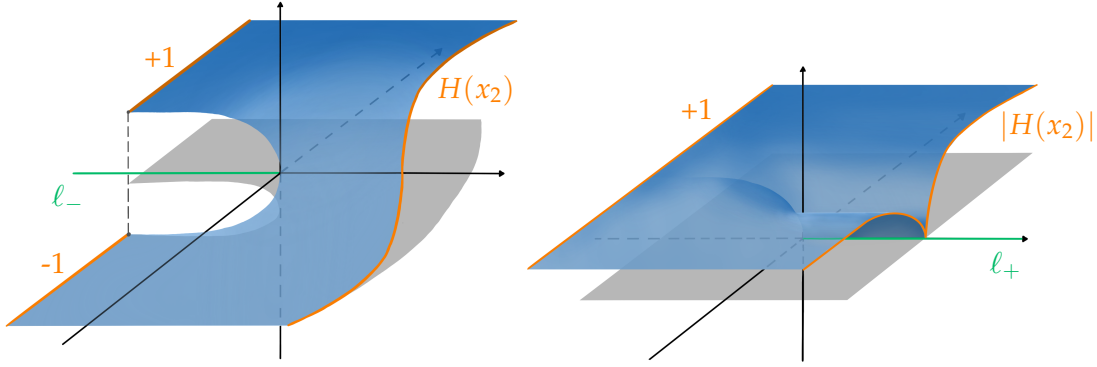
\begin{figure}
    \centering
    \begin{minipage}[t]{0.45\textwidth}
        \centering
        \begin{overpic}[width=1\textwidth]{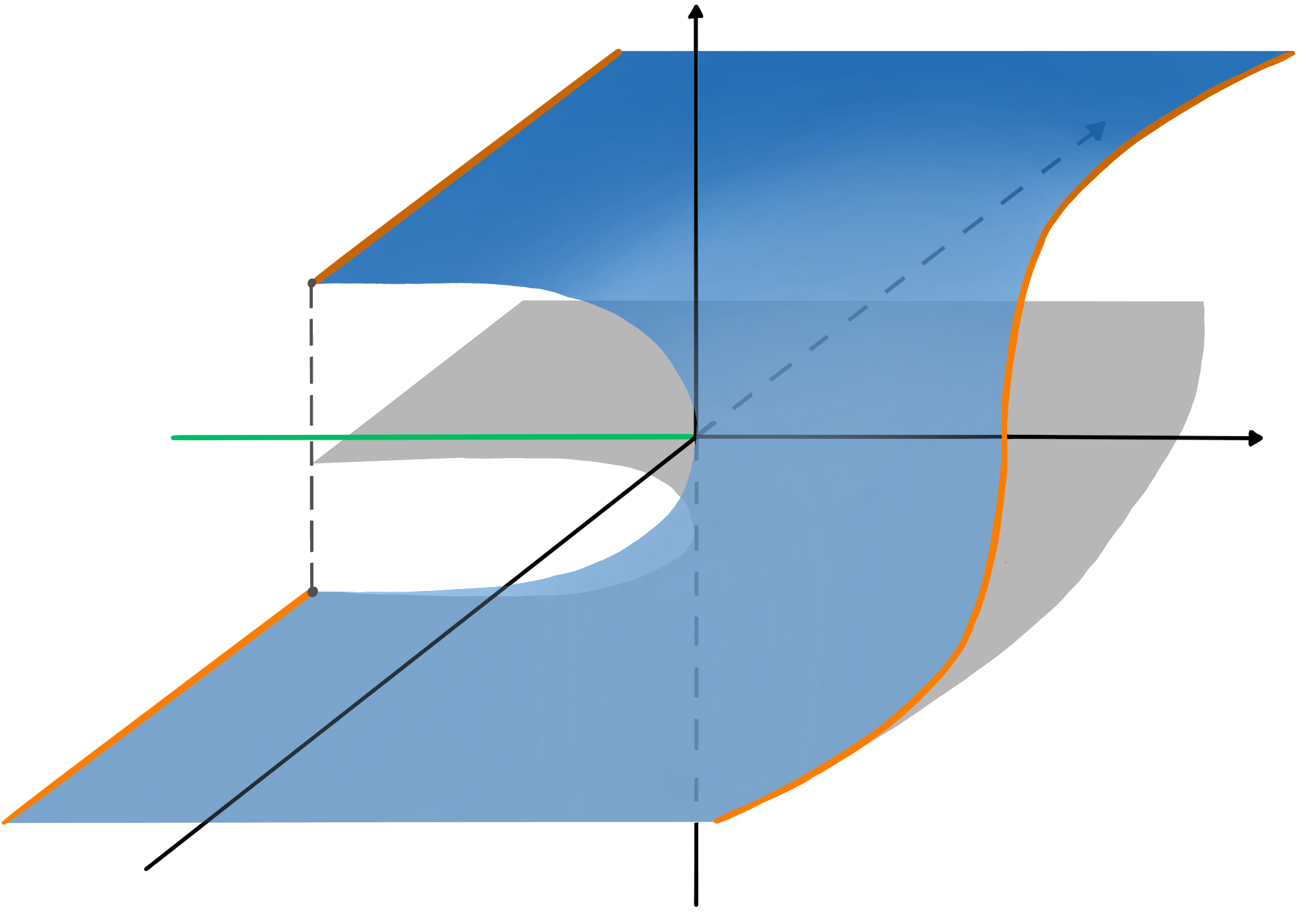}
        \put(27,58){\color{orange} +1}
        \put(7,18){\color{orange} -1}
        \put(9,31){\color{forestgreen} $\ell_-$}
        \put(84,52){\color{orange} $H(x_2)$}
    \end{overpic}
    \end{minipage}
    \begin{minipage}[t]{0.45\textwidth}
        \centering
        \begin{overpic}[width=1\textwidth]{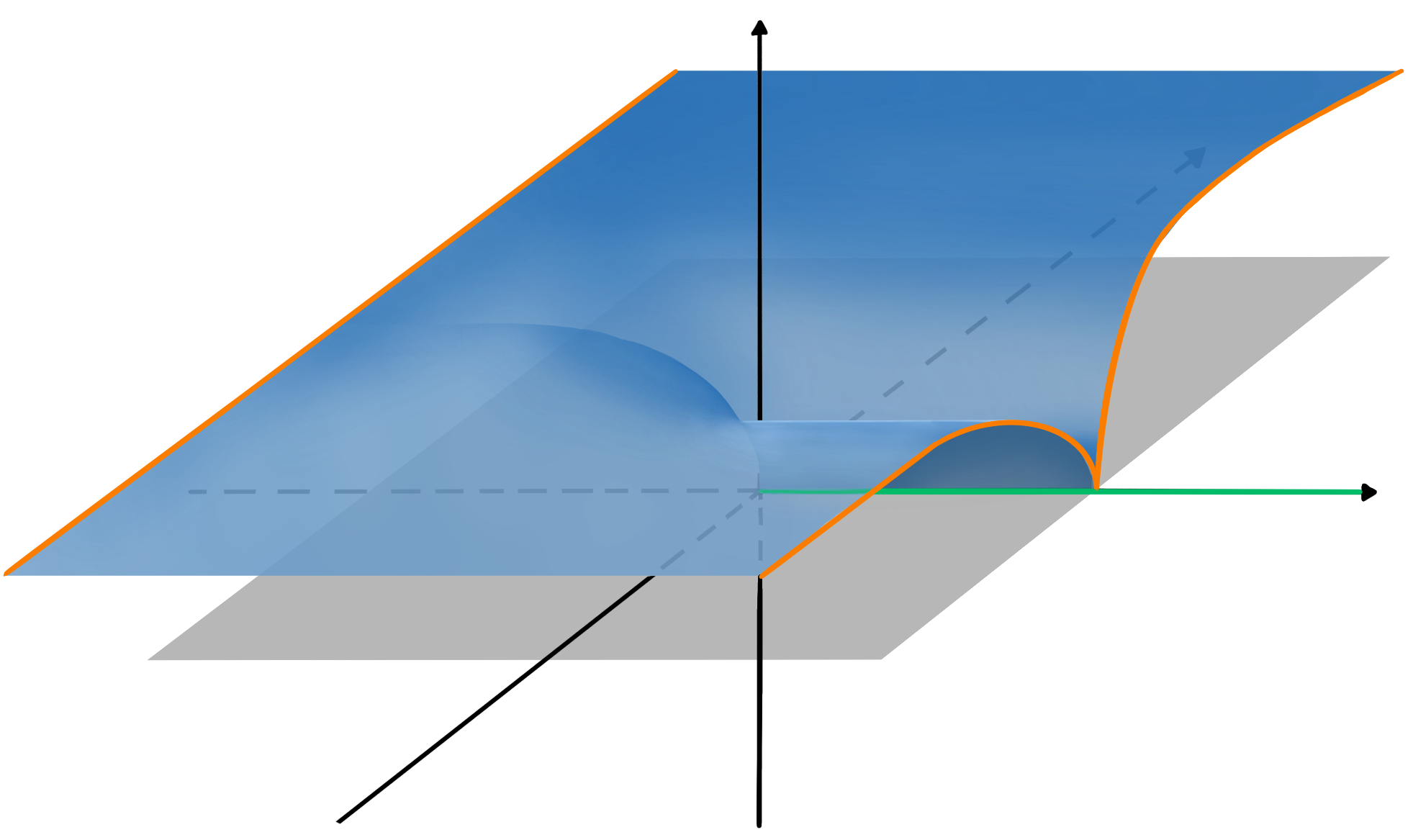}
        \put(30,48){\color{orange} +1}
        \put(90,18){\color{forestgreen} $\ell_+$}
        \put(92,45){\color{orange} $|H(x_2)|$}
    \end{overpic}
    \end{minipage}
    \caption[Trivialisations of the model solution]{The two trivialisations of the model solution obtained by removing the negative half line (left) and the positive half line (right).}
    \label{fig: U1}
\end{figure}

It will be practical to consider some canonical trivialisations of the model solution. By removing any half line $\ell$ based at the origin,
the set $\R^2_*\setminus\ell$ becomes trivialising for $\L$, in the sense that
$\pi^{-1}({\R^2_*\setminus\ell})\simeq (\R^2_*\setminus\ell)\times\R$. In
particular, $U\vert_{\R^2_*\setminus\ell}$ can be represented by a standard
function $\R^2_*\setminus\ell \to \R$. We describe now two trivialisations of $U$.

The first trivialisation is obtained by removing the negative half line
$\ell_-\coloneqq \{(x,0):x<0\}$. In this trivialisation the section is
represented by a function $U_-$ with a jump discontinuity at $\ell_-$, positive
in $\{z>0\}$, negative in $\{z<0\}$, and approximating $H(z)$ as
$x\to +\infty$; see \Cref{fig: U1,cor: decay U_1}.

The second trivialisation is obtained by removing the positive half-line
$\ell_+$. Then $U$ can be represented by a function
$U_+$, continuous over $\R^2$ and positive outside of $\ell_+$, where it
vanishes. Moreover, $U_+=|U_-|$ outside of $\{x=0\}$.

\subsection{Asymptotics of the model solution}\label{subs: asymp U}

In what follows we give a precise description of the asymptotic behaviour of the trivialisations $U_-$ and $U_+$.

\begin{lemma}\label{lem: qualitative positive end}
    The positive-end trivialisation satisfies
    \[
        \lim_{x\to+\infty}
        \lVert U_-(x,\cdot)-H\rVert_{L^\infty(\R)}
        =0 .
    \]
\end{lemma}

\begin{proof}
We first record a uniform tail non-collapse estimate. Given \(\eta>0\), the
Brezis--Oswald criterion gives \(R_\eta>0\) and a positive zero-Dirichlet
solution \(v_\eta\) of
\[
    \Delta v_\eta-W'(v_\eta)=0\quad\text{in }B_{R_\eta}(0),
    \qquad
    v_\eta=0\quad\text{on }\partial B_{R_\eta}(0),
\]
such that \(v_\eta(0)>1-\eta\). Translating this ball to any point \((x,z)\)
with \(z>R_\eta\) and using the same sliding comparison as above gives
\[
    U_-(x,z)\geq 1-\eta\qquad\text{for }z>R_\eta.
\]
Applying the same argument to \(-U_-\) gives
\[
    U_-(x,z)\leq -1+\eta\qquad\text{for }z<-R_\eta.
\]
Let \(x_n\to+\infty\) and define
\[
    V_n(x,z)\coloneqq U_-(x_n+x,z).
\]
On every fixed compact subset of \(\R^2\), the functions \(V_n\) are eventually
defined in the same scalar trivialisation and solve the Allen--Cahn equation.
The uniform \(L^\infty\) bound and Schauder estimates therefore give, after
passing to a subsequence, convergence in \(C^2_{\mathrm{loc}}(\R^2)\) to a bounded
entire solution \(V\). Moreover,
\[
    V(x,0)=0,\qquad V\geq0\text{ in }\{z>0\},\qquad V\leq0\text{ in }\{z<0\}.
\]
The non-collapse estimate above shows that \(V\) is not identically zero in
either half-plane; hence the strong maximum principle gives \(V>0\) in
\(\{z>0\}\) and \(V<0\) in \(\{z<0\}\). Thus the nodal set of \(V\) is the line
\(\{z=0\}\). By the standard flat-zero-set classification for entire
Allen--Cahn solutions (e.g.~ \cite[Corollary 5]{HamelLiuSicbaldiWangWei2019}), \(V\) is a one-dimensional heteroclinic. The sign and
translation are fixed by the inequalities above and by \(V(x,0)=0\), so
\(V=H(z)\). Since every sequence \(x_n\to+\infty\) has the same subsequential limit on
compact sets, \(U_-(x,\cdot)\to H\) locally uniformly in \(z\). Combining this
local convergence with the uniform upper/lower tail estimate above proves the
\(L^\infty(\R_z)\) convergence.
\end{proof}

The following lemma will be used to prove that the decay from the previous lemma is exponential.

\begin{lemma}\label{lem: decay U_1}
  There exists $R_{\mathrm{dec}}>0$ such that for all balls
  $B_R(p)\subset \{x>0\}$ with $R\geq R_{\mathrm{dec}}$ we have:
  \begin{equation*}
    \begin{split}
      |U_--H|(p) \leq \frac{1}{2}\lVert U-H\rVert_{L^\infty(B_R(p))}
    \end{split}
  \end{equation*}
\begin{proof}
We drop the minus subscript from $U_-$ for brevity. First, note that \(U-H\)
solves the equation
      \[
        \Delta (U-H)-W''(H)(U-H)=f,
        \qquad
        f=(W'(U)-W'(H))-W''(H)(U-H).
      \]
      Since \(U\) and \(H\) are bounded and \(W\) is smooth, Taylor's theorem
      gives
      \[
        |f(x,z)|\leq C|U(x,z)-H(z)|^2 .
      \]

We can prove the claim arguing via a contradiction. Assume there exists a
sequence of points $(x_n,z_n)$ and $R_n\to\infty$ such that
$B_{R_n}(x_n,z_n)\subset\{x>0\}$ and
  \begin{equation*}
    \begin{split}
      |U(x_n,z_n) - H(z_n)| > \frac{1}{2}\sup_{B_{R_n}(x_n,z_n)}|U_--H|.
    \end{split}
  \end{equation*}

      Since \(B_{R_n}(x_n,z_n)\subset\{x>0\}\), we have \(x_n\to+\infty\). By
      \Cref{lem: qualitative positive end},
      \[
          |U(x_n,z_n)-H(z_n)|\to0,
      \]
      and therefore \(c_n\coloneqq |U(x_n,z_n)-H(z_n)|^{-1}\to+\infty\).
Let $\psi_n(x,z) = c_n \big(U(x+x_n,z+z_n) - H(z+z_n)\big)$. By construction,
$\psi_n$ satisfies:
  \begin{equation*}
    \begin{cases}
      \Delta \psi_n - W''(H(z+z_n))\psi_n = f_n & \text{in } B_{R_n}(0,0) \\
      |\psi_n|_\infty \leq 2 & \text{in } B_{R_n}(0,0) \\
      |\psi_n(0,0)| = 1 & \text{for all }n\geq 0
    \end{cases}
  \end{equation*}
where \(f_n(x,z)=c_n f(x+x_n,z+z_n)\).
      Moreover, the contradictory assumption gives
      \[
        \sup_{B_{R_n}(x_n,z_n)}|U-H|
        <2|U(x_n,z_n)-H(z_n)|=\frac{2}{c_n}.
      \]
      Hence, using the quadratic Taylor estimate for \(f\),
      \[
        \lVert f_n\rVert_{L^\infty(B_{R_n}(0,0))}
        =
        \lVert c_n f(x+x_n,z+z_n)\rVert_{L^\infty(B_{R_n}(0,0))}
        \leq \frac{C}{c_n}\to0 .
      \]
      We have two limiting cases depending on whether the sequence $z_n$ is bounded or not:

\begin{itemize}
  \item If $\{z_n\}$ is bounded, then after passing to a subsequence $z_n\to z_\infty$ and $\psi_n$ converges to a bounded solution $\psi_\infty$ of the equation
  \[
    \Delta \psi_\infty-W''(H(z+z_\infty))\psi_\infty=0
    \quad\text{in }\R^2
  \]
  with $|\psi_\infty|_\infty \leq 2$ and $|\psi_\infty(0,0)|=1$. By \cite[Corollary 7.5]{PacardRitore} the only possibility is that $\psi_\infty(x,z)=c H'(z+z_\infty)$ for some constant $c\in \R$. Since \(U(x,0)-H(0)=0\), we must have $\psi_\infty(x,-z_\infty)=0$, for all $x\in\R$. Since $H'(z)>0$ for all $z\in\R$, this implies $c=0$, which contradicts $|\psi_\infty(0,0)|=1$.

  \item If $\{z_n\}$ is unbounded, then after passing to a subsequence, $\psi_n$ converges to a bounded solution $\psi_\infty$ of the equation $\Delta \psi_\infty - \kappa_W\psi_\infty=0$ in $\R^2$ with $|\psi_\infty|_\infty \leq 2$ and $|\psi_\infty(0,0)|=1$. Since \(\kappa_W>0\), the only bounded solution is $\psi_\infty\equiv 0$, which contradicts $|\psi_\infty(0,0)|=1$.
\end{itemize}
This exhausts the possible limiting cases.
\end{proof}
\end{lemma}

Applying \Cref{lem: decay U_1} successively in regions farther into $\{x>0\}$
gives the following asymptotics for $U_--H$. Indeed, if
\(M(s)\coloneqq\sup_{\{x\ge s,\ z\in\R\}}|U_-(x,z)-H(z)|\), then
\Cref{lem: decay U_1} gives \(M(s+R_{\mathrm{dec}})\le M(s)/2\), and iteration gives
\(M(s)\le Ce^{-\sigma s}\) for some \(\sigma>0\), after increasing \(C\).
\begin{corollary}\label{cor: decay U_1}
  There exist positive constants $C,\sigma>0$ such that for all $x>0$ and $z\in\R$ we have
  \begin{equation*}
    |U_-(x,z) - H(z)| \leq C e^{-\sigma x}.
  \end{equation*}
\end{corollary}

\begin{lemma}\label{lem: decay U_1 higher}
  After possibly decreasing the constant \(\sigma\) from \Cref{cor: decay U_1}, for any $k\in \mathbb{N}$, there exist constants $C_k>0$ such that for all $x\geq 1$ and $z\in\R$ we have
  \begin{equation*}
    |\nabla_{x,z}^k \big(U_-(x,z) -  H(z)\big)| \leq C_k e^{-\sigma x}.
  \end{equation*}
  In particular, $|\nabla_{x}^k U_-(x,z) | \leq C_k e^{-\sigma x}$ on the same region.
\end{lemma}

\begin{proof}
Dropping again the subscript from $U$, remember $U-H$ satisfies the equation
\[
    \Delta (U-H)-W''(H)(U-H)=f,
    \qquad
    f=(W'(U)-W'(H))-W''(H)(U-H).
\]
By Taylor's theorem and the boundedness of \(U\) and \(H\), the right-hand side
satisfies \(|f|\leq C|U-H|^2\). Thus, by the previous corollary, there exist
positive constants \(C,\sigma>0\), after possibly decreasing \(\sigma\), such that
\(|f(x,z)|\leq C e^{-\sigma x}\) for all \(x\geq0\) and \(z\in\R\). The result
follows from a simple induction argument iteratively differentiating the
equation for \(U_--H\) and applying Schauder estimates in regions contained in
\(\{x>0\}\), using unit balls centered in \(\{x\ge1\}\). At each step, the differentiated
Taylor remainder contains only lower-order derivatives, together with possible
highest-order terms multiplied by the exponentially small factor \(U-H\), which
can be absorbed after increasing the constants.
\end{proof}

Next, we study the asymptotics of $U(x,z)$ in the $z$ direction. We begin with $L^\infty$ estimates:

\begin{lemma}\label{lem: decay U_1 in z}
  There exist positive constants $C,\sigma>0$ such that for all $x\geq 1$ and $z\in\R$ we have:
  \begin{equation*}
    |U_-(x,z)-\mathrm{sgn}(z)| \leq C e^{-\sigma |z|}.
  \end{equation*}
\end{lemma}

\begin{proof}
    By the uniform tail estimate in \Cref{lem: qualitative positive end}, for
    \(Z\) sufficiently large the functions \(q_+=1-U_-\) on \(\{z\ge Z\}\) and
    \(q_-=1+U_-\) on \(\{z\le -Z\}\) satisfy equations
    \((\Delta-a_\pm)q_\pm=0\) with \(a_\pm\ge c_W>0\). Comparing on horizontal
    half-strips with the barriers \(A e^{-\mu(z-Z)}\) and
    \(A e^{-\mu(-z-Z)}\), for any \(\mu<\sqrt{c_W}\), gives the claim after
    enlarging \(C\) on \(|z|\le Z\).
\end{proof}

Combining \Cref{lem: decay U_1}, \Cref{lem: decay U_1 in z} and Schauder theory we can improve \Cref{lem: decay U_1 higher}.

\begin{corollary}\label{cor: estim U-}
    There exist positive constants $C_k,\sigma>0$ such that
  \begin{equation*}
    |\nabla^k_{x,z}(U_-(x,z)-H(z))| \leq C_k e^{-\sigma \sqrt{x^2+z^2}}.
  \end{equation*}
  on $\{(x,z)\in \R^2:x\geq 1\}$.
\end{corollary}

\begin{proof}
    After decreasing \(\sigma\), \Cref{lem: decay U_1 higher} gives
    \(C_ke^{-\sigma x}\) decay for all derivatives of \(U_--H\). The barrier in
    the proof of \Cref{lem: decay U_1 in z}, the exponential decay of
    \(H-\mathrm{sgn}\), and local Schauder estimates applied to
    \(q_+=1-U_-\) and \(q_-=1+U_-\) give \(C_ke^{-\sigma |z|}\) decay for the same
    derivatives. The result follows from
    \[
        \min\{e^{-\sigma x},e^{-\sigma |z|}\}
        \leq e^{-\sigma\sqrt{x^2+z^2}/\sqrt2},
    \]
    after decreasing \(\sigma\) once more.
\end{proof}

A completely analogous argument in the trivialisation away from \(\ell_+\)
establishes the asymptotics when \(x\to-\infty\).
\begin{lemma}\label{lem: estim U+}
    There exist positive constants $C_k,\sigma>0$ such that
  \begin{equation*}
    |\nabla^k_{x,z}(U_+(x,z)-1)| \leq C_k e^{-\sigma \sqrt{x^2+z^2}}.
  \end{equation*}
  whenever \(\dist((x,z),\ell_+)\ge 1\).
\end{lemma}
\begin{proof}
    In the positive trivialisation, \(U_+\) is a positive scalar solution on the
    region under consideration. Repeating the preceding qualitative compactness,
    barrier, and Schauder argument with \(1-U_+\) in place of \(U_--H\) gives
    \(C_ke^{-\sigma |x|}\) decay in the left end and \(C_ke^{-\sigma |z|}\) decay
    in the vertical ends. Combining the two estimates as above gives the stated
    radial bound away from a fixed tube around \(\ell_+\).
\end{proof}

After decreasing \(\sigma\) if necessary, we henceforth fix a single
\(\sigma>0\) for which all the preceding exponential decay estimates hold.

\subsection{Stability of the linearised operator}\label{subs: linearised op}
Compactly supported perturbations through sections lift to odd compactly
supported functions on the square-root cover. We use the natural
completed-cover class, which is slightly larger because the support may meet
the branch point. For \(\phi\in C_c^\infty(\R^2,\R)\), set
\begin{equation}\label{eq: stability lifting}
\begin{aligned}
    \mathbf L[\phi]
    &\coloneqq
    -\Delta\phi+4|\xi|^2W''(\U)\phi,\\
    Q_\U(\phi)
    &\coloneqq
    \int_{\R^2}
    \left(
    |\nabla\phi|^2+4|\xi|^2W''(\U)\phi^2
    \right)
    =
    \int_{\R^2}\mathbf L[\phi]\phi.
\end{aligned}
\end{equation}
We say that \(U\) is \emph{stable} if \(Q_\U(\phi)\geq0\) for every
\(\phi\in C_{c,\mathrm{odd}}^\infty(\R^2,\R)\), and \emph{strongly stable}
if the same inequality holds for every \(\phi\in C_c^\infty(\R^2,\R)\).

\begin{proposition}\label{prop: stability of L}
    The model solution \(U\) is strongly stable.
\end{proposition}
\begin{proof}
Set \(\omega\coloneqq\partial_{\xi_2}\U>0\). Differentiating the lifted
equation gives
\[
    \mathbf L[\omega]=-8\xi_2W'(\U)\geq0.
\]
For \(\phi\in C_c^\infty(\R^2,\R)\), writing
\(\phi=\omega\psi\) and integrating by parts gives the ground-state identity
\[
    Q_\U(\phi)
    =
    \int_{\R^2}\frac{\mathbf L[\omega]}{\omega}\phi^2
    +
    \int_{\R^2}\omega^2
    \left|\nabla\left(\frac{\phi}{\omega}\right)\right|^2
    \geq0.
\]
\end{proof}

\subsection{The kernel of the linearised operator}

\begin{proposition}[Upstairs odd non-degeneracy in $L^\infty$]\label{prop: nondegeneracy R2}
Let $g\in L^\infty(\R^2)$ be odd with respect to the origin and satisfy
\begin{equation*}
    -\Delta g+4|\xi|^2W''(\U)g=0\quad\text{in }\R^2.
\end{equation*}
Then $g=0$.
\end{proposition}

\begin{proof}
    By \Cref{prop: stability of L} and
    \cite[Theorem 1]{Fischer-ColbrieSchoen1980}, the equation
    \begin{equation*}
        -\Delta\varphi+4|\xi|^2W''(\U)\varphi=0\quad\text{on }\R^2
    \end{equation*}
    admits a positive solution $\varphi>0$. Write $g=\varphi\psi$. Then
    \begin{align*}
        0&=\Delta g-4|\xi|^2W''(\U)g\\
        &=\varphi\Delta\psi+2\nabla\varphi\cdot\nabla\psi
        +\psi\bigl(\Delta\varphi-4|\xi|^2W''(\U)\varphi\bigr)\\
        &=\varphi^{-1}\mathrm{div}(\varphi^2\nabla\psi).
    \end{align*}
    Let now $\eta$ be a compactly supported cut-off function and test the above
    against $\varphi\psi\eta^2$. Integrating  by parts gives
    \begin{equation*}
        0=\int\varphi^2|\nabla\psi|^2\eta^2
        +2\int \varphi^2\psi\eta\nabla\psi\cdot\nabla\eta.
    \end{equation*}
    Next,
    \begin{equation*}
        \left|2\int\varphi^2\psi\eta\nabla\psi\cdot\nabla\eta\right|
        \leq
        \frac12\int\varphi^2|\nabla\psi|^2\eta^2
        +C\int\varphi^2\psi^2|\nabla\eta|^2,
    \end{equation*}
    and therefore
    \begin{equation*}
        \frac12\int\varphi^2|\nabla\psi|^2\eta^2
        \leq C\lVert g\rVert^2_\infty\int|\nabla\eta|^2.
    \end{equation*}
    We choose
    \begin{equation*}
        \eta=\eta_R(\xi)\coloneqq \begin{cases}
            1&|\xi|<\sqrt{R}\\
            2\frac{\log R-\log|\xi|}{\log R}&\sqrt{R}\leq |\xi|\leq R\\
            0&|\xi|>R
        \end{cases}
    \end{equation*}
    and observe that
    \begin{equation*}
        \int|\nabla\eta_R|^2\leq \frac{C}{\log R}\to 0.
    \end{equation*}
    Sending $R\to +\infty$, we obtain
    \begin{equation*}
        \int_{\R^2}\varphi^2|\nabla\psi|^2=0,
    \end{equation*}
    and therefore $\psi$ is constant. Since $g$ is odd and $\varphi$ is
    positive, this constant must be zero. Hence $g=0$.
\end{proof}

\begin{corollary}[Non-degeneracy for sections]\label{cor: nondegeneracy-section}
Let $h$ be a bounded section of $\L\to\R^2_*$ satisfying
\begin{equation*}
    -\Delta h+W''(U)h=0\quad\text{in }\R^2_*.
\end{equation*}
Then $h=0$.
\end{corollary}

\begin{proof}
    Let $g$ be the odd lift of $h$ to the square-root cover. Then $g$ is bounded
    and satisfies
    \begin{equation*}
        -\Delta g+4|\xi|^2W''(\U)g=0\quad\text{in }\R^2.
    \end{equation*}
    The result follows from \Cref{prop: nondegeneracy R2}.
\end{proof}

By \Cref{cor: nondegeneracy-section}, $0$ is not an eigenvalue of $L$ in $L^\infty$. On the other hand $0$ can be approximated arbitrarily well in $L^\infty$, in the Rayleigh quotient sense. Let
\begin{equation*}
    L[\phi]\coloneqq -\Delta\phi+W''(U)\phi.
\end{equation*}
\begin{proposition}\label{prop: continuous spectrum}
    There is a sequence of sections $\{\varphi_k\}\subset L^2(\R^2,\L)$ such that
    \begin{equation*}
        \frac{\int_{\R^2}\varphi_kL\varphi_k}{\int_{\R^2}|\varphi_k|^2}\to 0
    \end{equation*}
    as $k\to +\infty$.
    \begin{proof}
        The idea is that near $x\approx +\infty$ the operator $L$ approaches the linearised operator of $2d$ Allen--Cahn about the heteroclinic solution, whose point spectrum contains 0 and the kernel is spanned by $H'(z)$. For $k\geq 0$, let
        \begin{equation*}
    \varphi_k\coloneqq \chi_kH'\left(z\right)
\end{equation*}
where $\chi_k=\chi_k(x)$ is a smooth cutoff supported on the strip
\begin{equation*}
    S_k\coloneqq \{x\in\R^2:|x-k|<\tfrac{k}{10}\}
\end{equation*}
satisfying $|\chi_k'|\leq C/k$ and $|\chi_k''|\leq C/k^2$ for some $C>0$. By
the discussion in \Cref{subs: trivialisations-model-solution}, $\varphi$ can be
regarded as a section. Then
\begin{align*}
    L\varphi_k&=-\chi_kH'''(z)-H'(z)\Delta \chi_k+W''(U)\chi_kH'\\
    &=-H'(z)\Delta\chi_k+[W''(U)-W''(H)]\chi_kH'.
\end{align*}
We compute, using also \Cref{cor: decay U_1}
\begin{equation*}
    \left|\int_{\R^2}\varphi_k L\varphi_k\right|\leq\left( \frac{C}{k^2}+Ce^{-k}\right)\int_{S_k}H'(z)^2dxdz\leq \frac{C'}k.
\end{equation*}
Noting that $\lVert\varphi_k\rVert_{L^2(\R^2)}^2\geq ak$, the conclusion follows.
\end{proof}
\end{proposition}

\begin{remark}\label{rem: rotation mode}
If the boundedness condition is removed a natural kernel does appear:
the one associated with rotations of the model solution. We denote this section
by
\begin{equation*}
    D_R U\coloneqq -x\partial_z U + z\partial_x U.
\end{equation*}
Here derivatives are understood as covariant derivatives of sections. In
practice, this means one chooses a local scalar representative of $U$ in a trivialisation
and differentiates that representative. Changing the representative by a sign
changes the derivative by the same sign, so the expression above defines a
section independently of the trivialisation. Differentiating the family of
rotated model solutions gives
\begin{equation*}
    (\Delta-W''(U))D_R U=0,
\end{equation*}
where we understand $\Delta$ again in the sense of sections. Thus $D_R U$ lies in the kernel of the
linearised operator. This rotation mode grows linearly, so it does not
contradict the bounded non-degeneracy above. The key point is that,
nevertheless, $D_R U$ will play a central role in the linearised theory
developed below.
\end{remark}

\section{Invertibility of the linearised operator}\label{sec: invert linear bdy}

Let $L\coloneqq-\Delta+W''(U)$ be the linearised operator about $U$. In this
section we study the solvability and a priori estimates for the inhomogeneous
equation
\begin{equation}\label{eq: boundary lin}
    \Delta\phi-W''(U)\phi=f
\end{equation}
with bounded right-hand side $f$, or equivalently $L\phi=-f$.

The main difficulty is that the bounded non-degeneracy proved in
\Cref{prop: nondegeneracy R2} is not uniform at infinity, as shown in
\Cref{prop: continuous spectrum}. In the trivialisation $U_-$, the operator
approaches, as
$x\to+\infty$, the classical Allen--Cahn operator linearised at the
heteroclinic $H$. The limiting one-dimensional operator has kernel generated by
$H'$, and the classical theory gives $L^\infty$ estimates only after imposing
$L^2$-orthogonality to this kernel. Without this orthogonality, bounded
right-hand sides can produce unbounded solutions. Thus $H'$ reappears
at infinity as an almost-kernel and must be separated from the equation.

\begin{definition}[$H'$-component]\label{def: Hprime-component}
Let $\varphi \in (L^\infty\cap C)(\R^2_*, \L)$. For $x>0$, we use the trivialisation obtained by
removing $\ell_-$, recalled in \Cref{subs: trivialisations-model-solution}, and
write
\begin{equation*}
    \varphi^\top(x)\coloneqq \int_{\R}\varphi(x,z)H'(z)\,dz.
\end{equation*}
The representative is chosen with the same sign convention as $U_-$, namely so
that, on the end $x>0$, the model solution has the orientation of the increasing
heteroclinic $H(z)$. We will use this notation for both the right-hand side
$f$ and the solution $\phi$.
\end{definition}

\begin{definition}[Admissible decay function]\label{def: adm decay}
We say that $g\in C^2_{\mathrm{loc}}(\R^+)\cap L^\infty(\R^+)$ is an admissible decay function if
\[
    g''\geq0,
    \qquad
    \lim_{x\to+\infty}g(x)=0,
    \qquad
    \lim_{x\to+\infty}xg'(x)=0.
\]  In particular, $g\geq0$ and $g'\leq0$.
\end{definition}

\begin{remark}
The main example relevant for us is
\(
g(x)=(1+x)^{2-\tau}.
\) with $\tau>2$.
\end{remark}

\begin{definition}[$g$-weighted tangential projection norm]\label{def:Ag-norm}
Let $g$ be an admissible decay function. For any bounded section \(f\) for which
the \(H'\)-component \(f^\top\) is defined on the positive end, set
\[
    A_g(f)\coloneqq
    \inf\left\{
        A\geq0:\ |f^\top(x)|\leq A g''(x)\text{ for all }x>0
    \right\},
\]
with the convention that \(A_g(f)=+\infty\) if no such \(A\) exists.
\end{definition}

The main result of this section is the following invertibility statement for the linearised operator, which takes into account the presence of the almost-kernel at infinity and gives decay estimates for the parts of the solution tangential to this kernel.

\begin{theorem}[Boundary invertibility]\label{lem: boundary invertibility}
	Let $g$ be an admissible decay function according to \Cref{def: adm decay}. For some $\gamma\in (0,1)$ and
	\(0<\lambda<\min\{\sigma,\sqrt{\kappa_W}\}\), let
	\(f\in C_{\ell_+,\lambda}^{0,\gamma}(\R^2_*,\L)\) satisfy \(A_g(f)<+\infty\).
    Then there exists a unique bounded section $\phi$ satisfying
		\begin{equation*}
			\Delta\phi-W''(U)\phi=f\quad \text{on }\R^2_*.
		\end{equation*}
		Moreover, there exists a constant \(C_g>0\), depending only on \(g\), such that
		\begin{equation*}
			\lVert\phi\rVert_{L^\infty(\R^2_*)}
			\leq C_g\left(\lVert f\rVert_{L^\infty(\R^2_*)}+A_g(f)\right).
		\end{equation*}
		Finally, the $H'$-component of $\phi$ at infinity is exactly given by the
		improper square-limit
		\begin{equation}\label{eq:phi-top-infty-formula}
			\phi^\top_\infty
            \coloneqq
            \lim_{R\to\infty}\int_{Q_R} fD_R U,
            \qquad Q_R=(-R,R)^2,
		\end{equation}
		and we have the following estimates for the convergence $\phi^\top(x)\to \phi^\top_\infty$ as $x\to +\infty$:
		\begin{equation}
        \begin{split}
            \label{eq: decay tangential part}
			|\phi^\top(x)-\phi^\top_\infty|&\leq c\left(e^{-\sigma x}\lVert\phi\rVert_{L^\infty}+A_g(f)g(x)\right)\\
            |\partial_x\phi^\top(x)|&\leq c\left(e^{-\sigma x}\lVert\phi\rVert_{L^\infty}+A_g(f)|g'(x)|\right)\\
            |\partial_x^2\phi^\top(x)|&\leq c\left(e^{-\sigma x}\lVert\phi\rVert_{L^\infty}+A_g(f)g''(x)\right)
        \end{split}
		\end{equation}
		for all $x>0$ and some $c>0$.
	\end{theorem}

To avoid the use of the superscript $\top$ along the proof, we introduce the notation
\begin{equation}\label{eq: alpha F}
	\alpha(x)\coloneqq \int_{\R}\phi(x,z)H'(z)dz,\quad F(x)\coloneqq \int_{\R}f(x,z)H'(z)dz.
\end{equation}
By the definition of \(A_g(f)\), the tangential projection satisfies
\begin{equation}\label{eq: assumption g}
    |F(x)|\leq A_g(f)g''(x),\qquad x>0.
\end{equation}

\begin{proposition}[Existence of bounded solutions]\label{prop:bounded-solvability-boundary}
Assuming \(f\in L^\infty(\R^2_*,\L)\) and \(A_g(f)<+\infty\), there exists a
unique bounded solution $\phi$ to the equation $\Delta\phi-W''(U)\phi=f$.
Moreover,
\[
    \lVert\phi\rVert_{L^\infty(\R^2_*)}
    \leq C_g\left(\lVert f\rVert_{L^\infty(\R^2_*)}+A_g(f)\right).
\]
\end{proposition}

\begin{proof}
  First, we solve the problem on a half-plane with the intention of later constructing an entire solution as a limit of solutions on half-planes.
 Let $R>0$ and consider the Dirichlet problem on $\{x<R\}$:
  \begin{equation}\label{eq: half-plane problem}
    \begin{cases}
      \Delta \phi_R - W''(U)\phi_R = f & \text{in } \{x<R\} \\
      \phi_R = 0 & \text{on } \{x=R\}.
    \end{cases}
  \end{equation}
	  The equations in this proof are equations for sections. On each bounded
	  exhaustion domain \(\Omega_L=[-L,R]\times[-L,L]\), we pass to the square-root
	  cover \(\widehat\Omega_L\), fill in the branch point, and solve the lifted
	  Dirichlet problem
	  \[
	      \bigl(-\Delta_\xi+4|\xi|^2W''(\U)\bigr)\widehat\psi_L
	      =
	      -4|\xi|^2\widehat f,
	      \qquad
	      \widehat\psi_L=0\quad\text{on }\partial\widehat\Omega_L .
	  \]
	  The positive supersolution used in \Cref{prop: stability of L} implies that
	  the Dirichlet first eigenvalue of this lifted operator is positive on
	  \(\widehat\Omega_L\). Hence the finite-domain problem has a unique solution,
	  and the maximum principle gives
	  \(\lVert\widehat\psi_L\rVert_\infty\le C(L,R)\lVert f\rVert_\infty\). Since the lifted
	  operator commutes with the deck transformation and the lifted datum is odd,
	  uniqueness gives \(\widehat\psi_L\circ\iota=-\widehat\psi_L\), so the solution
	  descends to a section. Bounded odd lifted limits extend removably across the
	  branch point, with value zero there.
	  In order to solve \eqref{eq: half-plane problem}, we start by using these bounded-domain solutions. For each $L>0$, there exists a section \(\psi_L\) over \(\Omega_L\setminus\{0\}\), with \(\Omega_L=[-L,R]\times [-L,L]\), such that \begin{equation*}
    \begin{cases}
      \Delta \psi_L - W''(U)\psi_L = f & \text{in } \Omega_L \\
      \psi_L = 0 & \text{on } \partial \Omega_L
    \end{cases}
  \end{equation*} with $\lVert\psi_L\rVert_\infty\leq C \lVert f\rVert_\infty$ for some constant $C=C(L)>0$. As $L \to \infty$, two things can happen: either the $\psi_L$'s remain uniformly bounded in $L^\infty$ norm, or they do not. In the first case, standard elliptic estimates imply that a subsequence of $\psi_L$ converges to the wanted bounded solution $\phi_R$ on the half-plane $\{x<R\}$. In the second case, we can find a sequence $L_n\to\infty$ such that $\lVert\psi_{L_n}\rVert_\infty = M_n \to \infty$. Defining ${\psi}_n = \psi_{L_n}/M_n$ we see that $\psi_n$ satisfies:
\begin{equation*}
  \begin{cases}
    \Delta \psi_n - W''(U)\psi_n = f_n & \text{in } \Omega_{L_n} \\
    \psi_n = 0 & \text{on } \partial \Omega_{L_n} \\
    \lVert\psi_n\rVert_{\infty} = 1 & \text{for all } n\\
    \lVert f_n\rVert_\infty \to 0 & \text{as } n\to\infty
  \end{cases}
	\end{equation*} where $f_n = f/M_n$. The pure-phase barriers used in Step 2
	below prevent the points where \(|\psi_n|\) is close to one from escaping to
	\(x=-\infty\) or to \(|z|=\infty\). Hence, after passing to a subsequence,
	local elliptic estimates give a non-trivial bounded solution \(\phi_R\) of
	\(\Delta \phi_R-W''(U)\phi_R=0\) in the half-plane \(\{x<R\}\), with zero
	Dirichlet data on \(\{x=R\}\).

Therefore, we have shown the following: for each $R>0$, there exists a non-trivial bounded function
\[
    \phi_R: \{x\leq R\}\to \R
\]
with zero Dirichlet condition on $\{x=R\}$ and satisfying one of the following equations:
      \begin{itemize}
        \item $\Delta \phi_R - W''(U)\phi_R = f$ in $\{x<R\}$, or
        \item $\Delta \phi_R - W''(U)\phi_R = 0$ in $\{x<R\}$.
      \end{itemize}

In order to prove the main claim, we would like to show that only the first case happens and, moreover, that $\lVert\phi_R\rVert_\infty$ remains uniformly bounded as $R\to\infty$. To this end, we argue by contradiction. This means  that we can make one of the following two assumptions: either the first case happens for a sequence $R_n\to\infty$ with $\lVert\phi_{R_n}\rVert_\infty \to \infty$, or the second case happens for a sequence $R_n\to\infty$ with $\lVert\phi_{R_n}\rVert_\infty=1$ for all $n$. Fortunately, we can treat both cases together: define $\psi_n = \phi_{R_n}/\lVert\phi_{R_n}\rVert_\infty$ which, in both cases, satisfies:
\begin{equation*}
  \begin{cases}
    \Delta \psi_n - W''(U)\psi_n = f_n & \text{in } \{x<R_n\} \\
    \psi_n = 0 & \text{on } \{x=R_n\} \\
    \lVert\psi_n\rVert_{\infty} = 1 & \text{for all } n\\
    \lVert f_n\rVert_\infty \to 0 & \text{as } n\to\infty
  \end{cases}
\end{equation*} where $f_n = f/\lVert\phi_{R_n}\rVert_\infty$ in the first case and $f_n=0$ in the second case. Moreover, we can assume there is a sequence $c_n\downarrow 0$, such that $F_n(x)=\int_\R f_n(x,z)H'(z)dz$ satisfies $|F_n(x)|\leq c_n g''(x)$ for all $x>0$. We will show, in four steps, that this leads to a contradiction.

\noindent \textbf{Step 1: Estimates on compact sets.} Fix \(L>0\). We first
claim that
\[
    \lVert\psi_n\rVert_{L^\infty(Q_L)}\to0,\qquad Q_L=[-L,L]\times[-L,L].
\]
Suppose not. After passing to a subsequence, there are points
\(q_n\in Q_L\setminus\{0\}\) such that \(|\psi_n(q_n)|\ge\delta>0\). If
\(q_n\to q_\infty\neq0\), local Schauder estimates near \(q_\infty\), followed
by a diagonal argument on compact subsets of \(\R^2_*\), produce a non-trivial
bounded section \(\psi\) satisfying
\[
    \Delta\psi-W''(U)\psi=0\quad\text{on }\R^2_*,
\]
contradicting \Cref{cor: nondegeneracy-section}. If \(q_n\to0\), pass to the
square-root coordinate. The odd lifts \(\widehat\psi_n\) extend across the
branch point with \(\widehat\psi_n(0)=0\), and the lifted equations have
uniformly bounded coefficients and right-hand sides on each fixed upstairs
ball. Hence interior estimates give a uniform \(C^1\) bound near the branch
point, so
\[
    |\psi_n(q_n)|=|\widehat\psi_n(\xi_n)|\le C|\xi_n|\to0,
    \qquad \pi(\xi_n)=q_n,
\]
again a contradiction. This proves the claim.

\noindent \textbf{Step 2: Estimates on $\{x\leq -L\}\cup \{|z|\geq L\}$}. We focus our argument on the region $\{z\geq L\}$, since on the remaining regions $\{z\leq -L\}$ and $\{x\leq -L\}$ the argument is the same, using the corresponding end asymptotic of \(U\). On this set, it is convenient to rewrite the equation for $\psi_n$ as:
\begin{equation*}
  \Delta \psi_n - \kappa_W\psi_n = h_n
\end{equation*}
where, by the end asymptotics of \(U\), $h_n = (W''(U)-\kappa_W)\psi_n + f_n$ satisfies $|h_n(x,z)|\leq K e^{-\sigma z} + |f_n(x,z)|$ for some positive constants $K,\sigma>0$ independent of $n$. Set \(c=K e^{-\sigma L}+\lVert f_n\rVert_\infty\). Then the function \(\psi_n-c/\kappa_W\) satisfies:
\[
    \Delta (\psi_n - \tfrac{c}{\kappa_W}) - \kappa_W(\psi_n - \tfrac{c}{\kappa_W}) = h_n + c\geq 0
\]
on the region $\{z>L\}$. Given $a>0$ define $v(z)=a e^{\sqrt{\kappa_W}z} + e^{\sqrt{\kappa_W}L}e^{-\sqrt{\kappa_W}z}$. Note $\Delta v(z)-\kappa_Wv(z)= 0$. Since $v(z)>1$ on $\{z=L\}$ and $\lim_{z\to\infty  }v(z)=\infty$, it follows that the function $\big(\psi_n(x,z)-\tfrac{c}{\kappa_W}\big)-v$ cannot attain a positive maximum in the region $\{z>L\}$, otherwise, there would be an interior positive maximum in which $0\geq \Delta (\psi_n - \tfrac{c}{\kappa_W}-v)-\kappa_W(\psi_n - \tfrac{c}{\kappa_W}-v)\geq 0$. Therefore:
\begin{equation*}
  \psi_n(x,z) \leq \tfrac{c}{\kappa_W} + v(z) = \tfrac{c}{\kappa_W} + a e^{\sqrt{\kappa_W}z} +  e^{\sqrt{\kappa_W}L}e^{-\sqrt{\kappa_W}z}
\end{equation*}
on the region $\{z>L\}$. By taking $a\to 0$ we get:
\begin{equation*}
  \begin{split}
    \psi_n(x,z) & \leq \tfrac{c}{\kappa_W} +  e^{\sqrt{\kappa_W}L} e^{-\sqrt{\kappa_W}z} \\
    & \leq \kappa_W^{-1} \Big( K e^{-\sigma L} + \lVert f_n\rVert_\infty \Big) + e^{\sqrt{\kappa_W}L} e^{-\sqrt{\kappa_W}z}
  \end{split}
\end{equation*}
For \(z>2L\) we obtain the better estimate:
\begin{equation*}
  \psi_n(x,z) \leq \kappa_W^{-1} \Big( K e^{-(\sigma/2)(2L)} + |f_n|_\infty \Big) + e^{-(\sqrt{\kappa_W}/2)(2L)}
\end{equation*}
which can be made arbitrarily small by choosing $L$ large and then $n$ large enough.

\begin{remark} We have obtained a one sided estimate on $\psi_n$, however note that our assumption is on the absolute value of the right-hand side, i.e.~ $|h_n|$, so the same estimate holds for $-\psi_n$, and together imply the estimate for $|\psi_n|$. Finally, we repeat the argument in the regions $\{x\leq -L\}$ and $\{z\leq -L\}$ using the same barrier with decay in the corresponding end variable.\end{remark}
In other words, we have that for any $\delta>0$, there exists $L>0$ such that for all $n$ large enough, $|\psi_n|\leq \delta$ on the region $\{x\leq -L\}\cup \{|z|\geq L\}$.

\noindent \textbf{Step 3: Tangential estimates on $\{x\geq L\}$}. Let $\delta>0$, $L>0$ to be chosen later and $n$ be large enough so that $R_n>L$. Let $\alpha_n(x)=\int_\R \psi_n(x,z)H'(z)dz$. The differentiations below, and the integration by parts in \(z\), are justified by local elliptic estimates and the exponential decay of \(H'\). Differentiating with respect to $x$ we get:
\begin{equation*}
\begin{split}
	\alpha_n''(x) &= \int_\R \partial_{xx}\psi_n(x,z)H'(z)dz \\
	&= \int_\R \Big(W''(U(x,z))\psi_n(x,z) + f_n(x,z) - \partial_{zz}\psi_n(x,z)\Big)H'(z) dz
\end{split}
\end{equation*}
Integrating by parts in $z$ and using the fact that $H'$ is in the kernel of the linearized operator $\partial_{zz} - W''(H(z))$ we obtain:
\begin{equation*}
  \begin{split}
  \alpha_n''(x) & = \int_\R \Big(W''(U(x,z)) - W''(H(z))\Big)\psi_n(x,z)H'(z) dz + \int_\R f_n(x,z)H'(z) dz
  \end{split}
\end{equation*} By the definition of \(A_g\) for \(f_n\) and the positive-end convergence of \(U(x,z)\) to \(H(z)\), we have:
\begin{equation*}
  |\alpha_n''(x)| \leq C e^{-\sigma x} + c_n g''(x)
\end{equation*} for some positive constant $C>0$ independent of $n$ and $f$.

Next, we use a barrier argument to derive precise asymptotics for $\alpha_n(x)$ for $x\geq L$. With
\[
    B_n(x)=C\sigma^{-2}e^{-\sigma x} + c_n g(x)
\]
consider the function:
\[
    G(x)=\alpha_n(x)+B_n(x)-\Big(B_n(L)+|\alpha_n(L)| \Big)
\]
Since \(g\) is admissible, \(B_n\) is decreasing. It follows automatically that $G(L)\leq 0$. Similarly, since $\alpha_n(R_n)=0$ we have $G(R_n)\leq 0$. In addition, $G$ satisfies the ODE:
\begin{equation*}
  G''(x) = \alpha_n''(x) + C e^{-\sigma x} + c_n g''(x) \geq 0
\end{equation*} This implies that $G(x)\leq 0$ for all $x\in[L,R_n]$. The same bound holds for $-\alpha_n(x)$ using a similar argument. Therefore, for all $x\in[L,R_n]$:
      \begin{equation*}
        |\alpha_n(x)| \leq |B_n(L)-B_n(x)|+|\alpha_n(L)|\leq |B_n(L)|+|\alpha_n(L)|
      \end{equation*}

Now, fix $\delta>0$. First choose \(L\), independently of \(n\) and \(x\), so that \(B_n(L)-B_n(x)\leq B_n(L)\leq \delta/2\) for all \(n\). Combining the results from Step 1 and Step 2, for this fixed \(L\) we can then take \(n\) large enough so that $|\alpha_n(L)|\leq \delta/2$. Therefore, we have shown that given $\delta>0$, there exists $L>0$ such that, $n$ large enough $\implies$
      $|\alpha_n(x)|\leq \delta$, for all $x\in[L,R_n]$

\noindent \textbf{Step 4: The contradiction.} Let $(x_n,z_n) \in \{x<R_n\}$ be such that $\psi_n(x_n,z_n)=1$. Combining the estimates from Steps 1 and 2, the only remaining possibility is that $x_n\to+\infty$ while $|z_n|$ remains bounded. If \(R_n-x_n\to+\infty\), recentering gives a bounded non-zero full-plane solution of \(\Delta\psi_\infty-W''(H(z))\psi_\infty=0\). By the standard spectral decomposition of \(L_H=-\partial_z^2+W''(H)\), whose bounded kernel is spanned by \(H'\) \cite[Corollary 7.5]{PacardRitore}, this solution is \(cH'\). Step 3 gives
\[
    0=\lim_{n\to\infty}\alpha_n(x_n)=c\int_\R (H')^2\,dz,
\]
so \(c=0\), a contradiction. If instead \(R_n-x_n\) remains bounded, the limit is a bounded solution on a half-plane with zero Dirichlet data on the vertical boundary; odd reflection across that boundary gives the same classification, and the boundary condition again forces \(c=0\).
The displayed \(L^\infty\) estimate follows from the same contradiction
argument: if no constant \(C_g\) existed, one could choose sources \(f_n\) with
\(\lVert f_n\rVert_{L^\infty}+A_g(f_n)=1\) whose bounded solutions satisfy
\(\lVert\phi_n\rVert_{L^\infty}\to\infty\); normalizing by
\(\lVert\phi_n\rVert_{L^\infty}\) gives the same four-step contradiction above.
Finally, uniqueness follows by \Cref{cor: nondegeneracy-section}. The proof is
concluded.

\end{proof}

\begin{proposition}[Asymptotics of the tangential part]\label{prop:tangential-asymptotics}
Let \(f\in C_{\ell_+,\lambda}^{0,\gamma}(\R^2_*,\L)\), let $\phi$ be a bounded
solution of \eqref{eq: boundary lin}, and let $\alpha$ and $F$ be defined by
\eqref{eq: alpha F}. Assume \(A_g(f)<+\infty\). Then the limit
$\alpha_\infty=\lim_{x\to+\infty}\alpha(x)$ exists,
$\lim_{x\to+\infty}x\alpha'(x)=0$, and the estimates
\eqref{eq: decay tangential part} hold with $\phi^\top$ replaced by $\alpha$ and $\phi_\infty^\top$ replaced by $\alpha_\infty$.
\end{proposition}

\begin{proof} We divide the proof into several steps. The first two steps establish the existence of the limit $\alpha_\infty$ and the decay of $\alpha'(x)$, while the last step proves the estimates in \eqref{eq: decay tangential part}.

\begin{claim}
The function $\alpha(x)$ defined in \eqref{eq: alpha F} is well-defined for $x>1$ and of class $C^2$.
\end{claim}

    Since \(f\in C_{\ell_+,\lambda}^{0,\gamma}(\R^2_*,\L)\), local Schauder
    estimates imply that \(\phi\in C_{\mathrm{loc}}^{2,\gamma}(\R^2_*,\L)\).
    In particular, on the region \(\{x>1\}\), which stays away from the puncture,
    the derivatives used below are uniformly bounded. Since $\phi$ is bounded and
    $H'(z)$ decays exponentially as $|z|\to \infty$, the integral defining
    $\alpha(x)$ converges for each $x>0$. To show that $\alpha(x)$ is $C^2$, we
    differentiate under the integral sign. The local Schauder estimates,
    together with the exponential decay of $H'(z)$, justify this
    differentiation. Thus, we have:
    \[
    \alpha'(x) = \int_{\R} \partial_x \phi(x,z) H'(z) dz,
    \]
    and
    \[
    \alpha''(x) = \int_{\R} \partial_x^2 \phi(x,z) H'(z) dz.
    \]
    Hence $\alpha(x)$ is indeed $C^2(\{x>1\})$.

\begin{claim}
The function $\alpha(x)$ defined above satisfies the following ODE:
\begin{equation*}
    \alpha''(x) = E(x) + F(x)
\end{equation*}
where
\begin{equation*}
  \begin{split}
    E(x) &= \int_{\R} \big( W''(U(x,z)) - W''(H(z)) \big) H'(z) \phi(x,z) dz
  \end{split}
\end{equation*}
and $F$ is as in \eqref{eq: alpha F}.
\end{claim}

    Using the equation for $\phi$ and the fact that $H''(z)=W'(H(z))$, we get:
    \begin{equation*}
      \begin{split}
        \alpha''(x) &= \partial_x^2 \alpha(x)\\
        &= \int_{\R} \partial_x^2 \phi(x,z) H'(z) dz \\
        &= \int_{\R} \big( W''(U(x,z))\phi(x,z) + f(x,z) \big) H'(z) dz - \int_{\R} \partial_z^2 \phi(x,z) H'(z) dz \\
        &= \int_{\R} W''(U(x,z)) H'(z) \phi(x,z) dz + \int_{\R} f(x,z) H'(z) dz - \int_{\R} \phi(x,z) H'''(z) dz \\
        &= \int_{\R} \big( W''(U(x,z)) - W''(H(z)) \big) H'(z) \phi(x,z) dz + \int_{\R} f(x,z) H'(z) dz,
      \end{split}
    \end{equation*}
    proving the claim.

\begin{claim}
  If \(A_g(f)<+\infty\), then $\alpha_\infty=\lim_{x\to\infty}\alpha(x)$ exists and $\lim_{x\to\infty}x \alpha'(x)=0.$
\end{claim}

There is $K>0$ such that $|E(x)|\leq K e^{-\sigma x}\lVert\phi\rVert_\infty$ for all $x>0$. Let
\[
    G(x) = \alpha(x) + (A_g(f)g(x)+\sigma^{-2}K e^{-\sigma x}|\phi|_\infty).
\]
Then:
\begin{equation*}
    \begin{split}
      G''(x)&= \alpha''(x) + (A_g(f)g''(x) + K e^{-\sigma x}\lVert\phi\rVert_\infty) \\
      &= (E(x) +  K e^{-\sigma x}|\phi|_\infty) + (F(x) + A_g(f)g''(x))\\
      &\geq 0
    \end{split}
\end{equation*}
By our assumptions on $\phi$ and $g$, it follows that $G$ is bounded and convex
on the interval $[1,\infty)$. The elementary convexity observation used in the
definition of admissible decay functions gives $\lim_{x\to\infty}G(x)$ and
$\lim_{x\to\infty}xG'(x)=0$. Since the same properties hold for $g$ and for the
exponential term, the formula for $G$ implies
$\lim_{x\to\infty}x\alpha'(x)=0$ and
$\lim_{x\to\infty}\alpha(x)$ exists. This finishes the proof of the claim.

Next, we estimate the rate of convergence for both the limits $\lim_{x\to\infty}\alpha(x)=\alpha_\infty$ and $\lim_{x\to\infty}\alpha'(x)=0$. More precisely, we have the following:

\begin{claim}
If \(A_g(f)<+\infty\), we have
  \begin{equation*}
  \begin{split}
      |\alpha(x)-\alpha_\infty| &\leq K\sigma^{-2}\lVert\phi\rVert_\infty e^{-\sigma x} + A_g(f)g(x) \\
      |\alpha'(x)| &\leq  K\sigma^{-1}\lVert\phi\rVert_\infty e^{-\sigma x} + A_g(f)|g'(x)|,\\
      |\alpha''(x)| &\leq K\lVert\phi\rVert_\infty e^{-\sigma x} + A_g(f)g''(x).
    \end{split}
  \end{equation*}
\end{claim}

  The estimate for \(\alpha''\) follows from Claim 2, the bound for \(E\) above,
  and \eqref{eq: assumption g}. For \(\alpha\) and \(\alpha'\), we use the
  formula below.

  From the formula $\alpha'(x) = -\int_x^\infty (E(t) + F(t)) dt$ and using the bounds on $E$ and $F$, as well as $\lim_{x\to\infty}g(x)=0$ we get:
  \begin{equation*}
    \begin{split}
      |\alpha'(x)| &\leq \int_x^\infty K e^{-\sigma t} |\phi|_\infty dt + A_g(f)\int_x^\infty g''(t) dt \\
      &\leq K\sigma^{-1} \lVert\phi\rVert_\infty e^{-\sigma x} - A_g(f)g'(x).
    \end{split}
  \end{equation*}
  Integrating again we get
  \begin{equation*}
    \begin{split}
      |\alpha(x) - \alpha_\infty| &\leq \int_x^\infty |\alpha'(t)| dt \\
      &\leq \int_x^\infty \sigma^{-1}K |\phi|_\infty e^{-\sigma t} dt - A_g(f)\int_x^\infty g'(t) dt \\
      &\leq K\sigma^{-2} \lVert\phi\rVert_\infty e^{-\sigma x} + A_g(f)g(x).
    \end{split}
  \end{equation*}
which proves the claim.
\end{proof}

\begin{proposition}[Closed formula for the tangential part at infinity]\label{prop:tangential-limit}
Under the assumptions of \Cref{prop:tangential-asymptotics}, we have
\begin{equation*}
    \alpha_\infty
    =
    \lim_{R\to\infty}\int_{Q_R} f(x,z)D_R U(x,z)\,dx\,dz
\end{equation*}
where $D_R U$ is the rotation kernel defined in \Cref{rem: rotation mode}.
\end{proposition}

\begin{proof}
The local Schauder estimates in the trivializing balls and the square-root
chart give
\(
M_\phi\coloneqq
\lVert\phi\rVert_{C_{*,1}^{2,\gamma}(\R^2_*,\L)}<+\infty.
\)
  Multiplying the equation $\Delta \phi - W''(U)\phi=f$ by $D_R U$ and
  integrating over the punctured domain
  \(Q_{R,\rho}\coloneqq Q_R\setminus B_\rho(0)\), where
  \(Q_R\coloneqq (-R,R)^2\), we get:
  \begin{equation*}
    \begin{split}
     \int_{Q_{R,\rho}} f D_R U dx dz &= \int_{Q_{R,\rho}} (\Delta \phi - W''(U)\phi) D_R U dx dz \\
    \end{split}
  \end{equation*}

  Integrating by parts on the right-hand side we get:
  \begin{equation*}
    \begin{split}
     \int_{Q_{R,\rho}} f D_R U dx dz &= \int_{\partial Q_R} \partial_\nu \phi D_R U - \phi \partial_\nu D_R U dS + I_\rho\\
     &\qquad + \int_{Q_{R,\rho}} \phi (\Delta D_R U - W''(U) D_R U) dx dz\\
    &= \int_{\partial Q_R} \partial_\nu \phi D_R U - \phi \partial_\nu D_R U dS + I_\rho,
    \end{split}
  \end{equation*}
where $\nu$ is the outer normal to $\partial Q_R$ and
\[
    I_\rho
    =
    \int_{\partial B_\rho(0)}
    \partial_\nu \phi D_R U-\phi\partial_\nu D_R U\,dS,
\]
with \(\nu\) denoting the outer normal to \(Q_{R,\rho}\) on the inner boundary.
The last equality follows from the fact that $D_R U$ is in the kernel of the
linearized operator, i.e. $\Delta D_R U - W''(U) D_R U=0$.
It remains only to remove the artificial inner boundary. In the square-root
coordinate at the puncture, the lifts of \(\phi\) and \(D_R U\) are odd
regular functions, hence their downstairs representatives satisfy
\[
    |\phi|+|D_R U|\lesssim \rho^{1/2},
    \qquad
    |\nabla\phi|+|\nabla D_R U|\lesssim \rho^{-1/2}
    \quad\text{on }\partial B_\rho(0).
\]
Therefore \(|I_\rho|\lesssim \rho\), and \(I_\rho\to0\) as
\(\rho\to0\). Letting \(\rho\to0\) gives the displayed identity with the
integral over \(Q_R\).

  Next, we compute the boundary integral on each side of the square $Q_R$. On
  the top, left and bottom sides, \(D_RU\) and \(\nabla D_RU\) decay
  exponentially, so the boundary integrals are
  \(O(R^2e^{-\sigma R}M_\phi)\). On the right side we have:
  \begin{equation*}
    \begin{split}
     I_{\operatorname{right}}&=\int_{-R}^R \partial_x \phi(R,z) D_R U(R,z) - \phi(R,z) \partial_x D_R U(R,z) dz \\
     &= \int_{-R}^R \partial_x \phi(R,z) \bigg(-R \partial_z U(R,z) + z \partial_x U(R,z)\bigg) dz \\
     &- \int_{-R}^R \phi(R,z) \bigg(-\partial_zU(R,z)-R \partial_{xz} U(R,z) + z \partial_{xx} U(R,z)\bigg) dz\\
     &= \int_{-R}^R \phi(R,z)\partial_z U(R,z)dz - R\int_{-R}^{R}\partial_x\phi(R,z)\partial_z U(R,z) + D_1\\
    \end{split}
  \end{equation*}
  where $D_1=\int_{-R}^R \partial_x \phi(R,z)\Big(z\partial_x U(R,z)\Big) + \phi(R,z)\Big(R\partial_{xz} U(R,z)-z\partial_{xx} U (R,z) \Big)dz$ satisfies $|D_1|\leq C R^2 e^{-\sigma R}M_\phi$ for some positive constants $C,\sigma>0$. Next, we exploit the fact that $|\partial_z U(R,z)-H'(z)|$ decays exponentially in $R$ uniformly in $z$ to write:
  \begin{equation*}
    \begin{split}
     I_{\operatorname{right}}&= \int_{-R}^R \phi(R,z)H'(z)dz - R\int_{-R}^{R}\partial_x\phi(R,z)H'(z) dz + D_1 + D_2\\
    \end{split}
  \end{equation*}
  where $D_2=\int_{-R}^R \phi(R,z)\Big(\partial_z U(R,z)-H'(z)\Big) dz - R\int_{-R}^{R}\partial_x\phi(R,z)\Big(\partial_z U(R,z)-H'(z)\Big) dz$ satisfies $|D_2|\leq C R^2 e^{-\sigma R}M_\phi$ for some $C>0$. Finally, we compare the integrals with the definition of $\alpha(R)$ and $\alpha'(R)$ to get:
  \begin{equation*}
    \begin{split}
     I_{\operatorname{right}}&= \alpha(R) - R \alpha'(R) + D_1 + D_2+ D_3\\
    \end{split}
  \end{equation*}
  where
  \[
      D_3=
      -\int_{|z|>R} \phi(R,z)H'(z)\,dz
      + R\int_{|z|>R}\partial_x\phi(R,z)H'(z)\,dz
  \]
  satisfies $|D_3|\leq C R e^{-\sigma R}M_\phi$ for some positive constant $C>0$. Combining all the estimates we get:
  \begin{equation*}
    \begin{split}
     |I_{\operatorname{right}}-\alpha(R)| &\leq C R^2 e^{-\sigma R}M_\phi + R |\alpha'(R)|
    \end{split}
  \end{equation*} which together with our estimates for $R$ imply the result.
\end{proof}

Finally, we present the proof of the main theorem of this section, which is a consequence of the previous propositions.

\begin{proof}[Proof of \Cref{lem: boundary invertibility}]
The existence, uniqueness, and \(L^\infty\) estimate for the bounded solution
follow from \Cref{prop:bounded-solvability-boundary}. The convergence estimates
in \eqref{eq: decay tangential part} follow from
\Cref{prop:tangential-asymptotics}, and the formula for $\phi^\top_\infty$
follows from \Cref{prop:tangential-limit}.
\end{proof}

\begin{corollary}[Weighted Schauder estimates for the boundary inverse]\label{cor: schauder-tan}
Let \(g\) be an admissible decay function, let \(\gamma\in (0,1)\), and let
\(f\in C_*^{0,\gamma}(\R^2_*,\L)\) satisfy \(A_g(f)<+\infty\). Here
\(C_*^{k,\gamma}\) denotes the unweighted section norm
\(C_{*,1}^{k,\gamma}\) from \Cref{def:scaled-c-star-holder-norm}. If \(\phi\)
is the bounded solution of
\[
    \Delta\phi-W''(U)\phi=f
    \qquad\text{on }\R^2_*,
\]
then
\[
    \lVert\phi\rVert_{C_*^{2,\gamma}(\R^2_*,\L)}
    \leq
    C_g\left(
    \lVert f\rVert_{C_*^{0,\gamma}(\R^2_*,\L)}
    +
    A_g(f)
    \right).
\]
Moreover, if \(0<\lambda<\min\{\sigma,\sqrt{\kappa_W}\}\) and
\(f\in C_{\ell_+,\lambda}^{0,\gamma}(\R^2_*,\L)\),  then
\[
    \lVert\phi\rVert_{C_{\ell_+,\lambda}^{2,\gamma}(\R^2_*,\L)}
    \leq
    C_{g,\lambda}\left(
    \lVert f\rVert_{C_{\ell_+,\lambda}^{0,\gamma}(\R^2_*,\L)}
    +
    A_g(f)
    \right).
\]
\end{corollary}

\begin{proof}
The \(L^\infty\) estimate in \Cref{prop:bounded-solvability-boundary} gives
\[
    \lVert\phi\rVert_{L^\infty(\R^2_*)}
    \leq
    C_g\left(\lVert f\rVert_{L^\infty(\R^2_*)}+A_g(f)\right).
\]
The unweighted \(C_*^{2,\gamma}\) estimate follows from this bound and the
standard local Schauder estimates in the trivializing balls. On the singular
ball at the puncture, one passes to the square-root coordinate. The lifted
equation has the form
\[
    \Delta\widetilde\phi
    -
    4|\xi|^2W''(\U)\widetilde\phi
    =
    4|\xi|^2\widetilde f,
\]
with smooth coefficients across \(\xi=0\). Bounded solutions have a removable
singularity on the cover, and oddness fixes the value at the branch point to be
zero; the same Schauder estimate therefore gives the singular part of the
\(C_*\)-norm.

If \(f\) is exponentially weighted away from \(\ell_+\), the same local
estimate gives the weighted Schauder bound once the corresponding weighted
\(L^\infty\) estimate is known. This \(L^\infty\) estimate is obtained by the
standard exponential barrier argument on the complement of a fixed tubular
neighborhood of \(\ell_+\), where \(W''(U)\) is uniformly positive because
\(U\) is exponentially close to a pure phase; compare
\cite[Lemma~8.3]{Guaraco-Marques-Neves19}. This proves the
\(C_{\ell_+,\lambda}^{2,\gamma}\) estimate.
\end{proof}

\section{Preliminaries for the gluing argument}
\label{sec: preliminaries-gluing}

\subsection{Geometric set up and coordinate systems}\label{subs: geometric-setup}
Let
$N\geq1$ and let $I_1,\dots,I_N$ be pairwise disjoint line segments in $\R^2$.
We write
\[
    \p\coloneqq \bigcup_{j=1}^N\partial I_j
    =\{p_1,\dots,p_{2N}\},
    \qquad
    \Gamma\coloneqq I_1\cup\cdots\cup I_N,
    \qquad
    \R^2_{\p}\coloneqq \R^2\setminus\p,
\]
and denote by $\L_{\p}\to\R^2_{\p}$ the line bundle constructed in \Cref{subs: line bundle}. We also fix a unit vector $\nu_j$ normal to the segment $I_j$, and, given $p\in \partial I_j$, we denote by $e_p$ the unit vector pointing towards the interior of $I_j$. In this context, we set the following boundary coordinates:
\begin{equation}\label{eq: coordinates Yp}
    Y_p\colon \R^2 \to \R^2,
    \qquad
    Y_p(x,z)=p+xe_p+z\nu_j.
\end{equation}
as well as the usual Fermi coordinates around the segment $I_j$:
\begin{equation}\label{eq: coordinates Xj}
    X_j\colon I_j\times \R \to \R^2,
    \qquad
    X_j(x,z)=x+z\nu_j.
\end{equation}
For \(\theta\in\R\), let
\[
    R_\theta
    \coloneqq
    \begin{pmatrix}
        \cos\theta&\sin\theta\\
        -\sin\theta&\cos\theta
    \end{pmatrix}
\]
denote clockwise rotation through the angle \(\theta\). Rotations about the
origin and positive dilations preserve the one-puncture double cover; using
their lifts from the identity and the induced identifications of \(\L\), write
\(U_\theta\coloneqq U\circ R_\theta\).

\subsection{Cut-off functions}
\label{subs: cutoff}

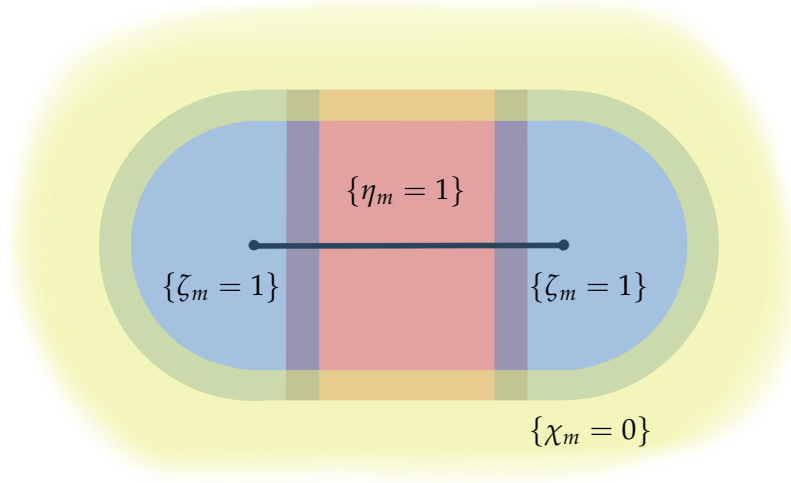
\begin{figure}
    \centering
    \begin{overpic}[width=.7\textwidth]{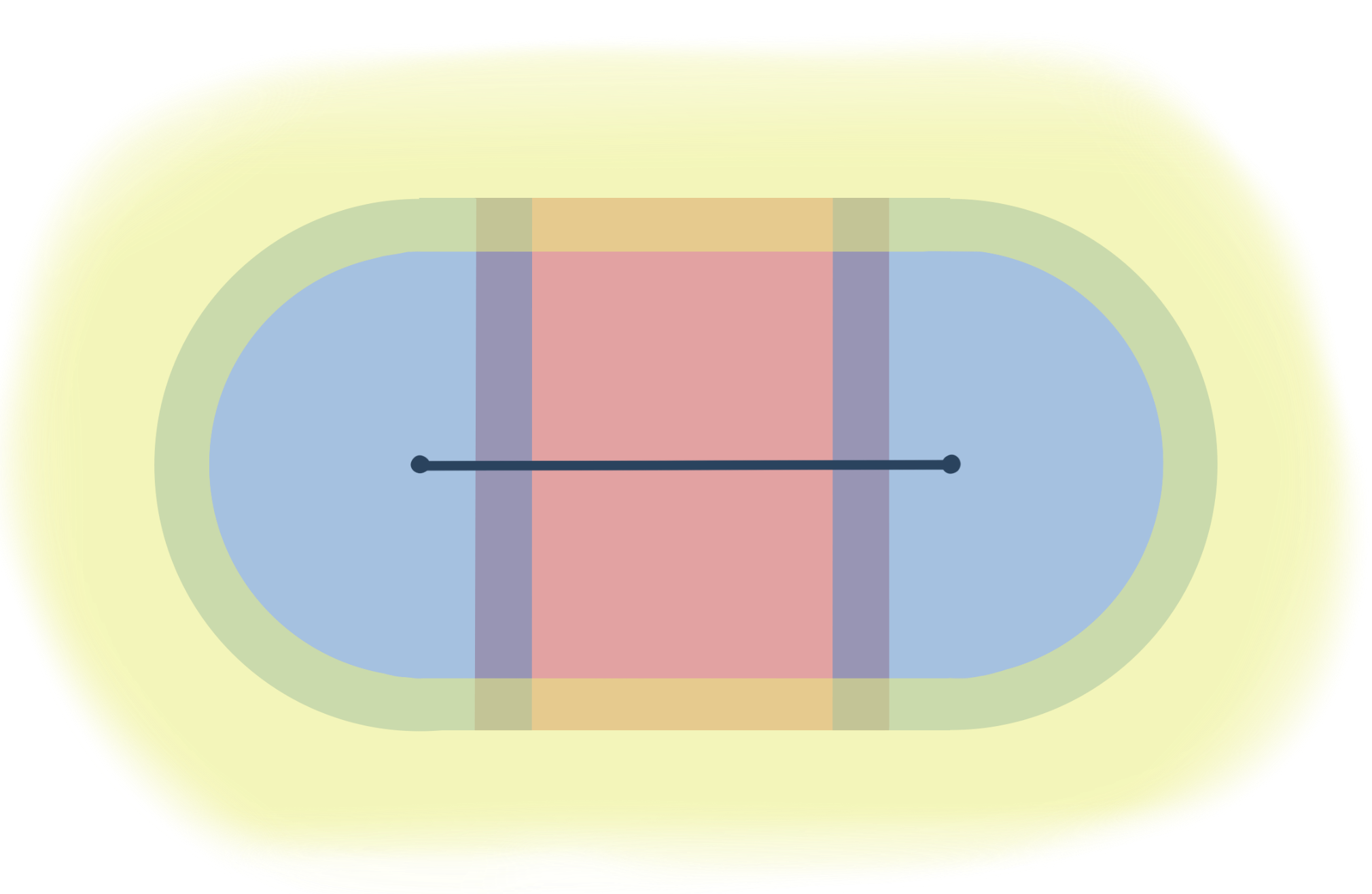}
    \put(65,7){$\{\chi_m=0\}$}
    \put(19,25){$\{\zeta_m=1\}$}
    \put(65,25){$\{\zeta_m=1\}$}
    \put(42,37){$\{\eta_m=1\}$}
    \end{overpic}
    \caption{A scheme of the cut-off regions near a single interval. Red represents the support of $\eta_m$, blue the support of $\zeta_m$ and yellow is the region where $\chi_m$ vanishes.}
    \label{fig: cutoff}
\end{figure}

Given $j=1,\dots,N$ and $m\in \N$ and $p\in \p$, we define four types of cut-off functions:
\begin{enumerate}
    \item $\chi_{j,m}$, which has support on a neighbourhood containing the segment $I_j$,
    \item $\chi_{\Ical,m}\coloneqq 1-\sum_{j=1}^N\chi_{j,m}$, which has support away from $\Gamma$,
    \item $\zeta_{p,m}$, which has support around the boundary point $p$, and
    \item $\eta_{j,m}:=\chi_{j,10} -\sum_{p\in \partial I_j}\zeta_{p,m}$, which has support in a tubular neighborhood around the interior of the interval $I_j$.
\end{enumerate}

To precisely define the functions in items (1) and (3) we need to fix a few
additional parameters. First, choose \(\delta>0\), using the convention that
\(\min_{j\ne k}\dist(I_j,I_k)=+\infty\) when \(N=1\), and, with \(\gamma\)
fixed as in \eqref{eq: exponential-decay-params}, choose \(\alpha>0\) such that
\begin{equation}\label{eq: segment-gluing-params}
    \delta<\frac{1}{24}
    \min\Big\{\min_j |I_j|,
    \min_{j\ne k}\dist(I_j,I_k)\Big\} \qquad \text{and} \qquad
    \alpha\in\left(\frac{1}{2+\gamma},\frac12\right).
\end{equation}
Roughly speaking, $\delta$ isolates the different segments, while $\alpha$
controls the order at which the support of the cut-off functions shrinks. The
lower bound on \(\alpha\) is the one used in the endpoint moment estimates
below. In addition, we need to fix a smooth non-increasing function $\varrho\colon\R\to[0,1]$ such that
$\varrho\equiv1$ on $(-\infty,1]$, $\varrho\equiv0$ on $[2,+\infty)$ and satisfying $|\varrho'|\leq1$.

For every integer \(m\geq1\), choose smooth cut-off functions with the support
and nesting properties described below. The following formulas should be
understood as schematic descriptions of the corresponding regions:
\[
    \chi_{j,m}(q)
    \sim
    \varrho\left(\frac{\dist(q,I_j)}{\delta\eps^\alpha}-m\right), \qquad \forall q \in \R^2_{\p}.
\]
and, given \(p\in \partial I_j\),
\[
    \zeta_{p,m}(q)\sim \chi_{j,10}(q)\cdot \varrho\left(\frac{x}{\delta\eps^\alpha}-m\right), \qquad \forall q=Y_p(x,z) \in \R^2_{\p}.
\]

Then
\begin{align*}
    \{\chi_{j,m}=1\}
    &=
    \{q \in \R^2_{\p}: \dist(q,I_j)\leq (m+1)\delta\eps^\alpha\}\\
    \{0<\chi_{j,m}<1\}
    &=
    \{q \in \R^2_{\p}: (m+1)\delta\eps^\alpha<\dist(q,I_j)<(m+2)\delta\eps^\alpha\}\\
    \{\zeta_{p,m}=1\}&=\{q =Y_p(x,z) \in \R^2_{\p}: x\leq (m+1)\delta\eps^\alpha\}\cap \{\chi_{j,10}=1\} \\
    \{0<\zeta_{p,m}<1\}
    &=
    \{q =Y_p(x,z) \in \R^2_{\p}: (m+1)\delta\eps^\alpha<x<(m+2)\delta\eps^\alpha\}\cap \{\chi_{j,10}=1\}
\end{align*}

We summarise a few immediate properties of these cut-off functions. The nesting
property below is used for \(m=1,\dots,9\), while the remaining properties are
used for \(m=1,\dots,10\).
\begin{itemize}
    \item For \(m=1,\dots,9\), $\spt\chi_{j,m}\subset\{\chi_{j,m+1}=1\}$ and
$\chi_{j,m}\chi_{j,m+1}=\chi_{j,m}$.
    \item By the choice of $\delta$, the supports
of $\chi_{j,m}$ and $\chi_{k,m}$ are disjoint for $j\ne k$
    \item The functions $\{\chi_{1,m},\dots,\chi_{N,m},\chi_{\Ical,m}\}$ form a partition of
unity in $\R^2$ while
\[
    \chi_{j,10}=\eta_{j,m}+\sum_{p\in \partial I_j}\zeta_{p,m}
\]
acts like a partition of unity on the region $\spt \chi_{j,9}$ surrounding the segment $I_j$.
    \item For \(k=1,2\), the cut-offs are chosen so that
    \[
        |\nabla^k\chi_{j,m}|+|\nabla^k\zeta_{p,m}|+|\nabla^k\eta_{j,m}|
        \lesssim
        \eps^{-k\alpha}.
    \]
\end{itemize}

For later use, choose fixed constants \(0<c_\delta<C_\delta\), depending only
on the cutoff profiles, such that, for every \(j\) and
\(p\in\partial I_j\),
\[
    \operatorname{pr}_x\left(
    \spt(\chi_{j,4}\nabla\zeta_{p,4})
    \cup\spt(\chi_{j,4}\Delta\zeta_{p,4})
    \right)
    \subset
    I_\eps\coloneqq
    [c_\delta\eps^\alpha,C_\delta\eps^\alpha].
\]
Here \(\operatorname{pr}_x\) denotes projection onto the inward boundary
coordinate in the chart \(Y_p\).

\subsection{Gluing conventions}
\label{subs: gluing-conventions}

Let \(\sigma>0\) be the common decay rate fixed after
\Cref{lem: estim U+}. Fix an exponential weight
\(0<\lambda<\min\{\sigma,\sqrt{\kappa_W}\}\) and a correction-size exponent
\(\beta\geq2\).
Since \(\alpha>1/(2+\gamma)>1/3\), we may choose \(\tau\) and define the
associated correction radius and boundary decay function by
\begin{align}
    2<\tau
    &<\frac{1+\alpha}{1-\alpha},
    \label{eq: boundary-decay-exponent-range}\\
    R_\eps
    &\coloneqq\eps^{\beta+1-\alpha},
    \label{eq: correction-radius}\\
    g_\tau(X)
    &\coloneqq(1+X)^{2-\tau},
    \qquad X\geq0.
    \label{eq: boundary-decay-function}
\end{align}
The condition \(\tau>2\) makes \(g_\tau\) admissible in the sense of
\Cref{def: adm decay}.

We now fix the bundle conventions used throughout the gluing construction.
Set
\[
    \Omega_\Gamma\coloneqq\R^2_{\p}\setminus\Gamma.
\]
This is a trivializing region for \(\L_{\p}\), and its inverse image in the
sign double cover has two connected components,
\[
    \pi^{-1}(\Omega_\Gamma)
    =
    \widetilde\Omega_+\sqcup\widetilde\Omega_-,
    \qquad
    \iota(\widetilde\Omega_+)=\widetilde\Omega_-.
\]
We fix the labels of these components once and for all.

\begin{definition}[The outer section]\label{def: outer section}
Let \(\widetilde\Ical\) be the odd function on the double cover given by
\(\widetilde\Ical=1\) on \(\widetilde\Omega_+\) and
\(\widetilde\Ical=-1\) on \(\widetilde\Omega_-\), and set
\(\widetilde\Ical=0\) over \(\Gamma\setminus\p\). We denote by \(\Ical\)
the corresponding pointwise section of \(\L_{\p}\).
\end{definition}

The section \(\Ical\) is discontinuous across \(\Gamma\setminus\p\). It will
only be used after multiplication by a cutoff that vanishes near \(\Gamma\).
In either outer trivialization it is represented by a constant pure phase.

For each \(j\in\{1,\dots,N\}\), set
\[
    T_j
    \coloneqq
    X_j\bigl(\operatorname{int}(I_j)\times(-\delta,\delta)\bigr).
\]
The tube \(T_j\) contains no point of \(\p\) and meets no other segment, so
\(\L_{\p}\) is trivial over \(T_j\). We choose the trivialization in which
the representative of \(\Ical\) is \(-1\) below \(I_j\) and \(+1\) above
\(I_j\), relative to the normal \(\nu_j\).

Fix \(p\in\partial I_j\). By the classification of connected double covers
of a punctured disk, there is, up to sign, a bundle identification
\begin{equation}\label{eq: boundary-bundle-identification}
    \mathcal T_p\colon
    \L|_{B_\delta(0)\setminus\{0\}}
    \longrightarrow
    \L_{\p}|_{B_\delta(p)\setminus\{p\}}
\end{equation}
covering \(Y_p\). We choose the sign so that \(U_+\) agrees with the positive
representative of \(\Ical\).

For \(\boldsymbol I\coloneqq(I_1,\dots,I_N)\) and
\(\h=(h_1,\dots,h_N)\), with \(h_j\colon I_j\to\R\), set
\[
    \lVert\h\rVert_{C^{k,\gamma}(\boldsymbol I)}
    \coloneqq
    \sup_{j=1,\dots,N}\lVert h_j\rVert_{C^{k,\gamma}(I_j)},
\]
and let \(C_0^{k,\gamma}(\boldsymbol I)\) be the subspace of tuples whose
components vanish on the boundaries of their intervals.

The shifts used in the gluing construction will be affine in a collar of every
boundary point. To encode this condition, we view \(\eta_{j,4}\) as a function
on \(I_j\) by restriction to the segment and set
\begin{equation}\label{eq: admissible-shift-space}
    \mathcal X_\eps
    \coloneqq
    \left\{
    \h\in C_0^{2,\gamma}(\boldsymbol I):
    \spt h_j''\subset\spt\eta_{j,4}
    \text{ for every }j
    \right\}.
\end{equation}
For \(r>0\), we denote by
\begin{equation}\label{eq: admissible-shift-ball}
    \mathcal B(r)
    \coloneqq
    \left\{
    \h\in\mathcal X_\eps:
    \lVert\h\rVert_{C^{2,\gamma}(\boldsymbol I)}\leq r
    \right\}
\end{equation}
the ball of radius \(r\) in \(\mathcal X_\eps\).

Let \(p\in\partial I_j\), and use the inward coordinate \(x\geq0\), so that
points of \(I_j\) near \(p\) are written as \(p+xe_p\). Since
\(h_j(p)=0\) and \(h_j''=0\) on the component of
\(I_j\setminus\spt\eta_{j,4}\) adjacent to \(p\), every
\(\h\in\mathcal X_\eps\) satisfies
\begin{equation}\label{eq: affine-boundary-collar}
    h_j(p+xe_p)=\theta_{\h}(p)x,
    \qquad
    \theta_{\h}(p)\coloneqq\partial_{e_p}h_j(p),
\end{equation}
throughout this boundary collar. In particular, if \(\h\in\mathcal B(r)\), then
\begin{equation}\label{eq: boundary-angle-bound}
    \lvert\theta_{\h}(p)\rvert\leq r
    \qquad\text{for every }p\in\p.
\end{equation}

The preceding coordinates will also be used after adapting them to a shift.

\begin{definition}[Adapted coordinate representatives]
\label{def: adapted-coordinate-representatives}
Let \(\h\in\mathcal X_\eps\), and let \(Q\) be a local section near
\(I_j\). Once a segment trivialization has been chosen, we denote its
representative in the shifted interior coordinates by
\begin{equation}\label{eq: interior-error-pullback}
    \widetilde Q_{j,\h}(x,t)
    \coloneqq
    Q\bigl(X_j(x,h_j(x)+\eps t)\bigr).
\end{equation}
Near \(p\in\partial I_j\), using the fixed bundle identification
\(\mathcal T_p\) covering \(Y_p\), we denote the representative
on the fixed boundary model by
\begin{equation}\label{eq: boundary-error-pullback}
    \widehat Q_{p,\h}(X,Z)
    \coloneqq
    \mathcal T_p^{-1}
    Q\bigl(Y_p(\eps R_{-\theta_{\h}(p)}(X,Z))\bigr).
\end{equation}
Here the second formula uses the lifted similarity convention above.
For a scalar coefficient, the same notation denotes ordinary composition,
without \(\mathcal T_p^{-1}\). Thus tildes with subscripts \(j,\h\) refer to
the shifted interior variables, whereas hats with subscripts \(p,\h\) refer
to the fixed, unrotated boundary variables. Whenever a localized
representative is measured on the full coordinate model, it is extended by
zero outside its coordinate domain.
\end{definition}

We use \(\mathsf C\) for physical Fermi norms and \(\mathcal C\) for stretched
norms, reserving \(C_\eps^{k,\gamma}\) for isotropic ball norms.

\begin{definition}[Physical product-space norms]
Let \(I\subset\R\) be an interval, let \(D\subset I\times\R\), and let
\(\psi=\psi(x,z)\) be defined on a neighborhood of \(D\). Set
\begin{align*}
    [\psi]_{\gamma,\eps;D}
    &\coloneqq
    \sup_{\substack{(x,z),(x',z')\in D\\(x,z)\ne(x',z')}}
    \frac{|\psi(x,z)-\psi(x',z')|}
    {(|x-x'|+\eps^{-1}|z-z'|)^\gamma},\\
    \lVert\psi\rVert_{\mathsf C^{k,\gamma}_\eps(D)}
    &\coloneqq
    \sum_{m+j\leq k}\eps^m
    \left(
    \lVert\partial_z^m\partial_x^j\psi\rVert_{L^\infty(D)}
    +
    [\partial_z^m\partial_x^j\psi]_{\gamma,\eps;D}
    \right),\\
    \lVert\psi\rVert_{\mathsf C^{k,\gamma}_{\eps,\lambda}(D)}
    &\coloneqq
    \lVert\cosh(\lambda z/\eps)\psi\rVert_{\mathsf C^{k,\gamma}_\eps(D)}.
\end{align*}
\end{definition}
Since \(\cosh(\lambda z/\eps)\asymp e^{\lambda|z|/\eps}\), the weighted
physical norm is the product-coordinate analogue of the section norm relative
to \(I\times\{0\}\) from \Cref{def: exp decay rel set}.

\begin{definition}[Stretched interior product-space norm]
\label{def: stretched interior norm}
Let \(I\subset\R\) be an interval, let \(D\subset I\times\R\), and let
\(v=v(x,t)\) be defined on a neighborhood of \(D\). Set
\begin{align*}
    d_\eps\bigl((x,t),(x',t')\bigr)
    &\coloneqq
    \eps^{-1}|x-x'|+|t-t'|,\\
    [v]_{\gamma,d_\eps;D}
    &\coloneqq
    \sup_{\substack{(x,t),(x',t')\in D\\(x,t)\ne(x',t')}}
    \frac{|v(x,t)-v(x',t')|}
    {d_\eps((x,t),(x',t'))^\gamma},\\
    \lVert v\rVert_{\mathcal C_\eps^{k,\gamma}(D)}
    &\coloneqq
    \sum_{a+b\leq k}\eps^a
    \left(
    \lVert\partial_x^a\partial_t^b v\rVert_{L^\infty(D)}
    +
    [\partial_x^a\partial_t^b v]_{\gamma,d_\eps;D}
    \right),\\
    \lVert v\rVert_{\mathcal C_{\eps,\lambda}^{k,\gamma}(D)}
    &\coloneqq
    \lVert\cosh(\lambda t)v\rVert_{\mathcal C_\eps^{k,\gamma}(D)}.
\end{align*}
\end{definition}
This norm is adapted to
\(\partial_{tt}+\eps^2\partial_{xx}-W''(H(t))\), since \(X=x/\eps\) gives
the usual uniformly elliptic variables \((X,t)\). When \(D\) is clear, we
suppress it from the notation for these norms and seminorms.

We also fix the normalization of the projection onto the heteroclinic kernel.
Set
\begin{equation}\label{eq: heteroclinic-projection}
    A_H\coloneqq\int_\R(H'(t))^2\,dt,
    \qquad
    \mathfrak p_j(G)(x)
    \coloneqq
    \frac{1}{A_H}\int_\R G(x,t)H'(t)\,dt .
\end{equation}

\section{Existence of sections modelling a family of line segments}\label{sec: LS}

In this section we carry out the gluing construction for a section $u$ solving
\cref{eq:section-allen-cahn} and whose nodal set is close to a finite family of
line segments. First, an approximate section $\omega_{\h}$ is obtained by placing a rotated boundary
model solution \(U\) at the endpoints and the heteroclinic \(H\) along the
interior of each segment, with its position determined by a shift \(\h\). The
main task is then to find a perturbation section \(\varphi\) such that
\(u=\omega_{\h}+\varphi\) is a genuine Allen--Cahn section on the associated
line bundle.

For the outer correction, fix a sufficiently small
\(\widetilde\lambda\in(0,\lambda)\) and set
\[
    r_{\mathrm{out}}(\eps)
    \coloneqq
    \exp\left(-\frac{\widetilde\lambda}{\eps^{1-\alpha}}\right),
\]
\[
    \mathscr B_{\mathrm{out}}(\eps)
    \coloneqq
    \left\{
        \psi\in
        C_{\Gamma,\eps,3\lambda/4}^{2,\gamma}
        (\R^2_{\p},\L_{\p}):
        \lVert\psi\rVert_{C_{\Gamma,\eps,3\lambda/4}^{2,\gamma}
        (\R^2_{\p},\L_{\p})}
        \le r_{\mathrm{out}}(\eps)
    \right\}.
\]
With these conventions, we prove in this section the following existence
result.

\begin{theorem}\label{thm: sol on intervals}
With the notation and parameters fixed above, there exists
\(\eps_*>0\), depending only on the fixed data and on
\(\delta,\alpha,\lambda,\gamma,\beta,\tau\), such that the following holds. For every
\(0<\eps\leq\eps_*\), there exists a smooth section
\[
    u=u_\eps\colon\R^2_{\p}\to\L_{\p}
\]
solving \eqref{eq:section-allen-cahn} whose nodal set is
close to the prescribed intervals in the Hausdorff distance. More precisely:
\[
    \dist_{\mathcal H}
    \left(\{u=0\},\Gamma\right)
    \lesssim \eps^\beta .
\]
Moreover, the following asymptotic estimates hold near the segments:
\begin{enumerate}
    \item For each \(j=1,\dots,N\), there is
    \(h_j\in C^{2,\gamma}_0(I_j,\R)\) such that
    \(\lVert h_j\rVert_{C^{2,\gamma}(I_j)}\lesssim \eps^\beta\) and
    \begin{equation*}
        \left\lVert
        u(x,z)
        -
        H\left(\frac{z-h_j(x)}{\eps}\right)
        \right\rVert_{C^{2,\gamma}_{I_j,\eps,\lambda}
        (\spt\eta_{j,1},\L_{\p})}
        \lesssim R_\eps,
    \end{equation*}
    where \((x,z)\) denote the Fermi coordinates from
    \eqref{eq: coordinates Xj}, and the \(H\)-term is interpreted as a section
    using the trivialization of \(\L_{\p}\) on \(\spt\eta_{j,1}\).

    \item For each \(j\) and each boundary point \(p\in\partial I_j\), define
    the angle
    \(
        \theta(p)\coloneqq \partial_{e_p}h_j(p).
    \)
    Then
    \begin{equation*}
        \sup_{p\in\p}|\theta(p)|
        \lesssim \eps^\beta
    \end{equation*} and
    writing \(R_{\theta(p)}\) for the clockwise rotation by angle \(\theta(p)\),
    \begin{equation*}
        \left\lVert
        u(x,z)
        -
        U\circ R_{\theta(p)} \left(\frac{x}{\eps},\frac{z}{\eps}\right)
        \right\rVert_{C_{*,\eps}^{2,\gamma}(\spt\zeta_{p,1},\L_{\p})}
        \lesssim R_\eps,
    \end{equation*}
    where \((x,z)\) are the boundary coordinates from
    \Cref{eq: coordinates Yp} around \(p\), and the model term is transported
    to \(\L_{\p}\) using the boundary identification \(\mathcal T_p\) from
    \eqref{eq: boundary-bundle-identification}.
\end{enumerate}
\end{theorem}
The proof separates the correction into boundary, interior, and outer
responses, makes these responses compatible for each fixed shift, and finally
adjusts the shift to remove the remaining translational obstruction.

With this strategy in mind, we now define the approximate solution and state
the results used in the proof. Their proofs are given in
\Cref{sec: ansazt,sec: proofs}.

\subsection{The approximate solution and its error of approximation}\label{subs: error of approximation}

For a section \(w\) of \(\L_{\p}\), set
\begin{equation}\label{eq: EoA}
    S(w)\coloneqq \eps^2\Delta w-W'(w).
\end{equation}

Fix \(K>0\), independently of \(\eps\); its value will be chosen in the final
fixed-point argument. All smallness thresholds in the intermediate results
below may depend on this fixed \(K\). Let
\begin{equation}\label{eq: shift-size-compatibility}
    \h=(h_1,\dots,h_N)\in\mathcal B(K\eps^\beta).
\end{equation}
Using the fixed segment trivializations and the boundary identifications
\(\mathcal T_p\), define the approximate solution by
\begin{equation}\label{eq: section-valued-global-ansatz}
\begin{split}
    \omega_{\h}
    ={}&\chi_{\Ical,5}\Ical\\
    &+\sum_{j=1}^N\chi_{j,5}
    \left[
        \eta_{j,3}
        H\left(\frac{z-h_j(x)}{\eps}\right)
        +
        \sum_{p\in\partial I_j}
        \zeta_{p,3}
        \mathcal T_p
        \left(
            U_{\theta_{\h}(p)}
            \left(\frac{x_p}{\eps},\frac{z_p}{\eps}\right)
        \right)
    \right].
\end{split}
\end{equation}
Here \((x,z)\) are the Fermi coordinates on the \(j\)-th segment tube and
\((x_p,z_p)=Y_p^{-1}(q)\) are the boundary coordinates near \(p\). Each
localized term is extended by zero outside its coordinate neighborhood. The
cutoff nesting makes \eqref{eq: section-valued-global-ansatz} a smooth global
section of \(\L_{\p}\).

\begin{proposition}[The approximate solution]\label{prop: expansion error}
    The error of approximation \(S(\omega_{\h})\) is nonzero only in
    \[
        \bigcup_{j=1}^N
        \left(
            \{0<\chi_{j,5}<1\}\cup\spt\eta_{j,3}
        \right).
    \]
    Moreover,
    \begin{equation}\label{eq: estimate S away from Ij}
        \lVert S(\omega_{\h})\rVert_
        {C^{0,\gamma}_{I_j,\eps,\lambda}
        (\{0<\chi_{j,5}<1\},\L_{\p})}
        \lesssim 1,
    \end{equation}
    \begin{equation}\label{eq: Bj-error-expansion}
        \lVert S(\omega_{\h})\rVert_
        {C^{0,\gamma}_{I_j,\eps,\lambda}
        (\spt\eta_{j,3},\L_{\p})}
        \lesssim \eps^{1+\beta}.
    \end{equation}
    After removing the leading tangential term, one also has
    \begin{equation}\label{eq: improved Linfty error}
        \left\lVert
        S(\omega_{\h})
        +
        \eps\eta_{j,3} h_j''(x)H'\left(\frac{z-h_j(x)}{\eps}\right)
        \right\rVert_{L^\infty(\spt\eta_{j,3},\L_{\p})}
        \lesssim \eps^{1+\beta+\alpha\gamma}.
    \end{equation}
\end{proposition}

The proof is given in \Cref{sec: ansazt}.

For the remainder of the fixed-shift construction, we fix an admissible
\(\h\) and abbreviate \(\omega=\omega_{\h}\).

\subsection{Splitting the correction equation}
\label{subs: correction-splitting}
We look for a solution $u\colon\R^2_{\p}\to\L_{\p}$ of the form
$u=\omega+\varphi$,
where $\varphi$ is a section of $\L_{\p}$. In other words, we aim at solving $S(\omega+\varphi)=0$, where $S$ is given by \eqref{eq: EoA}.
From now on, we will write $S(\omega+\varphi)=0$ as the following linear-nonlinear scheme:
\begin{equation}\label{eq: nonlinear scheme}
	L\varphi=-S(\omega)-\mathcal N_\omega(\varphi)
\end{equation}
where, for any base section \(w\),
\begin{align*}
    \begin{split}
        L\varphi&\coloneqq \eps^2\Delta\varphi-W''(\omega)\varphi,\\
        \mathcal N_w(\varphi)&\coloneqq
        S(w+\varphi)-S(w)-\left[\eps^2\Delta\varphi-W''(w)\varphi\right].
    \end{split}
\end{align*}

This expression is analogous to a Taylor expansion of $S(\omega+\varphi)$ around $\omega$, in which $L$ is the linearised operator of $S$ about $\omega$ and $\mathcal N_\omega(\varphi)$ is a nonlinear remainder which is quadratic on $\varphi$.

Roughly speaking, one would like to solve \eqref{eq: nonlinear scheme} via a fixed-point argument, by first inverting the operator $L$ and then showing that the map $\varphi\mapsto -L^{-1}[S(\omega)+\mathcal N_\omega(\varphi)]$ is a contraction in a sufficiently small ball of some appropriate Banach space. Unfortunately, this cannot be done in a straightforward way because of the presence of small (in $\eps$) eigenvalues for $L$ that make the estimate for $L^{-1}$ degenerate as $\eps\to 0$.

To tackle the solvability of \eqref{eq: nonlinear scheme} we break down the right-hand side into parts supported in different regimes. First, we decompose into an \emph{outer part}, which is supported away from the segment, and an \emph{interface part}, which is supported near the segment.

More precisely, we write the perturbation as
\begin{equation*}
	\varphi=\psi+\sum_{j=1}^N\chi_{j,2}\phi_j
\end{equation*}
where $\phi_1,\dots,\phi_N,\psi\colon\R_{\p}^2\to\L_{\p}$ are global sections.
Equation \eqref{eq: nonlinear scheme} then reads
\begin{equation}\label{eq: linear unsplit}
	\begin{split}
		\eps^2\Delta\psi
        &+
        \sum_{j=1}^N
        \left(
        \eps^2 \chi_{j,2}\Delta \phi_j
        +2\eps^2\nabla\chi_{j,2}\cdot\nabla\phi_j
        +\eps^2\phi_j \Delta\chi_{j,2}
        \right)\\
        &-
        W''(\omega)
        \left(
        \psi+\sum_{j=1}^N\chi_{j,2}\phi_j
        \right)
	=-S(\omega)-\mathcal N_\omega\left(\psi+ \sum_{j=1}^N\chi_{j,2}\phi_j\right).
	\end{split}
\end{equation}
Denoting $\phi=(\phi_1,\dots,\phi_N)$, we see that equation \eqref{eq: linear unsplit} can be solved by first solving the \emph{outer-interface system}:
\begin{equation}\label{eq: inner-outer system}
\begin{cases}
	\eps^2\Delta \phi_j-W''(\omega)\phi_j=E_j(\phi_j,\psi)& \text{on }\spt\chi_{j,2},\  j=1,\dots,N\\
	\eps^2\Delta\psi-V\psi=E_\out(\phi,\psi)& \text{on }\R^2_{\p}
\end{cases}
\end{equation}
where
\begin{equation*}
	V\coloneqq \chi_{\Ical,1}W''(\omega)+\kappa_W(1-\chi_{\Ical,1})
\end{equation*}
and the sections $E_j(\phi_j,\psi)\colon\R^2_{\p}\cap \spt\chi_{j,2}\to\L_{\p}$ and $E_\out(\phi,\psi)\colon\R^2_{\p}\to\L_{\p}$ are defined by
\begin{equation}
\begin{split}
	E_j(\phi_j,\psi)&\coloneqq
    -S(\omega)
    -\chi_{j,1}\mathcal N_\omega(\psi+\phi_j)
    +\chi_{j,1}[W''(\omega)-\kappa_W]\psi,\label{eq: Ej}\\
    E_\out(\phi,\psi)&\coloneqq
    -\chi_{\Ical,2}S(\omega)
    -\chi_{\Ical,1}\mathcal N_\omega\left(\psi+\sum_{k=1}^N\chi_{k,2}\phi_k\right)\\
    &\quad
    -\sum_{k=1}^N
    \left(
    2\eps^2\nabla\chi_{k,2}\cdot\nabla\phi_k
    +\eps^2\phi_k \Delta\chi_{k,2}
    \right).
\end{split}
\end{equation}
We emphasise that with this notation $E_j$ is a local section. The terms $S(\omega)$ and $W''(\omega)$ in its expression are understood after restricting
the global ansatz to $\spt\chi_{j,2}$. When the inner equation is inserted back into
\eqref{eq: linear unsplit}, the factor $\chi_{j,2}$ in the decomposition of
$\varphi$ localizes the term $S(\omega)$.

\subsection{The boundary-interior system}\label{subs: bdy-int system}

For the local problems we first hold the outer correction \(\psi\) fixed and
consider the system
\begin{equation}\label{eq: system reduced segments}
	\eps^2\Delta \phi_j-W''(\omega)\phi_j=E_j(\phi_j,\psi)\quad
    \text{on }\spt\chi_{j,2},\quad j=1,\dots,N.
\end{equation}
Here \(E_j\) is given by \eqref{eq: Ej}. After solving these local equations,
we impose compatibility with the outer equation.

To solve \eqref{eq: system reduced segments}, we now further decompose each $\phi_j$ in order to exploit the linear theories for both $U$ and $H$. For every $j=1,\dots,N$, the new functions are
\[
\left\{
\begin{aligned}
    \phi_j^{\mathrm{int}}&\colon I_j\times\R\to\R,\\
    \phi_p&\colon \R^2_*\to\L,\qquad p\in\partial I_j.
\end{aligned}
\right.
\]
Here \(\L\) is the one-puncture line bundle over \(\R^2_*\) fixed in
\Cref{subs: line bundle}. We interpret $\phi_j^{\mathrm{int}}$ as the interior
correction and each $\phi_p$ as the correction at the boundary point $p$. If
\(q\in\spt\chi_{j,2}\), we write
\[
    q=X_j(x,z),\qquad
    (x_p,z_p)=Y_p^{-1}(q),\quad p\in\partial I_j.
\]
These are trial variables in the algebraic decomposition; after solving the
response problems below, their values will be denoted by \(v_j\) and \(b_p\),
respectively. We look for $\phi_j$ of the form
\begin{equation}\label{eq: decomp phi}
	\phi_j(q)=
    \eta_{j,1}(q)\phi_j^{\mathrm{int}}\left(x,\frac{z-h_j(x)}{\eps}\right)
    +
    \sum_{p\in\partial I_j}\zeta_{p,4}(q)\phi_p(x_p,z_p),
    \quad q\in\spt \chi_{j,2}.
\end{equation}
As in \Cref{prop: expansion error}, the interior term is interpreted through the trivialization of $\L_{\p}$ used in the Fermi coordinates around $I_j$, while the $p$-summand is interpreted through the boundary coordinates $Y_p$. In the formulas below, the boundary equation at \(p\) is read after pulling all coefficients back by \(Y_p\), and the interior equation is read in the coordinates \(q=X_j(x,z)\). Actually, we will consider the change of coordinates
\[
    t=\frac{z-h_j(x)}{\eps}.
\]
and express the equations in variables $(x,t)$.
For the following calculation we fix $j$ and collect the two boundary corrections in the tuple
\[
    \phi_j^{\mathrm{bd}}\coloneqq (\phi_p)_{p\in\partial I_j}.
\]
It will also be useful to separate the leading curvature contribution in \(E_j\). Using \eqref{eq: Bj-error-expansion}, we write
\begin{equation}\label{eq: local remainder Rj}
\begin{split}
    \mathcal R_j(\phi_j,\psi)
    &\coloneqq
    -\left[
    S(\omega)+\eps\eta_{j,3}h_j''(x)H'(t)
    \right]\\
    &\quad
    -\chi_{j,1}\mathcal N_\omega(\psi+\phi_j)
    +\chi_{j,1}\left[W''(\omega)-\kappa_W\right]\psi .
\end{split}
\end{equation}
Thus
\[
    E_j(\phi_j,\psi)
    =
    \eps\eta_{j,3}h_j''(x)H'(t)
    +
    \mathcal R_j(\phi_j,\psi).
\]
The first term in \(\mathcal R_j\) is lower order by
\eqref{eq: improved Linfty error}, the second is quadratic in the perturbation,
and the last is exponentially small after solving the outer equation.

For \(v=v(x,t)\), define
\[
\begin{split}
    \mathfrak B_{j,h_j}[v]
    &\coloneqq
    h_j'^2\partial_{tt}v
    -2\eps h_j'\partial_{xt}v
    -\eps h_j''\partial_tv,\\
    m_j^{\mathrm{bd}}[v]
    &\coloneqq
    \frac{1}{A_H}\left(
    h_j'^2\int_\R vH'''
    +2\eps h_j'\int_\R \partial_xv H''
    \right),\\
    m_j^{\mathrm{int}}[v]
    &\coloneqq
    \frac{\eps h_j''}{A_H}\int_\R vH''.
\end{split}
\]
When the shift varies, we write
\(m_{j,h_j}^{\mathrm{bd}}\) and \(m_{j,h_j}^{\mathrm{int}}\) to display this
dependence.
All integrals in these definitions are taken in \(t\). Integration by parts
gives
\[
    \frac{1}{A_H}\int_\R\mathfrak B_{j,h_j}[v]H'
    =m_j^{\mathrm{bd}}[v]+m_j^{\mathrm{int}}[v].
\]
We therefore set
\[
    \mathfrak B_{j,h_j}^{\perp}[v]
    \coloneqq
    \mathfrak B_{j,h_j}[v]
    -\bigl(m_j^{\mathrm{bd}}[v]+m_j^{\mathrm{int}}[v]\bigr)H'
\]
so that \(\mathfrak B_{j,h_j}^{\perp}[v]\) is orthogonal to \(H'\). The split
isolates \(m_j^{\mathrm{bd}}\), which involves only \(v\) and \(\partial_xv\),
for distribution between the boundary and interior equations;
\(m_j^{\mathrm{int}}\), which contains \(h_j''\), remains in the interior
equation. We then define
\[
    \widetilde L_{H,h_j}v
    \coloneqq
    \partial_{tt}v+\eps^2\partial_{xx}v-W''(H(t))v
    +\mathfrak B_{j,h_j}^{\perp}[v].
\]

The next lemma rewrites the equation for $\phi_j$ as one boundary equation for
each \(p\in\partial I_j\), together with one interior equation depending on the
full tuple $\phi_j^{\mathrm{bd}}$. The cut-offs are chosen so that the
commutator terms produced by differentiating one correction are supported where
the other correction is already defined, and the remainder \(\mathcal R_j\) is
distributed using \(\eta_{j,4}+\sum_{p\in\partial I_j}\zeta_{p,4}=1\) on the
segment region. We record the system explicitly and postpone its verification
to \S\ref{sec: inner-bdy decomp}.

\begin{lemma}[The boundary-interior system]\label{lem: inner-bdy decomp}
    Fix an outer correction \(\psi\). The system
    \eqref{eq: system reduced segments} can be solved, via the decomposition
    \eqref{eq: decomp phi}, if the following system
    \begin{align}
		\eps^2\Delta\phi_p-W''(U_{\theta_{\h}(p)}(\cdot/\eps))\phi_p&= F_p(\phi_p,\phi_j^{\mathrm{int}};\psi)\quad \text{on }\R^2_*,\quad p\in\partial I_j,\label{eq: boundary eq2}\\
		\widetilde L_{H,h_j}\phi_j^{\mathrm{int}}&= F_j(\phi_j^{\mathrm{bd}},\phi_j^{\mathrm{int}};\psi)\quad \text{on }I_j\times \R\label{eq: interior eq2}
	\end{align}
is solved. The right-hand sides are
\begin{equation*}
    \begin{split}
F_p(\phi_p,\phi_j^{\mathrm{int}};\psi)&=-2\eps^2\chi_{j,4}\zeta_{p,5}\nabla\eta_{j,1}\cdot\nabla\phi_j^{\mathrm{int}}-\eps^2\chi_{j,4}\zeta_{p,5}\Delta\eta_{j,1}\phi_j^{\mathrm{int}}\\
        &+\chi_{j,4}\zeta_{p,5}\mathcal R_j(\phi_j,\psi)\\
	&+\chi_{j,4}\zeta_{p,5}\left(W''(\omega)-W''(U_{\theta_{\h}(p)}(\cdot/\eps))\right)\phi_p\\
        &+\chi_{j,4}\zeta_{p,5}\eta_{j,1}\left(W''(\omega)-W''(H(t))\right)\phi_j^{\mathrm{int}}\\
    &-\chi_{j,4}\zeta_{p,5}\eta_{j,1}
    m_j^{\mathrm{bd}}[\phi_j^{\mathrm{int}}]H'(t),\\
        F_j(\phi_j^{\mathrm{bd}},\phi_j^{\mathrm{int}};\psi)&=\eps\chi_{j,4}\eta_{j,3}h_j''(x)H'(t)\\
        &+\chi_{j,4}\eta_{j,4}\mathcal R_j(\phi_j,\psi)\\
        &-2\eps^2\chi_{j,4}\sum_{p\in\partial I_j}\nabla\zeta_{p,4}\cdot\nabla\phi_p-\eps^2\chi_{j,4}\sum_{p\in\partial I_j}\Delta\zeta_{p,4}\phi_p\\
        &-\chi_{j,4}\eta_{j,4}
        m_j^{\mathrm{bd}}[\phi_j^{\mathrm{int}}]H'(t)
        -\chi_{j,4}m_j^{\mathrm{int}}[\phi_j^{\mathrm{int}}]H'(t).
    \end{split}
\end{equation*}
Here \(\phi_j\) inside \(\mathcal R_j\) denotes the section
reconstructed by \eqref{eq: decomp phi}. In the boundary equation at \(p\),
every term in \(F_p\) is first restricted to the boundary neighborhood, pulled
back by \(Y_p\), and transported by \(\mathcal T_p^{-1}\) to the
one-puncture model bundle. On \(\spt\zeta_{p,5}\), the contribution from the other boundary point of
\(I_j\) vanishes. Thus the displayed \(F_p\) depends only on
\((\phi_p,\phi_j^{\mathrm{int}},\psi)\), as its notation indicates.
\end{lemma}

\subsection{Solving the boundary part}

The next two results, based on \Cref{lem: boundary invertibility}, concern one
boundary equation in \eqref{eq: boundary eq2}. We use the following common
data. Fix \(\eps>0\), \(j\in\{1,\dots,N\}\), \(p\in\partial I_j\), and
\(\h\in\mathcal B(K\eps^\beta)\), and let \(h_j\) be the \(j\)-th component
of \(\h\). Let
\(v\in\mathcal C_{\eps,\lambda}^{2,\gamma}(I_j\times\R)\) and
\(\psi\in\mathscr B_{\mathrm{out}}(\eps)\) satisfy
\begin{equation}\label{eq: orth v}
    \int_\R v(x,t)H'(t)\,dt=0,
    \qquad x\in I_j,
\end{equation}
and
\begin{equation}\label{eq: assumption phi}
    \lVert v\rVert_{\mathcal C_{\eps,\lambda}^{2,\gamma}(I_j\times\R)}
    \leq R_\eps.
\end{equation}
We set \(\ell_\theta\coloneqq R_{-\theta}\ell_+\) and write
\(b_p=\mathscr B_{p,h_j}(v,\psi)\) for the boundary response constructed
below. Its adapted representative is
\[
    \widehat b_p(Y)
    \coloneqq
    b_p(\eps R_{-\theta_{\h}(p)}Y).
\]
In the inward coordinate \(x\geq0\), abbreviate
\(h_j(p+xe_p)\) by \(h_j(x)\), set
\[
    \overline X_h(x)
    \coloneqq
    \eps^{-1}\bigl(
    x\cos\theta_{\h}(p)+h_j(x)\sin\theta_{\h}(p)
    \bigr),
\]
and define the physical-coordinate tangential part by
\[
    b_p^\top(x)
    \coloneqq
    \int_\R
    \widehat b_p\bigl(\overline X_h(x),Z\bigr)H'(Z)\,dZ,
\]
with the scalar representative and sign convention of
\Cref{def: Hprime-component}. When the limit exists, we denote the
corresponding constant at infinity by
\[
    b_{p,\infty}
    \coloneqq
    \lim_{X\to+\infty}\int_\R \widehat b_p(X,Z)H'(Z)\,dZ.
\]

\begin{proposition}[Existence and estimates]
\label{prop: invert boundary}
	There exists \(\eps_*>0\), depending only on the fixed data and on the
    preceding parameter choices, such that, for every \(0<\eps\leq\eps_*\)
    and every choice of the data above, the \(p\)-th boundary equation
    \eqref{eq: boundary eq2} admits a unique solution
    \(b_p=\mathscr B_{p,h_j}(v,\psi)\) in the closed ball
    \[
        \lVert b_p\rVert_{C_{\ell_{\theta_{\h}(p)},\eps,\lambda}^{2,\gamma}
        (\R^2_*,\L)}
        \leq R_\eps,
    \]
    and this solution satisfies
	\begin{equation}\label{eq: size phi_b}
		\lVert b_p\rVert_{C_{\ell_{\theta_{\h}(p)},\eps,\lambda}^{2,\gamma}(\R^2_*,\L)}
        \lesssim
        \eps^{1-\alpha}
        \lVert v\rVert_{\mathcal C_{\eps,\lambda}^{2,\gamma}(I_j\times \R)}
        +\eps^{2\beta-(1-\alpha)\tau}.
	\end{equation}
    In particular, the right-hand side of \eqref{eq: size phi_b} is
    \(o(R_\eps)\).
    Moreover, the constant \(b_{p,\infty}\) exists, and on the interval
    \(I_\eps\) fixed in \Cref{subs: cutoff},
	\begin{equation}\label{eq: decays tangential part}
	\begin{split}
        |b_{p,\infty}|
        &+\lVert b_p^\top\rVert_{L^\infty(I_\eps)}
        +\eps^{\alpha\gamma}[b_p^\top]_{\gamma,I_\eps}\\
        &+\eps^\alpha
        \lVert\partial_xb_p^\top\rVert_{L^\infty(I_\eps)}
        +\eps^{\alpha(1+\gamma)}
        [\partial_xb_p^\top]_{\gamma,I_\eps}
        \lesssim \eps^{2\beta-2+2\alpha}.
	\end{split}
	\end{equation}
\end{proposition}

\begin{proposition}[Lipschitz dependence]
\label{prop: invert boundary dependence}
    Let \(0<\eps\leq\eps_*\), and let
    \((\h^i,v^i,\psi^i)\), \(i=1,2\), be two choices of the data above for
    the same \(j\) and \(p\). Set
    \[
        \theta_i\coloneqq \theta_{\h^i}(p),
        \qquad
        \widehat b_p^i(Y)
        \coloneqq
        \mathscr B_{p,h_j^i}(v^i,\psi^i)(\eps R_{-\theta_i}Y),
        \qquad i=1,2.
    \]
    Then
	\begin{equation}\label{eq: Lip phib}
	\begin{split}
        \lVert\widehat b_p^1-\widehat b_p^2\rVert_{C_{\ell_+,\lambda}^{2,\gamma}(\R^2_*,\L)}
		&\lesssim
        \eps^{1-\alpha}
        \lVert v^1-v^2\rVert_{\mathcal C_{\eps,\lambda}^{2,\gamma}(I_j\times\R)}\\
		&+\eps^{-(1-\alpha)\tau}
        \lVert\psi^1-\psi^2\rVert_
        {C_{\Gamma,\eps,3\lambda/4}^{2,\gamma}(\R^2_{\p},\L_{\p})}\\
		&+\eps^{\beta-(1-\alpha)\tau}
        \lVert h_j^1-h_j^2\rVert_{C^{2,\gamma}(I_j)}.
	\end{split}
	\end{equation}
    If \(b_p^{\top,i}\) and \(b_{p,\infty}^i\) denote the corresponding
    tangential parts and constants, then
    \begin{equation}\label{eq: tangential dependence phib}
    \begin{split}
        &|b_{p,\infty}^1-b_{p,\infty}^2|
        +\lVert b_p^{\top,1}-b_p^{\top,2}\rVert_{L^\infty(I_\eps)}
        +\eps^{\alpha\gamma}
        [b_p^{\top,1}-b_p^{\top,2}]_{\gamma,I_\eps}\\
        &\quad
        +\eps^\alpha
        \lVert\partial_xb_p^{\top,1}-\partial_xb_p^{\top,2}\rVert_
        {L^\infty(I_\eps)}
        +\eps^{\alpha(1+\gamma)}
        [\partial_xb_p^{\top,1}-\partial_xb_p^{\top,2}]_{\gamma,I_\eps}\\
        &\qquad\lesssim
        \eps^{(1-\alpha)(\tau-1)}
        \lVert v^1-v^2\rVert_{\mathcal C_{\eps,\lambda}^{2,\gamma}(I_j\times\R)}\\
        &\qquad\quad
        +\eps^{-2(1-\alpha)}
        \lVert\psi^1-\psi^2\rVert_
        {C_{\Gamma,\eps,3\lambda/4}^{2,\gamma}(\R^2_{\p},\L_{\p})}
        +\eps^{\beta-2+2\alpha}
        \lVert h_j^1-h_j^2\rVert_{C^{2,\gamma}(I_j)}.
    \end{split}
    \end{equation}
\end{proposition}

The proofs of \Cref{prop: invert boundary,prop: invert boundary dependence}
can be found in \S\ref{sec: invert boundary}.

\subsection{Solving the interior part}
Fix \(j\in\{1,\dots,N\}\), a shift \(h_j\), and an outer correction \(\psi\).
For each trial interior correction
\(v\in\mathcal C_{\eps,\lambda}^{2,\gamma}(I_j\times\R)\) satisfying
\[
    v=0\quad\text{on }\partial I_j\times\R,
    \qquad
    \int_\R v(x,t)H'(t)\,dt=0,
    \qquad
    \lVert v\rVert_{\mathcal C_{\eps,\lambda}^{2,\gamma}(I_j\times\R)}
    \leq R_\eps,
\]
\Cref{prop: invert boundary} gives one boundary solution
\(\mathscr B_{p,h_j}(v,\psi)\) for each \(p\in\partial I_j\). We collect these
solutions in the tuple
\[
    \mathscr B_{j,h_j}(v,\psi)
    \coloneqq
    \left(\mathscr B_{p,h_j}(v,\psi)\right)_{p\in\partial I_j}.
\]
Substituting these boundary corrections into the interior right-hand side from
\Cref{lem: inner-bdy decomp}, we write
\begin{equation*}
	\mathcal F_{j,h_j}(v,\psi)
    \coloneqq
    F_j\left(\mathscr B_{j,h_j}(v,\psi),v;\psi\right).
\end{equation*}
The equation \eqref{eq: interior eq2} is not generally solvable with estimates uniform in $\eps>0$ because the first eigenvalue of the operator
\begin{equation*}
    \mathfrak L_{\eps,H}\coloneqq
    \partial_{tt}+\eps^2\partial_{xx}-W''(H(t))
\end{equation*}
is small in $\eps$. To fix this issue, we use an infinite dimensional Lyapunov--Schmidt reduction. For any interval \(I\subset\R\), we use the stretched
interior norm \(\mathcal C_{\eps,\lambda}^{0,\gamma}(I\times\R)\) from
\Cref{def: stretched interior norm} and define
\begin{equation*}
\begin{split}
	\Pi_I\colon \mathcal C_{\eps,\lambda}^{0,\gamma}&(I\times\R)\to
    \mathcal C_{\eps,\lambda}^{0,\gamma}(I\times\R)\\
	&f(x,t)\mapsto
    f(x,t)-\frac{1}{A_H}\int_\R f(x,t')H'(t')\,dt'\,H'(t).
\end{split}
\end{equation*}
All integrals in this definition are taken in the stretched normal variable
\(t\). Thus the projection convention uses \(dt\), not the physical measure
\(dz\).
With the projection operator at hand we can state the invertibility theory for
$\mathfrak L_{\eps,H}$. The following results are classical, see
\cite[Lemma 11.3]{PacardRitore} and \cite[Proposition 9.4]{PacardRitore},
respectively.
\begin{lemma}\label{lem: bound on Pi}
    The operator \(\Pi_I\) is a projection on
    \(\mathcal C_{\eps,\lambda}^{0,\gamma}(I\times\R)\) and satisfies
    \[
        \lVert\Pi_I f\rVert_{\mathcal C_{\eps,\lambda}^{0,\gamma}(I\times\R)}
        \leq C
        \lVert f\rVert_{\mathcal C_{\eps,\lambda}^{0,\gamma}(I\times\R)}
    \]
    for a constant \(C>0\) independent of \(\eps\) and \(I\).
\end{lemma}
\begin{lemma}\label{lem: invertibility LH}
		Let $I\subset\R$ be an interval. For any
	    \(f\in \mathcal C_{\eps,\lambda}^{0,\gamma}(I\times \R)\) satisfying
	    \(\Pi_I f=f\), there exists a unique
        \(\phi\in\mathcal C_{\eps,\lambda}^{2,\gamma}(I\times\R)\) solving
	\begin{equation*}
		\begin{cases}
			\mathfrak L_{\eps,H}\phi=f&\text{in }I\times\R\\
			\phi=0&\text{on }\partial I\times\R
		\end{cases}
	\end{equation*}
	satisfying
	\begin{equation*}
		\int_{\R}\phi(x,t)H'(t)dt=0,\quad \forall x\in I
	\end{equation*}
	and
	\begin{equation*}
				\lVert\phi\rVert_{\mathcal C_{\eps,\lambda}^{2,\gamma}(I\times \R)}\leq C\lVert f\rVert_{\mathcal C_{\eps,\lambda}^{0,\gamma}(I\times \R)}
	\end{equation*}
	for a constant $C>0$ independent of $\eps$ and $I$.
\end{lemma}
By construction, \(\mathfrak B_{j,h_j}^{\perp}[v]\) is orthogonal to \(H'\).
Hence \(\widetilde L_{H,h_j}\) preserves the orthogonal subspace whenever
\(v\) satisfies the boundary condition and
\(\int_\R v(x,t)H'(t)\,dt=0\). We therefore consider the projected equation
\begin{equation}\label{eq: interior eq3}
	\widetilde L_{H,h_j}v
    =
    \mathcal F_{j,h_j}(v,\psi)
    -c_j(v;h_j,\psi)(x)H'(t)
    \quad \text{on }I_j\times\R
\end{equation}
where
\[
    c_j(v;h_j,\psi)
    \coloneqq
    \mathfrak p_j\bigl(\mathcal F_{j,h_j}(v,\psi)\bigr).
\]
To solve the full equation we will show that, for a certain choice of vertical
shift \(h_j\), the coefficient retained in \eqref{eq: interior eq3} vanishes.
\begin{proposition}\label{prop: invertibility interior}
	There exists \(\eps_*>0\) such that, for every \(0<\eps\leq\eps_*\),
    every \(\h\in\mathcal B(K\eps^\beta)\), every
    \(j\in\{1,\dots,N\}\), and every
    \(\psi\in\mathscr B_{\mathrm{out}}(\eps)\), the following holds.
    Equation \eqref{eq: interior eq3} admits a unique solution
    \[
        v_j=\mathscr V_{j,h_j}(\psi)
    \]
    among the functions in
    \(\mathcal C_{\eps,\lambda}^{2,\gamma}(I_j\times\R)\) satisfying
    \[
        \int_\R v_j(x,t)H'(t)\,dt=0\quad(x\in I_j),
        \qquad
        v_j=0\quad\text{on }\partial I_j\times\R,
    \]
    and
    \[
        \lVert v_j\rVert_
        {\mathcal C_{\eps,\lambda}^{2,\gamma}(I_j\times\R)}
        \leq R_\eps,
    \]
    where the retained coefficient is
    \[
        c_j(h_j,\psi)
        \coloneqq
        c_j\bigl(\mathscr V_{j,h_j}(\psi);h_j,\psi\bigr).
    \]
	Moreover, if \((\h^i,\psi^i)\), \(i=1,2\), satisfy the same assumptions,
    then
				\begin{equation}\label{eq: Lip phiin}
				\begin{split}
                    &\lVert\mathscr V_{j,h_j^1}(\psi^1)
                    -\mathscr V_{j,h_j^2}(\psi^2)\rVert_
                    {\mathcal C_{\eps,\lambda}^{2,\gamma}(I_j\times \R)}\\
                    &\qquad\lesssim
                    \eps\lVert h_j^1-h_j^2\rVert_{C^{2,\gamma}(I_j)}
                    +\eps^{-(1-\alpha)\tau}
                    \lVert\psi^1-\psi^2\rVert_
                    {C_{\Gamma,\eps,3\lambda/4}^{2,\gamma}(\R^2_{\p},\L_{\p})}.
                \end{split}
				\end{equation}
\end{proposition}
The proof of \Cref{prop: invertibility interior} is given in
\S\ref{sec: invertibility interior}.

For \(\h\in\mathcal B(K\eps^\beta)\) and
\(\psi\in\mathscr B_{\mathrm{out}}(\eps)\), set
\[
    v_j(\h,\psi)
    \coloneqq
    \mathscr V_{j,h_j}(\psi),
    \qquad
    b_p(\h,\psi)
    \coloneqq
    \mathscr B_{p,h_j}\bigl(v_j(\h,\psi),\psi\bigr).
\]
Reconstruct \(\phi_{\h,j}(\psi)\) from these corrections by
\eqref{eq: decomp phi}, and define the local assembly map
\[
    \mathcal L_{\h}(\psi)
    \coloneqq
    \bigl(\phi_{\h,1}(\psi),\dots,\phi_{\h,N}(\psi)\bigr).
\]
\begin{lemma}[Local assembly]\label{lem: local assembly}
After possibly decreasing \(\eps_*\), for every \(0<\eps\leq\eps_*\),
\(\h\in\mathcal B(K\eps^\beta)\), and
\(\psi\in\mathscr B_{\mathrm{out}}(\eps)\),
\[
    \sup_j
    \lVert\mathcal L_{\h}(\psi)_j\rVert_
    {C_{I_j,\eps,\lambda}^{2,\gamma}
    (\spt\chi_{j,2},\L_{\p})}
    \lesssim R_\eps.
\]
For two admissible data sets \((\h^i,\psi^i)\), \(i=1,2\), one has
\begin{equation}\label{eq: local compatibility dependence}
\begin{split}
    &\sup_j\lVert\mathcal L_{\h^1}(\psi^1)_j
    -\mathcal L_{\h^2}(\psi^2)_j\rVert_
    {C_{I_j,\eps,\lambda}^{1,\gamma}
    (\spt\chi_{j,2},\L_{\p})}\\
    &\qquad\lesssim
    \eps^{-(1-\alpha)\tau}
    \lVert\psi^1-\psi^2\rVert_
    {C_{\Gamma,\eps,3\lambda/4}^{2,\gamma}
    (\R^2_{\p},\L_{\p})}
    +\left(\eps+\eps^{\beta-(1-\alpha)\tau}\right)
    \lVert\h^1-\h^2\rVert_{C^{2,\gamma}(\boldsymbol I)}.
\end{split}
\end{equation}
Here only one physical derivative is taken; the change of stretched normal
variable costs \(R_\eps/\eps=\eps^{\beta-\alpha}\lesssim\eps\), and is
therefore included in the last term.
\end{lemma}

The proof of \Cref{lem: local assembly} is given in
\S\ref{sec:local-assembly-proof}.

\subsection{The outer response}

For an admissible shift \(\h\) and an arbitrary local tuple \(\phi\), consider
the outer equation
\begin{equation}\label{eq: outer equation}
    \eps^2\Delta\psi-V\psi=E_\out(\phi,\psi)
    \quad\text{on }\R^2_{\p}.
\end{equation}

\begin{proposition}\label{prop: invert outer}
There exists \(\eps_*>0\) such that, for every \(0<\eps\le\eps_*\), every
\(\h\in\mathcal B(K\eps^\beta)\), and every
\(\phi=(\phi_1,\dots,\phi_N)\) satisfying
\begin{equation}\label{eq: norm bound on phi}
    \sup_j
    \lVert\phi_j\rVert_{C^{2,\gamma}_{I_j,\eps,\lambda}
    (\spt\chi_{j,2},\L_{\p})}
    \le 1,
\end{equation}
the outer equation has a unique solution in
\(\mathscr B_{\mathrm{out}}(\eps)\), which we denote by
\(\Psi_{\h}(\phi)\).
If \((\h^i,\phi^i)\), \(i=1,2\), satisfy the same assumptions, then
\begin{equation}\label{eq: contractivity psi-phi}
\begin{split}
    &\lVert\Psi_{\h^1}(\phi^1)-\Psi_{\h^2}(\phi^2)\rVert_
    {C_{\Gamma,\eps,3\lambda/4}^{2,\gamma}(\R^2_{\p},\L_{\p})}\\
    &\qquad\lesssim
    r_{\mathrm{out}}(\eps)
    \left(
        \lVert\h^1-\h^2\rVert_{C^{2,\gamma}(\boldsymbol I)}
        +
        \sup_j
        \lVert\phi_j^1-\phi_j^2\rVert_{C^{1,\gamma}_{I_j,\eps,\lambda}
        (\spt\chi_{j,2},\L_{\p})}
    \right).
\end{split}
\end{equation}
\end{proposition}

The proof is given in \S\ref{sec: invert outer}.

\subsection{The fixed-shift correction}

Define
\[
    \mathcal K_{\h}(\psi)
    \coloneqq
    \Psi_{\h}\bigl(\mathcal L_{\h}(\psi)\bigr),
    \qquad
    \psi\in\mathscr B_{\mathrm{out}}(\eps).
\]

Whenever \(\psi(\h)\) is a fixed point of \(\mathcal K_{\h}\), we use the
following notation. For \(j\in\{1,\dots,N\}\) and
\(p\in\partial I_j\), set
\[
    v_j(\h)\coloneqq v_j(\h,\psi(\h)),
    \qquad
    b_p(\h)\coloneqq b_p(\h,\psi(\h)),
    \qquad
    \phi_j(\h)\coloneqq\phi_{\h,j}(\psi(\h)),
\]
and
\[
    \varphi_{\h}
    \coloneqq
    \psi(\h)+\sum_{j=1}^N\chi_{j,2}\phi_j(\h),
    \qquad
    c_j(\h)\coloneqq c_j(h_j,\psi(\h)),
    \qquad
    u_{\h}^*\coloneqq\omega_{\h}+\varphi_{\h}.
\]
The adapted representative of the boundary response is
\[
    \widehat b_p(\h)(Y)
    \coloneqq
    b_p(\h)(\eps R_{-\theta_{\h}(p)}Y).
\]

\begin{proposition}[Existence and estimates]
\label{prop: fixed-shift-correction}
There exists \(\eps_*>0\) such that, for every \(0<\eps\le\eps_*\) and
every \(\h\in\mathcal B(K\eps^\beta)\), the map \(\mathcal K_{\h}\) has a
unique fixed point \(\psi(\h)\) in \(\mathscr B_{\mathrm{out}}(\eps)\).
The associated quantities defined above satisfy
\[
    \lVert v_j(\h)\rVert_{\mathcal C_{\eps,\lambda}^{2,\gamma}(I_j\times\R)}
    \lesssim R_\eps,
    \qquad
    \lVert b_p(\h)\rVert_
    {C_{\ell_{\theta_{\h}(p)},\eps,\lambda}^{2,\gamma}(\R^2_*,\L)}
    \lesssim R_\eps,
\]
\[
    \sup_j
    \lVert\phi_j(\h)\rVert_{C_{I_j,\eps,\lambda}^{2,\gamma}
    (\spt\chi_{j,2},\L_{\p})}
    \lesssim R_\eps,
\]
and
\begin{equation}\label{eq: fixed-shift-outer-bound}
    \lVert\psi(\h)\rVert_{C_{\Gamma,\eps,\lambda}^{2,\gamma}
    (\R^2_{\p},\L_{\p})}
    \lesssim \eps^{1-\alpha}R_\eps
    =o(R_\eps).
\end{equation}

The corrected section has the exact residual
\begin{equation}\label{eq: fixed-shift-residual}
\begin{split}
    S\bigl(\omega_{\h}+\varphi_{\h}\bigr)
    &=-\sum_{j=1}^N
    \chi_{j,2}\eta_{j,1}c_j(\h)H'(t_{j,\h}),\\
    t_{j,\h}&=\frac{z-h_j(x)}{\eps}.
\end{split}
\end{equation}

For every \(j\), in the Fermi coordinates around \(I_j\), one has
\begin{equation}\label{eq: fixed-shift-interior-profile}
    \left\lVert
        u_{\h}^*(x,z)
        -H\left(\frac{z-h_j(x)}{\eps}\right)
    \right\rVert_
    {C_{I_j,\eps,\lambda}^{2,\gamma}(\spt\eta_{j,1},\L_{\p})}
    \lesssim R_\eps
\end{equation}
and, for every \(p\in\partial I_j\), in the boundary coordinates around
\(p\), one has
\begin{equation}\label{eq: fixed-shift-boundary-profile}
    \left\lVert
        u_{\h}^*(x,z)
        -\mathcal T_p\left(
            U_{\theta_{\h}(p)}
            \left(\frac{x}{\eps},\frac{z}{\eps}\right)
        \right)
    \right\rVert_{C_{*,\eps}^{2,\gamma}(\spt\zeta_{p,1},\L_{\p})}
    \lesssim R_\eps .
\end{equation}
\end{proposition}

\begin{proposition}[Lipschitz dependence on the shift]
\label{prop: fixed-shift-dependence}
Let \(0<\eps\leq\eps_*\), and let
\(\h^1,\h^2\in\mathcal B(K\eps^\beta)\). Set
\[
    M\coloneqq
    \lVert\h^1-\h^2\rVert_{C^{2,\gamma}(\boldsymbol I)}.
\]
Then, for every \(j\) and \(p\in\partial I_j\),
\begin{equation}\label{eq: compatible correction dependence}
\begin{split}
    \lVert\psi(\h^1)-\psi(\h^2)\rVert_
    {C_{\Gamma,\eps,3\lambda/4}^{2,\gamma}
    (\R^2_{\p},\L_{\p})}
    &\lesssim e^{-\widetilde\lambda/\eps^{1-\alpha}}M,\\
    \lVert v_j(\h^1)-v_j(\h^2)\rVert_
    {\mathcal C_{\eps,\lambda}^{2,\gamma}(I_j\times\R)}
    &\lesssim\eps M,\\
    \lVert\widehat b_p(\h^1)-\widehat b_p(\h^2)\rVert_
    {C_{\ell_+,\lambda}^{2,\gamma}(\R^2_*,\L)}
    &\lesssim
    \left(\eps^{2-\alpha}
    +\eps^{\beta-(1-\alpha)\tau}\right)M.
\end{split}
\end{equation}
\end{proposition}

The proofs are given in \S\ref{sec: fixed-shift-correction}.

\subsection{The reduced obstruction}

For the compatible correction from \Cref{prop: fixed-shift-correction}, define
\begin{equation}\label{eq: definition Rj projected}
    R_j(\h)
    \coloneqq
    c_j(\h)-\eps h_j''.
\end{equation}

\begin{proposition}[Projected shift estimate]\label{prop: projected-h-error}
There exists a constant \(C>0\), depending only on the fixed data, such that,
for every \(K\geq1\), there are constants \(\eps_*=\eps_*(K)>0\) and
\(C_K>0\) for which the following holds whenever
\(0<\eps\leq\eps_*\). If
\(\h\in\mathcal B(K\eps^\beta)\), then
\[
    c_j(\h)=\eps h_j''+R_j(\h),
\]
\[
    \spt R_j(\h)\subset\spt\eta_{j,4}\subset\{\eta_{j,3}=1\}
    \qquad\text{for every }j,
\]
and
\begin{equation}\label{eq: Rj size}
    \lVert R_j(\h)\rVert_{C^{0,\gamma}(I_j)}
    \leq C\eps^{1+\beta}.
\end{equation}
Moreover, if \(\h^1,\h^2\in\mathcal B(K\eps^\beta)\), then
\begin{equation}\label{eq: Rj lip}
\begin{split}
    \lVert R_j(\h^1)-R_j(\h^2)\rVert_{C^{0,\gamma}(I_j)}
    &\leq C_K\left(
        \eps^{\beta-\alpha\gamma}
        +\eps^{4-\alpha(3+\gamma)}
    \right)
    \lVert\h^1-\h^2\rVert_{C^{2,\gamma}(\boldsymbol I)}.
\end{split}
\end{equation}
Both powers of \(\eps\) in \eqref{eq: Rj lip} are strictly larger than
one, so its right-hand side is
\(o(\eps)\lVert\h^1-\h^2\rVert_{C^{2,\gamma}(\boldsymbol I)}\).
\end{proposition}

The proof is given in \S\ref{sec: projected-h-error}.

\subsection{Adjusting the parameter $h$ and concluding the proof of \Cref{thm: sol on intervals}}

\begin{proof}[Proof of \Cref{thm: sol on intervals}]
By \Cref{prop: fixed-shift-correction}, every
\(\h\in\mathcal B(K\eps^\beta)\) determines a corrected section whose only
remaining error is the one displayed in \eqref{eq: fixed-shift-residual}.
By \Cref{prop: projected-h-error}, its coefficients satisfy
\(c_j(\h)=\eps h_j''+R_j(\h)\). Consequently, all these coefficients vanish
precisely when
\begin{equation}\label{eq: eq for h}
    \h''+\eps^{-1}R(\h)=0,
    \qquad
    R(\h)=(R_1(\h),\dots,R_N(\h)).
\end{equation}

Let
\[
    \mathcal Q\colon
    C^{0,\gamma}(\boldsymbol I)
    \longrightarrow
    C_0^{2,\gamma}(\boldsymbol I)
\]
be the componentwise Dirichlet inverse of \(d^2/dx^2\). The standard
one-dimensional Schauder estimate gives
\(\lVert\mathcal Qf\rVert_{C^{2,\gamma}(\boldsymbol I)}
\lesssim\lVert f\rVert_{C^{0,\gamma}(\boldsymbol I)}\), and
\eqref{eq: eq for h} is equivalent to
\begin{equation}\label{eq: fixed point h}
    \h=\mathcal W(\h),
    \qquad
    \mathcal W(\h)\coloneqq
    -\mathcal Q\bigl(\eps^{-1}R(\h)\bigr).
\end{equation}
Choose \(K\geq1\) sufficiently large and consider the ball
\(\mathcal B(K\eps^\beta)\).

The support assertion in \Cref{prop: projected-h-error} shows that
\((\mathcal W(\h))''=-\eps^{-1}R(\h)\) is supported in
\(\spt\eta_{j,4}\) on every \(I_j\). Since \(\mathcal Q\) imposes the
Dirichlet boundary condition, it follows that
\(\mathcal W(\h)\in\mathcal X_\eps\). Moreover, \eqref{eq: Rj size} and the
Schauder estimate give
\[
    \lVert\mathcal W(\h)\rVert_{C^{2,\gamma}(\boldsymbol I)}
    \lesssim
    \eps^{-1}\lVert R(\h)\rVert_{C^{0,\gamma}(\boldsymbol I)}
    \lesssim\eps^\beta.
\]
Thus \(K\) can be fixed, independently of \(\eps\), so that
\(\mathcal W(\mathcal B(K\eps^\beta))
\subset\mathcal B(K\eps^\beta)\). Fix this choice of \(K\). Finally,
\eqref{eq: Rj lip} gives
\[
    \lVert\mathcal W(\h^1)-\mathcal W(\h^2)\rVert_{C^{2,\gamma}(\boldsymbol I)}
    \lesssim_K
    \left(
        \eps^{\beta-\alpha\gamma-1}
        +\eps^{3-\alpha(3+\gamma)}
    \right)
    \lVert\h^1-\h^2\rVert_{C^{2,\gamma}(\boldsymbol I)}.
\]
Both exponents are positive, so \(\mathcal W\) is a contraction after
decreasing \(\eps\).

Let \(\h\) be the unique fixed point and let
\(\varphi_{\h}\) be the correction supplied by
\Cref{prop: fixed-shift-correction}. Equation \eqref{eq: fixed point h} gives
\(c_j(\h)=\eps h_j''+R_j(\h)=0\) for every \(j\). Hence
\eqref{eq: fixed-shift-residual} yields
\[
    S\bigl(\omega_{\h}+\varphi_{\h}\bigr)=0.
\]
Set \(u=\omega_{\h}+\varphi_{\h}\). Then \(u\) solves
\eqref{eq:section-allen-cahn}, and
local elliptic bootstrapping on the double cover shows that it is smooth.
Moreover,
\[
    \lVert\h\rVert_{C^{2,\gamma}(\boldsymbol I)}\lesssim\eps^\beta,
    \qquad
    \sup_{p\in\p}|\theta_{\h}(p)|\lesssim\eps^\beta,
\]
where the last estimate follows from \eqref{eq: boundary-angle-bound}. Set
\(\theta(p)=\theta_{\h}(p)\). Since \(u=u_{\h}^*\), the two asymptotic
estimates in the statement are exactly
\eqref{eq: fixed-shift-interior-profile} and
\eqref{eq: fixed-shift-boundary-profile}.

It remains to locate the nodal set. In the interior coordinates,
\eqref{eq: fixed-shift-interior-profile} gives
\[
    u=H(t)+O(R_\eps),
    \qquad
    \eps\partial_z u=H'(t)+O(R_\eps),
    \qquad
    t=\frac{z-h_j(x)}{\eps}.
\]
The phase gap for \(H\) confines the zeros to a fixed neighborhood of
\(t=0\). On this neighborhood \(H'\) is bounded below, so, for \(\eps\)
small, the implicit function theorem gives a unique zero on each interior
normal fibre, and
\[
    z=h_j(x)+O(\eps R_\eps).
\]

Near a boundary point, rotate the square-root coordinate by the corresponding
half-angle. The lift of the boundary model is then \(\U\), and
\(\partial_{\xi_2}\U>0\) by \Cref{prop: lifted-existence}. In the rescaled
chart \(y=\eps\xi^2\), \eqref{eq: fixed-shift-boundary-profile} and the
implicit function theorem therefore give a unique zero graph. On a fixed
bounded set in the \(\xi\)-plane, the graph lies at distance \(O(R_\eps)\)
from the lifted model line, and the squaring map converts this into a
downstairs displacement \(O(\eps R_\eps)\). On the positive end, the chain
rule and \Cref{cor: estim U-} give
\(\partial_{\xi_2}\U\simeq |\xi_1|\) along the zero set. Thus the lifted
displacement is \(O(R_\eps/|\xi_1|)\), while the derivative of
\(y=\eps\xi^2\) has size \(O(\eps|\xi_1|)\); the same downstairs estimate
follows. Uniform transversality on the intervening compact annuli completes
the comparison.
Since the boundary neighborhood has radius \(O(\eps^\alpha)\), rotating the
model through \(\theta(p)\) moves its nodal half-line by at most
\(O(\eps^{\alpha+\beta})\).

The phase gaps of \(H\), \(U\), and the outer pure phases, together with the
size estimates in \Cref{prop: fixed-shift-correction}, exclude any other
zeros: outside fixed neighborhoods of the model zero sets the approximate
solution is uniformly separated from zero, and the compatible correction
tends to zero with \(\eps\). Conversely, the interior and boundary zero graphs
cover each segment and approach its removed boundary points. Therefore
\[
    \dist_{\mathcal H}(\{u=0\},\Gamma)
    \lesssim
    \lVert\h\rVert_{C^0}
    +\eps R_\eps
    +\eps^{\alpha+\beta}
    \lesssim\eps^\beta.
\]
This concludes the proof of \Cref{thm: sol on intervals}.
\end{proof}

\section{Construction and estimates of the approximate solution}
\label{sec: ansazt}

Fix \(K>0\). Throughout this section, let
\(\h\in\mathcal B(K\eps^\beta)\), with \(\mathcal B(r)\) as in
\eqref{eq: admissible-shift-ball}.
These shifts satisfy the affine boundary-collar property from
\eqref{eq: admissible-shift-space}--\eqref{eq: affine-boundary-collar}. We use
the outer section \(\Ical\) and the fixed bundle conventions from
\Cref{sec: preliminaries-gluing}. For the residual computation we introduce
the auxiliary sections \(\mathcal H_{\h,j}\) in
\cref{subs: interior-heteroclinic-sections} and \(\mathcal U_{\h,p}\) in
\cref{subs: boundary-half-line-sections}. We then recall the approximate solution
\(\omega_{\h}\) in \cref{subs: approximate solution}
and compute its error of approximation, estimate its size as well as its Lipschitz dependence on \(\h\). Finally, we finish with similar estimates for the different profiles and their potentials.
Since \(\Ical\) is represented by a constant pure phase wherever it is used,
\(S(\Ical)=0\) on the outer transition regions.

\subsection{The shifted heteroclinic sections}
\label{subs: interior-heteroclinic-sections}

Fix \(j\in\{1,\dots,N\}\), and use the tube \(T_j\) and its fixed
trivialization from \Cref{sec: preliminaries-gluing}.

\begin{definition}[The shifted heteroclinic section]
\label{def: interior heteroclinic section}
For the admissible shift \(\h\) fixed above, let
\(\mathcal H_{\h,j}\) be the global section represented on \(T_j\) by
\[
    \mathcal H_{\h,j}(X_j(x,z))
    =
    \eta_{j,1}(X_j(x,z))
    H\left(\frac{z-h_j(x)}{\eps}\right)
\]
in the chosen trivialization, and set it equal to zero outside \(T_j\).
\end{definition}

By the cutoff construction, \(\spt\eta_{j,1}\Subset T_j\), so the zero extension is smooth.

\subsection{The rotated boundary sections}
\label{subs: boundary-half-line-sections}

Fix \(p\in\partial I_j\), and use the boundary coordinates
\(q=Y_p(x,z)=p+xe_p+z\nu_j\). Since \(p\) belongs to a unique segment, it
determines \(j\), which will therefore be omitted from the notation. We use
the fixed boundary identification \(\mathcal T_p\) from
\eqref{eq: boundary-bundle-identification}.

Given an admissible shift \(\h\) as above, recall that
\(
    \theta_{\h}(p)
    =
    \partial_{e_p}h_j(p),
\)
as in \eqref{eq: affine-boundary-collar}.
We use the lifted-similarity convention from
\Cref{subs: geometric-setup} when writing
\(U_{\theta_{\h}(p)}(x/\eps,z/\eps)\) as a section over the \((x,z)\)-disk.

\begin{definition}[The rotated boundary section]
\label{def: boundary half-line section}
Let \(\mathcal U_{\h,p}\) be the global section given on
\(B_\delta(p)\setminus\{p\}\) by
\[
    \mathcal U_{\h,p}(Y_p(x,z))
    =
    \zeta_{p,8}(Y_p(x,z))
    \mathcal T_p
    \left(
        U_{\theta_{\h}(p)}\left(\frac{x}{\eps},\frac{z}{\eps}\right)
    \right),
\]
and set it equal to zero outside \(B_\delta(p)\).
\end{definition}

For \(\eps\) sufficiently small,
\(\spt\zeta_{p,8}\Subset B_\delta(p)\), so this zero extension is smooth.
Only the dependence of \(\mathcal U_{\h,p}\) on \(\eps\) is suppressed from
the notation.

Wherever \(\zeta_{p,8}=1\), the definition of \(\mathcal U_{\h,p}\) and the
equation satisfied by \(U\) give
\[
    S(\mathcal U_{\h,p})=0.
\]

\subsection{The approximate solution}
\label{subs: approximate solution}

Define the section associated with the \(j\)-th segment by
\[
    \mathcal P_{\h,j}
    \coloneqq
    \eta_{j,3}\mathcal H_{\h,j}
    +
    \sum_{p\in\partial I_j}\zeta_{p,3}\mathcal U_{\h,p}.
\]
By \eqref{eq: section-valued-global-ansatz}, the approximate solution can be
written in the computational form
\[
    \omega_{\h}
    =
    \chi_{\Ical,5}\Ical
    +
    \sum_{j=1}^N\chi_{j,5}\mathcal P_{\h,j}.
\]
This is an identity of sections of \(\L_{\p}\). By the cutoff nesting,
\(\eta_{j,1}=1\) on \(\spt(\chi_{j,5}\eta_{j,3})\), and
\(\zeta_{p,8}=1\) on \(\spt(\chi_{j,5}\zeta_{p,3})\). Thus the cutoffs used
to define the global model sections do not alter the expression above.

On \(\spt\chi_{j,5}\), the partition of unity gives
\[
    \eta_{j,3}
    +
    \sum_{p\in\partial I_j}\zeta_{p,3}
    =1.
\]
Consequently, in the segment trivialization,
\begin{equation}\label{eq: interval ansatz}
    \mathcal P_{\h,j}
    =
    H\left(\frac{z-h_j(x)}{\eps}\right)
    +
    \sum_{p\in\partial I_j}\zeta_{p,3}
    \left[
        U_{\theta_{\h}(p)}
        \left(\frac{x_p}{\eps},\frac{z_p}{\eps}\right)
        -
        H\left(\frac{z-h_j(x)}{\eps}\right)
    \right],
\end{equation}
where \((x_p,z_p)=Y_p^{-1}(q)\) in the boundary neighborhood of \(p\).

\subsection{Error of the approximate solution}
\label{subs: proof-expansion-error}

We first record the interpolation defect associated with the potential. If
\(a\) and \(b\) belong to the same fibre and \(s\in[0,1]\), set
\begin{equation}\label{eq: interpolation-defect}
    \mathfrak I_W(a,b;s)
    \coloneqq
    W'\bigl((1-s)a+sb\bigr)
    -(1-s)W'(a)-sW'(b).
\end{equation}
Since \(W'\) is odd, this definition is independent of the choice of local
trivialization. We also note that $\mathfrak I_W(a,b;s)=0$ when $s$ is either $0$ or $1$.

For \(q=X_j(x,z)\), write
\[
    t_{j,\h}(x,z)
    \coloneqq
    \frac{z-h_j(x)}{\eps}.
\]
Expressions involving \(H'(t_{j,\h})\) and \(H''(t_{j,\h})\) below are
understood in the segment trivialization used to define
\(\mathcal H_{\h,j}\).

The following lemma is entirely computational, and we omit its proof.

\begin{lemma}[Error of approximation]\label{lem: residual-identity}
For \(\h\in\mathcal B(K\eps^\beta)\),
\begin{equation}\label{eq: residual-identity}
\begin{split}
    S(\omega_{\h})
    =
    \sum_{j=1}^N
    \Bigg[
        \mathcal E^{\mathrm{out}}_{j,\h}
        +
        \chi_{j,5}
        \Bigg(
            \mathcal E^{\mathrm{tan}}_{j,\h}
            +
            \mathcal E^{\mathrm{ort}}_{j,\h}
            +
            \sum_{p\in\partial I_j}
            \left(
                \mathcal E^{\mathrm{trans}}_{j,p,\h}
                +
                \mathcal E^{\mathrm{pot}}_{j,p,\h}
            \right)
        \Bigg)
    \Bigg],
\end{split}
\end{equation}
where
\begin{align}
    \mathcal E^{\mathrm{out}}_{j,\h}
    \coloneqq{}&
    \eps^2\Delta\chi_{j,5}
    \bigl(\mathcal P_{\h,j}-\Ical\bigr)
    +
    2\eps^2\nabla\chi_{j,5}\cdot
    \nabla\bigl(\mathcal P_{\h,j}-\Ical\bigr)
    -
    \mathfrak I_W
    \bigl(\Ical,\mathcal P_{\h,j};\chi_{j,5}\bigr),
    \label{eq: outer-residual-term}\\
    \mathcal E^{\mathrm{tan}}_{j,\h}
    \coloneqq{}&
    -\eps h_j''(x)H'\bigl(t_{j,\h}(x,z)\bigr),
    \label{eq: tangential-residual-term}\\
    \mathcal E^{\mathrm{ort}}_{j,\h}
    \coloneqq{}&
    \eta_{j,3}\bigl(h_j'(x)\bigr)^2
    H''\bigl(t_{j,\h}(x,z)\bigr),
    \label{eq: orthogonal-residual-term}\\
    \mathcal E^{\mathrm{trans}}_{j,p,\h}
    \coloneqq{}&
    \eps^2\Delta\zeta_{p,3}
    \bigl(\mathcal U_{\h,p}-\mathcal H_{\h,j}\bigr)+
    2\eps^2\nabla\zeta_{p,3}\cdot
    \nabla\bigl(\mathcal U_{\h,p}-\mathcal H_{\h,j}\bigr),
    \label{eq: transition-residual-term}\\
    \mathcal E^{\mathrm{pot}}_{j,p,\h}
    \coloneqq{}&
    -\mathfrak I_W
    \bigl(
        \mathcal H_{\h,j},
        \mathcal U_{\h,p};
        \zeta_{p,3}
    \bigr).
    \label{eq: potential-residual-term}
\end{align}
The contributions in \eqref{eq: residual-identity} satisfy the following
support containments:
\begin{itemize}
    \item
    \(\spt\mathcal E^{\mathrm{out}}_{j,\h}
    \subset\overline{\{0<\chi_{j,5}<1\}}\);
    \item
    \(\spt\bigl(\chi_{j,5}\mathcal E^{\mathrm{tan}}_{j,\h}\bigr)
    \subset\spt\eta_{j,4}\);
    \item
    \(\spt\mathcal E^{\mathrm{ort}}_{j,\h}
    \subset\spt\eta_{j,3}\);
    \item
    \(\spt\bigl(
        \chi_{j,5}\mathcal E^{\mathrm{trans}}_{j,p,\h}
    \bigr)
    \cup
    \spt\bigl(
        \chi_{j,5}\mathcal E^{\mathrm{pot}}_{j,p,\h}
    \bigr)
    \subset
    \spt\chi_{j,5}\cap
    \overline{\{0<\zeta_{p,3}<1\}}\).
\end{itemize}
\end{lemma}

We next estimate the terms in \eqref{eq: residual-identity}.
Physical norms on the displayed subsets of the segment region are read in
the segment coordinates.

\begin{lemma}[Size of the residual terms]
\label{lem: approximation-error-size}
There are constants \(C,c>0\), uniform for
\(\h\in\mathcal B(K\eps^\beta)\), such that the following estimates hold.
\begin{align}
    \lVert\mathcal E^{\mathrm{out}}_{j,\h}\rVert_
    {C^{0,\gamma}_{I_j,\eps,\lambda}
    (\{0<\chi_{j,5}<1\})}
    &\leq
    C\exp\left(-\frac{c}{\eps^{1-\alpha}}\right),
    \label{eq: outer-error-size}\\
    \lVert\chi_{j,5}\mathcal E^{\mathrm{tan}}_{j,\h}\rVert_
    {\mathsf C^{0,\gamma}_{\eps,\lambda}(\spt\eta_{j,4})}
    &\leq C\eps^{1+\beta},
    \label{eq: tangential-error-size}\\
    \lVert\mathcal E^{\mathrm{ort}}_{j,\h}\rVert_
    {\mathsf C^{0,\gamma}_{\eps,\lambda}(\spt\eta_{j,3})}
    &\leq C\eps^{2\beta-\alpha\gamma},
    \label{eq: orthogonal-error-size}\\
    \lVert\chi_{j,5}\mathcal E^{\mathrm{trans}}_{j,p,\h}\rVert_
    {\mathsf C^{0,\gamma}_{\eps,\lambda}
    (\{0<\zeta_{p,3}<1\})}
    &\leq C\left(
        \eps^{2+2\beta-\alpha(2+\gamma)}
        +
        \eps^{1+3\beta-\alpha(1+\gamma)}
    \right),
    \label{eq: transition-error-size}\\
    \lVert\chi_{j,5}\mathcal E^{\mathrm{pot}}_{j,p,\h}\rVert_
    {\mathsf C^{0,\gamma}_{\eps,\lambda}
    (\{0<\zeta_{p,3}<1\})}
    &\leq C\eps^{4\beta-\alpha\gamma}.
    \label{eq: potential-error-size}
\end{align}
In particular,
\begin{equation}\label{eq: approximation-error-aggregate}
    \lVert S(\omega_{\h})\rVert_
    {\mathsf C^{0,\gamma}_{\eps,\lambda}(\spt\eta_{j,3})}
    \leq C\eps^{1+\beta}.
\end{equation}
Moreover,
\begin{equation}\label{eq: outer-approximation-error-global}
    \lVert\chi_{\Ical,2}S(\omega_{\h})\rVert_
    {C^{0,\gamma}_{\Gamma,\eps,\lambda}
    (\R^2_{\p},\L_{\p})}
    \leq
    C\exp\left(-\frac{c}{\eps^{1-\alpha}}\right).
\end{equation}
\end{lemma}

\begin{proof}
\emph{Estimate \eqref{eq: outer-error-size}.}
Choose
\[
    \lambda<\lambda_*<\min\{\sigma,\sqrt{\kappa_W}\},
\]
where \(\sigma\) is the decay rate in
\Cref{cor: estim U-,lem: estim U+}. The standard heteroclinic estimates give
the same decay with rate \(\lambda_*\) for \(H\) and its derivatives.

Let \(B_\eps(q)\) be one of the trivializing balls whose closure is contained
in \(\{0<\chi_{j,5}<1\}\), and set \(d(q)=\dist(q,I_j)\). Then
\(d(q)\simeq\eps^\alpha\). Choose the outer representative in which
\(\Ical\) is represented by the constant \(+1\). The interior model is then
represented by \(|H|\), and each boundary model by
\(U_+\circ R_{\theta_{\h}(p)}\). Since
\(\lVert h_j\rVert_\infty+|\theta_{\h}(p)|\leq C\eps^\beta\), the transition sets of
these profiles remain at distance at least \(d(q)/2\) from \(B_\eps(q)\).
The model estimates therefore imply
\[
    \lVert\mathcal P_{\h,j}-\Ical\rVert_
    {C^{2,\gamma}_\eps(B_\eps(q),\L_{\p})}
    \leq
    C\eps^{-M}\exp\left(-\frac{\lambda_*d(q)}{\eps}\right)
\]
for some fixed \(M>0\), where the polynomial factor also accounts for the
derivatives of the tangential partition of unity.
The rescaled cutoff construction gives, for \(k=1,2\),
\[
    \lVert\nabla^k\chi_{j,5}\rVert_
    {C^{0,\gamma}_\eps(B_\eps(q))}
    \leq C\eps^{-\alpha(k+\gamma)}.
\]
Taylor's theorem also gives the local H\"older estimate
\[
    \lVert\mathfrak I_W(a,b;s)\rVert_{C^{0,\gamma}_\eps(B_\eps(q))}
    \leq
    C\lVert s(1-s)\rVert_{C^{0,\gamma}_\eps(B_\eps(q))}
    \lVert a-b\rVert_{C^{0,\gamma}_\eps(B_\eps(q))}^2
\]
for the bounded profiles occurring here. Applying these estimates to the
three terms in \eqref{eq: outer-residual-term}, and then applying the weight
in the norm relative to \(I_j\), gives, for some fixed \(M>0\),
\[
    e^{\lambda d(q)/\eps}
    \lVert\mathcal E^{\mathrm{out}}_{j,\h}\rVert_
    {C^{0,\gamma}_\eps(B_\eps(q),\L_{\p})}
    \leq
    C\eps^{-M}
    \exp\left(-\frac{(\lambda_*-\lambda)d(q)}{\eps}\right).
\]
Since \(d(q)\simeq\eps^\alpha\), the exponential absorbs the polynomial
factor. Taking the supremum over the outer transition proves
\eqref{eq: outer-error-size}.

\emph{Estimate \eqref{eq: tangential-error-size}.}
The heteroclinic decay and
\(\lVert\h\rVert_{C^{2,\gamma}(\boldsymbol I)}\leq K\eps^\beta\) give
\[
    \lVert H'(t_{j,\h})\rVert_
    {\mathsf C^{0,\gamma}_{\eps,\lambda}(\spt\eta_{j,4})}
    \leq C.
\]
Indeed, \(|h_j|/\eps\leq C\eps^{\beta-1}\), so the normal weight centered
at \(z=0\) is uniformly equivalent to the one centered at \(z=h_j(x)\);
the tangential variation of the profile is controlled by
\(|h_j'|/\eps\leq C\eps^{\beta-1}\). On the region where
\(\chi_{j,5}=1\), the weighted product estimate therefore gives
\[
    \lVert\eps h_j''H'(t_{j,\h})\rVert_
    {\mathsf C^{0,\gamma}_{\eps,\lambda}(\spt\eta_{j,4}\cap\{\chi_{j,5}=1\})}
    \leq
    C\eps\lVert h_j''\rVert_{C^{0,\gamma}(I_j)}
    \leq C\eps^{1+\beta}.
\]
On the transition of \(\chi_{j,5}\), one has
\(|z-h_j(x)|\geq c\eps^\alpha\). Hence \(H'(t_{j,\h})\) and its scaled
derivatives are exponentially small there, and this decay absorbs every
H\"older loss from \(\chi_{j,5}\). This proves
\eqref{eq: tangential-error-size}.

\emph{Estimate \eqref{eq: orthogonal-error-size}.}
The same shifted-profile estimate gives
\[
    \lVert(h_j')^2H''(t_{j,\h})\rVert_
    {\mathsf C^{0,\gamma}_{\eps,\lambda}(\spt\eta_{j,3})}
    \leq C\eps^{2\beta}.
\]
The fixed-profile construction of the cutoffs at tangential scale
\(\eps^\alpha\) gives
\[
    \lVert\eta_{j,3}\rVert_
    {\mathsf C^{0,\gamma}_\eps(\spt\eta_{j,3})}
    \leq C\eps^{-\alpha\gamma}.
\]
The weighted product estimate now yields
\[
    \lVert\eta_{j,3}(h_j')^2H''(t_{j,\h})\rVert_
    {\mathsf C^{0,\gamma}_{\eps,\lambda}(\spt\eta_{j,3})}
    \leq C\eps^{2\beta-\alpha\gamma},
\]
which is \eqref{eq: orthogonal-error-size}.

\emph{Estimate \eqref{eq: transition-error-size}.}
Fix \(p\in\partial I_j\), use the inward boundary coordinate \(x\), and
write \(\theta=\theta_{\h}(p)\). The transition of \(\zeta_{p,3}\) is
contained in the affine boundary collar, so
\begin{equation}\label{eq: expansion h near end}
    h_j(p+xe_p)=\theta x,
    \qquad
    h_j'(p+xe_p)=\theta,
    \qquad
    h_j''(p+xe_p)=0.
\end{equation}
Moreover, \(x\simeq\eps^\alpha\) and
\(|\theta|\leq K\eps^\beta\) on this transition. In the segment
trivialization, the representatives of the two uncut models are
\[
    U_-(r,s),
    \qquad
    H(t),
\]
where
\[
    r=\frac{x\cos\theta+z\sin\theta}{\eps},
    \qquad
    s=\frac{-x\sin\theta+z\cos\theta}{\eps},
    \qquad
    t=\frac{z-\theta x}{\eps}.
\]
The use of \(U_-\) here follows from the compatibility between the bundle
identification in \eqref{eq: boundary-bundle-identification} and the segment
trivialization. The cutoff nesting also gives \(\eta_{j,1}=\zeta_{p,8}=1\)
on the region under consideration, so these are the uncut representatives of
\(\mathcal H_{\h,j}\) and \(\mathcal U_{\h,p}\). Set
\[
    D\coloneqq U_-(r,s)-H(t)
    =\bigl(U_-(r,s)-H(s)\bigr)+\bigl(H(s)-H(t)\bigr).
\]
Since \(r\geq c\eps^{\alpha-1}\), the first difference and all its needed
scaled derivatives are bounded by
\(C\eps^{-M}\exp(-c/\eps^{1-\alpha})\). For the second difference, direct
substitution gives
\begin{equation}\label{eq: boundary-phase-gap}
    s-t
    =
    t(\cos\theta-1)
    +
    \frac{x}{\eps}(\theta\cos\theta-\sin\theta).
\end{equation}
Using
\[
    |\cos\theta-1|\leq C\theta^2,
    \qquad
    |\theta\cos\theta-\sin\theta|\leq C|\theta|^3,
\]
the mean value theorem and the exponential decay of the derivatives of \(H\)
give
\[
    \lVert\cosh(\lambda z/\eps)D\rVert_
    {L^\infty(\spt\chi_{j,5}\cap\{0<\zeta_{p,3}<1\})}
    \leq
    C\left(
        e^{-c/\eps^{1-\alpha}}
        +|\theta|^2
        +|\theta|^3\eps^{\alpha-1}
    \right)
    \leq C\eps^{2\beta}.
\]
For the tangential derivative, at fixed \(z\),
\[
    \partial_x\bigl(H(s)-H(t)\bigr)
    =
    \frac{\theta-\sin\theta}{\eps}H'(s)
    +
    \frac{\theta}{\eps}\bigl(H'(t)-H'(s)\bigr).
\]
Together with \eqref{eq: boundary-phase-gap}, this gives
\[
    \lVert\cosh(\lambda z/\eps)\partial_xD\rVert_
    {L^\infty(\spt\chi_{j,5}\cap\{0<\zeta_{p,3}<1\})}
    \leq
    C\left(
        e^{-c/\eps^{1-\alpha}}
        +\frac{|\theta|^3}{\eps}
        +\frac{|\theta|^4\eps^\alpha}{\eps^2}
    \right)
    \leq C\eps^{3\beta-1}.
\]
Set \(w(z)=\cosh(\lambda z/\eps)\). Differentiating once more gives
\[
\begin{aligned}
    \lVert\partial_x(wD)\rVert_\infty
    +\lVert\eps\partial_z(wD)\rVert_\infty
    &\leq C\eps^{2\beta},\\
    \lVert\partial_x(w\partial_xD)\rVert_\infty
    +\lVert\eps\partial_z(w\partial_xD)\rVert_\infty
    &\leq C\eps^{3\beta-1}
\end{aligned}
\]
on \(\spt\chi_{j,5}\cap\{0<\zeta_{p,3}<1\}\). Here scaled normal
differentiation preserves the preceding orders, while a second tangential
differentiation costs at most
\(|\theta|/\eps=O(\eps^{\beta-1})\); all differentiated radial remainders are
still exponentially small. For pairs at physical scaled distance at most one,
the mean value theorem in the metric
\(|x-x'|+\eps^{-1}|z-z'|\) gives the required H\"older quotients. For pairs
at distance at least one, the corresponding supremum bounds do so. Thus we
have proved
\begin{equation}\label{eq: estimate variable difference}
\begin{split}
    \lVert\mathcal U_{\h,p}-\mathcal H_{\h,j}\rVert_
    {\mathsf C^{0,\gamma}_{\eps,\lambda}
    (\spt\chi_{j,5}\cap\{0<\zeta_{p,3}<1\})}
    &\leq C\eps^{2\beta},\\
    \lVert\partial_x(\mathcal U_{\h,p}-\mathcal H_{\h,j})\rVert_
    {\mathsf C^{0,\gamma}_{\eps,\lambda}
    (\spt\chi_{j,5}\cap\{0<\zeta_{p,3}<1\})}
    &\leq C\eps^{3\beta-1}.
\end{split}
\end{equation}

On \(\spt\chi_{j,5}\), the nesting gives \(\chi_{j,10}=1\). By the chosen
product construction of the boundary cutoffs, one therefore has
\[
    \zeta_{p,3}(Y_p(x,z))
    =\varrho\left(\frac{x}{\delta\eps^\alpha}-3\right)
\]
on this region. Consequently,
\[
    \nabla\zeta_{p,3}
    =\partial_x\zeta_{p,3}e_p,
    \qquad
    \Delta\zeta_{p,3}=\partial_x^2\zeta_{p,3}.
\]
Its fixed-profile construction gives, for \(m=1,2\),
\[
    \lVert\partial_x^m\zeta_{p,3}\rVert_
    {\mathsf C^{0,\gamma}_\eps
    (\spt\chi_{j,5}\cap\{0<\zeta_{p,3}<1\})}
    \leq C\eps^{-\alpha(m+\gamma)}.
\]
On the region where \(\chi_{j,5}=1\), the product estimate and
\eqref{eq: estimate variable difference} therefore give
\[
\begin{split}
    \lVert\eps^2\Delta\zeta_{p,3}
    (\mathcal U_{\h,p}-\mathcal H_{\h,j})\rVert_
    {\mathsf C^{0,\gamma}_{\eps,\lambda}
    (\{\chi_{j,5}=1\}\cap\{0<\zeta_{p,3}<1\})}
    &\leq C\eps^{2+2\beta-\alpha(2+\gamma)},\\
    \lVert2\eps^2\nabla\zeta_{p,3}\cdot
    \nabla(\mathcal U_{\h,p}-\mathcal H_{\h,j})\rVert_
    {\mathsf C^{0,\gamma}_{\eps,\lambda}
    (\{\chi_{j,5}=1\}\cap\{0<\zeta_{p,3}<1\})}
    &\leq C\eps^{1+3\beta-\alpha(1+\gamma)}.
\end{split}
\]
On the transition of \(\chi_{j,5}\), the profile difference and all the
derivatives used above are exponentially small in the normal direction. This
decay absorbs the H\"older losses produced by multiplication by
\(\chi_{j,5}\). The estimated section vanishes outside
\(\spt\chi_{j,5}\), so the last two estimates prove
\eqref{eq: transition-error-size} on the full domain stated there.

\emph{Estimate \eqref{eq: potential-error-size}.}
For bounded \(a,b\) and \(s\in[0,1]\), Taylor's theorem gives
\begin{equation}\label{eq: interpolation-defect-quadratic}
    \mathfrak I_W(a,b;s)
    =s(1-s)(b-a)^2\mathfrak R_W(a,b;s),
\end{equation}
where \(\mathfrak R_W\) is smooth and uniformly bounded on the range of the
profiles. Set
\[
    G(a,b,s)\coloneqq s(1-s)\mathfrak R_W(a,b;s).
\]
The individual profiles have uniformly bounded unweighted scaled H\"older
norms, while
\[
\begin{split}
    \lVert G(\mathcal H_{\h,j},\mathcal U_{\h,p},\zeta_{p,3})\rVert_\infty
    &\leq C,\\
    [G(\mathcal H_{\h,j},\mathcal U_{\h,p},\zeta_{p,3})]_{\gamma,\eps}
    &\leq C\eps^{-\alpha\gamma}
\end{split}
\]
on \(\spt\chi_{j,5}\cap\{0<\zeta_{p,3}<1\}\). Furthermore,
\(\cosh(\lambda z/\eps)D^2
=(\cosh(\lambda z/\eps)D)D\), so
\[
    \lVert D^2\rVert_
    {\mathsf C^{0,\gamma}_{\eps,\lambda}
    (\spt\chi_{j,5}\cap\{0<\zeta_{p,3}<1\})}
    \leq C\eps^{4\beta}.
\]
On \(\{\chi_{j,5}=1\}\cap\{0<\zeta_{p,3}<1\}\),
the H\"older product rule, applied using the supremum and seminorm bounds in
the preceding two displays, gives
\[
    \lVert\mathcal E^{\mathrm{pot}}_{j,p,\h}\rVert_
    {\mathsf C^{0,\gamma}_{\eps,\lambda}
    (\{\chi_{j,5}=1\}\cap\{0<\zeta_{p,3}<1\})}
    \leq C\eps^{-\alpha\gamma}\eps^{4\beta}
    =C\eps^{4\beta-\alpha\gamma}.
\]
On the transition of \(\chi_{j,5}\), the profile difference is exponentially
small in the normal direction, so the cutoff loss is again absorbed. The
estimated section vanishes outside \(\spt\chi_{j,5}\), and hence this proves
\eqref{eq: potential-error-size} on the stated domain.

\emph{Estimate \eqref{eq: approximation-error-aggregate}.}
Relative to the exponent \(1+\beta\), the four exponents appearing in
\eqref{eq: orthogonal-error-size}--\eqref{eq: potential-error-size} have the
gaps
\[
\begin{aligned}
    (2\beta-\alpha\gamma)-(1+\beta)
    &=\beta-1-\alpha\gamma>0,\\
    \bigl(2+2\beta-\alpha(2+\gamma)\bigr)-(1+\beta)
    &=1+\beta-\alpha(2+\gamma)>0,\\
    \bigl(1+3\beta-\alpha(1+\gamma)\bigr)-(1+\beta)
    &=2\beta-\alpha(1+\gamma)>0,\\
    (4\beta-\alpha\gamma)-(1+\beta)
    &=3\beta-1-\alpha\gamma>0.
\end{aligned}
\]
These inequalities follow from \(\beta\geq2\), \(\alpha<1/2\), and
\(\gamma<1\). Where \(\chi_{j,5}=1\),
\eqref{eq: orthogonal-error-size} applies directly to the orthogonal term in
the residual identity. On the transition of \(\chi_{j,5}\), the factor
\(H''(t_{j,\h})\) is exponentially small, so its decay absorbs the H\"older
loss from this cutoff. On the overlap of the outer transition with
\(\spt\eta_{j,3}\), the same profile-tail calculation used for
\eqref{eq: outer-error-size}, now performed in the physical product norm,
gives an exponentially small bound. The change of tangential H\"older scale
costs only a fixed power of \(\eps^{-1}\), which is absorbed by the
exponential decay.
Thus \eqref{eq: residual-identity}, the estimates just proved, and the
separation of the segment tubes give
\eqref{eq: approximation-error-aggregate}.

\emph{Estimate \eqref{eq: outer-approximation-error-global}.}
The cutoff \(\chi_{\Ical,2}\) vanishes on a fixed multiple of the
\(\eps^\alpha\)-neighborhood of \(\Gamma\). In particular, it vanishes on
every singular ball \(B_{R_0\eps}(p)\) for \(\eps\) small. Let
\(B_\eps(q)\) be a trivializing ball meeting the support of
\(\chi_{\Ical,2}S(\omega_{\h})\), and set
\(d(q)=\dist(q,\Gamma)\). Then \(d(q)\geq c\eps^\alpha\).

The outer term satisfies the required local estimate by the calculation for
\eqref{eq: outer-error-size}. For the inner terms, the tangential and
orthogonal contributions contain \(H'(t_{j,\h})\) and
\(H''(t_{j,\h})\), respectively, and hence decay in the normal direction.
For a transition term, the decomposition
\[
    D=\bigl(U_-(r,s)-H(s)\bigr)+\bigl(H(s)-H(t)\bigr)
\]
gives radial decay for the first summand; on the present ball, the phase-gap
factor in the second summand multiplies a normally decaying derivative of
\(H\). The potential term is quadratic in this same difference. Therefore,
for some fixed \(M>0\), all terms in the residual identity satisfy
\[
    \lVert S(\omega_{\h})\rVert_
    {C^{0,\gamma}_\eps(B_\eps(q),\L_{\p})}
    \leq
    C\eps^{-M}\exp\left(-\frac{\lambda_* d(q)}{\eps}\right).
\]
This estimate also covers the transition of \(\chi_{\Ical,2}\); its cutoff
variation introduces only additional polynomial factors. Multiplication by
the weight in \eqref{eq: outer-approximation-error-global} leaves
\[
    C\eps^{-M}
    \exp\left(-\frac{(\lambda_*-\lambda)d(q)}{\eps}\right)
    \leq
    C\exp\left(-\frac{c}{\eps^{1-\alpha}}\right).
\]
Taking the supremum over the trivializing balls and summing over the finitely
many segment regions proves \eqref{eq: outer-approximation-error-global}.
\end{proof}

The tangential term is the only contribution at the leading order. On
\(\spt\chi_{j,10}\), we remove its contribution to the residual and define
\begin{equation}\label{eq: approximation-remainder}
    \mathcal E^{\mathrm{rem}}_{j,\h}
    \coloneqq
    S(\omega_{\h})
    -
    \chi_{j,5}\mathcal E^{\mathrm{tan}}_{j,\h}
    =
    S(\omega_{\h})
    +
    \eps\chi_{j,5}h_j''(x)H'\bigl(t_{j,\h}(x,z)\bigr).
\end{equation}

Recall from \Cref{def: adapted-coordinate-representatives} that tildes with
subscripts \(j,\h\) denote shifted interior representatives, while hats with
subscripts \(p,\h\) denote representatives on the fixed boundary model.

For the tangential boundary projection, we use the admissible decay function
\(g_\tau\) from \eqref{eq: boundary-decay-function}. For
\(p\in\partial I_j\), let
\(\widehat{\mathcal E}^{\mathrm{rem}}_{p,\h}\) denote the boundary
representative of
\(\chi_{j,4}\zeta_{p,5}\mathcal E^{\mathrm{rem}}_{j,\h}\).

For the interior localization, define
\[
    a_{j,\h}(x)
    =
    \frac{1}{A_H}\int_{\R}
    (\chi_{j,4}\eta_{j,3})
    \bigl(X_j(x,h_j(x)+\eps t)\bigr)
    (H'(t))^2\,dt
\]
and let \(\widetilde{\mathcal E}^{\mathrm{rem}}_{j,\h}\) denote the shifted
interior representative of
\(\chi_{j,4}\eta_{j,4}\mathcal E^{\mathrm{rem}}_{j,\h}\).

\begin{lemma}[Size of the approximation remainder]
\label{lem: approximation-remainder-size}
There is a constant \(C>0\), uniform for
\(\h\in\mathcal B(K\eps^\beta)\), such that
\begin{align}
    \lVert\mathcal E^{\mathrm{rem}}_{j,\h}\rVert_
    {\mathsf C^{0,\gamma}_{\eps,\lambda}(\spt\eta_{j,3})}
    &\leq C\eps^{2\beta-\alpha\gamma},
    \label{eq: approximation-remainder-size}\\
    \lVert\mathcal E^{\mathrm{rem}}_{j,\h}\rVert_
    {L^\infty(\spt\eta_{j,3})}
    &\leq C\eps^{2\beta}
    \leq C\eps^{1+\beta+\alpha\gamma},
    \label{eq: approximation-error-Linfty}\\
    \lVert\widehat{\mathcal E}^{\mathrm{rem}}_{p,\h}\rVert_
    {C^{0,\gamma}_{\ell_+,\lambda}(\R^2_*,\L)}
    +
    \lVert(\widehat{\mathcal E}^{\mathrm{rem}}_{p,\h})^\top\rVert_
    {L^\infty(\R^+)}
    &\leq C\eps^{2\beta},
    \label{eq: boundary-approximation-error-size}\\
    A_{g_\tau}
    \bigl(\widehat{\mathcal E}^{\mathrm{rem}}_{p,\h}\bigr)
    &\leq
    C\eps^{2\beta-(1-\alpha)\tau},
    \label{eq: boundary-approximation-error-Ag}\\
    \lVert\widetilde{\mathcal E}^{\mathrm{rem}}_{j,\h}\rVert_
    {\mathcal C^{0,\gamma}_{\eps,\lambda}(I_j\times\R)}
    &\leq C\eps^{2\beta},
    \label{eq: interior-approximation-error-size}\\
    \left\lVert
        \mathfrak p_j
        \bigl(\widetilde{\mathcal E}^{\mathrm{rem}}_{j,\h}\bigr)
    \right\rVert_{C^{0,\gamma}(I_j)}
    +
    \lVert\eps(a_{j,\h}-\eta_{j,3})h_j''\rVert_{C^{0,\gamma}(I_j)}
    &\leq C_m\eps^m,
    \label{eq: projected-approximation-error-size}
\end{align}
where \(p\in\partial I_j\), \(m>0\) is arbitrary, and \(C_m\) may depend
on \(m\).
\end{lemma}

\begin{proof}
Set \(\Omega_j\coloneqq\spt\eta_{j,3}\). The separation of the segment
neighborhoods implies that all terms indexed by \(k\ne j\) in
\eqref{eq: residual-identity} vanish on \(\Omega_j\). Hence
\eqref{eq: residual-identity} and \eqref{eq: approximation-remainder} give
\begin{equation}\label{eq: remainder-local-decomposition}
\begin{split}
    \mathcal E^{\mathrm{rem}}_{j,\h}
    ={}&
    \mathcal E^{\mathrm{out}}_{j,\h}
    +
    \chi_{j,5}\mathcal E^{\mathrm{ort}}_{j,\h}\\
    &+
    \sum_{p\in\partial I_j}\chi_{j,5}
    \left(
        \mathcal E^{\mathrm{trans}}_{j,p,\h}
        +
        \mathcal E^{\mathrm{pot}}_{j,p,\h}
    \right)
    \qquad\text{on }\Omega_j.
\end{split}
\end{equation}

\emph{Estimate \eqref{eq: approximation-remainder-size}.}
Set \(q=2\beta-\alpha\gamma\). On the intersection of \(\Omega_j\) with
the outer transition, repeat the profile-tail calculation used for
\eqref{eq: outer-error-size} in the physical product norm. The change in the
tangential H\"older scale produces only a fixed polynomial loss, and hence,
for some fixed \(M>0\),
\begin{equation}\label{eq: remainder-outer-bound}
    \lVert\mathcal E^{\mathrm{out}}_{j,\h}\rVert_
    {\mathsf C^{0,\gamma}_{\eps,\lambda}(\Omega_j)}
    \leq
    C\eps^{-M}\exp\left(-\frac{c}{\eps^{1-\alpha}}\right)
    \leq C\eps^q.
\end{equation}
For the orthogonal contribution, \eqref{eq: orthogonal-error-size}, its
direct supremum estimate, and the cutoff bound give
\[
    \lVert\mathcal E^{\mathrm{ort}}_{j,\h}\rVert_
    {\mathsf C^{0,\gamma}_{\eps,\lambda}(\Omega_j)}
    \leq C\eps^q,
    \qquad
    \lVert\cosh(\lambda z/\eps)
    \mathcal E^{\mathrm{ort}}_{j,\h}\rVert_{L^\infty(\Omega_j)}
    \leq C\eps^{2\beta},
\]
and
\[
    [\chi_{j,5}]_{\gamma,\eps}
    \leq C\eps^{-\alpha\gamma}.
\]
The H\"older product rule therefore yields
\begin{equation}\label{eq: remainder-orthogonal-bound}
    \lVert\chi_{j,5}\mathcal E^{\mathrm{ort}}_{j,\h}\rVert_
    {\mathsf C^{0,\gamma}_{\eps,\lambda}(\Omega_j)}
    \leq
    C\left(\eps^q+\eps^{-\alpha\gamma}\eps^{2\beta}\right)
    =C\eps^q.
\end{equation}
The support containments in \Cref{lem: residual-identity} and the restriction
of \eqref{eq: transition-error-size}--
\eqref{eq: potential-error-size} to \(\Omega_j\) give
\begin{align*}
    \lVert\chi_{j,5}\mathcal E^{\mathrm{trans}}_{j,p,\h}\rVert_
    {\mathsf C^{0,\gamma}_{\eps,\lambda}(\Omega_j)}
    &\leq
    C\left(
        \eps^{q+2(1-\alpha)}
        +
        \eps^{q+1+\beta-\alpha}
    \right),\\
    \lVert\chi_{j,5}\mathcal E^{\mathrm{pot}}_{j,p,\h}\rVert_
    {\mathsf C^{0,\gamma}_{\eps,\lambda}(\Omega_j)}
    &\leq C\eps^{q+2\beta}.
\end{align*}
All three added exponents are positive because \(\alpha<1/2\) and
\(\beta\geq2\). Summing these estimates in
\eqref{eq: remainder-local-decomposition} over the two points of
\(\partial I_j\) proves \eqref{eq: approximation-remainder-size}.

\emph{Estimate \eqref{eq: approximation-error-Linfty}.}
The formula for the orthogonal term gives directly
\[
    \lVert\chi_{j,5}\mathcal E^{\mathrm{ort}}_{j,\h}\rVert_{L^\infty(\Omega_j)}
    \leq
    \lVert h_j'\rVert_{L^\infty(I_j)}^2\lVert H''\rVert_{L^\infty(\R)}
    \leq C\eps^{2\beta}.
\]
By \eqref{eq: remainder-outer-bound},
\[
    \lVert\mathcal E^{\mathrm{out}}_{j,\h}\rVert_{L^\infty(\Omega_j)}
    \leq
    C\eps^{-M}\exp\left(-\frac{c}{\eps^{1-\alpha}}\right)
    \leq C\eps^{2\beta}.
\]
For the remaining terms, every exponential weight is at least one, so
\eqref{eq: transition-error-size} and \eqref{eq: potential-error-size} imply
\begin{align*}
    \lVert\chi_{j,5}\mathcal E^{\mathrm{trans}}_{j,p,\h}\rVert_
    {L^\infty(\Omega_j)}
    &\leq C\left(
        \eps^{2+2\beta-\alpha(2+\gamma)}
        +
        \eps^{1+3\beta-\alpha(1+\gamma)}
    \right),\\
    \lVert\chi_{j,5}\mathcal E^{\mathrm{pot}}_{j,p,\h}\rVert_
    {L^\infty(\Omega_j)}
    &\leq C\eps^{4\beta-\alpha\gamma}.
\end{align*}
Relative to \(2\beta\), the three polynomial gaps are
\[
    2-\alpha(2+\gamma),
    \qquad
    1+\beta-\alpha(1+\gamma),
    \qquad
    2\beta-\alpha\gamma.
\]
They are positive because \(\alpha<1/2\), \(\gamma<1\), and
\(\beta\geq2\). Thus
\eqref{eq: remainder-local-decomposition} gives the
first inequality in \eqref{eq: approximation-error-Linfty}. The second
follows from
\[
    2\beta-(1+\beta+\alpha\gamma)
    =\beta-1-\alpha\gamma>0.
\]

\emph{Estimate \eqref{eq: boundary-approximation-error-size}.}
Fix \(p\in\partial I_j\), set
\[
    \theta=\theta_{\h}(p),
    \qquad
    L_\eps=\eps^{\alpha-1},
\]
and use the fixed boundary coordinates \((X,Z)\) from
\eqref{eq: boundary-error-pullback}. Set
\[
    \kappa_{p,\h}
    \coloneqq
    \widehat{(\chi_{j,4}\zeta_{p,5})}_{p,\h}.
\]
This pulled-back localization is supported where
\(|X|+|Z|\leq CL_\eps\) and has uniformly bounded
\(C^{0,\gamma}\)-norm. The nesting of the cutoffs gives
\(\chi_{j,5}=1\) on its support. Consequently, the outer term vanishes,
and the localization excludes the terms associated with the other point of
\(\partial I_j\). Thus, in the fixed boundary coordinates,
\begin{equation}\label{eq: localized-boundary-remainder-decomposition}
    \widehat{\mathcal E}^{\mathrm{rem}}_{p,\h}
    =
    \kappa_{p,\h}
    \left(
        \widehat{\mathcal E}^{\mathrm{ort}}_{j,\h}
        +
        \widehat{\mathcal E}^{\mathrm{trans}}_{j,p,\h}
        +
        \widehat{\mathcal E}^{\mathrm{pot}}_{j,p,\h}
    \right),
\end{equation}
where the three representatives on the right are needed only on
\(\spt\kappa_{p,\h}\). The localized remainder vanishes where
\(\zeta_{p,3}=1\). On its nonzero support, the physical coordinates satisfy
\(c\eps^\alpha\leq x\leq C\eps^\alpha\) and
\(|z|\leq C\eps^\alpha\). Since \(|\theta|\leq C\eps^\beta\), it follows
that
\begin{equation}\label{eq: boundary-remainder-scaled-support}
    cL_\eps\leq X\leq CL_\eps,
    \qquad
    |Z|\leq CL_\eps
\end{equation}
on the nonzero support of \(\widehat{\mathcal E}^{\mathrm{rem}}_{p,\h}\).

We first estimate the two terms supported in the transition of
\(\zeta_{p,3}\). This transition lies in the affine boundary collar. If
\(x=\eps(X\cos\theta-Z\sin\theta)\), then the shifted heteroclinic is
represented there by \(H(T_\theta)\), where
\[
    T_\theta(X,Z)
    =
    (\sin\theta-\theta\cos\theta)X
    +
    (\cos\theta+\theta\sin\theta)Z.
\]
In the \(U_-\)-trivialization, set
\[
    \widehat D(X,Z)
    \coloneqq
    U_-(X,Z)-H(T_\theta(X,Z)).
\]
Let
\[
    \widehat\Omega^{\mathrm{tr}}_{p,\h}
    \coloneqq
    \spt\kappa_{p,\h}
    \cap
    \{0<\widehat\zeta_{p,3}<1\}.
\]
Since
\[
    T_\theta-Z
    =O(\theta^3X)+O(\theta^2Z),
\]
the model decay, the decay of the heteroclinic, and
\(|X|\leq CL_\eps\) on this transition give
\begin{equation}\label{eq: boundary-profile-difference-for-remainder}
    \lVert\widehat D\rVert_
    {C^{1,\gamma}_{\ell_+,\lambda}
    (\widehat\Omega^{\mathrm{tr}}_{p,\h})}
    \leq
    C\left(
        e^{-cL_\eps}
        +|\theta|^2
        +|\theta|^3L_\eps
    \right)
    \leq C\eps^{2\beta}.
\end{equation}
Here the last inequality uses \(|\theta|\leq C\eps^\beta\) and
\(\beta+\alpha-1>0\).

On \(\spt\kappa_{p,\h}\), the factor \(\chi_{j,10}\) in
\(\zeta_{p,3}\) is equal to one. Thus its pullback
\(\widehat\zeta_{p,3}\) is a one-variable fixed profile whose transition
has scale \(L_\eps\), and
\[
    \lVert\nabla^k\widehat\zeta_{p,3}\rVert_{C^{0,\gamma}}
    \leq CL_\eps^{-k},
    \qquad k=1,2.
\]
Because the dilation and rotation in
\eqref{eq: boundary-error-pullback} transform \(\eps^2\Delta\) into the
ordinary Laplacian, the transition term satisfies the exact identity
\[
    \widehat{\mathcal E}^{\mathrm{trans}}_{j,p,\h}
    =
    (\Delta\widehat\zeta_{p,3})\widehat D
    +
    2\nabla\widehat\zeta_{p,3}
    \mathbin{\cdot}\nabla\widehat D.
\]
It follows from
\eqref{eq: boundary-profile-difference-for-remainder} that
\begin{equation}\label{eq: boundary-transition-remainder-bound}
    \lVert\kappa_{p,\h}
    \widehat{\mathcal E}^{\mathrm{trans}}_{j,p,\h}\rVert_
    {C^{0,\gamma}_{\ell_+,\lambda}(\R^2_*,\L)}
    \leq
    C(L_\eps^{-2}+L_\eps^{-1})\eps^{2\beta}
    \leq C\eps^{2\beta}.
\end{equation}
The quadratic formula \eqref{eq: interpolation-defect-quadratic} similarly
gives
\begin{equation}\label{eq: boundary-potential-remainder-bound}
    \lVert\kappa_{p,\h}
    \widehat{\mathcal E}^{\mathrm{pot}}_{j,p,\h}\rVert_
    {C^{0,\gamma}_{\ell_+,\lambda}(\R^2_*,\L)}
    \leq C\eps^{4\beta}.
\end{equation}

It remains to estimate the orthogonal term. On its support inside the
boundary localization, write
\[
    t_{\h}(X,Z)
    =
    X\sin\theta+Z\cos\theta
    -\eps^{-1}h_j(p+x e_p).
\]
The bounds for \(h_j\) and \(\theta\), together with
\(|X|+|Z|\leq CL_\eps\), give
\[
    |t_{\h}-Z|
    \leq
    C\left(
        |\theta|L_\eps
        +\theta^2L_\eps
        +\eps^{\beta-1}
    \right)
    =o(1).
\]
The first derivatives of \(t_{\h}\) are uniformly bounded. Moreover,
\eqref{eq: boundary-remainder-scaled-support} gives \(X>0\), so the distance
to \(\ell_+\) is \(|Z|\) throughout this support. The cutoff
factors have uniformly bounded scaled H\"older norms, while
\(\lVert h_j'\rVert_{C^{0,\gamma}}\leq C\eps^\beta\). Therefore
\begin{equation}\label{eq: boundary-orthogonal-remainder-bound}
    \lVert\kappa_{p,\h}
    \widehat{\mathcal E}^{\mathrm{ort}}_{j,\h}\rVert_
    {C^{0,\gamma}_{\ell_+,\lambda}(\R^2_*,\L)}
    \leq C\eps^{2\beta}.
\end{equation}
Combining \eqref{eq: localized-boundary-remainder-decomposition} with
\eqref{eq: boundary-transition-remainder-bound}--
\eqref{eq: boundary-orthogonal-remainder-bound} proves the first norm estimate in
\eqref{eq: boundary-approximation-error-size}. The zero extension causes no
loss: the localized remainder is smooth at the outer edge of its support and
vanishes in a neighborhood of the puncture, where the ansatz is exactly the
boundary model.

For \(X>0\), the weight relative to \(\ell_+\) is equivalent to
\(e^{\lambda|Z|}\). Hence the estimate just proved gives
\[
\begin{split}
    \left|
        (\widehat{\mathcal E}^{\mathrm{rem}}_{p,\h})^\top(X)
    \right|
    &\leq
    C\eps^{2\beta}
    \int_\R e^{-\lambda|Z|}|H'(Z)|\,dZ\\
    &\leq C\eps^{2\beta},
\end{split}
\]
which is the second norm estimate in
\eqref{eq: boundary-approximation-error-size}.

\emph{Estimate \eqref{eq: boundary-approximation-error-Ag}.}
The support of the boundary localization implies
\[
    \spt
    (\widehat{\mathcal E}^{\mathrm{rem}}_{p,\h})^\top
    \subset
    \{0<X\leq CL_\eps\}.
\]
Since
\[
    g_\tau''(X)
    =
    (\tau-1)(\tau-2)(1+X)^{-\tau},
\]
the definition of \(A_{g_\tau}\) and
\eqref{eq: boundary-approximation-error-size} yield
\[
\begin{split}
    A_{g_\tau}
    \bigl(\widehat{\mathcal E}^{\mathrm{rem}}_{p,\h}\bigr)
    &\leq
    C\sup_{0<X\leq CL_\eps}
    (1+X)^\tau
    \left|
        (\widehat{\mathcal E}^{\mathrm{rem}}_{p,\h})^\top(X)
    \right|\\
    &\leq
    C\eps^{2\beta}L_\eps^\tau
    =
    C\eps^{2\beta-(1-\alpha)\tau}.
\end{split}
\]
This proves \eqref{eq: boundary-approximation-error-Ag}.

\emph{Estimate \eqref{eq: interior-approximation-error-size}.}
Define
\[
    b_{j,\h}(x,t)
    \coloneqq
    (\chi_{j,4}\eta_{j,4})
    \bigl(X_j(x,h_j(x)+\eps t)\bigr).
\]
The nesting and product construction of the cutoffs give the exact support
identities
\[
    (\chi_{j,4}\eta_{j,4})\chi_{j,5}\eta_{j,3}
    =\chi_{j,4}\eta_{j,4},
    \qquad
    (\chi_{j,4}\eta_{j,4})\mathcal E^{\mathrm{out}}_{j,\h}=0,
\]
and, for each \(p\in\partial I_j\),
\[
    (\chi_{j,4}\eta_{j,4})
    \mathcal E^{\mathrm{trans}}_{j,p,\h}
    =
    (\chi_{j,4}\eta_{j,4})
    \mathcal E^{\mathrm{pot}}_{j,p,\h}
    =0.
\]
Thus the shifted representative satisfies the exact identity
\begin{equation}\label{eq: exact-interior-remainder}
    \widetilde{\mathcal E}^{\mathrm{rem}}_{j,\h}(x,t)
    =
    b_{j,\h}(x,t)(h_j'(x))^2H''(t).
\end{equation}
The cutoff estimates give, for any two points in \(I_j\times\R\),
\[
\begin{split}
    |b_{j,\h}(x,t)-b_{j,\h}(x',t')|
    &\leq
    C\eps^{-\alpha}
    \left(
        (1+\lVert h_j'\rVert_\infty)|x-x'|
        +\eps|t-t'|
    \right)\\
    &\leq
    C\eps^{1-\alpha}
    d_\eps\bigl((x,t),(x',t')\bigr).
\end{split}
\]
Moreover,
\[
    [(h_j')^2]_{\gamma,d_\eps}
    \leq
    \eps^\gamma[(h_j')^2]_{C^{0,\gamma}(I_j)}
    \leq C\eps^{2\beta+\gamma}.
\]
Consequently,
\[
    \lVert b_{j,\h}\rVert_
    {\mathcal C^{0,\gamma}_\eps(I_j\times\R)}
    \leq C,
    \qquad
    \lVert(h_j')^2\rVert_
    {\mathcal C^{0,\gamma}_\eps(I_j\times\R)}
    \leq C\eps^{2\beta}.
\]
Since \(\cosh(\lambda t)H''(t)\) has a fixed
\(C^{0,\gamma}\)-norm, the product
estimate in \eqref{eq: exact-interior-remainder} proves
\eqref{eq: interior-approximation-error-size}.

\emph{Estimate \eqref{eq: projected-approximation-error-size}.}
When a cutoff \(\eta_{j,m}\) occurs in an interval norm below, it denotes the
restriction \(x\mapsto\eta_{j,m}(X_j(x,0))\). Let
\[
    b^0_{j,\h}(x)\coloneqq b_{j,\h}(x,0).
\]
For some fixed \(c>0\), the product construction of the cutoffs gives
\[
    b_{j,\h}(x,t)=b^0_{j,\h}(x)
    \qquad\text{when }|t|\leq c\eps^{\alpha-1}.
\]
Indeed, \(\lVert h_j\rVert_\infty=O(\eps^\beta)=o(\eps^\alpha)\), so the points in
this range of \(t\) remain in the normal core where \(\chi_{j,4}=1\),
\(\chi_{j,10}=1\), and \(\eta_{j,4}\) is independent of the normal
variable. The cutoff estimates also give, for some fixed \(M>0\),
\begin{equation}\label{eq: interior-localizer-holder-bound}
    \lVert b_{j,\h}(\cdot,t)-b^0_{j,\h}\rVert_{C^{0,\gamma}(I_j)}
    \leq C\eps^{-M}
    \qquad\text{for every }t\in\R.
\end{equation}
Using
\[
    \int_\R H''H'
    =
    \frac12\int_\R\bigl((H')^2\bigr)'
    =0,
\]
we obtain from \eqref{eq: exact-interior-remainder}
\[
    \mathfrak p_j
    \bigl(\widetilde{\mathcal E}^{\mathrm{rem}}_{j,\h}\bigr)(x)
    =
    \frac{(h_j'(x))^2}{A_H}
    \int_\R
    \bigl(b_{j,\h}(x,t)-b^0_{j,\h}(x)\bigr)
    H''(t)H'(t)\,dt.
\]
The integrand is supported where
\(|t|\geq c\eps^{\alpha-1}\). Combining
\eqref{eq: interior-localizer-holder-bound} with the exponential decay of
\(H'H''\) gives
\[
    \left\lVert
        \int_\R
        (b_{j,\h}-b^0_{j,\h})H''H'\,dt
    \right\rVert_{C^{0,\gamma}(I_j)}
    \leq
    C\eps^{-M}
    e^{-c/\eps^{1-\alpha}}.
\]
This quantity is \(O_m(\eps^m)\) for every \(m>0\), and hence so is the
first term in \eqref{eq: projected-approximation-error-size}.

Finally, set
\[
    c_{j,\h}(x,t)
    \coloneqq
    (\chi_{j,4}\eta_{j,3})
    \bigl(X_j(x,h_j(x)+\eps t)\bigr).
\]
On the same core in the \(t\)-variable,
\(c_{j,\h}(x,t)=\eta_{j,3}(x)\). This follows from the same normal-core
argument, with \(\eta_{j,3}\) in place of \(\eta_{j,4}\). Moreover, after
increasing \(M\) if necessary,
\[
    \lVert c_{j,\h}(\cdot,t)-\eta_{j,3}\rVert_{C^{0,\gamma}(I_j)}
    \leq C\eps^{-M}
    \qquad\text{for every }t\in\R.
\]
Since
\(A_H=\int_\R(H')^2\), the definition of \(a_{j,\h}\) gives
\[
    a_{j,\h}(x)-\eta_{j,3}(x)
    =
    \frac1{A_H}\int_\R
    \bigl(c_{j,\h}(x,t)-\eta_{j,3}(x)\bigr)(H'(t))^2\,dt.
\]
The same tail estimate shows that this difference is
\(O_m(\eps^m)\) in \(C^{0,\gamma}(I_j)\) for every \(m>0\). Multiplication
by \(\eps h_j''\), using
\(\lVert h_j''\rVert_{C^{0,\gamma}(I_j)}\leq C\eps^\beta\), proves the second term
in \eqref{eq: projected-approximation-error-size} and completes the proof.
\end{proof}

\subsection{Dependence of the approximation error on $\h$}

We now compare the same terms for two admissible shifts. The boundary terms
must be placed in the common unrotated coordinates of
\eqref{eq: boundary-error-pullback}; subtracting their coordinate
representatives before this pullback would also compare two different model
half-lines.

\begin{lemma}[Dependence of the approximation error on the shift]
\label{lem: approximation-error-dependence}
Let \(\h^1,\h^2\in\mathcal B(K\eps^\beta)\), and set
\[
    M
    \coloneqq
    \lVert\h^1-\h^2\rVert_{C^{2,\gamma}(\boldsymbol I)}.
\]
There are constants \(C,c>0\), uniform in \(\h^1\) and \(\h^2\), such that
\begin{align}
    \lVert\mathcal E^{\mathrm{out}}_{j,\h^1}
    -\mathcal E^{\mathrm{out}}_{j,\h^2}\rVert_
    {C^{0,\gamma}_{I_j,\eps,\lambda}
    (\{0<\chi_{j,5}<1\})}
    &\leq
    C\exp\left(-\frac{c}{\eps^{1-\alpha}}\right)M,
    \label{eq: outer-error-dependence}\\
    \lVert\widetilde{\mathcal E}^{\mathrm{tan}}_{j,\h^1}
    -\widetilde{\mathcal E}^{\mathrm{tan}}_{j,\h^2}\rVert_
    {\mathcal C^{0,\gamma}_{\eps,\lambda}(I_j\times\R)}
    &\leq C\eps M,
    \label{eq: tangential-error-dependence}\\
    \lVert\widetilde{\mathcal E}^{\mathrm{ort}}_{j,\h^1}
    -\widetilde{\mathcal E}^{\mathrm{ort}}_{j,\h^2}\rVert_
    {\mathcal C^{0,\gamma}_{\eps,\lambda}(I_j\times\R)}
    &\leq C\eps^\beta M,
    \label{eq: orthogonal-error-dependence}\\
    \lVert\widehat{\chi_{j,5}
    \mathcal E^{\mathrm{trans}}_{j,p,\h^1}}_{p,\h^1}
    -\widehat{\chi_{j,5}
    \mathcal E^{\mathrm{trans}}_{j,p,\h^2}}_{p,\h^2}\rVert_
    {C^{0,\gamma}_{\ell_+,\lambda}(\R^2_*,\L)}
    &\leq C\eps^{1+\beta-\alpha}M,
    \label{eq: transition-error-dependence}\\
    \lVert\widehat{\chi_{j,5}
    \mathcal E^{\mathrm{pot}}_{j,p,\h^1}}_{p,\h^1}
    -\widehat{\chi_{j,5}
    \mathcal E^{\mathrm{pot}}_{j,p,\h^2}}_{p,\h^2}\rVert_
    {C^{0,\gamma}_{\ell_+,\lambda}(\R^2_*,\L)}
    &\leq C\eps^{3\beta}M,
    \label{eq: potential-error-dependence}\\
    \lVert\chi_{\Ical,2}
    (S(\omega_{\h^1})-S(\omega_{\h^2}))\rVert_
    {C^{0,\gamma}_{\Gamma,\eps,3\lambda/4}
    (\R^2_{\p},\L_{\p})}
    &\leq
    C\exp\left(-\frac{c}{\eps^{1-\alpha}}\right)M,
    \label{eq: outer-approximation-error-dependence}\\
        \lVert\widehat{\mathcal E}^{\mathrm{rem}}_{p,\h^1}
        -\widehat{\mathcal E}^{\mathrm{rem}}_{p,\h^2}\rVert_
        {C^{0,\gamma}_{\ell_+,\lambda}(\R^2_*,\L)}
        +
        \lVert\bigl(
            \widehat{\mathcal E}^{\mathrm{rem}}_{p,\h^1}
            -\widehat{\mathcal E}^{\mathrm{rem}}_{p,\h^2}
        \bigr)^\top\rVert_{L^\infty(\R^+)}
    &\leq C\eps^\beta M,
    \label{eq: boundary-approximation-error-dependence}\\
    A_{g_\tau}
    \left(
        \widehat{\mathcal E}^{\mathrm{rem}}_{p,\h^1}
        -\widehat{\mathcal E}^{\mathrm{rem}}_{p,\h^2}
    \right)
    &\leq
    C\eps^{\beta-(1-\alpha)\tau}M,
    \label{eq: boundary-approximation-error-dependence-Ag}\\
    \lVert\widetilde{\mathcal E}^{\mathrm{rem}}_{j,\h^1}
    -\widetilde{\mathcal E}^{\mathrm{rem}}_{j,\h^2}\rVert_
    {\mathcal C^{0,\gamma}_{\eps,\lambda}(I_j\times\R)}
    &\leq C\eps^\beta M
    \leq C\eps M
    \label{eq: interior-approximation-error-dependence}
\end{align}
and
\begin{multline}\label{eq: projected-approximation-error-dependence}
    \left\lVert
        \mathfrak p_j\left(
            \widetilde{\mathcal E}^{\mathrm{rem}}_{j,\h^1}
            -\widetilde{\mathcal E}^{\mathrm{rem}}_{j,\h^2}
        \right)
    \right\rVert_{C^{0,\gamma}(I_j)}
    +
    \left\lVert
        \eps(a_{j,\h^1}-\eta_{j,3})(h_j^1)''
        -\eps(a_{j,\h^2}-\eta_{j,3})(h_j^2)''
    \right\rVert_{C^{0,\gamma}(I_j)}\\
    \leq C_m\eps^m M,
\end{multline}

where \(j\in\{1,\dots,N\}\), \(p\in\partial I_j\), \(m>0\) is arbitrary,
and \(C_m\) may depend on \(m\).
\end{lemma}

\begin{proof}
Set
\[
    \h^s\coloneqq \h^2+s(\h^1-\h^2),
    \qquad 0\leq s\leq1.
\]
The ball \(\mathcal B(K\eps^\beta)\) is convex, so every \(\h^s\) is an
admissible shift. In particular,
\begin{equation}\label{eq: common-shift-difference-bounds}
\begin{aligned}
    \lVert h_j^i\rVert_{C^{2,\gamma}(I_j)}&\leq K\eps^\beta,
    &\lVert h_j^1-h_j^2\rVert_{C^{2,\gamma}(I_j)}&\leq M,\\
    |\theta_{\h^i}(p)|&\leq K\eps^\beta,
    &|\theta_{\h^1}(p)-\theta_{\h^2}(p)|&\leq M.
\end{aligned}
\end{equation}
We also write
\[
    L_\eps\coloneqq\eps^{\alpha-1}.
\]
Thus \(L_\eps\to\infty\), and for every fixed \(N,m>0\), after decreasing
the constant in the exponential if necessary,
\begin{equation}\label{eq: exponential-absorbs-powers}
    \eps^{-N}e^{-cL_\eps}
    \leq C_Ne^{-cL_\eps/2}
    \leq C_{N,m}\eps^m.
\end{equation}

\emph{Estimate \eqref{eq: outer-error-dependence}.}
Fix a trivializing ball \(B_\eps(q)\) contained in
\(\{0<\chi_{j,5}<1\}\), and set \(d(q)=\dist(q,I_j)\). Then
\(d(q)\simeq\eps^\alpha\). Choose
\[
    \lambda<\lambda_*<\min\{\sigma,\sqrt{\kappa_W}\}.
\]
In a fixed local trivialization, the mean-value formulas for a translated
heteroclinic and a rotated boundary profile are
\begin{align*}
    &H\left(\frac{z-h_j^1(x)}{\eps}\right)
    -H\left(\frac{z-h_j^2(x)}{\eps}\right)\\
    &\qquad={}
    -\frac{h_j^1(x)-h_j^2(x)}{\eps}
    \int_0^1
    H'\left(
        \frac{z-h_j^2(x)-s(h_j^1(x)-h_j^2(x))}{\eps}
    \right)\,ds,
    \\
    &U_{\theta_1}(y)-U_{\theta_2}(y)
    =
    (\theta_1-\theta_2)
    \int_0^1
    \left.
    \partial_\theta\bigl(U(R_\theta y)\bigr)
    \right|_{\theta=\theta_s}\,ds,
\end{align*}
where \(\theta_i=\theta_{\h^i}(p)\) and
\(\theta_s=\theta_2+s(\theta_1-\theta_2)\). The profile decay and
\eqref{eq: common-shift-difference-bounds} therefore give, for some fixed
\(N_0>0\),
\begin{equation}\label{eq: outer-profile-difference}
    \lVert\mathcal P_{\h^1,j}-\mathcal P_{\h^2,j}\rVert_
    {C^{2,\gamma}_\eps(B_\eps(q),\L_{\p})}
    \leq
    C\eps^{-N_0}
    e^{-\lambda_*d(q)/\eps}M.
\end{equation}
Here the fixed power of \(\eps^{-1}\) accounts for differentiation of the
translations, rotations, and tangential cutoffs.

Subtracting \eqref{eq: outer-residual-term} for the two shifts gives the
exact identity
\begin{align*}
    \mathcal E^{\mathrm{out}}_{j,\h^1}
    -\mathcal E^{\mathrm{out}}_{j,\h^2}
    ={}&
    \eps^2\Delta\chi_{j,5}
    (\mathcal P_{\h^1,j}-\mathcal P_{\h^2,j})\\
    &+2\eps^2\nabla\chi_{j,5}\mathbin{\cdot}
    \nabla(\mathcal P_{\h^1,j}-\mathcal P_{\h^2,j})\\
    &-\Bigl[
        \mathfrak I_W(\Ical,\mathcal P_{\h^1,j};\chi_{j,5})
        -\mathfrak I_W(\Ical,\mathcal P_{\h^2,j};\chi_{j,5})
    \Bigr].
\end{align*}
The cutoff derivative bounds, the smooth dependence of
\(\mathfrak I_W\) on its second profile, and
\eqref{eq: outer-profile-difference} imply, for some fixed \(N>0\),
\[
    e^{\lambda d(q)/\eps}
    \lVert\mathcal E^{\mathrm{out}}_{j,\h^1}
    -\mathcal E^{\mathrm{out}}_{j,\h^2}\rVert_
    {C^{0,\gamma}_\eps(B_\eps(q),\L_{\p})}
    \leq
    C\eps^{-N}
    e^{-(\lambda_*-\lambda)d(q)/\eps}M.
\]
Since \(d(q)\simeq\eps^\alpha\),
\eqref{eq: exponential-absorbs-powers} proves
\eqref{eq: outer-error-dependence}.

\emph{Estimate \eqref{eq: tangential-error-dependence}.}
The shifted interior pullback fixes the argument of the heteroclinic, and
hence
\[
    \widetilde{\mathcal E}^{\mathrm{tan}}_{j,\h^i}(x,t)
    =-\eps(h_j^i)''(x)H'(t).
\]
Consequently,
\[
    \widetilde{\mathcal E}^{\mathrm{tan}}_{j,\h^1}
    -\widetilde{\mathcal E}^{\mathrm{tan}}_{j,\h^2}
    =-\eps\bigl((h_j^1)''-(h_j^2)''\bigr)H'(t).
\]
If \(a=(h_j^1)''-(h_j^2)''\), then
\[
    \lVert a\rVert_{L^\infty(I_j)}+[a]_{\gamma,d_\eps}
    \leq
    \lVert a\rVert_{C^{0,\gamma}(I_j)}
    \leq M.
\]
Since \(\cosh(\lambda t)H'(t)\) has a fixed
\(C^{0,\gamma}(\R)\)-norm, the weighted product estimate gives
\eqref{eq: tangential-error-dependence}.

\emph{Estimate \eqref{eq: orthogonal-error-dependence}.}
Set
\[
    \widetilde\eta_i(x,t)
    \coloneqq
    \eta_{j,3}\bigl(X_j(x,h_j^i(x)+\eps t)\bigr),
    \qquad
    q_i(x)\coloneqq\bigl((h_j^i)'(x)\bigr)^2.
\]
Then the exact pulled-back formula is
\[
    \widetilde{\mathcal E}^{\mathrm{ort}}_{j,\h^i}
    =\widetilde\eta_iq_iH''(t),
\]
and therefore
\begin{equation}\label{eq: orthogonal-error-exact-difference}
\begin{split}
    \widetilde{\mathcal E}^{\mathrm{ort}}_{j,\h^1}
    -\widetilde{\mathcal E}^{\mathrm{ort}}_{j,\h^2}
    ={}&
    (\widetilde\eta_1-\widetilde\eta_2)q_1H''(t)\\
    &+\widetilde\eta_2(q_1-q_2)H''(t).
\end{split}
\end{equation}
The fixed-profile cutoff construction and the mean value theorem in the
normal variable give
\begin{equation}\label{eq: shifted-eta-difference}
\begin{split}
    \lVert\widetilde\eta_i\rVert_
    {\mathcal C^{0,\gamma}_\eps(I_j\times\R)}&\leq C,\\
    \lVert\widetilde\eta_1-\widetilde\eta_2\rVert_
    {\mathcal C^{0,\gamma}_\eps(I_j\times\R)}
    &\leq C\eps^{-\alpha}M.
\end{split}
\end{equation}
For the second estimate, the supremum bound costs \(\eps^{-\alpha}M\).
The corresponding stretched Lipschitz bounds cost at most
\(C\eps^{1-2\alpha}M\), which is harmless because \(\alpha<1/2\), and
they imply the stated H\"older bound. Moreover,
\begin{equation}\label{eq: squared-slope-difference}
\begin{split}
    \lVert q_i\rVert_{\mathcal C^{0,\gamma}_\eps(I_j\times\R)}
    &\leq C\eps^{2\beta},\\
    \lVert q_1-q_2\rVert_{\mathcal C^{0,\gamma}_\eps(I_j\times\R)}
    &\leq C\eps^\beta M,
\end{split}
\end{equation}
where the second line follows by factoring
\(q_1-q_2=((h_j^1)'+(h_j^2)')((h_j^1)'-(h_j^2)')\).
Applying \eqref{eq: shifted-eta-difference}--
\eqref{eq: squared-slope-difference} to
\eqref{eq: orthogonal-error-exact-difference} yields
\[
    \lVert\widetilde{\mathcal E}^{\mathrm{ort}}_{j,\h^1}
    -\widetilde{\mathcal E}^{\mathrm{ort}}_{j,\h^2}\rVert_
    {\mathcal C^{0,\gamma}_{\eps,\lambda}(I_j\times\R)}
    \leq
    C\bigl(\eps^{2\beta-\alpha}+\eps^\beta\bigr)M
    \leq C\eps^\beta M.
\]
This proves \eqref{eq: orthogonal-error-dependence}.

\emph{Estimate \eqref{eq: transition-error-dependence}.}
Fix \(p\in\partial I_j\), and abbreviate
\[
    \theta_i\coloneqq\theta_{\h^i}(p),
    \qquad
    \widehat\zeta_i
    \coloneqq\widehat{(\zeta_{p,3})}_{p,\h^i},
    \qquad
    \widehat\chi_i
    \coloneqq\widehat{(\chi_{j,5})}_{p,\h^i}.
\]
Set
\begin{equation}\label{eq: localized-boundary-transition-domain}
\begin{split}
    \widehat\Omega_i^{\mathrm{tr}}
    &\coloneqq
    \spt\widehat\chi_i
    \cap\overline{\{0<\widehat\zeta_i<1\}},\\
    \widehat\Omega_{12}^{\mathrm{tr}}
    &\coloneqq
    \widehat\Omega_1^{\mathrm{tr}}
    \cup\widehat\Omega_2^{\mathrm{tr}}.
\end{split}
\end{equation}
On this union, the nesting gives \(\chi_{j,10}=1\), and the cutoffs in the
definitions of the two model sections are inactive. Moreover,
\(cL_\eps\leq X\leq CL_\eps\).
Let \(D_i\) be the boundary representative of
\(\mathcal U_{\h^i,p}-\mathcal H_{\h^i,j}\) in the coordinates associated
with \(\h^i\). The affine boundary-collar condition gives, throughout
\(\widehat\Omega_{12}^{\mathrm{tr}}\),
\[
    D_i(X,Z)=U_-(X,Z)-H(T_i(X,Z)),
\]
where
\begin{equation}\label{eq: common-coordinate-phase}
    T_i(X,Z)
    =
    (\sin\theta_i-\theta_i\cos\theta_i)X
    +(\cos\theta_i+\theta_i\sin\theta_i)Z.
\end{equation}
The identities
\[
    T_\theta-Z=O(\theta^3X)+O(\theta^2Z),
    \qquad
    \partial_\theta T_\theta
    =\theta\sin\theta X+\theta\cos\theta Z,
\]
together with the decay of \(U_--H(Z)\) as \(X\to\infty\), give
\begin{equation}\label{eq: boundary-profile-lipschitz-bounds}
\begin{split}
    \lVert D_i\rVert_
    {C^{1,\gamma}_{\ell_+,\lambda}
    (\widehat\Omega_{12}^{\mathrm{tr}},\L)}
    &\leq C\eps^{2\beta},\\
    \lVert D_1-D_2\rVert_
    {C^{1,\gamma}_{\ell_+,\lambda}
    (\widehat\Omega_{12}^{\mathrm{tr}},\L)}
    &\leq C\eps^\beta M.
\end{split}
\end{equation}
In the second line, the factor \(|Z|\) is absorbed by heteroclinic decay, while
\(|X|\leq CL_\eps\) and \(\beta+\alpha-1>0\) control the tangential term.

The pulled-back fixed-profile cutoffs satisfy, for \(k=0,1,2\),
\begin{equation}\label{eq: boundary-cutoff-lipschitz-bounds}
\begin{split}
    \lVert\nabla^k\widehat\zeta_i\rVert_{C^{0,\gamma}(\R^2)}
    &\leq CL_\eps^{-k},\\
    \lVert\nabla^k(\widehat\zeta_1-\widehat\zeta_2)\rVert_
    {C^{0,\gamma}(\R^2)}
    &\leq CL_\eps^{-k}M.
\end{split}
\end{equation}
On \(\widehat\Omega_{12}^{\mathrm{tr}}\), set
\[
    \mathcal C_i
    \coloneqq
    (\Delta\widehat\zeta_i)D_i
    +2\nabla\widehat\zeta_i\mathbin{\cdot}\nabla D_i.
\]
This is precisely the boundary representative of
\(\mathcal E^{\mathrm{trans}}_{j,p,\h^i}\) wherever
\(\widehat\chi_i\neq0\). Subtracting gives the exact identity
\begin{align*}
    \mathcal C_1-\mathcal C_2
    ={}&
    (\Delta\widehat\zeta_1)(D_1-D_2)
    +2\nabla\widehat\zeta_1\mathbin{\cdot}\nabla(D_1-D_2)\\
    &+\Delta(\widehat\zeta_1-\widehat\zeta_2)D_2
    +2\nabla(\widehat\zeta_1-\widehat\zeta_2)
    \mathbin{\cdot}\nabla D_2.
\end{align*}
By \eqref{eq: boundary-profile-lipschitz-bounds} and
\eqref{eq: boundary-cutoff-lipschitz-bounds},
\begin{align}
    \lVert\mathcal C_1-\mathcal C_2\rVert_
    {C^{0,\gamma}_{\ell_+,\lambda}
    (\widehat\Omega_{12}^{\mathrm{tr}},\L)}
    &\leq
    C\left[
        (L_\eps^{-2}+L_\eps^{-1})\eps^\beta
        +(L_\eps^{-2}+L_\eps^{-1})\eps^{2\beta}
    \right]M
    \leq C\eps^{1+\beta-\alpha}M,
    \label{eq: boundary-commutator-difference-bound}\\
    \lVert\mathcal C_i\rVert_
    {C^{0,\gamma}_{\ell_+,\lambda}
    (\widehat\Omega_{12}^{\mathrm{tr}},\L)}
    &\leq
    C(L_\eps^{-2}+L_\eps^{-1})\eps^{2\beta}
    \leq CL_\eps^{-1}\eps^{2\beta}.
    \label{eq: boundary-commutator-fixed-bound}
\end{align}
Fixed-profile cutoff calculus also gives
\[
    \lVert\widehat\chi_i\rVert_{C^{0,\gamma}(\widehat\Omega_{12}^{\mathrm{tr}})}
    \leq C,
    \qquad
    \lVert\widehat\chi_1-\widehat\chi_2\rVert_{
        C^{0,\gamma}(\widehat\Omega_{12}^{\mathrm{tr}})}
    \leq CM.
\]
The localized representatives vanish outside
\(\widehat\Omega_{12}^{\mathrm{tr}}\), and on this union their difference is
\begin{equation}\label{eq: localized-transition-exact-difference}
    \widehat\chi_1\mathcal C_1-\widehat\chi_2\mathcal C_2
    =
    \widehat\chi_1(\mathcal C_1-\mathcal C_2)
    +(\widehat\chi_1-\widehat\chi_2)\mathcal C_2.
\end{equation}
The product estimate, \eqref{eq: boundary-commutator-difference-bound}, and
\eqref{eq: boundary-commutator-fixed-bound} prove
\eqref{eq: transition-error-dependence} on the full boundary model.

\emph{Estimate \eqref{eq: potential-error-dependence}.}
By the quadratic identity \eqref{eq: interpolation-defect-quadratic}, in the
same boundary coordinates set
\[
    \mathcal V_i\coloneqq-G_iD_i^2,
\]
where \(G_i\) is a smooth bounded function of the two profiles and of
\(\widehat\zeta_i\). On \(\widehat\Omega_{12}^{\mathrm{tr}}\), the
fixed-profile cutoff estimates give
\begin{equation}\label{eq: boundary-potential-coefficient-difference}
\begin{split}
    \lVert G_i\rVert_{C^{0,\gamma}(\widehat\Omega_{12}^{\mathrm{tr}})}
    &\leq C,\\
    \lVert G_1-G_2\rVert_{C^{0,\gamma}(\widehat\Omega_{12}^{\mathrm{tr}})}
    &\leq CM.
\end{split}
\end{equation}
The exact difference is
\[
    \mathcal V_1-\mathcal V_2
    =-G_1(D_1^2-D_2^2)-(G_1-G_2)D_2^2.
\]
Since \(D_1^2-D_2^2=(D_1+D_2)(D_1-D_2)\),
\eqref{eq: boundary-profile-lipschitz-bounds} and
\eqref{eq: boundary-potential-coefficient-difference} give
\begin{align}
    \lVert\mathcal V_1-\mathcal V_2\rVert_
    {C^{0,\gamma}_{\ell_+,\lambda}
    (\widehat\Omega_{12}^{\mathrm{tr}},\L)}
    &\leq C\eps^{3\beta}M,
    \label{eq: boundary-potential-bare-difference}\\
    \lVert\mathcal V_i\rVert_
    {C^{0,\gamma}_{\ell_+,\lambda}
    (\widehat\Omega_{12}^{\mathrm{tr}},\L)}
    &\leq C\eps^{4\beta}.
    \label{eq: boundary-potential-bare-size}
\end{align}
The localized representatives vanish outside
\(\widehat\Omega_{12}^{\mathrm{tr}}\), and their exact difference is
\begin{equation}\label{eq: localized-potential-exact-difference}
    \widehat\chi_1\mathcal V_1-\widehat\chi_2\mathcal V_2
    =
    \widehat\chi_1(\mathcal V_1-\mathcal V_2)
    +(\widehat\chi_1-\widehat\chi_2)\mathcal V_2.
\end{equation}
The product estimate, \eqref{eq: boundary-potential-bare-difference}, and
\eqref{eq: boundary-potential-bare-size} now give
\(C(\eps^{3\beta}+\eps^{4\beta})M\).
This proves \eqref{eq: potential-error-dependence}.

\emph{Estimate \eqref{eq: outer-approximation-error-dependence}.}
This global estimate must also include the exterior tails of the terms inside
the segment brackets in \eqref{eq: residual-identity}. Choose
\[
    \frac{3\lambda}{4}<\lambda_*<
    \min\{\sigma,\sqrt{\kappa_W}\}.
\]
The contribution from every singular ball is zero because
\(\chi_{\Ical,2}\) vanishes near \(\Gamma\). If a trivializing ball
\(B_\eps(q)\) meets \(\spt\chi_{\Ical,2}\), set
\(d(q)=\dist(q,\Gamma)\). Either the ansatz is independent of the shift on
that ball, or \(d(q)\geq c\eps^\alpha\). The same mean-value formulas used in
\eqref{eq: outer-profile-difference}, now applied to every profile entering
the ansatz, give, for a fixed \(N>0\),
\[
    \lVert\omega_{\h^1}-\omega_{\h^2}\rVert_
    {C^{2,\gamma}_\eps(B_\eps(q),\L_{\p})}
    \leq
    C\eps^{-N}e^{-\lambda_*d(q)/\eps}M.
\]
Because \(W'\) is smooth on the bounded range of the ansatz, the same bound,
with a possibly larger \(N\), holds for
\(S(\omega_{\h^1})-S(\omega_{\h^2})\) in
\(C^{0,\gamma}_\eps(B_\eps(q),\L_{\p})\). Multiplication by
\(\chi_{\Ical,2}\) costs only another fixed power of \(\eps^{-1}\).
After inserting the global weight, the resulting local bound is
\[
    C\eps^{-N}
    e^{-(\lambda_*-3\lambda/4)d(q)/\eps}M
    \leq
    C\exp\left(-\frac{c}{\eps^{1-\alpha}}\right)M.
\]
Taking the supremum over the trivializing balls proves
\eqref{eq: outer-approximation-error-dependence}.

\emph{Estimate \eqref{eq: boundary-approximation-error-dependence}.}
Set
\[
    \kappa_i
    \coloneqq
    \widehat{(\chi_{j,4}\zeta_{p,5})}_{p,\h^i},
    \qquad
    F_i
    \coloneqq
    \widehat{\mathcal E}^{\mathrm{rem}}_{p,\h^i}.
\]
Define the common boundary localization region by
\[
    \widehat\Omega_{12}^{\mathrm{bd}}
    \coloneqq
    \spt\kappa_1\cup\spt\kappa_2.
\]
Fixed-profile cutoff calculus gives
\begin{equation}\label{eq: boundary-localizer-difference}
    \lVert\kappa_i\rVert_{C^{0,\gamma}(\R^2)}\leq C,
    \qquad
    \lVert\kappa_1-\kappa_2\rVert_{C^{0,\gamma}(\R^2)}\leq CM,
\end{equation}
and
\[
    \widehat\Omega_{12}^{\mathrm{bd}}
    \subset\{|X|+|Z|\leq CL_\eps\}.
\]
On this domain, write
\[
    Q_i
    \coloneqq
    \widehat{\mathcal E}^{\mathrm{ort}}_{j,\h^i}
    +\widehat{\mathcal E}^{\mathrm{trans}}_{j,p,\h^i}
    +\widehat{\mathcal E}^{\mathrm{pot}}_{j,p,\h^i}.
\]
The exact localized decomposition
\eqref{eq: localized-boundary-remainder-decomposition} becomes
\begin{equation}\label{eq: boundary-remainder-difference-decomposition}
    F_i=\kappa_iQ_i.
\end{equation}
The raw estimates in the proof of
\Cref{lem: approximation-remainder-size} give
\begin{equation}\label{eq: common-boundary-remainder-raw-bound}
    \lVert Q_i\rVert_
    {C^{0,\gamma}_{\ell_+,\lambda}
    (\widehat\Omega_{12}^{\mathrm{bd}},\L)}
    \leq C\eps^{2\beta}.
\end{equation}
Indeed, each pullback of \(\chi_{j,5}\) equals one on its own localization
region. On the part of the other localization region where this is not so,
the distance to the segment is comparable to \(\eps^\alpha\), and the same
profile-tail estimate used above is exponentially small.

It remains to compare the orthogonal contributions to \(Q_1\) and \(Q_2\).
Set
\[
    \widehat\Omega_{12}^{\mathrm{ort}}
    \coloneqq
    \widehat\Omega_{12}^{\mathrm{bd}}
    \cap
    \bigl(
        \spt\widehat{(\eta_{j,3})}_{p,\h^1}
        \cup
        \spt\widehat{(\eta_{j,3})}_{p,\h^2}
    \bigr).
\]
The separation between the boundary point and the support of
\(\eta_{j,3}\), together with \eqref{eq: common-shift-difference-bounds},
ensures that both inward coordinates below belong to \(I_j\) on this
domain. For \((X,Z)\in\widehat\Omega_{12}^{\mathrm{ort}}\), set
\begin{align*}
    x_i&=\eps(X\cos\theta_i-Z\sin\theta_i),\\
    t_i&=X\sin\theta_i+Z\cos\theta_i
    -\eps^{-1}h_j^i(p+x_i e_p),\\
    q_i&=\bigl((h_j^i)'(p+x_i e_p)\bigr)^2,
    \qquad
    \widehat\eta_i=\widehat{(\eta_{j,3})}_{p,\h^i}.
\end{align*}
Then
\[
    \widehat{\mathcal E}^{\mathrm{ort}}_{j,\h^i}
    =\widehat\eta_iq_iH''(t_i).
\]
Outside \(\widehat\Omega_{12}^{\mathrm{ort}}\), both orthogonal terms
vanish. The phase satisfies
\begin{equation}\label{eq: boundary-orthogonal-phase-difference}
\begin{split}
    \lVert t_i-Z\rVert_
    {C^{0,\gamma}(\widehat\Omega_{12}^{\mathrm{ort}})}
    &\leq
    C\bigl(\eps^{\beta-1}+\eps^\beta L_\eps\bigr)=o(1),\\
    \lVert t_1-t_2\rVert_
    {C^{0,\gamma}(\widehat\Omega_{12}^{\mathrm{ort}})}
    &\leq CL_\eps M.
\end{split}
\end{equation}
The first line makes the weights centered at \(t_i=0\) uniformly equivalent
to the weight centered at \(Z=0\). The cutoff bounds,
\eqref{eq: common-shift-difference-bounds}, and the mean value theorem now
give
\begin{align}
    \lVert\widehat\eta_i\rVert_
    {C^{0,\gamma}(\widehat\Omega_{12}^{\mathrm{ort}})}
    &\leq C,
    &
    \lVert\widehat\eta_1-\widehat\eta_2\rVert_
    {C^{0,\gamma}(\widehat\Omega_{12}^{\mathrm{ort}})}
    &\leq CM,
    \label{eq: boundary-eta-difference}\\
    \lVert q_i\rVert_
    {C^{0,\gamma}(\widehat\Omega_{12}^{\mathrm{ort}})}
    &\leq C\eps^{2\beta},
    &
    \lVert q_1-q_2\rVert_
    {C^{0,\gamma}(\widehat\Omega_{12}^{\mathrm{ort}})}
    &\leq C\eps^\beta M,
    \label{eq: boundary-squared-slope-difference}\\
    \lVert H''(t_1)-H''(t_2)\rVert_
    {C^{0,\gamma}_{\ell_+,\lambda}
    (\widehat\Omega_{12}^{\mathrm{ort}},\L)}
    &\leq CL_\eps M.
    \label{eq: boundary-Hpp-difference}
\end{align}
The exact difference of the two orthogonal terms is
\begin{align*}
    &\widehat\eta_1q_1H''(t_1)
    -\widehat\eta_2q_2H''(t_2)\\
    &\qquad={}
    (\widehat\eta_1-\widehat\eta_2)q_1H''(t_1)
    +\widehat\eta_2(q_1-q_2)H''(t_1)\\
    &\qquad\quad
    +\widehat\eta_2q_2\bigl(H''(t_1)-H''(t_2)\bigr).
\end{align*}
It follows from \eqref{eq: boundary-eta-difference}--
\eqref{eq: boundary-Hpp-difference} that
\begin{align*}
    \lVert\widehat{\mathcal E}^{\mathrm{ort}}_{j,\h^1}
    -\widehat{\mathcal E}^{\mathrm{ort}}_{j,\h^2}\rVert_
    {C^{0,\gamma}_{\ell_+,\lambda}
    (\widehat\Omega_{12}^{\mathrm{ort}},\L)}
    &\leq
    C\bigl(\eps^\beta+\eps^{2\beta}
    +\eps^{2\beta}L_\eps\bigr)M\\
    &\leq C\eps^\beta M,
\end{align*}
because \(\beta+\alpha-1>0\). Together with
\eqref{eq: transition-error-dependence} and
\eqref{eq: potential-error-dependence}, and with the exponentially small
profile-tail part described after
\eqref{eq: common-boundary-remainder-raw-bound}, this gives
\begin{equation}\label{eq: common-boundary-Q-difference}
    \lVert Q_1-Q_2\rVert_
    {C^{0,\gamma}_{\ell_+,\lambda}
    (\widehat\Omega_{12}^{\mathrm{bd}},\L)}
    \leq C\eps^\beta M.
\end{equation}
Finally,
\[
    F_1-F_2
    =\kappa_1(Q_1-Q_2)+(\kappa_1-\kappa_2)Q_2.
\]
Equations \eqref{eq: boundary-localizer-difference}--
\eqref{eq: common-boundary-Q-difference} prove
\[
    \lVert F_1-F_2\rVert_
    {C^{0,\gamma}_{\ell_+,\lambda}(\R^2_*,\L)}
    \leq C\eps^\beta M.
\]
For \(X>0\), the weight relative to \(\ell_+\) is equivalent to
\(e^{\lambda|Z|}\). Hence
\[
    |(F_1-F_2)^\top(X)|
    \leq
    C\eps^\beta M
    \int_\R e^{-\lambda|Z|}|H'(Z)|\,dZ
    \leq C\eps^\beta M.
\]
This proves both terms in
\eqref{eq: boundary-approximation-error-dependence}.

\emph{Estimate \eqref{eq: boundary-approximation-error-dependence-Ag}.}
The support of the boundary localizers gives
\[
    \spt(F_1-F_2)^\top\subset\{0<X\leq CL_\eps\}.
\]
Since \(g_\tau''(X)=(\tau-1)(\tau-2)(1+X)^{-\tau}\), the definition of
\(A_{g_\tau}\) and the projection estimate just proved yield
\begin{align*}
    A_{g_\tau}(F_1-F_2)
    &\leq
    C\sup_{0<X\leq CL_\eps}
    (1+X)^\tau|(F_1-F_2)^\top(X)|\\
    &\leq C\eps^\beta L_\eps^\tau M
    =C\eps^{\beta-(1-\alpha)\tau}M.
\end{align*}
This proves \eqref{eq: boundary-approximation-error-dependence-Ag}.

\emph{Estimate \eqref{eq: interior-approximation-error-dependence}.}
Set
\[
    b_i(x,t)
    \coloneqq
    (\chi_{j,4}\eta_{j,4})
    \bigl(X_j(x,h_j^i(x)+\eps t)\bigr),
    \qquad
    q_i(x)\coloneqq\bigl((h_j^i)'(x)\bigr)^2.
\]
The exact identity \eqref{eq: exact-interior-remainder} reads
\[
    \widetilde{\mathcal E}^{\mathrm{rem}}_{j,\h^i}
    =b_iq_iH''(t).
\]
In addition to \eqref{eq: squared-slope-difference}, the cutoff construction
and the normal mean value theorem give
\[
    \lVert b_i\rVert_{\mathcal C^{0,\gamma}_\eps(I_j\times\R)}\leq C,
    \qquad
    \lVert b_1-b_2\rVert_{\mathcal C^{0,\gamma}_\eps(I_j\times\R)}
    \leq C\eps^{-\alpha}M.
\]
Therefore
\[
    b_1q_1H''-b_2q_2H''
    =b_1(q_1-q_2)H''+(b_1-b_2)q_2H'',
\]
and the weighted product estimate gives
\[
    \lVert\widetilde{\mathcal E}^{\mathrm{rem}}_{j,\h^1}
    -\widetilde{\mathcal E}^{\mathrm{rem}}_{j,\h^2}\rVert_
    {\mathcal C^{0,\gamma}_{\eps,\lambda}(I_j\times\R)}
    \leq
    C\bigl(\eps^\beta+\eps^{2\beta-\alpha}\bigr)M
    \leq C\eps^\beta M.
\]
Since \(\beta\geq2\) and \(0<\eps<1\), this is also bounded by
\(C\eps M\). This proves
\eqref{eq: interior-approximation-error-dependence}.

\emph{Estimate \eqref{eq: projected-approximation-error-dependence}.}
First set
\[
    b_i^0(x)\coloneqq b_i(x,0),
    \qquad
    d_i(x,t)\coloneqq b_i(x,t)-b_i^0(x).
\]
The normal-core argument in the proof of
\eqref{eq: projected-approximation-error-size} is uniform for the segment
joining \(\h^1\) to \(\h^2\). Thus, for some fixed \(c>0\),
\[
    d_i(\cdot,t)=0
    \qquad\text{when }|t|\leq cL_\eps.
\]
The cutoff bounds and the mean value theorem with respect to the shift give,
for some fixed \(N>0\),
\begin{align*}
    \lVert d_i(\cdot,t)\rVert_{C^{0,\gamma}(I_j)}
    &\leq C\eps^{-N},\\
    \lVert d_1(\cdot,t)-d_2(\cdot,t)\rVert_{C^{0,\gamma}(I_j)}
    &\leq C\eps^{-N}M
\end{align*}
for every \(t\in\R\). Define
\[
    J_i(x)
    \coloneqq
    \frac1{A_H}\int_\R d_i(x,t)H''(t)H'(t)\,dt.
\]
The exponential decay of \(H'H''\) then gives
\begin{equation}\label{eq: projected-remainder-tail-difference}
\begin{split}
    \lVert J_i\rVert_{C^{0,\gamma}(I_j)}
    &\leq C\eps^{-N}e^{-cL_\eps},\\
    \lVert J_1-J_2\rVert_{C^{0,\gamma}(I_j)}
    &\leq C\eps^{-N}e^{-cL_\eps}M.
\end{split}
\end{equation}
Since \(\int_\R H''H'=0\),
\[
    \mathfrak p_j
    (\widetilde{\mathcal E}^{\mathrm{rem}}_{j,\h^i})
    =q_iJ_i.
\]
In the ordinary interval norm, the same factorization used in
\eqref{eq: squared-slope-difference} gives
\begin{equation}\label{eq: ordinary-squared-slope-difference}
\begin{split}
    \lVert q_i\rVert_{C^{0,\gamma}(I_j)}
    &\leq C\eps^{2\beta},\\
    \lVert q_1-q_2\rVert_{C^{0,\gamma}(I_j)}
    &\leq C\eps^\beta M.
\end{split}
\end{equation}
Consequently,
\[
    \mathfrak p_j
    (\widetilde{\mathcal E}^{\mathrm{rem}}_{j,\h^1}
    -\widetilde{\mathcal E}^{\mathrm{rem}}_{j,\h^2})
    =(q_1-q_2)J_1+q_2(J_1-J_2).
\]
Combining \eqref{eq: ordinary-squared-slope-difference},
\eqref{eq: projected-remainder-tail-difference}, and
\eqref{eq: exponential-absorbs-powers} proves that the first norm in
\eqref{eq: projected-approximation-error-dependence} is bounded by
\(C_m\eps^mM\) for every \(m>0\).

For the second norm, set
\[
    c_i(x,t)
    \coloneqq
    (\chi_{j,4}\eta_{j,3})
    \bigl(X_j(x,h_j^i(x)+\eps t)\bigr),
    \qquad
    K_i\coloneqq a_{j,\h^i}-\eta_{j,3}.
\]
The definition of \(a_{j,\h^i}\) gives the exact formula
\[
    K_i(x)
    =
    \frac1{A_H}\int_\R
    \bigl(c_i(x,t)-\eta_{j,3}(x)\bigr)(H'(t))^2\,dt.
\]
On the same normal core, \(c_i(x,t)=\eta_{j,3}(x)\). The cutoff bounds,
the mean value theorem with respect to the shift, and the exponential decay
of \((H')^2\) therefore give, after increasing \(N\) if necessary,
\begin{equation}\label{eq: coefficient-tail-difference}
\begin{split}
    \lVert K_i\rVert_{C^{0,\gamma}(I_j)}
    &\leq C\eps^{-N}e^{-cL_\eps},\\
    \lVert K_1-K_2\rVert_{C^{0,\gamma}(I_j)}
    &\leq C\eps^{-N}e^{-cL_\eps}M.
\end{split}
\end{equation}
The exact algebraic decomposition is
\[
    \eps K_1(h_j^1)''-\eps K_2(h_j^2)''
    =
    \eps K_1\bigl((h_j^1)''-(h_j^2)''\bigr)
    +\eps(K_1-K_2)(h_j^2)''.
\]
Using \eqref{eq: common-shift-difference-bounds},
\eqref{eq: coefficient-tail-difference}, and
\eqref{eq: exponential-absorbs-powers}, the second norm in
\eqref{eq: projected-approximation-error-dependence} is also bounded by
\(C_m\eps^mM\). This completes the proof of
\eqref{eq: projected-approximation-error-dependence}.
\end{proof}

\subsection{Profile and potential comparisons}
\label{subs: profile-potential-comparisons}

The boundary equation compares the approximate solution with two different
models.  These comparisons take place on different regions and therefore have
different sizes.  To state them precisely, fix \(p\in\partial I_j\).
On the segment-side overlap \(x>0\), let \(H_{\h,j}\) be the uncut model
section represented by
\[
    H_{\h,j}(Y_p(x,z))
    =
    H\left(
        \frac{z-h_j(p+xe_p)}{\eps}
    \right).
\]
On the full boundary neighborhood of \(p\), let \(U_{\h,p}\) be represented
by
\[
    U_{\h,p}(Y_p(x,z))
    =
    \mathcal T_p
    \left(
        U_{\theta_{\h}(p)}\left(\frac{x}{\eps},\frac{z}{\eps}\right)
    \right).
\]
The first formula is read in the segment trivialization, which is compatible
with the chosen boundary trivialization on the overlap. Every localized
expression involving \(H_{\h,j}\) below is multiplied by \(\eta_{j,1}\) or
\(\eta_{j,3}\), and is extended by zero outside the segment-side overlap.

\begin{lemma}[Fixed-shift profile and potential comparisons]
\label{lem: profile-potential-comparisons}
There is a constant \(C>0\), uniform for
\(\h\in\mathcal B(K\eps^\beta)\), such that
\begin{align}
    \left\lVert
        \widehat{
            \chi_{j,4}\zeta_{p,5}
            (\omega_{\h}-U_{\h,p})
        }_{p,\h}
    \right\rVert_{L^\infty(\R^2_*,\L)}
    &\leq C\eps^{\beta-1+\alpha(2+\gamma)},
    \label{eq: U-profile-comparison-Linfty}\\
    \left\lVert
        \widehat{
            \chi_{j,4}\zeta_{p,5}
            [W''(\omega_{\h})-W''(U_{\h,p})]
        }_{p,\h}
    \right\rVert_{L^\infty(\R^2_*)}
    &\leq C\eps^{\beta-1+\alpha(2+\gamma)},
    \label{eq: U-potential-comparison-Linfty}\\
    \left\lVert
        \widehat{
            \chi_{j,4}\zeta_{p,5}
            (\omega_{\h}-U_{\h,p})
        }_{p,\h}
    \right\rVert_{C^{0,\gamma}_{\ell_+,\lambda}(\R^2_*,\L)}
    &\leq C\eps^{\beta-1+2\alpha},
    \label{eq: U-profile-comparison-Holder}\\
    \left\lVert
        \widehat{
            \chi_{j,4}\zeta_{p,5}
            [W''(\omega_{\h})-W''(U_{\h,p})]
        }_{p,\h}
    \right\rVert_{C^{0,\gamma}_{\ell_+,\lambda}(\R^2_*)}
    &\leq C\eps^{\beta-1+2\alpha}.
    \label{eq: U-potential-comparison-Holder}
\end{align}
On the smaller affine overlap, the comparison with the interior model is
stronger:
\begin{equation}\label{eq: full-overlap-profile-comparison}
    \lVert U_{\h,p}-H_{\h,j}\rVert_
    {C^{2,\gamma}_{I_j,\eps,\lambda}
    (\spt(\chi_{j,5}\eta_{j,1}\zeta_{p,3}),\L_{\p})}
    \leq C\eps^{2\beta}.
\end{equation}
Its localized profile and potential consequences are
\begin{align}
    \left\lVert
        \widehat{
            \chi_{j,4}\zeta_{p,5}\eta_{j,1}
            (\omega_{\h}-H_{\h,j})
        }_{p,\h}
    \right\rVert_{C^{0,\gamma}_{\ell_+,\lambda}(\R^2_*,\L)}
    &\leq C\eps^{2\beta},
    \label{eq: H-profile-comparison-size}\\
    \left\lVert
        \widehat{
            \chi_{j,4}\zeta_{p,5}\eta_{j,1}
            [W''(\omega_{\h})-W''(H_{\h,j})]
        }_{p,\h}
    \right\rVert_{C^{0,\gamma}_{\ell_+,\lambda}(\R^2_*)}
    &\leq C\eps^{2\beta}.
    \label{eq: H-potential-comparison-size}
\end{align}
\end{lemma}

\begin{proof}
On \(\spt(\chi_{j,4}\zeta_{p,5})\), cutoff nesting gives
\(\chi_{j,5}=1\), and the neighborhood associated with the other boundary
point of \(I_j\) is disjoint. Hence
\[
    \omega_{\h}
    =\eta_{j,3}H_{\h,j}+\zeta_{p,3}U_{\h,p},
    \qquad
    \eta_{j,3}+\zeta_{p,3}=1.
\]
It follows that
\begin{equation}\label{eq: two-local-profile-identities}
\begin{aligned}
    \chi_{j,4}\zeta_{p,5}(\omega_{\h}-U_{\h,p})
    &=
    \chi_{j,4}\zeta_{p,5}\eta_{j,3}
    (H_{\h,j}-U_{\h,p}),\\
    \chi_{j,4}\zeta_{p,5}\eta_{j,1}
    (\omega_{\h}-H_{\h,j})
    &=
    \chi_{j,4}\zeta_{p,5}\eta_{j,1}\zeta_{p,3}
    (U_{\h,p}-H_{\h,j}).
\end{aligned}
\end{equation}

Write \(\theta=\theta_{\h}(p)\). After the boundary pullback,
\(U_{\h,p}\) is represented by the fixed section \(U(X,Z)\). Set
\[
    t_\theta(X,Z)
    =
    (\cos\theta+\theta\sin\theta)Z
    +(\sin\theta-\theta\cos\theta)X.
\]

We first use the identity for \(\omega_{\h}-U_{\h,p}\). The corresponding
localized expression vanishes outside
\(\spt(\chi_{j,4}\zeta_{p,5}\eta_{j,3})\). On this support the inward
coordinate is positive and \(O(\eps^\alpha)\). Let
\(a_{p,\eps}\simeq\eps^\alpha\) denote the inner edge of that collar in the
inward coordinate \(x\), and set
\[
    r_{j,\h}(x)
    =
    h_j(p+xe_p)-\theta_{\h}(p)x.
\]
The function \(r_{j,\h}\) and its first two derivatives vanish at the affine
side of the collar, with the second derivative understood by continuity.
The support condition on \(h_j''\) and its \(C^{0,\gamma}\)-bound imply
\begin{align}
    |r_{j,\h}(x)|
    &\leq
    C\eps^\beta(x-a_{p,\eps})_+^{2+\gamma},
    &
    |r_{j,\h}(x)|
    &\leq
    C\eps^\beta(x-a_{p,\eps})_+^2.
    \label{eq: nonaffine-shift-Taylor}
\end{align}
In the pulled-back boundary coordinates, the exact interior phase is
\[
    t_{\h}(X,Z)
    =
    t_{\theta_{\h}(p)}(X,Z)
    -\eps^{-1}r_{j,\h}
    \bigl(\eps(\cos\theta X-\sin\theta Z)\bigr).
\]
The first estimate in \eqref{eq: nonaffine-shift-Taylor} gives the sharper
factor
\[
    C\eps^{\beta-1+\alpha(2+\gamma)}
\]
in \(L^\infty\), whereas the quadratic estimate and the scaled composition
estimate give
\[
    C\eps^{\beta-1+2\alpha}
\]
in the full weighted \(C^{0,\gamma}\)-norm. The affine phase error and the
convergence of \(U\) to \(H\) are smaller on this region. This proves
\eqref{eq: U-profile-comparison-Linfty} and
\eqref{eq: U-profile-comparison-Holder}. Applying the smooth composition
estimate for \(W''\) proves \eqref{eq: U-potential-comparison-Linfty} and
\eqref{eq: U-potential-comparison-Holder}.

We next use the identity for \(\omega_{\h}-H_{\h,j}\). Its localization by
\(\eta_{j,1}\zeta_{p,3}\) lies in the affine boundary collar, where
\(H_{\h,j}\) is represented by \(H(t_\theta)\) and
\[
    t_\theta-Z
    =
    O(\theta^2)Z+O(\theta^3)X.
\]
On this localization, \(X\simeq\eps^{\alpha-1}\). The convergence of
\(U(X,Z)\) to \(H(Z)\), the exponential decay of the derivatives of \(H\),
and \(|\theta|\leq K\eps^\beta\) give contributions bounded by
\[
    C\left(
        e^{-c/\eps^{1-\alpha}}
        +\eps^{2\beta}
        +\eps^{3\beta+\alpha-1}
    \right)
    \leq C\eps^{2\beta},
\]
which proves \eqref{eq: full-overlap-profile-comparison}.
The second identity in \eqref{eq: two-local-profile-identities} and the
uniform cutoff bounds prove \eqref{eq: H-profile-comparison-size}. The smooth
composition estimate for \(W''\) gives
\eqref{eq: H-potential-comparison-size}.
\end{proof}

\begin{lemma}[Dependence of the profile and potential comparisons on the shift]
\label{lem: profile-potential-comparison-dependence}
For two shifts \(\h^1,\h^2\in\mathcal B(K\eps^\beta)\), set
\[
    M=\lVert\h^1-\h^2\rVert_{C^{2,\gamma}(\boldsymbol I)}.
\]
There is a constant \(C>0\), independent of \(\h^1,\h^2\), such that
\begin{equation}\label{eq: U-profile-comparison-dependence-Linfty}
\left\lVert
\begin{aligned}
&\widehat{
    \chi_{j,4}\zeta_{p,5}
    (\omega_{\h^1}-U_{\h^1,p})
}_{p,\h^1}\\
&\quad-
\widehat{
    \chi_{j,4}\zeta_{p,5}
    (\omega_{\h^2}-U_{\h^2,p})
}_{p,\h^2}
\end{aligned}
\right\rVert_{L^\infty(\R^2_*,\L)}
\leq C\eps^{-1+\alpha(2+\gamma)}M .
\end{equation}
\begin{equation}\label{eq: U-potential-comparison-dependence-Linfty}
\left\lVert
\begin{aligned}
&\widehat{
    \chi_{j,4}\zeta_{p,5}
    [W''(\omega_{\h^1})-W''(U_{\h^1,p})]
}_{p,\h^1}\\
&\quad-
\widehat{
    \chi_{j,4}\zeta_{p,5}
    [W''(\omega_{\h^2})-W''(U_{\h^2,p})]
}_{p,\h^2}
\end{aligned}
\right\rVert_{L^\infty(\R^2_*)}
\leq C\eps^{-1+\alpha(2+\gamma)}M .
\end{equation}
\begin{equation}\label{eq: U-profile-comparison-dependence-Holder}
\left\lVert
\begin{aligned}
&\widehat{
    \chi_{j,4}\zeta_{p,5}
    (\omega_{\h^1}-U_{\h^1,p})
}_{p,\h^1}\\
&\quad-
\widehat{
    \chi_{j,4}\zeta_{p,5}
    (\omega_{\h^2}-U_{\h^2,p})
}_{p,\h^2}
\end{aligned}
\right\rVert_{C^{0,\gamma}_{\ell_+,\lambda}(\R^2_*,\L)}
\leq C\eps^{-1+2\alpha}M .
\end{equation}
\begin{equation}\label{eq: U-potential-comparison-dependence-Holder}
\left\lVert
\begin{aligned}
&\widehat{
    \chi_{j,4}\zeta_{p,5}
    [W''(\omega_{\h^1})-W''(U_{\h^1,p})]
}_{p,\h^1}\\
&\quad-
\widehat{
    \chi_{j,4}\zeta_{p,5}
    [W''(\omega_{\h^2})-W''(U_{\h^2,p})]
}_{p,\h^2}
\end{aligned}
\right\rVert_{C^{0,\gamma}_{\ell_+,\lambda}(\R^2_*)}
\leq C\eps^{-1+2\alpha}M .
\end{equation}
For the comparison with the interior model,
\begin{equation}\label{eq: H-profile-comparison-dependence}
\left\lVert
\begin{aligned}
&\widehat{
    \chi_{j,4}\zeta_{p,5}\eta_{j,1}
    (\omega_{\h^1}-H_{\h^1,j})
}_{p,\h^1}\\
&\quad-
\widehat{
    \chi_{j,4}\zeta_{p,5}\eta_{j,1}
    (\omega_{\h^2}-H_{\h^2,j})
}_{p,\h^2}
\end{aligned}
\right\rVert_{C^{0,\gamma}_{\ell_+,\lambda}(\R^2_*,\L)}
\leq C\eps^\beta M .
\end{equation}
\begin{equation}\label{eq: H-potential-comparison-dependence}
\left\lVert
\begin{aligned}
&\widehat{
    \chi_{j,4}\zeta_{p,5}\eta_{j,1}
    [W''(\omega_{\h^1})-W''(H_{\h^1,j})]
}_{p,\h^1}\\
&\quad-
\widehat{
    \chi_{j,4}\zeta_{p,5}\eta_{j,1}
    [W''(\omega_{\h^2})-W''(H_{\h^2,j})]
}_{p,\h^2}
\end{aligned}
\right\rVert_{C^{0,\gamma}_{\ell_+,\lambda}(\R^2_*)}
\leq C\eps^\beta M .
\end{equation}
\end{lemma}

\begin{proof}[Proof of \Cref{lem: profile-potential-comparison-dependence}]
Write
\[
    \theta_i=\theta_{\h^i}(p),
    \qquad i=1,2.
\]
Set
\[
    t_\theta(X,Z)
    =
    (\cos\theta+\theta\sin\theta)Z
    +(\sin\theta-\theta\cos\theta)X.
\]

We first compare the expressions involving \(U_{\h^i,p}\). They vanish
outside the corresponding pullbacks of
\(\spt(\chi_{j,4}\zeta_{p,5}\eta_{j,3})\). The mismatch between these two
supports is controlled by the variation of the pulled-back cutoffs and the
fixed-shift estimates. On their common part, use the inner collar coordinate
\(a_{p,\eps}\simeq\eps^\alpha\) as in the preceding proof and set
\[
    r_i(x)
    =
    h_j^i(p+xe_p)-\theta_i x.
\]
The support condition on \((h_j^i)''\), together with its
\(C^{0,\gamma}\)-bound, gives
\begin{align}
    |r_i(x)|
    &\leq
    C\eps^\beta(x-a_{p,\eps})_+^{2+\gamma},
    &
    |r_i(x)|
    &\leq
    C\eps^\beta(x-a_{p,\eps})_+^2,
    \label{eq: nonaffine-shift-Taylor-fixed}\\
    |r_i'(x)|
    &\leq
    C\eps^\beta(x-a_{p,\eps})_+^{1+\gamma},
    &
    |r_i'(x)|
    &\leq
    C\eps^\beta(x-a_{p,\eps})_+,
    \label{eq: nonaffine-shift-derivative}\\
    |r_1(x)-r_2(x)|
    &\leq
    CM(x-a_{p,\eps})_+^{2+\gamma},
    &
    |r_1(x)-r_2(x)|
    &\leq
    CM(x-a_{p,\eps})_+^2.
    \label{eq: nonaffine-shift-Taylor-difference}
\end{align}
For \((X,Z)\) in this common part, define
\[
    x_i
    =
    \eps(\cos\theta_i X-\sin\theta_i Z).
\]
Then \(0<x_i\leq C\eps^\alpha\),
\[
    |x_1-x_2|
    \leq
    C\eps(|X|+|Z|)M
    \leq C\eps^\alpha M,
\]
and the interior phase in the coordinates associated with \(\h^i\) is
\[
    t_{\h^i}(X,Z)
    =
    t_{\theta_i}(X,Z)
    -\eps^{-1}r_i(x_i).
\]
The exact decomposition
\[
    r_1(x_1)-r_2(x_2)
    =
    [r_1(x_1)-r_2(x_1)]
    +
    [r_2(x_1)-r_2(x_2)]
\]
allows us to apply \eqref{eq: nonaffine-shift-Taylor-difference} to the first
term and \eqref{eq: nonaffine-shift-derivative} to the second. It gives the
\(L^\infty\)-factor
\[
    C\eps^{-1+\alpha(2+\gamma)}M,
\]
whereas the quadratic estimate and the scaled composition estimate give
\[
    C\eps^{-1+2\alpha}M
\]
in the full weighted \(C^{0,\gamma}\)-norm. The affine phase and pulled-back
cutoff differences are smaller. Applying these bounds to the first identity
in \eqref{eq: two-local-profile-identities} proves
\eqref{eq: U-profile-comparison-dependence-Linfty} and
\eqref{eq: U-profile-comparison-dependence-Holder}.

The smooth composition estimate for \(W''\) now proves
\eqref{eq: U-potential-comparison-dependence-Linfty} and
\eqref{eq: U-potential-comparison-dependence-Holder}.

We next compare the expressions involving \(H_{\h^i,j}\). Their
localizations lie in the affine boundary collar. There the boundary pullback
of \(U_{\h^i,p}\) is the same fixed section \(U\), whereas
\(H_{\h^i,j}\) is represented by \(H(t_{\theta_i})\). Since
\[
    \partial_\theta t_\theta(X,Z)
    =
    \theta\cos\theta Z+\theta\sin\theta X,
\]
the mean value theorem and heteroclinic decay bound the \(Z\)-contribution by
\(C\eps^\beta M\). Moreover,
\(\theta\sin\theta=O(\theta^2)\), and hence the \(X\)-contribution is bounded
by
\[
    C\eps^{2\beta+\alpha-1}M
    \leq C\eps^\beta M
\]
because \(\beta+\alpha-1>0\). The variation of the pulled-back cutoffs is
\(O(M)\) and multiplies the fixed-shift bound
\eqref{eq: H-profile-comparison-size}. This proves
\eqref{eq: H-profile-comparison-dependence}; the smooth composition estimate
for \(W''\) proves \eqref{eq: H-potential-comparison-dependence}.
\end{proof}

\noindent\emph{Localized potential products.}
Let \(a\) denote any of the scalar potential coefficients in the preceding
two lemmas, including their shift differences. It is extended by zero on the
fixed boundary model, and its positive-side support satisfies
\[
    \spt a\cap\{X>0\}
    \subset
    \{0<X\leq C\eps^{\alpha-1}\}.
\]
Consequently, every
\(\Phi\in C^{0,\gamma}_{\ell_+,\lambda}(\R^2_*,\L)\) satisfies
\begin{align}
    \lVert a\Phi\rVert_{C^{0,\gamma}_{\ell_+,\lambda}(\R^2_*,\L)}
    &\leq
    C\lVert a\rVert_{C^{0,\gamma}_{\ell_+,\lambda}(\R^2_*)}
    \lVert\Phi\rVert_{C^{0,\gamma}_{\ell_+,\lambda}(\R^2_*,\L)},
    \label{eq: localized-potential-product}\\
    \lVert(a\Phi)^\top\rVert_{L^\infty(\R^+)}
    &\leq
    C\lVert a\rVert_{L^\infty(\R^2_*)}
    \lVert\Phi\rVert_{C^{0,\gamma}_{\ell_+,\lambda}(\R^2_*,\L)},
    \label{eq: localized-potential-projection}\\
    A_{g_\tau}(a\Phi)
    &\leq
    C\eps^{-(1-\alpha)\tau}
    \lVert a\rVert_{L^\infty(\R^2_*)}
    \lVert\Phi\rVert_{C^{0,\gamma}_{\ell_+,\lambda}(\R^2_*,\L)}.
    \label{eq: localized-potential-Ag}
\end{align}
Indeed, the first estimate is the weighted product estimate, and the second
follows by integrating against \(H'\). The support condition and
\(g_\tau''(X)\simeq(1+X)^{-\tau}\) give the last estimate.

The remaining estimates control the ansatz and its potential in the outer
region, together with the background coefficient multiplying the outer
correction in the boundary equation.

\begin{lemma}[Outer-region and boundary-background estimates]
\label{lem: outer-background-comparisons}
There are constants \(C,c>0\), uniform for
\(\h,\h^1,\h^2\in\mathcal B(K\eps^\beta)\), such that, with
\[
    M=\lVert\h^1-\h^2\rVert_{C^{2,\gamma}(\boldsymbol I)},
\]
\begin{align}
    \left\lVert
        \chi_{\Ical,1}
        [W''(\omega_{\h})-\kappa_W]
    \right\rVert_
    {C^{2,\gamma}_{\Gamma,\eps,3\lambda/4}(\R^2_{\p})}
    &\leq
    C e^{-c/\eps^{1-\alpha}},
    \label{eq: outer-potential-size}\\
    \left\lVert
        \chi_{\Ical,1}
        [W''(\omega_{\h^1})-W''(\omega_{\h^2})]
    \right\rVert_
    {C^{0,\gamma}_{\Gamma,\eps,3\lambda/4}(\R^2_{\p})}
    &\leq
    C e^{-c/\eps^{1-\alpha}}M,
    \label{eq: outer-potential-dependence}\\
    \left\lVert
        \chi_{\Ical,1}
        (\omega_{\h^1}-\omega_{\h^2})
    \right\rVert_
    {C^{2,\gamma}_{\Gamma,\eps,3\lambda/4}
    (\R^2_{\p},\L_{\p})}
    &\leq
    C e^{-c/\eps^{1-\alpha}}M.
    \label{eq: outer-ansatz-dependence}
\end{align}
At the boundary one has the uniform, rather than small, estimate
\begin{equation}\label{eq: boundary-background-potential}
    \left\lVert
        \widehat{
            \chi_{j,1}
            [W''(\omega_{\h})-\kappa_W]
        }_{p,\h}
    \right\rVert_
    {C^{0,\gamma}_{\ell_+,\lambda/4}(\R^2_*)}
    \leq C.
\end{equation}
For two shifts, the corresponding coefficients satisfy
\begin{equation}\label{eq: boundary-background-potential-dependence}
\left\lVert
\begin{aligned}
&\widehat{
    \chi_{j,1}
    [W''(\omega_{\h^1})-\kappa_W]
}_{p,\h^1}\\
&\quad-
\widehat{
    \chi_{j,1}
    [W''(\omega_{\h^2})-\kappa_W]
}_{p,\h^2}
\end{aligned}
\right\rVert_{C^{0,\gamma}_{\ell_+,\lambda/4}(\R^2_*)}
\leq C\eps^{-1}M.
\end{equation}
\end{lemma}

\begin{proof}
Choose \(\lambda_*\) such that
\[
    \frac{3\lambda}{4}<\lambda_*<
    \min\{\sigma,\sqrt{\kappa_W}\}.
\]
We first prove \eqref{eq: outer-potential-size}. If an
\(\eps\)-ball centered at \(q\) meets \(\spt\chi_{\Ical,1}\), then either
the ansatz is already equal to a pure phase on that ball or
\[
    d(q)\coloneqq\dist(q,\Gamma)\geq c\eps^\alpha.
\]
The exponential convergence of the heteroclinic and boundary models, the
scaled cutoff bounds, and the smoothness of \(W''\) give, for some fixed
\(N>0\),
\[
    \left\lVert
        \chi_{\Ical,1}[W''(\omega_{\h})-\kappa_W]
    \right\rVert_{C^{2,\gamma}_\eps(B_\eps(q))}
    \leq C\eps^{-N}e^{-\lambda_*d(q)/\eps}.
\]
After multiplication by the weight
\(e^{3\lambda d(q)/(4\eps)}\), the right-hand side is bounded by
\(Ce^{-c/\eps^{1-\alpha}}\). This proves
\eqref{eq: outer-potential-size}.

For \eqref{eq: outer-potential-dependence}, apply the mean value theorem to
the translations and rotations of every profile entering the ansatz. On the
same ball,
\[
    \left\lVert
        \chi_{\Ical,1}
        [W''(\omega_{\h^1})-W''(\omega_{\h^2})]
    \right\rVert_{C^{0,\gamma}_\eps(B_\eps(q))}
    \leq
    C\eps^{-N}e^{-\lambda_*d(q)/\eps}M.
\]
The preceding weighted argument proves
\eqref{eq: outer-potential-dependence}. The same mean value formulas, before
composition with \(W''\), give
\[
    \left\lVert
        \chi_{\Ical,1}
        (\omega_{\h^1}-\omega_{\h^2})
    \right\rVert_{C^{2,\gamma}_\eps(B_\eps(q),\L_{\p})}
    \leq
    C\eps^{-N}e^{-\lambda_*d(q)/\eps}M,
\]
and hence prove \eqref{eq: outer-ansatz-dependence}.

We next prove \eqref{eq: boundary-background-potential}. In the adapted
boundary coordinates, the boundary profile is fixed. The transition sets of
the other profiles remain within \(C\eps^{\beta-1}\) of \(\ell_+\), their
scaled \(C^{1,\gamma}\)-norms are uniform, and their potentials decay
exponentially away from these transition sets. Since \(\beta\geq2\), the
decay measured from \(\ell_+\), together with the uniform pulled-back cutoff
bounds, gives
\[
    \left\lVert
        \widehat{
            \chi_{j,1}[W''(\omega_{\h})-\kappa_W]
        }_{p,\h}
    \right\rVert_{C^{0,\gamma}_{\ell_+,\lambda/4}(\R^2_*)}
    \leq C.
\]

Finally, compare the two adapted boundary pullbacks. For \(i=1,2\), set
\[
    q_i(X,Z)
    =Y_p\bigl(\eps R_{-\theta_{\h^i}(p)}(X,Z)\bigr).
\]
The scalar representatives have the exact difference
\begin{align*}
&\chi_{j,1}(q_1)
    [W''(\omega_{\h^1}(q_1))-\kappa_W]
-\chi_{j,1}(q_2)
    [W''(\omega_{\h^2}(q_2))-\kappa_W]\\
&\quad=
[\chi_{j,1}(q_1)-\chi_{j,1}(q_2)]
    [W''(\omega_{\h^1}(q_1))-\kappa_W]\\
&\qquad\quad+
\chi_{j,1}(q_2)
    [W''(\omega_{\h^1}(q_1))-W''(\omega_{\h^2}(q_2))].
\end{align*}
The pulled-back cutoff difference is \(O(\eps^{-\alpha}M)\); after
multiplication by the decaying potential coefficient, the first term has the
same bound in
\(C^{0,\gamma}_{\ell_+,\lambda/4}(\R^2_*)\). The mean value theorem applied
to the translated and rotated profiles bounds the second term by
\(C\eps^{-1}M\) in that norm. Since \(\alpha<1\), the weighted product
estimate gives
\[
\left\lVert
\begin{aligned}
&\widehat{
    \chi_{j,1}[W''(\omega_{\h^1})-\kappa_W]
}_{p,\h^1}\\
&\quad-
\widehat{
    \chi_{j,1}[W''(\omega_{\h^2})-\kappa_W]
}_{p,\h^2}
\end{aligned}
\right\rVert_{C^{0,\gamma}_{\ell_+,\lambda/4}(\R^2_*)}
\leq C\eps^{-1}M,
\]
which is \eqref{eq: boundary-background-potential-dependence}.
\end{proof}

\begin{proof}[Proof of \Cref{prop: expansion error}]
The section constructed in \eqref{eq: section-valued-global-ansatz} has the
local description stated in the proposition.  Its residual and its support
are given by \Cref{lem: residual-identity}.  The estimates
\eqref{eq: estimate S away from Ij} and
\eqref{eq: Bj-error-expansion} follow from
\eqref{eq: outer-approximation-error-global} and
\eqref{eq: approximation-error-aggregate}, respectively.  Finally,
the expression in brackets in \eqref{eq: improved Linfty error} differs from
\(\mathcal E^{\mathrm{rem}}_{j,\h}\) by
\[
    \eps(\eta_{j,3}-\chi_{j,5})h_j''H'(t_{j,\h}).
\]
This difference lies in the exponentially small normal cutoff tails. Hence
\eqref{eq: approximation-error-Linfty}, together with this tail estimate and
\(2\beta\geq1+\beta+\alpha\gamma\), proves the last estimate.
\end{proof}

\section{Technical estimates for the gluing construction}\label{sec: proofs}

\subsection{Proof of \Cref{lem: inner-bdy decomp}}\label{sec: inner-bdy decomp}

\begin{proof}
Fix \(j\). We write
\[
    L_\omega\coloneqq \eps^2\Delta-W''(\omega),
    \qquad
    U_p\coloneqq U_{\theta_{\h}(p)}(\cdot/\eps).
\]
In the coordinates \(q=X_j(x,z)\), \(t=(z-h_j(x))/\eps\), any function
\(v=v(x,t)\) satisfies
\begin{equation}\label{eq: physical-stretched-operator}
    \eps^2\Delta v-W''(H(t))v
    =
    \widetilde L_{H,h_j}v
    +\bigl(m_j^{\mathrm{bd}}[v]+m_j^{\mathrm{int}}[v]\bigr)H'.
\end{equation}

Assume that \(\phi_p\), \(p\in\partial I_j\), and
\(\phi_j^{\mathrm{int}}\) solve \eqref{eq: boundary eq2}--\eqref{eq: interior eq2},
and reconstruct \(\phi_j\) by \eqref{eq: decomp phi}. Expanding
\(L_\omega\phi_j\) gives, on \(\spt\chi_{j,2}\),
\[
\begin{split}
    L_\omega\phi_j
    &=
    \eta_{j,1}\bigl(\eps^2\Delta\phi_j^{\mathrm{int}}
    -W''(\omega)\phi_j^{\mathrm{int}}\bigr)\\
    &\quad+
    \sum_{p\in\partial I_j}
    \zeta_{p,4}
    \bigl(\eps^2\Delta\phi_p-W''(\omega)\phi_p\bigr)\\
    &\quad+
    2\eps^2\nabla\eta_{j,1}\cdot\nabla\phi_j^{\mathrm{int}}
    +\eps^2\Delta\eta_{j,1}\phi_j^{\mathrm{int}}\\
    &\quad+
    \sum_{p\in\partial I_j}
    \left(
    2\eps^2\nabla\zeta_{p,4}\cdot\nabla\phi_p
    +\eps^2\Delta\zeta_{p,4}\phi_p
    \right).
\end{split}
\]
The two equations in the lemma imply
\[
\begin{split}
    \eps^2\Delta\phi_p-W''(\omega)\phi_p
    &=
    F_p
    -
    \bigl(W''(\omega)-W''(U_p)\bigr)\phi_p,\\
    \eps^2\Delta\phi_j^{\mathrm{int}}-W''(\omega)\phi_j^{\mathrm{int}}
    &=
    F_j
    +\bigl(m_j^{\mathrm{bd}}[\phi_j^{\mathrm{int}}]
    +m_j^{\mathrm{int}}[\phi_j^{\mathrm{int}}]\bigr)H'
    -
    \bigl(W''(\omega)-W''(H(t))\bigr)\phi_j^{\mathrm{int}}.
\end{split}
\]
We now substitute the definitions of \(F_p\) and \(F_j\). On \(\spt\chi_{j,2}\),
the cut-offs satisfy
\[
\begin{aligned}
    \chi_{j,4}&=1\quad\text{on }\spt\chi_{j,2},\\
    \zeta_{p,5}&=1\quad\text{on }\spt\chi_{j,2}\cap\spt\zeta_{p,4},\\
    \spt\chi_{j,2}\cap
    \bigl(\spt\nabla\eta_{j,1}\cup\spt\Delta\eta_{j,1}\bigr)
    &\subset \bigcup_{p\in\partial I_j}\{\zeta_{p,4}=1\},\\
    \spt\chi_{j,2}\cap
    \bigl(\spt\nabla\zeta_{p,4}\cup\spt\Delta\zeta_{p,4}\bigr)
    &\subset\{\eta_{j,1}=1\},\\
    \spt\chi_{j,2}\cap
    \bigl(\spt\eta_{j,4}\cup\spt\eta_{j,3}\bigr)
    &\subset\{\eta_{j,1}=1\}.
\end{aligned}
\]
Hence the commutator terms involving \(\eta_{j,1}\) are cancelled by the
corresponding terms in the boundary equations, and the commutator terms involving
\(\zeta_{p,4}\) are cancelled by the corresponding terms in the interior
equation.

The remainder terms recombine because
\[
    \eta_{j,4}+\sum_{p\in\partial I_j}\zeta_{p,4}=1
    \quad\text{on }\spt\chi_{j,2},
\]
and \(\eta_{j,1}=1\) on \(\spt\chi_{j,2}\cap\spt\eta_{j,4}\). Thus their total contribution is
\(\mathcal R_j(\phi_j,\psi)\). The leading term contributes
\[
    \eta_{j,1}\eps\eta_{j,3}h_j''(x)H'(t)
    =
    \eps\eta_{j,3}h_j''(x)H'(t).
\]
The boundary potential mismatch
\(\bigl(W''(\omega)-W''(U_p)\bigr)\phi_p\) cancels directly. The remaining
interior mismatch is multiplied by \(\eta_{j,4}\), where the approximate
solution is the pure shifted heteroclinic; hence
\[
    W''(\omega)=W''(H(t))
    \qquad\text{on }\spt\chi_{j,2}\cap\spt\eta_{j,4}.
\]
Finally, the scalar terms obtained from \(\mathfrak B_{j,h_j}\) cancel. Indeed,
writing \(m^{\mathrm{bd}}=m_j^{\mathrm{bd}}[\phi_j^{\mathrm{int}}]\)
and \(m^{\mathrm{int}}=m_j^{\mathrm{int}}[\phi_j^{\mathrm{int}}]\), their
sum is
\[
\begin{split}
    &\eta_{j,1}(m^{\mathrm{bd}}+m^{\mathrm{int}})H'
    -\eta_{j,1}\sum_{p\in\partial I_j}
    \zeta_{p,4}m^{\mathrm{bd}}H'
    -\eta_{j,1}\eta_{j,4}m^{\mathrm{bd}}H'
    -\eta_{j,1}m^{\mathrm{int}}H'\\
    &\qquad=
    \eta_{j,1}m^{\mathrm{bd}}H'
    \left(1-\eta_{j,4}
    -\sum_{p\in\partial I_j}\zeta_{p,4}\right)=0.
\end{split}
\]
Combining these cancellations gives
\[
    L_\omega\phi_j
    =
    \eps\eta_{j,3}h_j''(x)H'(t)
    +
    \mathcal R_j(\phi_j,\psi)
    =
    E_j(\phi_j,\psi),
\]
which is precisely \eqref{eq: system reduced segments}.
\end{proof}

\subsection{Proofs of the boundary-response propositions}
\label{sec: invert boundary}

\begin{proof}[Proof of \Cref{prop: invert boundary}]
Fix \(j\in\{1,\dots,N\}\), a boundary point \(p\in\partial I_j\), and
\(h_j\), \(v\), and \(\psi\) as in the statement.
In this proof, we write
\(F_p(\phi_p;v,h_j,\psi)\) for the right-hand side in
\eqref{eq: boundary eq2}, with the fixed data \((v,h_j,\psi)\) displayed
explicitly. Set
\[
    \theta\coloneqq\theta_{\h}(p),
    \qquad
    U_{\theta,\eps}(y)\coloneqq U_\theta(y/\eps),
    \qquad y=(x,z).
\]
We use the hatted boundary representatives from
\Cref{def: adapted-coordinate-representatives}, suppressing the subscripts
\(p,\h\). Thus the rescaling and rotation \(y=\eps R_{-\theta}Y\) turn
\eqref{eq: boundary eq2} into
\begin{equation}\label{eq: boundary eq3}
    \Delta\widehat\phi_p-W''(U)\widehat\phi_p
    =\widehat F_p(\phi_p;v,h_j,\psi)
    \qquad\text{on }\R^2_* .
\end{equation}
The change of variables identifies
\(C^{k,\gamma}_{\ell_\theta,\eps,\lambda}\) with
\(C^{k,\gamma}_{\ell_+,\lambda}\).
Throughout this proof, a full-model norm whose domain is suppressed is taken
over \((\R^2_*,\L)\), and a stretched interior norm whose domain is
suppressed is taken over \(I_j\times\R\).

Although the operator on the left-hand side of
\eqref{eq: boundary eq3} is the one considered in
\Cref{lem: boundary invertibility}, the right-hand side depends on the unknown
\(\phi_p\). We must therefore examine this dependence carefully. Denote the
physical representative of \(v\) by
\[
    v_h(x,z)\coloneqq
    v\left(x,\frac{z-h_j(x)}{\eps}\right).
\]
Before the rescaling, the right-hand side
\(F_p(\phi_p;v,h_j,\psi)\) decomposes as
\[
\begin{split}
    A_1&=-2\eps^2\chi_{j,4}\zeta_{p,5}
    \nabla\eta_{j,1}\cdot\nabla v_h,\\
    A_2&=-\eps^2\chi_{j,4}\zeta_{p,5}\Delta\eta_{j,1}v_h,\\
    A_3&=\chi_{j,4}\zeta_{p,5}\mathcal R_j(\phi_j,\psi),\\
    B_1&=\chi_{j,4}\zeta_{p,5}
    \bigl(W''(\omega)-W''(U_{\theta,\eps})\bigr)\phi_p,\\
    B_2&=\chi_{j,4}\zeta_{p,5}\eta_{j,1}
    \bigl(W''(\omega)-W''(H(t))\bigr)v_h,\\
    C_{\mathrm{bd}}
    &=-\chi_{j,4}\zeta_{p,5}\eta_{j,1}
    m_j^{\mathrm{bd}}[v]H'(t).
\end{split}
\]
For fixed \(h_j\), \(v\), and
\(\psi\), the terms \(A_1,A_2,B_2\), and \(C_{\mathrm{bd}}\) are independent of
\(\phi_p\). The terms depending on \(\phi_p\) are \(B_1\) and \(A_3\).
Here \(B_1\) is linear in \(\phi_p\) and \(A_3\) depends
nonlinearly on it through \(\mathcal R_j\). The precise dependence of \(A_3\)
will be examined in the estimates below, but since we localise
\(\mathcal R_j\) by the boundary cut-off \(\zeta_{p,5}\), no boundary
correction other than \(\phi_p\) enters \(A_3\) directly. Consequently, we
must solve \eqref{eq: boundary eq3} by a contraction mapping argument.

For an admissible
\(g\), let \(\mathcal G_U(f)\) denote the unique bounded solution furnished by
\Cref{lem: boundary invertibility} whenever
\(f\in C_{\ell_+,\lambda}^{0,\gamma}(\R^2_*,\L)\) and
\(A_g(f)<+\infty\). By \Cref{cor: schauder-tan},
\[
    \lVert\mathcal G_U(f)\rVert_{C_{\ell_+,\lambda}^{2,\gamma}}
    \leq
    C_{g,\lambda}
    \left(
        \lVert f\rVert_{C_{\ell_+,\lambda}^{0,\gamma}}+A_g(f)
    \right).
\]
Consider the ball
\[
    \mathcal B_p\coloneqq
    \left\{
        \phi_p\in C_{\ell_\theta,\eps,\lambda}^{2,\gamma}(\R^2_*,\L):
        \lVert\phi_p\rVert_{C_{\ell_\theta,\eps,\lambda}^{2,\gamma}}
        \leq R_\eps
    \right\}.
\]
For \(\phi_p\in\mathcal B_p\), define
\[
    T_{h_j,v,\psi}(\phi_p)(y)
    \coloneqq
    \left[
        \mathcal G_U\bigl(\widehat F_p(\phi_p;v,h_j,\psi)\bigr)
    \right](\eps^{-1}R_\theta y).
\]
Write \(T=T_{h_j,v,\psi}\). By construction, the fixed points of \(T\) in
\(\mathcal B_p\) are precisely the solutions of
\eqref{eq: boundary eq2} in this ball. We will prove, uniformly over the fixed
data above, that, for \(\eps\) sufficiently small,
\begin{equation}\label{eq: boundary self-map}
    T\bigl(\mathcal B_p\bigr)
    \subseteq \mathcal B_p,
\end{equation}
and that there exists \(q_\eps\to0\) such that
\begin{equation}\label{eq: boundary contraction}
    \lVert T(\phi_p^1)-T(\phi_p^2)\rVert_{C_{\ell_\theta,\eps,\lambda}^{2,\gamma}}
    \leq
    q_\eps
    \lVert\phi_p^1-\phi_p^2\rVert_{C_{\ell_\theta,\eps,\lambda}^{2,\gamma}}
\end{equation}
for every \(\phi_p^1,\phi_p^2\in\mathcal B_p\).

We first prove \eqref{eq: boundary self-map}. For
\(\phi_p\in\mathcal B_p\), write
\(\widehat F_{\phi_p}\coloneqq
\widehat F_p(\phi_p;v,h_j,\psi)\).
By \Cref{cor: schauder-tan}, applied with the admissible decay function
\(g_\tau\) from \eqref{eq: boundary-decay-function}, it suffices to prove that
\[
    \lVert\widehat F_{\phi_p}\rVert_{C_{\ell_+,\lambda}^{0,\gamma}}
    +A_{g_\tau}(\widehat F_{\phi_p})
    \leq c_\eps R_\eps
\]
uniformly for \(\phi_p\in\mathcal B_p\), where \(c_\eps\to0\).

In what follows, for each term
\(Q\in\{A_1,A_2,A_3,B_1,B_2,C_{\mathrm{bd}}\}\), we write
\(\widehat Q(Y)\coloneqq Q(\eps R_{-\theta}Y)\) for the scaled and rotated counterpart. Similarly, given a domain \(\Omega\subset\R^2\), we write
\(\widehat\Omega\coloneqq\eps^{-1}R_\theta\Omega\).

We begin estimating the
H\"older norm of \(A_1\) and \(A_2\). Set
\(
    \Omega_{p,\eta}
    \coloneqq
    \spt(\chi_{j,4}\zeta_{p,5}\nabla\eta_{j,1})
    \cup
    \spt(\chi_{j,4}\zeta_{p,5}\Delta\eta_{j,1}).
\)
From the definition of the cut-offs, the set \(\Omega_{p,\eta}\) is contained in the
boundary collar
\(
    2\delta\eps^\alpha\leq x_p\leq3\delta\eps^\alpha.
\)
Both the boundary and interior coordinates are valid on this collar, and their
weights are uniformly comparable by the coordinate comparison used in
\eqref{eq: estimate variable difference}. Since
\(\partial_xv_h=\partial_xv-\eps^{-1}h_j'\partial_tv\) and
\(\partial_zv_h=\eps^{-1}\partial_tv\), \Cref{def: stretched interior norm} and
\(\lVert h_j'\rVert_\infty\lesssim\eps^\beta\) give
\[
\begin{split}
    &\lVert v_h\rVert_{C_{\ell_\theta,\eps,\lambda}^{0,\gamma}(\Omega_{p,\eta})}
    +\eps
    \lVert\nabla v_h\rVert_{C_{\ell_\theta,\eps,\lambda}^{0,\gamma}(\Omega_{p,\eta})}
    \lesssim
    \lVert v\rVert_{\mathcal C_{\eps,\lambda}^{2,\gamma}(I_j\times\R)}.
\end{split}
\]
Moreover, \(R_{-\theta}\) is an isometry sending \(\ell_+\) to
\(\ell_\theta\), so the rotation introduces no additional factor. The scaled
cut-off estimates are $\lVert\nabla\eta_{j,1}\rVert_{\mathsf C_\eps^{0,\gamma}(\Omega_{p,\eta})}
    \lesssim\eps^{-\alpha}$ and $\lVert\Delta\eta_{j,1}\rVert_{\mathsf C_\eps^{0,\gamma}(\Omega_{p,\eta})}
    \lesssim\eps^{-2\alpha}.$

The remaining cut-offs occur undifferentiated and have uniformly bounded
scaled H\"older norm. Therefore
\[
\begin{aligned}
    \lVert\widehat A_1\rVert
    _{C_{\ell_+,\lambda}^{0,\gamma}(\widehat\Omega_{p,\eta})}
    &\lesssim
    \eps^2\eps^{-\alpha}\eps^{-1}
    \lVert v\rVert_{\mathcal C_{\eps,\lambda}^{2,\gamma}(I_j\times\R)}
    =\eps^{1-\alpha}
    \lVert v\rVert_{\mathcal C_{\eps,\lambda}^{2,\gamma}(I_j\times\R)},\\
    \lVert\widehat A_2\rVert
    _{C_{\ell_+,\lambda}^{0,\gamma}(\widehat\Omega_{p,\eta})}
    &\lesssim
    \eps^2\eps^{-2\alpha}
    \lVert v\rVert_{\mathcal C_{\eps,\lambda}^{2,\gamma}(I_j\times\R)}
    =\eps^{2-2\alpha}
    \lVert v\rVert_{\mathcal C_{\eps,\lambda}^{2,\gamma}(I_j\times\R)}.
\end{aligned}
\]

We next estimate \(A_3\). Set
\(\Omega_{p,3}\coloneqq\spt(\chi_{j,4}\zeta_{p,5})\). On
\(\Omega_{p,3}\cap\spt\chi_{j,1}\), the nesting of the cut-offs gives
\(\chi_{j,2}=1\), so \eqref{eq: decomp phi} applies. Moreover, the supports of
\(\zeta_{p,5}\) and \(\zeta_{p',4}\), for
\(p'\in\partial I_j\setminus\{p\}\), are disjoint, since an intersection
would imply \(|I_j|\leq13\delta\eps^\alpha<24\delta<|I_j|\). Hence
\[
    \phi_j=\eta_{j,1}v_h+\zeta_{p,4}\phi_p
    \quad\text{on }\Omega_{p,3}\cap\spt\chi_{j,1},
\]
and the scaled cut-off bounds and the coordinate comparison above give
\[
    \lVert\phi_j\rVert_{C_{\ell_\theta,\eps,\lambda}^{2,\gamma}
    (\Omega_{p,3}\cap\spt\chi_{j,1})}
    \lesssim
    \lVert v\rVert_{\mathcal C_{\eps,\lambda}^{2,\gamma}(I_j\times\R)}
    +\lVert\phi_p\rVert_{C_{\ell_\theta,\eps,\lambda}^{2,\gamma}}.
\]

On \(\Omega_{p,3}\), the nesting gives \(\chi_{j,5}=1\), while the
support condition on the shift gives \(\eta_{j,3}h_j''=h_j''\).  Hence the
approximation term in \(A_3\) is exactly the boundary representative of
\(\mathcal E^{\mathrm{rem}}_{j,\h}\) from
\eqref{eq: approximation-remainder}.  By
\eqref{eq: boundary-approximation-error-size},
\[
    \left\lVert\chi_{j,4}\zeta_{p,5}
    \bigl(S(\omega)+\eps\eta_{j,3}h_j''H'\bigr)\right\rVert_
    {C_{\ell_\theta,\eps,\lambda}^{0,\gamma}(\Omega_{p,3})}
    \lesssim\eps^{2\beta}.
\]
On \(\Omega_{p,3}\), the coordinate comparison above makes the
\(\Gamma\)-weight and the boundary weight with rate \(3\lambda/4\) uniformly
comparable. Since \(\phi_j\) has weight \(\lambda\) and
\(2(3\lambda/4)>\lambda\), the quadratic estimate for \(\mathcal N_\omega\) gives
\[
\begin{split}
    &\lVert\chi_{j,4}\zeta_{p,5}\chi_{j,1}\mathcal N_\omega(\psi+\phi_j)\rVert_
    {C_{\ell_\theta,\eps,\lambda}^{0,\gamma}(\Omega_{p,3})}
    \lesssim
    \lVert v\rVert_{\mathcal C_{\eps,\lambda}^{2,\gamma}(I_j\times\R)}^2
    +\lVert\phi_p\rVert_{C_{\ell_\theta,\eps,\lambda}^{2,\gamma}}^2
    +\exp\left(-\frac{\widetilde\lambda}{\eps^{1-\alpha}}\right).
\end{split}
\]
Finally, \eqref{eq: boundary-background-potential} shows that the coefficient
of \(\psi\) has boundary weight \(\lambda/4\).  Since \(\psi\) has weight
\(3\lambda/4\), their product is exponentially small in the norm with weight
\(\lambda\). Combining the three terms in
\eqref{eq: local remainder Rj} and rotating gives
\[
\begin{split}
    \lVert\widehat A_3\rVert_{C_{\ell_+,\lambda}^{0,\gamma}
    (\widehat\Omega_{p,3})}
    \lesssim{}&
    \eps^{2\beta}
    +\lVert v\rVert_{\mathcal C_{\eps,\lambda}^{2,\gamma}(I_j\times\R)}^2
    +\lVert\phi_p\rVert_{C_{\ell_\theta,\eps,\lambda}^{2,\gamma}}^2
    +\exp\left(-\frac{\widetilde\lambda}{\eps^{1-\alpha}}\right)
    =o(R_\eps).
\end{split}
\]
The last relation follows from \eqref{eq: assumption phi},
\(\phi_p\in\mathcal B_p\), and
\(2\beta>\beta+1-\alpha\).

We now estimate \(B_1\) and \(B_2\).  The two coefficients are precisely the
localized potential comparisons studied in
\Cref{lem: profile-potential-comparisons}.  Equations
\eqref{eq: U-potential-comparison-Holder} and
\eqref{eq: H-potential-comparison-size}, followed by the weighted product
estimate, give
\[
\begin{aligned}
    \lVert\widehat B_1\rVert_{C_{\ell_+,\lambda}^{0,\gamma}
    (\widehat\Omega_{p,3})}
    &\lesssim
    \eps^{\beta-1+2\alpha}
    \lVert\phi_p\rVert_{C_{\ell_\theta,\eps,\lambda}^{2,\gamma}},\\
    \lVert\widehat B_2\rVert_{C_{\ell_+,\lambda}^{0,\gamma}
    (\widehat\Omega_{p,3})}
    &\lesssim
    \eps^{2\beta}
    \lVert v\rVert_{\mathcal C_{\eps,\lambda}^{2,\gamma}(I_j\times\R)}.
\end{aligned}
\]
Since \(\beta-1+2\alpha>0\), both terms are
\(o(R_\eps)\).

It remains to estimate \(C_{\mathrm{bd}}\). Set
\(\Omega_{p,C}\coloneqq
\spt(\chi_{j,4}\zeta_{p,5}\eta_{j,1})\). The exponential decay of
\(H''\) and \(H'''\), together with
\Cref{def: stretched interior norm}, the coordinate comparison above, and
\(\lVert h_j\rVert_{C^{2,\gamma}(I_j)}\lesssim\eps^\beta\) give directly
\[
    \lVert\widehat C_{\mathrm{bd}}\rVert_{C_{\ell_+,\lambda}^{0,\gamma}
    (\widehat\Omega_{p,C})}
    \lesssim
    \eps^\beta
    \lVert v\rVert_{\mathcal C_{\eps,\lambda}^{2,\gamma}(I_j\times\R)}
    =o(R_\eps).
\]
Combining the estimates for
\(A_1,A_2,A_3,B_1,B_2\), and \(C_{\mathrm{bd}}\), we obtain
\[
    \lVert\widehat F_{\phi_p}\rVert_{C_{\ell_+,\lambda}^{0,\gamma}}
    \leq c_\eps R_\eps,
    \qquad c_\eps\to0.
\]

It remains to estimate \(A_{g_\tau}(\widehat F_{\phi_p})\). Each scaled term
is supported, for \(X>0\), in \(\{X\leq C\eps^{\alpha-1}\}\). Since
\(g_\tau''(X)=(\tau-1)(\tau-2)(1+X)^{-\tau}\), the definition of
\(A_{g_\tau}\) gives
\begin{equation}\label{eq: boundary Ag support}
    A_{g_\tau}(Q)
    \lesssim
    \eps^{-(1-\alpha)\tau}
    \lVert Q^\top\rVert_{L^\infty(\{X>0\})}
\end{equation}
for every such term \(Q\).

We first consider \(\widehat A_1+\widehat A_2\). On the transition region of
\(\eta_{j,1}\), the differentiated cut-offs are horizontal, up to
exponentially small normal tails. Since \(h_j(x)=\theta x\) in this collar,
the stretched normal variable in the unrotated boundary coordinates is
\[
    t=(\cos\theta+\theta\sin\theta)Z
      +(\sin\theta-\theta\cos\theta)X.
\]
At \(\theta=0\), the leading terms in the tangential projections of
\(A_1\) and \(A_2\) are therefore multiples of
\[
    \int_{\R}\partial_xv(x,Z)H'(Z)\,dZ
    \quad\text{and}\quad
    \int_{\R}v(x,Z)H'(Z)\,dZ,
\]
and both vanish by \eqref{eq: orth v} and its \(x\)-derivative.  Moreover,
\[
    \partial_xv_h
    =\partial_xv-\frac{\theta}{\eps}\partial_tv.
\]
The second term, together with the changes in the two coordinate variables,
contributes at most
\(\eps^{2}\eps^{-\alpha}|\theta|\eps^{-1}\lVert v\rVert\);
the corresponding discrepancy in \(A_2\) is smaller by a factor
\(\eps^{1-\alpha}\).  Since
\(|\theta|\lesssim\eps^\beta\), this gives
\[
    \lVert(\widehat A_1+\widehat A_2)^\top\rVert_{L^\infty(\{X>0\})}
    \lesssim
    \eps^{\beta+1-\alpha}
    \lVert v\rVert_{\mathcal C_{\eps,\lambda}^{2,\gamma}(I_j\times\R)}.
\]

For \(A_3\), equations \eqref{eq: boundary-approximation-error-size} and
\eqref{eq: boundary-approximation-error-Ag} control the approximation term
directly.  The quadratic estimate for \(\mathcal N_\omega\), the fixed outer bound, and
\(\phi_p\in\mathcal B_p\) control the remaining terms.  Thus
\[
    \lVert\widehat A_3^\top\rVert_{L^\infty(\{X>0\})}
    \lesssim
    \eps^{2\beta}
    +R_\eps^2
    +e^{-\widetilde\lambda/\eps^{1-\alpha}}.
\]
Equations \eqref{eq: U-potential-comparison-Linfty},
\eqref{eq: H-potential-comparison-size}, and
\eqref{eq: localized-potential-projection}, together with the definition of
\(m_j^{\mathrm{bd}}\), give
\[
\begin{aligned}
    \lVert\widehat B_1^\top\rVert_{L^\infty(\{X>0\})}
    &\lesssim
    \eps^{\beta-1+\alpha(2+\gamma)}
    \lVert\phi_p\rVert_{C_{\ell_\theta,\eps,\lambda}^{2,\gamma}},\\
    \lVert\widehat B_2^\top\rVert_{L^\infty(\{X>0\})}
    &\lesssim
    \eps^{2\beta}
    \lVert v\rVert_{\mathcal C_{\eps,\lambda}^{2,\gamma}},\\
    \lVert\widehat C_{\mathrm{bd}}^\top\rVert_{L^\infty(\{X>0\})}
    &\lesssim\eps^\beta
    \lVert v\rVert_{\mathcal C_{\eps,\lambda}^{2,\gamma}}.
\end{aligned}
\]

Applying \eqref{eq: boundary Ag support}, and using the sharper Section 8
estimate for the approximation term in \(A_3\), we obtain
\[
\begin{aligned}
    A_{g_\tau}(\widehat A_1+\widehat A_2)
    &\lesssim
    \eps^{\beta+1-\alpha-(1-\alpha)\tau}
    \lVert v\rVert_{\mathcal C_{\eps,\lambda}^{2,\gamma}},\\
    A_{g_\tau}(\widehat A_3)
    &\lesssim
    \eps^{2\beta-(1-\alpha)\tau}
    +\eps^{-(1-\alpha)\tau}R_\eps^2
    +\eps^{-(1-\alpha)\tau}
    e^{-\widetilde\lambda/\eps^{1-\alpha}},\\
    A_{g_\tau}(\widehat B_1)
    &\lesssim
    \eps^{\beta-1+\alpha(2+\gamma)-(1-\alpha)\tau}
    \lVert\phi_p\rVert_{C_{\ell_\theta,\eps,\lambda}^{2,\gamma}},\\
    A_{g_\tau}(\widehat B_2)
    &\lesssim
    \eps^{2\beta-(1-\alpha)\tau}
    \lVert v\rVert_{\mathcal C_{\eps,\lambda}^{2,\gamma}},\\
    A_{g_\tau}(\widehat C_{\mathrm{bd}})
    &\lesssim
    \eps^{\beta-(1-\alpha)\tau}
    \lVert v\rVert_{\mathcal C_{\eps,\lambda}^{2,\gamma}}.
\end{aligned}
\]
By \eqref{eq: boundary-decay-exponent-range},
\((1-\alpha)\tau<1+\alpha\).  Together with \(\beta\geq2\) and the
standing condition \(\alpha(2+\gamma)>1\), this implies
\[
\begin{aligned}
    \beta-(1-\alpha)\tau&>1-\alpha,\\
    \beta-1+\alpha(2+\gamma)-(1-\alpha)\tau&>1-\alpha,\\
    2\beta-(1-\alpha)\tau-(\beta+1-\alpha)&>0.
\end{aligned}
\]
Together with \(R_\eps\lesssim\eps^\beta\), these inequalities show that
every term in the preceding display is \(o(R_\eps)\); the exponential
term is smaller than every power of \(\eps\).
Consequently,
\begin{equation}\label{eq: boundary Ag bound}
    A_{g_\tau}(\widehat F_{\phi_p})
    \leq c_\eps R_\eps,
    \qquad c_\eps\to0.
\end{equation}

Combining this with the weighted H\"older estimate and applying
\Cref{cor: schauder-tan}, we conclude that, for \(\eps\) sufficiently small,
\[
    \lVert T(\phi_p)\rVert_{C_{\ell_\theta,\eps,\lambda}^{2,\gamma}}
    \lesssim
    \lVert\widehat F_{\phi_p}\rVert_{C_{\ell_+,\lambda}^{0,\gamma}}
    +A_{g_\tau}(\widehat F_{\phi_p})
    \leq c_\eps R_\eps
    \leq R_\eps.
\]
This proves \eqref{eq: boundary self-map}.

We next prove \eqref{eq: boundary contraction}. Fix
\(\phi_p^1,\phi_p^2\in\mathcal B_p\), and set
\[
    \delta\phi_p\coloneqq\phi_p^1-\phi_p^2,
    \qquad
    \delta\widehat F\coloneqq
    \widehat F_{\phi_p^1}-\widehat F_{\phi_p^2}.
\]
Since \(h_j\), \(v\), and \(\psi\) are fixed, the same rotation and
rescaling are used for both corrections. The terms \(A_1,A_2,B_2\), and
\(C_{\mathrm{bd}}\) cancel. Moreover, the support identity used in the
\(A_3\)-estimate
gives
\[
    \phi_j^1-\phi_j^2=\zeta_{p,4}\delta\phi_p
    \quad\text{on }\Omega_{p,3}\cap\spt\chi_{j,1}.
\]
Hence, before rescaling, the two remaining differences are
\[
\begin{split}
    \delta B_1
    &=\chi_{j,4}\zeta_{p,5}
    \bigl(W''(\omega)-W''(U_{\theta,\eps})\bigr)\delta\phi_p,\\
    \delta A_3
    &=-\chi_{j,4}\zeta_{p,5}\chi_{j,1}
    \bigl(\mathcal N_\omega(\psi+\phi_j^1)-\mathcal N_\omega(\psi+\phi_j^2)\bigr).
\end{split}
\]
Write \(\delta\widehat B_1\) and \(\delta\widehat A_3\) for their scaled
counterparts. Then
\(\delta\widehat F=\delta\widehat B_1+\delta\widehat A_3\).

The full H\"older part of the boundary matching estimate and the quadratic
Lipschitz estimate for \(\mathcal N_\omega\) give
\[
\begin{split}
    \lVert\delta\widehat B_1\rVert_{C_{\ell_+,\lambda}^{0,\gamma}}
    &\lesssim
    \eps^{\beta-1+2\alpha}
    \lVert\delta\phi_p\rVert_{C_{\ell_\theta,\eps,\lambda}^{2,\gamma}},\\
    \lVert\delta\widehat A_3\rVert_{C_{\ell_+,\lambda}^{0,\gamma}}
    &\lesssim
    R_\eps
    \lVert\delta\phi_p\rVert_{C_{\ell_\theta,\eps,\lambda}^{2,\gamma}}.
\end{split}
\]
For the second estimate, the Lipschitz coefficient is bounded by the fixed
outer bound, the norm of \(v\), and the norms of \(\phi_p^1\) and
\(\phi_p^2\), and hence by \(R_\eps\). The product involving \(\psi\)
has the required boundary weight because its two factors have rates
\(3\lambda/4\) and \(\lambda\).

Using instead the \(L^\infty\)-part of the boundary matching estimate and
\(H'\in L^1(\R)\), we obtain
\[
\begin{aligned}
    \lVert(\delta\widehat B_1)^\top\rVert_{L^\infty(\{X>0\})}
    &\lesssim
    \eps^{\beta-1+\alpha(2+\gamma)}
    \lVert\delta\phi_p\rVert_{C_{\ell_\theta,\eps,\lambda}^{2,\gamma}},\\
    \lVert(\delta\widehat A_3)^\top\rVert_{L^\infty(\{X>0\})}
    &\lesssim
    R_\eps
    \lVert\delta\phi_p\rVert_{C_{\ell_\theta,\eps,\lambda}^{2,\gamma}}.
\end{aligned}
\]
Both terms have the support considered in
\eqref{eq: boundary Ag support}, so
\[
\begin{split}
    A_{g_\tau}(\delta\widehat B_1)
    &\lesssim
    \eps^{\beta-1+\alpha(2+\gamma)-(1-\alpha)\tau}
    \lVert\delta\phi_p\rVert_{C_{\ell_\theta,\eps,\lambda}^{2,\gamma}},\\
    A_{g_\tau}(\delta\widehat A_3)
    &\lesssim
    \eps^{-(1-\alpha)\tau}R_\eps
    \lVert\delta\phi_p\rVert_{C_{\ell_\theta,\eps,\lambda}^{2,\gamma}}.
\end{split}
\]

Since \(R_\eps\lesssim\eps^\beta\) and
\(\alpha(2+\gamma)>1\), the four coefficients above are bounded by a
constant times \(\eps^{\beta-(1-\alpha)\tau}\).  The choice
\eqref{eq: boundary-decay-exponent-range} gives
\(\beta-(1-\alpha)\tau>1-\alpha>0\), so this power tends to zero. It follows from
the linearity of \(\mathcal G_U\) and \Cref{cor: schauder-tan} that
\[
\begin{split}
    \lVert T(\phi_p^1)-T(\phi_p^2)\rVert_{C_{\ell_\theta,\eps,\lambda}^{2,\gamma}}
    &\lesssim
    \lVert\delta\widehat F\rVert_{C_{\ell_+,\lambda}^{0,\gamma}}
    +A_{g_\tau}(\delta\widehat F)\\
    &\lesssim
    \eps^{\beta-(1-\alpha)\tau}
    \lVert\delta\phi_p\rVert_{C_{\ell_\theta,\eps,\lambda}^{2,\gamma}}.
\end{split}
\]
This proves \eqref{eq: boundary contraction}. For \(\eps\) sufficiently
small, the contraction mapping theorem gives a unique fixed point of \(T\) in
\(\mathcal B_p\), and hence a solution of \eqref{eq: boundary eq2}
satisfying the stated uniqueness condition. We denote this fixed point by
\(b_p=\mathscr B_{p,h_j}(v,\psi)\), as in the statement; in the remaining
calculations \(\phi_p\) denotes this fixed point, and its tangential constant
\(\phi_\infty\) is \(b_{p,\infty}\). At this fixed point, the
term-by-term estimates above give, up to an exponentially small term,
\[
\begin{split}
    \lVert\widehat F_{\phi_p}\rVert_{C_{\ell_+,\lambda}^{0,\gamma}}
    +A_{g_\tau}(\widehat F_{\phi_p})
    \lesssim{}&
    \eps^{1-\alpha}
    \lVert v\rVert_{\mathcal C_{\eps,\lambda}^{2,\gamma}(I_j\times\R)}
    +\eps^{2\beta-(1-\alpha)\tau}\\
    &+\left(
    \eps^{\beta-1+\alpha(2+\gamma)-(1-\alpha)\tau}
    +
    \eps^{-(1-\alpha)\tau}R_\eps
    \right)
    \lVert\phi_p\rVert_{C_{\ell_\theta,\eps,\lambda}^{2,\gamma}}.
\end{split}
\]
The coefficient in parentheses tends to zero.  Applying
\Cref{cor: schauder-tan} and absorbing the last term therefore yields
\[
    \lVert\phi_p\rVert_{C_{\ell_\theta,\eps,\lambda}^{2,\gamma}}
    \lesssim
    \eps^{1-\alpha}
    \lVert v\rVert_{\mathcal C_{\eps,\lambda}^{2,\gamma}(I_j\times\R)}
    +\eps^{2\beta-(1-\alpha)\tau}.
\]
This proves \eqref{eq: size phi_b}.

We conclude with the tangential estimates. The term-by-term bounds used above
retain the sharper information
\begin{equation}\label{eq: sharp boundary projected rhs}
    A_{g_\tau}(\widehat F_{\phi_p})
    \lesssim
    \eps^{2\beta-(1-\alpha)\tau},
    \qquad
    \lVert\widehat F_{\phi_p}^{\top}\rVert_{L^\infty(\{X>1\})}
    \lesssim\eps^{2\beta}.
\end{equation}
Indeed, the approximation term is controlled by
\eqref{eq: boundary-approximation-error-size} and
\eqref{eq: boundary-approximation-error-Ag}. The orthogonality calculation
for \(A_1+A_2\), the localized potential estimates, and
\(\lVert v\rVert+\lVert\phi_p\rVert\lesssim R_\eps\) show that every
correction-dependent term satisfies the same bounds. Since
\(g_\tau(X)=(1+X)^{2-\tau}\),
\eqref{eq: decay tangential part} and the physical relation
\(x=\eps X+O(\eps^\beta X)\) give, on \(I_\eps\),
\[
\begin{split}
    \lVert\phi_p^\top-\phi_\infty\rVert_{L^\infty}
    &\lesssim\eps^{2\beta-2+2\alpha},\\
    \lVert\partial_x\phi_p^\top\rVert_{L^\infty}
    &\lesssim\eps^{2\beta-2+\alpha},
    \qquad
    \lVert\partial_x^2\phi_p^\top\rVert_{L^\infty}
    \lesssim\eps^{2\beta-2}.
\end{split}
\]
The interpolation inequality
\([f]_\gamma\lesssim\lVert f\rVert_\infty^{1-\gamma}
\lVert f'\rVert_\infty^\gamma\) now gives the two H\"older bounds in
\eqref{eq: decays tangential part}.

It remains to bound the constant. Choose \(\varpi>0\) smaller than the fixed
boundary cut-off constants. On
\(\{X<\varpi\eps^{\alpha-1}\}\), all linear gluing terms vanish, and hence
\[
    \lVert\widehat F_{\phi_p}\rVert_{L^\infty}
    \lesssim R_\eps^2
    +e^{-\widetilde\lambda/\eps^{1-\alpha}}.
\]
Using \eqref{eq:phi-top-infty-formula}, split its integral at
\(X=\varpi\eps^{\alpha-1}\). On the complementary region,
\(D_RU+XH'\) decays exponentially, while
\(\spt\widehat F_{\phi_p}\subset\{X\leq C\eps^{\alpha-1}\}\).
Together with \eqref{eq: sharp boundary projected rhs}, this gives
\[
\begin{split}
    |\phi_\infty|
    &\lesssim
    \eps^{2\alpha-2}
    \left(
    \eps^{2\beta}+R_\eps^2
    \right)
    +e^{-c/\eps^{1-\alpha}}\\
    &\lesssim\eps^{2\beta-2+2\alpha},
\end{split}
\]
which completes the proof of \eqref{eq: decays tangential part}.
\end{proof}

The Lipschitz estimates follow by comparing two instances of the fixed-point
construction above, using the same decomposition of the right-hand side.

\begin{proof}[Proof of \Cref{prop: invert boundary dependence}]
We first compare the dependence on \(v\) and \(\psi\) while keeping \(h_j\)
fixed. Let \(v^i\) and \(\psi^i\), \(i=1,2\), satisfy the
assumptions above, and let \(\phi_p^i\) be the corresponding fixed points.
Write \(\phi_\infty^i=b_{p,\infty}^i\). Since \(h_j\) is the same for both
data sets, so are \(\theta\) and the change of variables. Set

\[
\begin{split}
    D_\phi&\coloneqq
    \lVert\widehat\phi_p^1-\widehat\phi_p^2\rVert_
    {C_{\ell_+,\lambda}^{2,\gamma}},\\
    D_v&\coloneqq
    \lVert v^1-v^2\rVert_{
    \mathcal C_{\eps,\lambda}^{2,\gamma}(I_j\times\R)},\\
    D_\psi&\coloneqq
    \lVert\psi^1-\psi^2\rVert_
    {C_{\Gamma,\eps,3\lambda/4}^{2,\gamma}
    (\R^2_{\p},\L_{\p})}.
\end{split}
\]
Adding and subtracting the right-hand sides with intermediate data
\((\phi_p^2,v^1,h_j,\psi^1)\) and
\((\phi_p^2,v^2,h_j,\psi^1)\) separates the variations in
\(\phi_p\), \(v\), and \(\psi\).
More precisely, set
\[
\begin{aligned}
    \delta_\phi\widehat F
    &\coloneqq
    \widehat F_p(\phi_p^1;v^1,h_j,\psi^1)
    -\widehat F_p(\phi_p^2;v^1,h_j,\psi^1),\\
    \delta_v\widehat F
    &\coloneqq
    \widehat F_p(\phi_p^2;v^1,h_j,\psi^1)
    -\widehat F_p(\phi_p^2;v^2,h_j,\psi^1),\\
    \delta_\psi\widehat F
    &\coloneqq
    \widehat F_p(\phi_p^2;v^2,h_j,\psi^1)
    -\widehat F_p(\phi_p^2;v^2,h_j,\psi^2).
\end{aligned}
\]

The estimates proving \eqref{eq: boundary contraction} remain valid for the
first difference because their nonlinear coefficient uses only
\[
    \lVert v^1\rVert_{\mathcal C_{\eps,\lambda}^{2,\gamma}}
    +\sum_{i=1}^2
    \lVert\phi_p^i\rVert_{C_{\ell_\theta,\eps,\lambda}^{2,\gamma}}
    +e^{-\widetilde\lambda/\eps^{1-\alpha}}
    \lesssim\eps^\beta.
\]
Thus
\[
    \lVert\delta_\phi\widehat F\rVert_{C_{\ell_+,\lambda}^{0,\gamma}}
    +A_{g_\tau}(\delta_\phi\widehat F)
    \lesssim
    \eps^{\beta-(1-\alpha)\tau}D_\phi.
\]

We now vary \(v\), with \(\phi_p^2\), \(h_j\), and \(\psi^1\) fixed. The
terms that change are \(A_1,A_2,A_3,B_2\), and \(C_{\mathrm{bd}}\). In the
weighted H\"older norm, the coefficients for
\(A_1+A_2,A_3,B_2\), and \(C_{\mathrm{bd}}\) are,
respectively,
\(\eps^{1-\alpha}\), \(\eps^\beta\),
\(\eps^{2\beta}\), and \(\eps^\beta\). Hence
\[
    \lVert\delta_v\widehat F\rVert_{C_{\ell_+,\lambda}^{0,\gamma}}
    \lesssim\eps^{1-\alpha}D_v.
\]
For the tangential projections, the orthogonality of \(v^1-v^2\) and its
\(x\)-derivative, together with \(H'\in L^1(\R)\), gives
\[
\begin{aligned}
    A_{g_\tau}(\delta_v\widehat A_1+\delta_v\widehat A_2)
    &\lesssim
    \eps^{\beta+1-\alpha-(1-\alpha)\tau}D_v,
    &
    A_{g_\tau}(\delta_v\widehat A_3)
    &\lesssim
    \eps^{\beta-(1-\alpha)\tau}D_v,\\
    A_{g_\tau}(\delta_v\widehat B_2)
    &\lesssim
    \eps^{2\beta-(1-\alpha)\tau}D_v,
    &
    A_{g_\tau}(\delta_v\widehat C_{\mathrm{bd}})
    &\lesssim
    \eps^{\beta-(1-\alpha)\tau}D_v.
\end{aligned}
\]
The choice \eqref{eq: boundary-decay-exponent-range} gives
\(\beta-(1-\alpha)\tau>1-\alpha\); all the coefficients in this display are
therefore \(O(\eps^{1-\alpha})\). Hence
\[
    \lVert\delta_v\widehat F\rVert_{C_{\ell_+,\lambda}^{0,\gamma}}
    +A_{g_\tau}(\delta_v\widehat F)
    \lesssim\eps^{1-\alpha}D_v.
\]

Finally, only \(A_3\) changes when \(\psi\) varies. The nonlinear difference
has Lipschitz coefficient \(O(\eps^\beta)\). For the remaining term, the
coefficient
\(\chi_{j,1}(W''(\omega)-\kappa_W)\) has boundary weight \(\lambda/4\),
while \(\psi^1-\psi^2\) has rate \(3\lambda/4\). Thus
\[
    \lVert\delta_\psi\widehat F\rVert_{C_{\ell_+,\lambda}^{0,\gamma}}
    \lesssim D_\psi.
\]
Its support reaches \(X=O(\eps^{\alpha-1})\), so
\eqref{eq: boundary Ag support} gives
\[
    A_{g_\tau}(\delta_\psi\widehat F)
    \lesssim\eps^{-(1-\alpha)\tau}D_\psi.
\]

Combining the three variations yields
\[
\begin{split}
    &\lVert\widehat F_p(\phi_p^1;v^1,h_j,\psi^1)
    -\widehat F_p(\phi_p^2;v^2,h_j,\psi^2)\rVert_
    {C_{\ell_+,\lambda}^{0,\gamma}}\\
    &\quad+
    A_{g_\tau}\bigl(
    \widehat F_p(\phi_p^1;v^1,h_j,\psi^1)
    -\widehat F_p(\phi_p^2;v^2,h_j,\psi^2)
    \bigr)\\
    &\qquad\lesssim
    \eps^{\beta-(1-\alpha)\tau}D_\phi
    +\eps^{1-\alpha}D_v
    +\eps^{-(1-\alpha)\tau}D_\psi.
\end{split}
\]
Applying \Cref{cor: schauder-tan} to the difference of the fixed-point
equations and absorbing the first term gives
\begin{equation}\label{eq: boundary fixed-h dependence}
    D_\phi
    \lesssim
    \eps^{1-\alpha}D_v
    +\eps^{-(1-\alpha)\tau}D_\psi.
\end{equation}

We finally vary \(h_j\). Let \(h_j^i\), \(i=1,2\), satisfy the assumptions
in the statement, set \(\theta_i=\partial_{e_p}h_j^i(p)\), and write
\[
    D_h\coloneqq
    \lVert h_j^1-h_j^2\rVert_{C^{2,\gamma}(I_j)}.
\]
Both boundary equations are compared in the common unrotated variable \(Y\).
Thus a fixed function \(\widehat\phi\) corresponds, for the \(i\)-th data set,
to \(\phi_i(y)=\widehat\phi(\eps^{-1}R_{\theta_i}y)\). With this
identification, the estimates above apply unchanged to the variations in
\(\widehat\phi\), \(v\), and \(\psi\); it remains to estimate the variation
of the right-hand side with \(\widehat\phi\), \(v\), and \(\psi\) fixed.

For \(A_1+A_2\), the chain rules used in their fixed-data estimate involve
only \(v\) and its first derivatives. Since \(v\in
\mathcal C_{\eps,\lambda}^{2,\gamma}\), the mean value theorem uses at most
its second derivatives. More precisely, if \((x_i,t_i)\) are the interior
coordinates of \(\eps R_{-\theta_i}Y\), then on
\(\widehat\Omega_{p,\eta}\)
\[
    \eps^{-1}|x_1-x_2|+|t_1-t_2|
    \lesssim\eps^{\alpha-1}D_h.
\]
Differentiating the affine-collar formula for \(t\) above with respect to
\(\theta\), and using \(\lVert v\rVert\leq R_\eps\), shows that both the full
H\"older difference and the tangential projection before applying
\eqref{eq: boundary Ag support} are \(O(\eps^\beta D_h)\).  Together with
the cut-off estimates, this gives
\[
    \lVert\delta_h(\widehat A_1+\widehat A_2)\rVert_
    {C_{\ell_+,\lambda}^{0,\gamma}}
    +A_{g_\tau}\bigl(\delta_h(\widehat A_1+\widehat A_2)\bigr)
    \lesssim\eps^{\beta-(1-\alpha)\tau}D_h.
\]
The potential terms are controlled by
\Cref{lem: profile-potential-comparison-dependence} and
\eqref{eq: localized-potential-product}--\eqref{eq: localized-potential-Ag}.
After multiplication by \(\phi_p\) or \(v\), both their full and projected
norms are \(O(\eps^{\beta-(1-\alpha)\tau}D_h)\).

For \(A_3\), the approximation part is controlled directly by
\eqref{eq: boundary-approximation-error-dependence} and
\eqref{eq: boundary-approximation-error-dependence-Ag}.  The nonlinear
difference has coefficient \(O(R_\eps)\), and
\eqref{eq: boundary-background-potential-dependence} controls the coefficient
of the exponentially small outer correction. Consequently,
\[
    \lVert\delta_h\widehat A_3\rVert_{C_{\ell_+,\lambda}^{0,\gamma}}
    +A_{g_\tau}(\delta_h\widehat A_3)
    \lesssim\eps^{\beta-(1-\alpha)\tau}D_h.
\]

It remains to consider the boundary scalar.  From its definition,
\[
\begin{split}
    m_{j,h_j^1}^{\mathrm{bd}}[v]
    -m_{j,h_j^2}^{\mathrm{bd}}[v]
    =\frac1{A_H}\Bigg(&
    \bigl((h_j^1)'^2-(h_j^2)'^2\bigr)\int_\R vH'''\\
    &+2\eps\bigl((h_j^1)'-(h_j^2)'\bigr)
    \int_\R\partial_xvH''\Bigg).
\end{split}
\]
The stretched derivative bounds and \(\lVert h_j^i\rVert_{C^{2,\gamma}}\lesssim
\eps^\beta\) therefore give a full weighted bound
\(CR_\eps D_h\).  The same calculation after composition with the two
boundary coordinates uses at most one further stretched derivative of \(v\).
Together with \eqref{eq: boundary Ag support}, it gives
\[
    \lVert\delta_h\widehat C_{\mathrm{bd}}\rVert_
    {C_{\ell_+,\lambda}^{0,\gamma}}
    +A_{g_\tau}(\delta_h\widehat C_{\mathrm{bd}})
    \lesssim\eps^{\beta-(1-\alpha)\tau}D_h.
\]

Combining these estimates with the preceding telescoping argument and
absorbing the variation in \(\widehat\phi_p\) gives
\begin{equation}\label{eq: boundary full dependence}
\begin{split}
    \lVert\widehat\phi_p^1-\widehat\phi_p^2\rVert_
    {C_{\ell_+,\lambda}^{2,\gamma}}
    \lesssim{}&
    \eps^{1-\alpha}
    \lVert v^1-v^2\rVert_{\mathcal C_{\eps,\lambda}^{2,\gamma}}
    +\eps^{-(1-\alpha)\tau}
    \lVert\psi^1-\psi^2\rVert_{C_{\Gamma,\eps,3\lambda/4}^{2,\gamma}}\\
    &+\eps^{\beta-(1-\alpha)\tau}
    \lVert h_j^1-h_j^2\rVert_{C^{2,\gamma}(I_j)}.
\end{split}
\end{equation}

It remains to prove the tangential dependence estimate. Set
\[
\begin{aligned}
    \mathcal D_p\coloneqq{}&
    \eps^{1-\alpha}
    \lVert v^1-v^2\rVert_{\mathcal C_{\eps,\lambda}^{2,\gamma}(I_j\times\R)}\\
    &+\eps^{-(1-\alpha)\tau}
    \lVert\psi^1-\psi^2\rVert_
    {C_{\Gamma,\eps,3\lambda/4}^{2,\gamma}(\R^2_{\p},\L_{\p})}\\
    &+\eps^{\beta-(1-\alpha)\tau}
    \lVert h_j^1-h_j^2\rVert_{C^{2,\gamma}(I_j)}.
\end{aligned}
\]
The difference estimates preceding \eqref{eq: boundary full dependence}, after
absorbing the variation in \(\widehat\phi_p\), give
\(A_{g_\tau}(\widehat F^1-\widehat F^2)\lesssim\mathcal D_p\).
Set \(L_\eps\coloneqq\eps^{\alpha-1}\).  Every
transition-supported term in \(\widehat F^1-\widehat F^2\), including the
potential term multiplying \(\widehat\phi_p^1-\widehat\phi_p^2\), is
supported in \(cL_\eps<X<CL_\eps\).  On the positive end,
\[
    D_RU(X,Z)=-XH'(Z)+\rho(X,Z),
    \qquad
    \int_{\R}|\rho(X,Z)|\,dZ\lesssim e^{-\sigma X}.
\]
Since
\(g_\tau''(X)=(\tau-1)(\tau-2)(1+X)^{-\tau}\), the moment formula gives
\[
\begin{split}
    \left|
    \int_{\{cL_\eps<X<CL_\eps\}}
    (\widehat F^1-\widehat F^2)D_RU
    \right|
    &\lesssim
    \mathcal D_p
    \int_{cL_\eps}^{CL_\eps}Xg_\tau''(X)\,dX
    +e^{-cL_\eps}\mathcal D_p\\
    &\lesssim
    L_\eps^{2-\tau}\mathcal D_p.
\end{split}
\]
In the core \(0<X<cL_\eps\), the approximation and linear gluing terms
vanish.  The remaining terms have size
\(O(R_\eps\mathcal D_p)\) for the nonlinear part and
\[
    O\left(
    \lVert\psi^1-\psi^2\rVert_
    {C_{\Gamma,\eps,3\lambda/4}^{2,\gamma}
    (\R^2_{\p},\L_{\p})}
    \right)
\]
for the background part.  Exponential localization in \(Z\) makes the normal
integral uniform, while \(|D_RU|\lesssim 1+X\) there.  Hence the core
contribution is bounded by
\[
    CL_\eps^2
    \left(
        R_\eps\mathcal D_p
        +\lVert\psi^1-\psi^2\rVert_
        {C_{\Gamma,\eps,3\lambda/4}^{2,\gamma}
        (\R^2_{\p},\L_{\p})}
    \right)
    \lesssim L_\eps^{2-\tau}\mathcal D_p.
\]
Indeed,
\(R_\eps\lesssim\eps^{(1-\alpha)\tau}\) and the definition of
\(\mathcal D_p\) gives
\(\lVert\psi^1-\psi^2\rVert\lesssim
\eps^{(1-\alpha)\tau}\mathcal D_p\).  Since
\(L_\eps^{2-\tau}=\eps^{(1-\alpha)(\tau-2)}\), the transition and core
estimates imply
\[
    |\phi_\infty^1-\phi_\infty^2|
    \lesssim
    \eps^{(1-\alpha)(\tau-2)}\mathcal D_p.
\]
Together with \eqref{eq: decay tangential part}, this gives after the common
conversion \(x=\eps X\), but before the two \(h_j\)-dependent physical
coordinate conversions,
\begin{equation}\label{eq: boundary tangential difference}
\begin{split}
    |\delta\phi_\infty|
    +\lVert\delta\phi_p^\top\rVert_{L^\infty(I_\eps)}
    &\lesssim
    \eps^{(1-\alpha)(\tau-2)}\mathcal D_p,\\
    \lVert\partial_x\delta\phi_p^\top\rVert_{L^\infty(I_\eps)}
    &\lesssim
    \eps^{-1+(1-\alpha)(\tau-1)}\mathcal D_p,\\
    \lVert\partial_x^2\delta\phi_p^\top\rVert_{L^\infty(I_\eps)}
    &\lesssim
    \eps^{(1-\alpha)\tau-2}\mathcal D_p,\\
    [\delta\phi_p^\top]_{\gamma,I_\eps}
    &\lesssim
    \eps^{(1-\alpha)(\tau-2)-\alpha\gamma}\mathcal D_p,\\
    [\partial_x\delta\phi_p^\top]_{\gamma,I_\eps}
    &\lesssim
    \eps^{-1+(1-\alpha)(\tau-1)-\alpha\gamma}\mathcal D_p,
\end{split}
\end{equation}
where \(\delta\phi_p^\top=\phi_p^{\top,1}-\phi_p^{\top,2}\) and
\(\delta\phi_\infty=\phi_\infty^1-\phi_\infty^2\).  The third line follows
directly from the second-derivative estimate in
\eqref{eq: decay tangential part}.  Interpolating the first two pointwise
bounds gives the fourth line, and interpolating the second and third gives the
fifth.  It remains only to compare the two physical coordinate conversions.
On \(I_\eps\), Taylor's theorem gives
\[
    |(h_j^1-h_j^2)(x)-(\theta_1-\theta_2)x|
    \lesssim D_hx^{2+\gamma},
    \qquad
    |(h_j^1-h_j^2)'(x)-(\theta_1-\theta_2)|
    \lesssim D_hx^{1+\gamma}.
\]
Together with the individual estimates above, these bounds show that the
combined contribution of the two coordinate conversions to the scaled
\(C^{1,\gamma}\)-quantity on the left-hand side of
\eqref{eq: tangential dependence phib} is
\(O(\eps^{\beta-2+2\alpha}D_h)\).  Multiplying the value,
first-derivative, and H\"older estimates in
\eqref{eq: boundary tangential difference} by the corresponding powers of
\(\eps^\alpha\), and then expanding \(\mathcal D_p\), gives
\eqref{eq: tangential dependence phib}.

\end{proof}

\subsection{Proof of \Cref{prop: invertibility interior}}\label{sec: invertibility interior}
\begin{proof}
Fix \(j\), \(h_j\), and \(\psi\), and work in the shifted variables
\((x,t)\), where \(t=(z-h_j(x))/\eps\).  Let \(\mathbb B_j\) be the closed
ball of radius \(R_\eps\), defined in \eqref{eq: correction-radius}, in the
subspace of \(\mathcal C_{\eps,\lambda}^{2,\gamma}(I_j\times\R)\) determined
by
\[
    v=0\quad\text{on }\partial I_j\times\R,
    \qquad
    \int_\R v(x,t)H'(t)\,dt=0\quad(x\in I_j).
\]

We first invert the operator on the left-hand side of
\eqref{eq: interior eq3}.  The definition of
\(\mathfrak B_{j,h_j}^{\perp}\) gives
\[
    \lVert\mathfrak B_{j,h_j}^{\perp}[v]\rVert_
    {\mathcal C_{\eps,\lambda}^{0,\gamma}(I_j\times\R)}
    \lesssim
    \eps^\beta
    \lVert v\rVert_{\mathcal C_{\eps,\lambda}^{2,\gamma}(I_j\times\R)}.
\]
It is orthogonal to \(H'\), so, if \(G_H\) is the inverse from
\Cref{lem: invertibility LH}, the Neumann series defines
\[
    \widetilde G_{H,h_j}
    \coloneqq
    \bigl(I+G_H\mathfrak B_{j,h_j}^{\perp}\bigr)^{-1}G_H
\]
and gives
\begin{equation}\label{eq: inverse Ltilde}
    \lVert\widetilde G_{H,h_j}f\rVert_
    {\mathcal C_{\eps,\lambda}^{2,\gamma}(I_j\times\R)}
    \lesssim
    \lVert f\rVert_{\mathcal C_{\eps,\lambda}^{0,\gamma}(I_j\times\R)}.
\end{equation}
For \(v\in\mathbb B_j\), consider the map
\[
    v\longmapsto
    \widetilde G_{H,h_j}\Pi_{I_j}
    \mathcal F_{j,h_j}(v,\psi).
\]
Since \(\Pi_{I_j}\) is bounded,
\[
    \lVert\Pi_{I_j}\mathcal F_{j,h_j}(v,\psi)\rVert_
    {\mathcal C_{\eps,\lambda}^{0,\gamma}(I_j\times\R)}
    \lesssim
    \lVert\mathcal F_{j,h_j}(v,\psi)\rVert_
    {\mathcal C_{\eps,\lambda}^{0,\gamma}(I_j\times\R)}.
\]
Thus it is enough to estimate the full right-hand side, including its
\(H'\)-component.

For this purpose, write
\[
    \mathcal F_{j,h_j}(v,\psi)
    =
    P_1+P_2+Q_1+Q_2+M_{\mathrm{bd}}+M_{\mathrm{int}},
\]
where
\[
\begin{aligned}
    P_1&=\eps\chi_{j,4}\eta_{j,3}h_j''H',
    &
    P_2&=\chi_{j,4}\eta_{j,4}\mathcal R_j(\phi_j,\psi),\\
    Q_1&=-2\eps^2\chi_{j,4}
    \sum_{p\in\partial I_j}\nabla\zeta_{p,4}\cdot\nabla b_p,
    &
    Q_2&=-\eps^2\chi_{j,4}
    \sum_{p\in\partial I_j}\Delta\zeta_{p,4}b_p,\\
    M_{\mathrm{bd}}
    &=-\chi_{j,4}\eta_{j,4}m_j^{\mathrm{bd}}[v]H',
    &
    M_{\mathrm{int}}
    &=-\chi_{j,4}m_j^{\mathrm{int}}[v]H',
\end{aligned}
\]
where
\[
    b_p=\mathscr B_{p,h_j}(v,\psi)
\]
and \(\phi_j\) is reconstructed from \(v\) and these boundary responses.
All six terms are read in the shifted interior coordinates.

We first prove that this map sends \(\mathbb B_j\) into itself.  The shift
bound gives
\begin{equation}\label{eq: interior P1 size}
    \lVert P_1\rVert_{\mathcal C_{\eps,\lambda}^{0,\gamma}(I_j\times\R)}
    \lesssim\eps^{\beta+1}.
\end{equation}
For \(P_2\), cutoff nesting identifies its approximation part with
\(-\widetilde{\mathcal E}^{\mathrm{rem}}_{j,\h}\).  Hence
\eqref{eq: interior-approximation-error-size}, the quadratic estimate for
\(\mathcal N_\omega\), \eqref{eq: size phi_b}, and the outer bound give
\begin{equation}\label{eq: interior P2 size}
\begin{split}
    \lVert P_2\rVert_{\mathcal C_{\eps,\lambda}^{0,\gamma}(I_j\times\R)}
    \lesssim
    \eps^{2\beta}+R_\eps^2+r_{\mathrm{out}}(\eps).
\end{split}
\end{equation}
The stretched cutoff bounds are
\[
    \eps^2
    \lVert\nabla\zeta_{p,4}\rVert_
    {\mathcal C_\eps^{0,\gamma}(I_j\times\R)}
    \lesssim\eps^{2-\alpha},
    \qquad
    \eps^2
    \lVert\Delta\zeta_{p,4}\rVert_
    {\mathcal C_\eps^{0,\gamma}(I_j\times\R)}
    \lesssim\eps^{2-2\alpha}.
\]
On the commutator collar, the boundary and shifted-interior exponential
weights are uniformly comparable.  Thus the boundary estimate can be used
directly in the stretched interior norm.
After the physical derivative in \(Q_1\) is taken into account,
\eqref{eq: size phi_b} therefore yields
\begin{align}
    \lVert Q_1\rVert_{\mathcal C_{\eps,\lambda}^{0,\gamma}(I_j\times\R)}
    &\lesssim
    \eps^{2-2\alpha}
    \lVert v\rVert_{\mathcal C_{\eps,\lambda}^{2,\gamma}(I_j\times\R)}
    +\eps^{2\beta+1-\alpha-(1-\alpha)\tau},
    \label{eq: interior Q1 size}\\
    \lVert Q_2\rVert_{\mathcal C_{\eps,\lambda}^{0,\gamma}(I_j\times\R)}
    &\lesssim
    \eps^{3-3\alpha}
    \lVert v\rVert_{\mathcal C_{\eps,\lambda}^{2,\gamma}(I_j\times\R)}
    +\eps^{2\beta+2-2\alpha-(1-\alpha)\tau}.
    \label{eq: interior Q2 size}
\end{align}
Finally, the definitions of the two moments give
\begin{align}
    \lVert M_{\mathrm{bd}}\rVert_
    {\mathcal C_{\eps,\lambda}^{0,\gamma}(I_j\times\R)}
    &\lesssim
    \eps^\beta
    \lVert v\rVert_{\mathcal C_{\eps,\lambda}^{2,\gamma}(I_j\times\R)},
    \label{eq: interior Mbd size}\\
    \lVert M_{\mathrm{int}}\rVert_
    {\mathcal C_{\eps,\lambda}^{0,\gamma}(I_j\times\R)}
    &\lesssim
    \eps^{\beta+1}
    \lVert v\rVert_{\mathcal C_{\eps,\lambda}^{2,\gamma}(I_j\times\R)}.
    \label{eq: interior Mint size}
\end{align}

For \(v\in\mathbb B_j\), division of
\eqref{eq: interior P1 size}--\eqref{eq: interior Mint size} by \(R_\eps\)
produces only positive powers of \(\eps\).  The two powers involving
\(\tau\) are positive because
\[
    \beta-(1-\alpha)\tau>1-\alpha,
    \qquad
    \beta+1-\alpha-(1-\alpha)\tau>\beta-2\alpha>0.
\]
The outer term is smaller than every power of \(\eps\).  Consequently,
\begin{equation}\label{eq: smallness Fin}
    \lVert\mathcal F_{j,h_j}(v,\psi)\rVert_
    {\mathcal C_{\eps,\lambda}^{0,\gamma}(I_j\times\R)}
    \leq c_\eps R_\eps,
    \qquad
    c_\eps\longrightarrow0.
\end{equation}

We next fix \(h_j,\psi\) and compare \(v_1,v_2\in\mathbb B_j\).  The term
\(P_1\) cancels.  The quadratic estimate, the boundary dependence
\eqref{eq: Lip phib}, and the two moment definitions give, respectively,
\[
\begin{aligned}
    \lVert P_2(v_1)-P_2(v_2)\rVert_
    {\mathcal C_{\eps,\lambda}^{0,\gamma}(I_j\times\R)}
    &\lesssim R_\eps
    \lVert v_1-v_2\rVert_{\mathcal C_{\eps,\lambda}^{2,\gamma}(I_j\times\R)},\\
    \lVert Q_1(v_1)-Q_1(v_2)\rVert_
    {\mathcal C_{\eps,\lambda}^{0,\gamma}(I_j\times\R)}
    &\lesssim\eps^{2-2\alpha}
    \lVert v_1-v_2\rVert_{\mathcal C_{\eps,\lambda}^{2,\gamma}(I_j\times\R)},\\
    \lVert Q_2(v_1)-Q_2(v_2)\rVert_
    {\mathcal C_{\eps,\lambda}^{0,\gamma}(I_j\times\R)}
    &\lesssim\eps^{3-3\alpha}
    \lVert v_1-v_2\rVert_{\mathcal C_{\eps,\lambda}^{2,\gamma}(I_j\times\R)},\\
    \lVert M_{\mathrm{bd}}(v_1)-M_{\mathrm{bd}}(v_2)\rVert_
    {\mathcal C_{\eps,\lambda}^{0,\gamma}(I_j\times\R)}
    &+
    \lVert M_{\mathrm{int}}(v_1)-M_{\mathrm{int}}(v_2)\rVert_
    {\mathcal C_{\eps,\lambda}^{0,\gamma}(I_j\times\R)}\\
    &\lesssim\eps^\beta
    \lVert v_1-v_2\rVert_{\mathcal C_{\eps,\lambda}^{2,\gamma}(I_j\times\R)}.
\end{aligned}
\]
Since
\[
    R_\eps+\eps^{3-3\alpha}+\eps^\beta
    =o(\eps^{2-2\alpha}),
\]
we obtain
\begin{equation}\label{eq: contract Fin}
\begin{split}
    &\lVert\mathcal F_{j,h_j}(v_1,\psi)
    -\mathcal F_{j,h_j}(v_2,\psi)\rVert_
    {\mathcal C_{\eps,\lambda}^{0,\gamma}(I_j\times\R)}\\
    &\qquad\lesssim
    \eps^{2-2\alpha}
    \lVert v_1-v_2\rVert_{\mathcal C_{\eps,\lambda}^{2,\gamma}(I_j\times\R)}.
\end{split}
\end{equation}
Equations \eqref{eq: inverse Ltilde}, \eqref{eq: smallness Fin}, and
\eqref{eq: contract Fin} show that this map is a contraction of
\(\mathbb B_j\) into itself when \(\eps\) is sufficiently small.  Its unique
fixed point is \(v_j=\mathscr V_{j,h_j}(\psi)\).

It remains to compare two data sets.  Let
\[
    v_i=\mathscr V_{j,h_j^i}(\psi^i),
    \qquad i=1,2,
\]
and set
\[
\begin{aligned}
    D_v&\coloneqq
    \lVert v_1-v_2\rVert_{\mathcal C_{\eps,\lambda}^{2,\gamma}(I_j\times\R)},\\
    D_h&\coloneqq
    \lVert h_j^1-h_j^2\rVert_{C^{2,\gamma}(I_j)},\\
    D_\psi&\coloneqq
    \lVert\psi^1-\psi^2\rVert_
    {C_{\Gamma,\eps,3\lambda/4}^{2,\gamma}
    (\R^2_{\p},\L_{\p})}.
\end{aligned}
\]
We use the decomposition
\[
\begin{split}
    &\mathcal F_{j,h_j^1}(v_1,\psi^1)
    -\mathcal F_{j,h_j^2}(v_2,\psi^2)\\
    &\quad=
    \bigl[
        \mathcal F_{j,h_j^1}(v_1,\psi^1)
        -\mathcal F_{j,h_j^1}(v_2,\psi^2)
    \bigr]\\
    &\qquad+
    \bigl[
        \mathcal F_{j,h_j^1}(v_2,\psi^2)
        -\mathcal F_{j,h_j^2}(v_2,\psi^2)
    \bigr].
\end{split}
\]
First keep the shift \(h_j=h_j^1\) fixed.  Equation
\eqref{eq: Lip phib} gives
\[
    \lVert\mathscr B_{p,h_j}(v_1,\psi^1)
    -\mathscr B_{p,h_j}(v_2,\psi^2)\rVert_
    {C_{\ell_{\theta_{\h^1}(p)},\eps,\lambda}^{2,\gamma}(\R^2_*,\L)}
    \lesssim
    \eps^{1-\alpha}D_v+
    \eps^{-(1-\alpha)\tau}D_\psi.
\]
The preceding term-by-term estimates then give
\begin{equation}\label{eq: interior fixed-h data rhs}
\begin{split}
    &\lVert\mathcal F_{j,h_j}(v_1,\psi^1)
    -\mathcal F_{j,h_j}(v_2,\psi^2)\rVert_
    {\mathcal C_{\eps,\lambda}^{0,\gamma}(I_j\times\R)}\\
    &\qquad\lesssim
    \rho_\eps D_v+
    \eps^{1-\alpha-(1-\alpha)\tau}D_\psi,
    \qquad
    \rho_\eps\longrightarrow0.
\end{split}
\end{equation}
Indeed, the \(Q_1\)-coefficients are
\(\eps^{2-2\alpha}\) and
\(\eps^{1-\alpha-(1-\alpha)\tau}\), the \(Q_2\)-coefficients are
\(\eps^{3-3\alpha}\) and
\(\eps^{2-2\alpha-(1-\alpha)\tau}\), and the nonlinear and moment terms are
smaller.

Now keep \(v=v_2\) and \(\psi=\psi^2\) fixed and vary the shift.  The term \(P_1\)
contributes \(O(\eps D_h)\).  By
\eqref{eq: interior-approximation-error-dependence}, the approximation part of
\(P_2\) contributes \(O(\eps^\beta D_h)\).  The quadratic estimate,
\eqref{eq: Lip phib}, and the exponentially small outer bound give
\[
\begin{split}
    \lVert P_2(v,h_j^1,\psi)-P_2(v,h_j^2,\psi)\rVert_
    {\mathcal C_{\eps,\lambda}^{0,\gamma}(I_j\times\R)}
    \lesssim{}&
    \left(
        \eps^\beta
        +\eps^{\beta-(1-\alpha)\tau}R_\eps
        +\eps^{-1}R_\eps^2
        +\eps^{-1}r_{\mathrm{out}}(\eps)
    \right)D_h\\
    &\lesssim\eps D_h.
\end{split}
\]
In the common unrotated boundary
coordinates, set
\[
    \widehat b_p^i(Y)
    =
    \mathscr B_{p,h_j^i}(v,\psi)
    \bigl(\eps R_{-\theta_{\h^i}(p)}Y\bigr),
    \qquad i=1,2.
\]
Equation \eqref{eq: Lip phib} gives
\[
    \lVert\widehat b_p^1-\widehat b_p^2\rVert_
    {C_{\ell_+,\lambda}^{2,\gamma}(\R^2_*,\L)}
    \lesssim
    \eps^{\beta-(1-\alpha)\tau}D_h.
\]
On the commutator collar, the two boundary coordinate maps differ by
\(O(\eps^{\alpha-1}D_h)\) in the scaled \(C^{1,\gamma}\)-norm.  Write
\(\delta_{\mathrm{coord}}Q_k\) for the part of the \(Q_k\)-difference caused
only by composing the boundary response with these two coordinate maps.
The \(C^{2,\gamma}\)-to-\(C^{1,\gamma}\) composition estimate and
\eqref{eq: size phi_b} give
\[
\begin{aligned}
    \lVert\delta_{\mathrm{coord}}Q_1\rVert_
    {\mathcal C_{\eps,\lambda}^{0,\gamma}(I_j\times\R)}
    &\lesssim
    \left(
        \eps^{1-\alpha}R_\eps
        +\eps^{2\beta-(1-\alpha)\tau}
    \right)D_h,\\
    \lVert\delta_{\mathrm{coord}}Q_2\rVert_
    {\mathcal C_{\eps,\lambda}^{0,\gamma}(I_j\times\R)}
    &\lesssim
    \eps^{1-\alpha}
    \left(
        \eps^{1-\alpha}R_\eps
        +\eps^{2\beta-(1-\alpha)\tau}
    \right)D_h.
\end{aligned}
\]
The variations of the pulled-back \(\chi_{j,4}\) and the undifferentiated
cutoff factors are included in these bounds; the differentiated
\(\zeta_{p,4}\)-factors are fixed and horizontal on this collar.  Both
coordinate terms are \(o(\eps D_h)\).  The direct boundary-response
differences contribute
\[
    O\left(
        \eps^{\beta+1-\alpha-(1-\alpha)\tau}D_h
    \right)
    \quad\text{and}\quad
    O\left(
        \eps^{\beta+2-2\alpha-(1-\alpha)\tau}D_h
    \right),
\]
respectively.  Both are \(O(\eps D_h)\), since
\[
    \beta-\alpha-(1-\alpha)\tau>\beta-1-2\alpha>0,
    \qquad
    \beta+1-2\alpha-(1-\alpha)\tau>\beta-3\alpha>0.
\]
Finally,
\[
\begin{aligned}
    \lVert M_{\mathrm{bd}}^1-M_{\mathrm{bd}}^2\rVert_
    {\mathcal C_{\eps,\lambda}^{0,\gamma}(I_j\times\R)}
    &\lesssim
    \eps^\beta D_v+R_\eps D_h,\\
    \lVert M_{\mathrm{int}}^1-M_{\mathrm{int}}^2\rVert_
    {\mathcal C_{\eps,\lambda}^{0,\gamma}(I_j\times\R)}
    &\lesssim
    \eps^{\beta+1}D_v+\eps R_\eps D_h.
\end{aligned}
\]
Combining the fixed-shift and varying-shift comparisons, and weakening the
outer coefficient, gives
\begin{equation}\label{eq: interior full data rhs}
\begin{split}
    &\lVert\mathcal F_{j,h_j^1}(v_1,\psi^1)
    -\mathcal F_{j,h_j^2}(v_2,\psi^2)\rVert_
    {\mathcal C_{\eps,\lambda}^{0,\gamma}(I_j\times\R)}\\
    &\qquad\lesssim
    \rho_\eps D_v+\eps D_h
    +\eps^{-(1-\alpha)\tau}D_\psi.
\end{split}
\end{equation}

The inverse has the corresponding shift dependence.  Directly from the
definition of \(\mathfrak B_{j,h_j}^{\perp}\),
\[
    \lVert(\mathfrak B_{j,h_j^1}^{\perp}
    -\mathfrak B_{j,h_j^2}^{\perp})v\rVert_
    {\mathcal C_{\eps,\lambda}^{0,\gamma}(I_j\times\R)}
    \lesssim
    D_h
    \lVert v\rVert_{\mathcal C_{\eps,\lambda}^{2,\gamma}(I_j\times\R)}.
\]
The resolvent identity therefore gives
\[
    \lVert(\widetilde G_{H,h_j^1}-\widetilde G_{H,h_j^2})f\rVert_
    {\mathcal C_{\eps,\lambda}^{2,\gamma}(I_j\times\R)}
    \lesssim
    D_h
    \lVert f\rVert_{\mathcal C_{\eps,\lambda}^{0,\gamma}(I_j\times\R)}.
\]
Here \(f\) is required to satisfy \(\Pi_{I_j}f=f\), as in
\Cref{lem: invertibility LH}.  By \eqref{eq: smallness Fin}, the
inverse-variation term in the two fixed-point equations is
\(o(R_\eps)D_h=o(\eps D_h)\).  Subtracting those equations and using
\eqref{eq: interior full data rhs} gives
\[
    D_v
    \lesssim
    \rho_\eps D_v
    +\eps D_h
    +\eps^{-(1-\alpha)\tau}D_\psi
    +o(\eps D_h).
\]
Absorbing the \(\rho_\eps D_v\)-term proves \eqref{eq: Lip phiin}.
\end{proof}

\subsection{Proof of \Cref{lem: local assembly}}
\label{sec:local-assembly-proof}

\begin{proof}
The scaled product and composition estimates applied to
\eqref{eq: decomp phi} give
\[
    \lVert\mathcal L_{\h}(\psi)_j\rVert_
    {C_{I_j,\eps,\lambda}^{2,\gamma}
    (\spt\chi_{j,2},\L_{\p})}
    \lesssim
    \lVert v_j(\h,\psi)\rVert_
    {\mathcal C_{\eps,\lambda}^{2,\gamma}(I_j\times\R)}
    +
    \sum_{p\in\partial I_j}
    \lVert b_p(\h,\psi)\rVert_
    {C_{\ell_{\theta_{\h}(p)},\eps,\lambda}^{2,\gamma}
    (\R^2_*,\L)}.
\]
The cutoff derivatives cost powers of
\(\eps^{1-\alpha}\) and \(\eps^{2-2\alpha}\), so their multiplication
operators are uniformly bounded in the displayed scaled norm.  The size
estimate now follows from \Cref{prop: invert boundary} and
\Cref{prop: invertibility interior}.

For the difference estimate, set
\[
    D_h=\lVert\h^1-\h^2\rVert_{C^{2,\gamma}(\boldsymbol I)},
    \qquad
    D_\psi=
    \lVert\psi^1-\psi^2\rVert_
    {C_{\Gamma,\eps,3\lambda/4}^{2,\gamma}
    (\R^2_{\p},\L_{\p})}.
\]
Equation \eqref{eq: Lip phiin} gives
\[
    \lVert v_j(\h^1,\psi^1)-v_j(\h^2,\psi^2)\rVert_
    {\mathcal C_{\eps,\lambda}^{2,\gamma}(I_j\times\R)}
    \lesssim
    \eps D_h+\eps^{-(1-\alpha)\tau}D_\psi.
\]
For \(i=1,2\), set
\[
    \widehat b_p(\h^i,\psi^i)(Y)
    \coloneqq
    b_p(\h^i,\psi^i)
    \bigl(\eps R_{-\theta_{\h^i}(p)}Y\bigr).
\]
Substituting this estimate into \eqref{eq: Lip phib} gives, in the common
unrotated boundary coordinates,
\[
\begin{split}
    \lVert\widehat b_p(\h^1,\psi^1)
    -\widehat b_p(\h^2,\psi^2)\rVert_
    {C_{\ell_+,\lambda}^{2,\gamma}(\R^2_*,\L)}
    \lesssim{}&
    \left(
        \eps+\eps^{\beta-(1-\alpha)\tau}
    \right)D_h\\
    &+\eps^{-(1-\alpha)\tau}D_\psi.
\end{split}
\]
It remains to compare the reconstructed sections in physical coordinates.
For the interior term, changing the stretched normal variable costs
\(R_\eps D_h/\eps\).  For a boundary term, changing from the common
unrotated coordinates to the two physical boundary coordinates costs
\(\eps^{\alpha-1}R_\eps D_h\).  Since
\[
    \frac{R_\eps}{\eps}
    =\eps^{\beta-\alpha}\lesssim\eps,
    \qquad
    \eps^{\alpha-1}R_\eps=\eps^\beta\lesssim\eps,
\]
both coordinate contributions are absorbed by the \(\eps D_h\)-term.
Applying the scaled product estimate to \eqref{eq: decomp phi} now proves
\eqref{eq: local compatibility dependence}.
\end{proof}

\subsection{Proof of \Cref{prop: invert outer}}\label{sec: invert outer}

We first record the linear estimate used in the fixed-point argument.

\begin{lemma}\label{lem: invertibility out}
For a shift tuple \(\h\) satisfying \eqref{eq: shift-size-compatibility}, let
\[
    V\coloneqq
    \chi_{\Ical,1}W''(\omega)
    +
   \kappa_W(1-\chi_{\Ical,1}).
\]
There exists \(\eps_*>0\), independent of \(\h\), such that, for every \(0<\eps\leq\eps_*\) and
every \(f\in C_{*,\eps}^{0,\gamma}(\R^2_{\p},\L_{\p})\), the equation
\[
    \eps^2\Delta\phi-V\phi=f
    \qquad\text{on }\R^2_{\p}
\]
has a unique solution
\(\phi=\mathcal Gf\in C_{*,\eps}^{2,\gamma}(\R^2_{\p},\L_{\p})\).
The solution operator \(\mathcal G\) satisfies
\begin{equation}\label{eq: bound on G}
\begin{split}
    \lVert\mathcal Gf\rVert_
    {C_{*,\eps}^{2,\gamma}(\R^2_{\p},\L_{\p})}
    &\lesssim
    \lVert f\rVert_{C_{*,\eps}^{0,\gamma}(\R^2_{\p},\L_{\p})},\\
    \lVert\mathcal Gf\rVert_{C_{\Gamma,\eps,\nu}^{2,\gamma}
    (\R^2_{\p},\L_{\p})}
    &\lesssim
    \lVert f\rVert_{C_{\Gamma,\eps,\nu}^{0,\gamma}
    (\R^2_{\p},\L_{\p})},
    \qquad
    \nu\in\left\{\frac{3\lambda}{4},\lambda\right\}.
\end{split}
\end{equation}
The weighted estimate applies whenever its right-hand side is finite.
\end{lemma}

\begin{proof}
The potential can be written as
\[
    V=\kappa_W+\chi_{\Ical,1}\bigl(W''(\omega)-\kappa_W\bigr).
\]
By \eqref{eq: outer-potential-size}, the second term is exponentially small
in the scaled \(C^{2,\gamma}\) norm. Since
\(\lambda<\sqrt{\kappa_W}\), after decreasing \(\eps_*\) we have
\[
    V-\lambda^2\geq c>0,
    \qquad
    \lVert V\rVert_{C_{*,\eps}^{2,\gamma}(\R^2_{\p})}\lesssim1 .
\]

Let \(\pi\colon\widetilde{\R^2_{\p}}\to\R^2_{\p}\) be the double cover from
\Cref{sec: preliminaries}. We complete it over each boundary point by adding
the branch point in the coordinate \(\zeta\mapsto p_i+\zeta^2\), and define
\[
    g=\rho\pi^*|dx|^2 ,
\]
where \(\rho\) is a smooth positive deck-invariant function equal to
\((4|\zeta|^2)^{-1}\) near the added points and to \(1\) away from them.
Since \(\pi^*|dx|^2=4|\zeta|^2|d\zeta|^2\) near a branch point,
\(g\) extends to a smooth complete metric on the completed cover
\(\widetilde M\).

Let \(\widetilde V=V\circ\pi\), and let \(\widetilde f\) and
\(\widetilde\phi\) denote the odd lifts of the right-hand side and the unknown.
The lifted equation is
\[
    \eps^2\Delta_g\widetilde \phi
    -
    \rho^{-1}\widetilde V\widetilde \phi
    =
    \rho^{-1}\widetilde f .
\]
Here \(\rho^{-1}\) extends smoothly by zero at the added points. Set
\[
    M=\lVert\widetilde f/\widetilde V\rVert_{L^\infty}.
\]
On deck-invariant zero-Dirichlet exhaustions, the constants \(M\) and
\(-M\) are respectively an upper and a lower barrier for the lifted equation.
The maximum principle and local compactness therefore give a bounded global
solution with \(\lVert\widetilde\phi\rVert_{L^\infty}\leq M\).

To prove uniqueness among bounded solutions, observe that outside a fixed
compact set one has \(\rho=1\) and \(\widetilde V\geq c_0>0\). If
\(r(q)=\sqrt{1+|q|^2}\), then, for \(a>0\) sufficiently small,
\[
    (-\eps^2\Delta_g+\widetilde V)e^{ar/\eps}
    \geq
    \bigl(c_0-a^2-O(\eps)\bigr)e^{ar/\eps}>0 .
\]
Comparison on expanding exhaustions with a multiple of this growing
supersolution shows that every bounded homogeneous solution vanishes.
Uniqueness and deck invariance of the exhaustions imply that the solution is
odd, and it therefore descends to a section.

On trivializing balls, the \(L^\infty\) estimate is upgraded by the standard
scaled Schauder estimates. Near a branch point, after writing
\(\zeta=\sqrt\eps\eta\), the equation becomes
\[
    \Delta_\eta\widetilde\phi_\eps
    -4|\eta|^2\widetilde V_\eps\widetilde\phi_\eps
    =
    4|\eta|^2\widetilde f_\eps
\]
on a fixed ball. The ordinary Schauder estimate in this chart controls the
inner part of the singular disk. Its remaining fixed annulus is covered by
trivializing balls, where the two scaled descriptions are uniformly
equivalent by the compatibility observation following
\Cref{def:singular-ball-norm}. This proves the first estimate in
\eqref{eq: bound on G}.

It remains to prove the weighted estimates. Fix
\(\nu\in\{3\lambda/4,\lambda\}\). For
\(d_j(q)\coloneqq\dist(q,I_j)\), set
\[
    b(q)\coloneqq
    \sum_{j=1}^N
    \exp\left(-\frac{\nu}{\eps}d_j(q)\right).
\]
Each segment is convex, so \(\Delta d_j\geq0\) weakly. Consequently,
\[
    (-\eps^2\Delta+V)
    \exp\left(-\frac{\nu}{\eps}d_j\right)
    \geq
    \left(V-\nu^2\right)
    \exp\left(-\frac{\nu}{\eps}d_j\right)
\]
in the weak sense. Thus \((-\eps^2\Delta+V)b\geq cb\), while
\[
    b(q)\simeq
    \exp\left(-\frac{\nu}{\eps}\dist(q,\Gamma)\right).
\]
The definition of the weighted norm gives
\(\lvert f\rvert\lesssim
\lVert f\rVert_{C_{\Gamma,\eps,\nu}^{0,\gamma}
(\R^2_{\p},\L_{\p})}b\). Applying the comparison
principle on the lifted zero-Dirichlet exhaustions, first with small disks about
the branch points removed, and using that the odd lift tends to zero there,
yields the corresponding weighted \(L^\infty\) estimate for \(\mathcal Gf\)
after passing to the limit. On a ball of radius \(O(\eps)\), the distance to \(\Gamma\)
changes by \(O(\eps)\), so the weights are uniformly comparable and the scaled
Schauder estimate gives the weighted bound on every trivializing ball. The
weight factor in the singular-ball term is one because \(p_i\in\Gamma\), so
the square-root estimate above applies. This proves the weighted estimates in
\eqref{eq: bound on G}.
\end{proof}
\begin{proof}[Proof of \Cref{prop: invert outer}]
Fix \(\h\) and \(\phi=(\phi_1,\dots,\phi_N)\), and set
\[
    \Phi_\phi\coloneqq \sum_{k=1}^N\chi_{k,2}\phi_k,
    \qquad
    \mathcal C_\phi\coloneqq
    \sum_{k=1}^N
    \left(
    2\eps^2\nabla\chi_{k,2}\cdot\nabla\phi_k
    +\eps^2\phi_k\Delta\chi_{k,2}
    \right).
\]
Then
\[
    E_\out(\phi,\psi)
    =-\chi_{\Ical,2}S(\omega)
    -\chi_{\Ical,1}\mathcal N_\omega(\psi+\Phi_\phi)
    -\mathcal C_\phi.
\]
For
\(m=0,2\) and \(\Omega\subset\R^2_{\p}\), write
\[
    \lVert u\rVert_{m,\out;\Omega}
    \coloneqq
    \lVert u\rVert_{C_{\Gamma,\eps,3\lambda/4}^{m,\gamma}
    (\Omega,\L_{\p})},
\]
and omit \(\Omega\) when it is \(\R^2_{\p}\). Let
\(T_\phi(\psi)\coloneqq\mathcal G(E_\out(\phi,\psi))\). By construction, the
fixed points of \(T_\phi\) in \(\mathscr B_{\mathrm{out}}(\eps)\) are
precisely the solutions of \eqref{eq: outer equation} in that ball.

We first prove that \(T_\phi\) is a self-map. On the supports of
\(\chi_{\Ical,1}\chi_{k,2}\), \(\nabla\chi_{k,2}\), and
\(\Delta\chi_{k,2}\),
\(\dist(q,\Gamma)=\dist(q,I_k)\gtrsim\eps^\alpha\). Hence
\eqref{eq: norm bound on phi} gives
\[
    e^{3\lambda\dist(q,\Gamma)/(4\eps)}
    \lVert\phi_k\rVert_{C_\eps^{2,\gamma}(B_\eps(q),\L_{\p})}
    \lesssim
    \exp\left(-\frac{\lambda}{4\eps}\dist(q,I_k)\right)
    \lesssim
    \exp\left(-\frac{c}{\eps^{1-\alpha}}\right)
\]
for some \(c>0\). Since the cut-offs vary on the scale \(\eps^\alpha\),
their fixed-profile construction gives
\[
    \lVert\nabla^k\chi_{j,2}\rVert_{L^\infty}
    +\eps^{\alpha\gamma}[\nabla^k\chi_{j,2}]_\gamma
    \lesssim\eps^{-k\alpha},
    \qquad k=1,2.
\]
The resulting polynomial costs are absorbed by the same exponential after
decreasing \(c\).
Set
\[
    \Omega_{\Ical,\eps}
    \coloneqq
    \left\{q\in\R^2_{\p}:
    \dist(q,\spt\chi_{\Ical,1})<3R_0\eps\right\}.
\]
The same exponential estimate controls \(\Phi_\phi\) on this set. Moreover,
\eqref{eq: outer-approximation-error-global} gives the corresponding bound
for \(\chi_{\Ical,2}S(\omega)\). The fixed \(\widetilde\lambda\) was chosen
smaller than half of the decay rates in these estimates, so
\begin{equation}\label{eq: outer-localized-smallness}
    \lVert\Phi_\phi\rVert_{0,\out;\Omega_{\Ical,\eps}}
    +\lVert\mathcal C_\phi\rVert_{0,\out}
    +\lVert\chi_{\Ical,2}S(\omega)\rVert_{0,\out}
    \lesssim r_{\mathrm{out}}(\eps)^2.
\end{equation}
The quadratic estimate for \(\mathcal N_\omega\), the weighted product estimate, and
\eqref{eq: outer-localized-smallness} give
\[
    \lVert\chi_{\Ical,1}\mathcal N_\omega(\psi+\Phi_\phi)\rVert_{0,\out}
    \lesssim
    \bigl(\lVert\psi\rVert_{2,\out}
    +r_{\mathrm{out}}(\eps)^2\bigr)^2.
\]
Thus, for \(\psi\in\mathscr B_{\mathrm{out}}(\eps)\),
\[
    \lVert E_\out(\phi,\psi)\rVert_{0,\out}
    \lesssim r_{\mathrm{out}}(\eps)^2.
\]
The weighted estimate in \Cref{lem: invertibility out} shows that
\(T_\phi(\psi)\in\mathscr B_{\mathrm{out}}(\eps)\) for small \(\eps\).

For \(\psi_1,\psi_2\in\mathscr B_{\mathrm{out}}(\eps)\), the weighted
estimate in
\Cref{lem: invertibility out} and the quadratic Lipschitz estimate give
\[
\begin{split}
    \lVert T_\phi(\psi_1)-T_\phi(\psi_2)\rVert_{2,\out}
    &\lesssim
    \lVert E_\out(\phi,\psi_1)-E_\out(\phi,\psi_2)\rVert_{0,\out}\\
    &\lesssim
    r_{\mathrm{out}}(\eps)
    \lVert\psi_1-\psi_2\rVert_{2,\out}.
\end{split}
\]
Thus \(T_\phi\) is a contraction on \(\mathscr B_{\mathrm{out}}(\eps)\) for \(\eps\)
sufficiently small. Together with the preceding self-map estimate, the
contraction mapping theorem gives a unique solution of
\eqref{eq: outer equation} in \(\mathscr B_{\mathrm{out}}(\eps)\).

We next compare two local tuples while keeping the shift fixed. Let
\(\psi^\ell=\Psi_{\h}(\phi^\ell)\), \(\ell=1,2\), and set
\[
    D\coloneqq
    \sup_j
    \lVert\phi_j^1-\phi_j^2\rVert_
    {C^{1,\gamma}_{I_j,\eps,\lambda}
    (\spt\chi_{j,2},\L_{\p})}.
\]
The same support estimates used in
\eqref{eq: outer-localized-smallness} give
\begin{equation}\label{eq: outer-localized-difference}
\begin{split}
    &\lVert\Phi_{\phi^1}-\Phi_{\phi^2}\rVert_
    {0,\out;\Omega_{\Ical,\eps}}
    +\lVert\mathcal C_{\phi^1}-\mathcal C_{\phi^2}\rVert_{0,\out}
    \lesssim r_{\mathrm{out}}(\eps)^2D.
\end{split}
\end{equation}
On \(\spt\chi_{\Ical,1}\), the two arguments of \(\mathcal N_\omega\) satisfy
\[
\begin{split}
    \lVert\psi^\ell+\Phi_{\phi^\ell}\rVert_
    {0,\out;\Omega_{\Ical,\eps}}
    &\lesssim r_{\mathrm{out}}(\eps),\\
    \lVert(\psi^1+\Phi_{\phi^1})
    -(\psi^2+\Phi_{\phi^2})\rVert_
    {0,\out;\Omega_{\Ical,\eps}}
    &\lesssim
    \lVert\psi^1-\psi^2\rVert_{2,\out}
    +r_{\mathrm{out}}(\eps)^2D.
\end{split}
\]
The quadratic Lipschitz estimate for \(\mathcal N_\omega\) therefore gives
\[
\begin{split}
    &\lVert E_\out(\phi^1,\psi^1)-E_\out(\phi^2,\psi^2)\rVert_{0,\out}\\
    &\qquad\lesssim
    r_{\mathrm{out}}(\eps)\lVert\psi^1-\psi^2\rVert_{2,\out}
    +r_{\mathrm{out}}(\eps)^2D.
\end{split}
\]
Since \(\psi^\ell=\mathcal G(E_\out(\phi^\ell,\psi^\ell))\), applying
\Cref{lem: invertibility out} and absorbing the first term yields
\[
    \lVert\psi^1-\psi^2\rVert_{2,\out}
    \lesssim r_{\mathrm{out}}(\eps)D.
\]

Finally, let both the shift and the local tuple vary. Write
\(\omega_i=\omega_{\h^i}\), and let \(V_i\) and
\(E_{\out,i}\) denote the corresponding quantities. Set
\[
    M\coloneqq
    \lVert\h^1-\h^2\rVert_{C^{2,\gamma}(\boldsymbol I)}.
\]
Equations \eqref{eq: outer-potential-dependence},
\eqref{eq: outer-approximation-error-dependence}, and
\eqref{eq: outer-ansatz-dependence}, with the choice of
\(\widetilde\lambda\) made above, give
\[
    \lVert V_1-V_2\rVert_{0,\out}
    +\lVert\chi_{\Ical,2}(S(\omega_1)-S(\omega_2))\rVert_{0,\out}
    +\lVert\chi_{\Ical,1}(\omega_1-\omega_2)\rVert_{2,\out}
    \lesssim r_{\mathrm{out}}(\eps)^2M.
\]
The same two argument estimates hold for the varying shifts. Taylor's theorem,
including the dependence of the quadratic remainder on \(\omega_i\), gives
\[
\begin{split}
    &\lVert\chi_{\Ical,1}
    [\mathcal N_{\omega_1}(\psi^1+\Phi_{\phi^1})
    -\mathcal N_{\omega_2}(\psi^2+\Phi_{\phi^2})]\rVert_{0,\out}\\
    &\qquad\lesssim
    r_{\mathrm{out}}(\eps)\lVert\psi^1-\psi^2\rVert_{2,\out}
    +r_{\mathrm{out}}(\eps)^2(M+D).
\end{split}
\]
Together with \eqref{eq: outer-localized-difference}, this yields
\[
\begin{split}
    &\lVert E_{\out,1}(\phi^1,\psi^1)
    -E_{\out,2}(\phi^2,\psi^2)
    +(V_1-V_2)\psi^2\rVert_{0,\out}\\
    &\qquad\lesssim
    r_{\mathrm{out}}(\eps)\lVert\psi^1-\psi^2\rVert_{2,\out}
    +r_{\mathrm{out}}(\eps)^2(M+D).
\end{split}
\]
Subtracting the two outer equations, applying the uniformly bounded inverse
for \(\eps^2\Delta-V_1\), and absorbing the first term proves
\eqref{eq: contractivity psi-phi} in full.
\end{proof}

\subsection{Proofs of the fixed-shift correction propositions}
\label{sec: fixed-shift-correction}

\begin{proof}[Proof of \Cref{prop: fixed-shift-correction}]
After decreasing \(\eps_*\), \Cref{lem: local assembly} gives
\[
    \sup_j
    \lVert\mathcal L_{\h}(\psi)_j\rVert_
    {C_{I_j,\eps,\lambda}^{2,\gamma}
    (\spt\chi_{j,2},\L_{\p})}
    \leq1
\]
for every \(\psi\in\mathscr B_{\mathrm{out}}(\eps)\). Hence
\Cref{prop: invert outer} shows that \(\mathcal K_{\h}\) maps
\(\mathscr B_{\mathrm{out}}(\eps)\) into itself. If
\(\psi^1,\psi^2\) belong to this ball, then
\eqref{eq: contractivity psi-phi} and
\eqref{eq: local compatibility dependence}, with the shift fixed, give
\[
\begin{split}
    &\lVert\mathcal K_{\h}(\psi^1)-\mathcal K_{\h}(\psi^2)\rVert_
    {C_{\Gamma,\eps,3\lambda/4}^{2,\gamma}
    (\R^2_{\p},\L_{\p})}\\
    &\qquad\lesssim
    r_{\mathrm{out}}(\eps)\eps^{-(1-\alpha)\tau}
    \lVert\psi^1-\psi^2\rVert_
    {C_{\Gamma,\eps,3\lambda/4}^{2,\gamma}
    (\R^2_{\p},\L_{\p})}.
\end{split}
\]
The coefficient tends to zero. Thus \(\mathcal K_{\h}\) is a contraction and
has a unique fixed point \(\psi(\h)\) in
\(\mathscr B_{\mathrm{out}}(\eps)\).

The bounds for \(v_j(\h)\) and \(\phi_j(\h)\) follow from
\Cref{prop: invertibility interior,lem: local assembly}. For the boundary
response, \eqref{eq: size phi_b} gives
\[
\begin{split}
    \lVert b_p(\h)\rVert_
    {C_{\ell_{\theta_{\h}(p)},\eps,\lambda}^{2,\gamma}
    (\R^2_*,\L)}
    &\lesssim
    \eps^{1-\alpha}R_\eps
    +\eps^{2\beta-(1-\alpha)\tau}
    \lesssim R_\eps.
\end{split}
\]
Indeed, \eqref{eq: boundary-decay-exponent-range} implies
\[
    2\beta-(1-\alpha)\tau-(\beta+1-\alpha)
    >\beta-2\geq0.
\]

We next improve the weight of the outer correction. Set
\[
    \Phi_{\h}\coloneqq
    \sum_{j=1}^N\chi_{j,2}\phi_j(\h),
    \qquad
    \mathcal C_{\h}\coloneqq
    \sum_{j=1}^N
    \left(
        2\eps^2\nabla\chi_{j,2}\cdot\nabla\phi_j(\h)
        +\eps^2\phi_j(\h)\Delta\chi_{j,2}
    \right).
\]
The local assembly estimate and the scaled cutoff bounds give
\[
\begin{split}
    \lVert\Phi_{\h}\rVert_
    {C_{\Gamma,\eps,\lambda}^{2,\gamma}
    (\R^2_{\p},\L_{\p})}
    &\lesssim R_\eps,\\
    \lVert\mathcal C_{\h}\rVert_
    {C_{\Gamma,\eps,\lambda}^{0,\gamma}
    (\R^2_{\p},\L_{\p})}
    &\lesssim\eps^{1-\alpha}R_\eps.
\end{split}
\]
Estimate \eqref{eq: outer-approximation-error-global} gives
\[
    \lVert\chi_{\Ical,2}S(\omega_{\h})\rVert_
    {C_{\Gamma,\eps,\lambda}^{0,\gamma}
    (\R^2_{\p},\L_{\p})}
    \lesssim
    \exp\left(-\frac{c}{\eps^{1-\alpha}}\right).
\]
Since \(2(3\lambda/4)>\lambda\),
\((3\lambda/4)+\lambda>\lambda\), and \(2\lambda>\lambda\), the weighted
product estimates give
\[
\begin{aligned}
    \lVert\psi(\h)^2\rVert_
    {C_{\Gamma,\eps,\lambda}^{0,\gamma}(\R^2_{\p})}
    &\lesssim
    \lVert\psi(\h)\rVert_
    {C_{\Gamma,\eps,3\lambda/4}^{2,\gamma}(\R^2_{\p},\L_{\p})}^2,\\
    \lVert\psi(\h)\Phi_{\h}\rVert_
    {C_{\Gamma,\eps,\lambda}^{0,\gamma}(\R^2_{\p})}
    &\lesssim
    \lVert\psi(\h)\rVert_
    {C_{\Gamma,\eps,3\lambda/4}^{2,\gamma}(\R^2_{\p},\L_{\p})}
    \lVert\Phi_{\h}\rVert_
    {C_{\Gamma,\eps,\lambda}^{2,\gamma}(\R^2_{\p},\L_{\p})},\\
    \lVert\Phi_{\h}^2\rVert_
    {C_{\Gamma,\eps,\lambda}^{0,\gamma}(\R^2_{\p})}
    &\lesssim
    \lVert\Phi_{\h}\rVert_
    {C_{\Gamma,\eps,\lambda}^{2,\gamma}(\R^2_{\p},\L_{\p})}^2.
\end{aligned}
\]
Together with the uniformly bounded Taylor coefficients in \(\mathcal N_{\omega_{\h}}\), these
estimates imply
\[
    \lVert\chi_{\Ical,1}\mathcal N_{\omega_{\h}}(\psi(\h)+\Phi_{\h})\rVert_
    {C_{\Gamma,\eps,\lambda}^{0,\gamma}
    (\R^2_{\p},\L_{\p})}
    \lesssim
    R_\eps^2+r_{\mathrm{out}}(\eps)^2.
\]
Applying \eqref{eq: bound on G} with weight \(\lambda\) to the outer equation
at the fixed point gives
\[
\begin{split}
    \lVert\psi(\h)\rVert_
    {C_{\Gamma,\eps,\lambda}^{2,\gamma}
    (\R^2_{\p},\L_{\p})}
    \lesssim{}&
    \eps^{1-\alpha}R_\eps+R_\eps^2
    +r_{\mathrm{out}}(\eps)^2
    +\exp\left(-\frac{c}{\eps^{1-\alpha}}\right)\\
    &\lesssim\eps^{1-\alpha}R_\eps,
\end{split}
\]
where
\[
    R_\eps^2+r_{\mathrm{out}}(\eps)^2
    +\exp\left(-\frac{c}{\eps^{1-\alpha}}\right)
    =o\bigl(\eps^{1-\alpha}R_\eps\bigr).
\]
This proves \eqref{eq: fixed-shift-outer-bound}.

We now verify the residual identity. Write
\[
    L_{\omega_{\h}}\coloneqq \eps^2\Delta-W''(\omega_{\h}).
\]
Repeating the calculation in
\Cref{lem: inner-bdy decomp} with the projected interior equation gives
\[
    L_{\omega_{\h}}\phi_j(\h)
    =
    E_j(\phi_j(\h),\psi(\h))
    -\eta_{j,1}c_j(\h)H'(t_{j,\h}).
\]
On the other hand, substituting the outer equation into the expansion of
\(S(\omega_{\h}+\varphi_{\h})\) gives
\[
    S(\omega_{\h}+\varphi_{\h})
    =
    \sum_{j=1}^N
    \chi_{j,2}
    \bigl(
        L_{\omega_{\h}}\phi_j(\h)
        -E_j(\phi_j(\h),\psi(\h))
    \bigr).
\]
Here the approximation errors cancel by
\(\chi_{\Ical,2}+\sum_j\chi_{j,2}=1\), the cutoff commutators cancel those in
\(E_\out\), and the nonlinear terms cancel because
\(\chi_{j,1}\chi_{j,2}=\chi_{j,1}\) and
\(\Phi_{\h}=\phi_j(\h)\) on \(\spt\chi_{j,1}\). Finally,
\[
    V-W''(\omega_{\h})
    =
    -\sum_{j=1}^N
    \chi_{j,1}[W''(\omega_{\h})-\kappa_W],
\]
so the terms multiplying \(\psi(\h)\) cancel as well. The preceding two
identities prove \eqref{eq: fixed-shift-residual}.

Finally, we prove the profile estimates. On
\(\spt\eta_{j,1}\), \eqref{eq: full-overlap-profile-comparison} controls the
boundary overlap, while the exponential convergence of the models controls
the outer cutoff region. Hence
\[
\begin{split}
    &\left\lVert
        \omega_{\h}(x,z)
        -H\left(\frac{z-h_j(x)}{\eps}\right)
    \right\rVert_
    {C_{I_j,\eps,\lambda}^{2,\gamma}
    (\spt\eta_{j,1},\L_{\p})}\\
    &\qquad\lesssim
    \eps^{2\beta}
    +\exp\left(-\frac{c}{\eps^{1-\alpha}}\right).
\end{split}
\]
On this region the cutoff geometry gives
\(\dist(\cdot,\Gamma)=\dist(\cdot,I_j)\), and hence
\[
\begin{aligned}
    \lVert\chi_{j,2}\phi_j(\h)\rVert_
    {C_{I_j,\eps,\lambda}^{2,\gamma}
    (\spt\eta_{j,1},\L_{\p})}
    &\lesssim R_\eps,\\
    \lVert\psi(\h)\rVert_
    {C_{I_j,\eps,\lambda}^{2,\gamma}
    (\spt\eta_{j,1},\L_{\p})}
    &\lesssim\eps^{1-\alpha}R_\eps.
\end{aligned}
\]
These estimates prove
\eqref{eq: fixed-shift-interior-profile}. Near \(p\in\partial I_j\), cutoff
nesting makes \(\omega_{\h}\) equal to the rotated boundary model on the
core of \(\spt\zeta_{p,1}\); their difference on the remaining cutoff region
is exponentially small by the radial convergence of \(U\). Thus
\[
\begin{split}
    &\left\lVert
        \omega_{\h}(x,z)
        -\mathcal T_p\left(
            U_{\theta_{\h}(p)}
            \left(\frac{x}{\eps},\frac{z}{\eps}\right)
        \right)
    \right\rVert_
    {C_{*,\eps}^{2,\gamma}(\spt\zeta_{p,1},\L_{\p})}\\
    &\qquad\lesssim
    \exp\left(-\frac{c}{\eps^{1-\alpha}}\right).
\end{split}
\]
The scaled multiplication estimates and the corresponding unweighted
restrictions of the local and outer bounds give
\[
\begin{aligned}
    \lVert\chi_{j,2}\phi_j(\h)\rVert_
    {C_{*,\eps}^{2,\gamma}(\spt\zeta_{p,1},\L_{\p})}
    &\lesssim R_\eps,\\
    \lVert\psi(\h)\rVert_
    {C_{*,\eps}^{2,\gamma}(\spt\zeta_{p,1},\L_{\p})}
    &\lesssim\eps^{1-\alpha}R_\eps.
\end{aligned}
\]
This proves
\eqref{eq: fixed-shift-boundary-profile}.
\end{proof}

The dependence estimates follow by comparing the fixed points corresponding
to two shifts and invoking the local and outer Lipschitz estimates above.

\begin{proof}[Proof of \Cref{prop: fixed-shift-dependence}]
Let \(\h^1,\h^2\in\mathcal B(K\eps^\beta)\), and set
\[
    M=\lVert\h^1-\h^2\rVert_{C^{2,\gamma}(\boldsymbol I)},
    \qquad
    D_\psi=
    \lVert\psi(\h^1)-\psi(\h^2)\rVert_
    {C_{\Gamma,\eps,3\lambda/4}^{2,\gamma}
    (\R^2_{\p},\L_{\p})}.
\]
Since both outer corrections are fixed points,
\eqref{eq: contractivity psi-phi} and
\eqref{eq: local compatibility dependence} give
\[
    D_\psi
    \lesssim
    r_{\mathrm{out}}(\eps)
    \left[
        \eps^{-(1-\alpha)\tau}D_\psi
        +
        \left(
            1+\eps+\eps^{\beta-(1-\alpha)\tau}
        \right)M
    \right].
\]
Absorbing the first term yields
\[
    D_\psi\lesssim r_{\mathrm{out}}(\eps)M.
\]
Equation \eqref{eq: Lip phiin} then gives
\[
    \lVert v_j(\h^1)-v_j(\h^2)\rVert_
    {\mathcal C_{\eps,\lambda}^{2,\gamma}(I_j\times\R)}
    \lesssim\eps M.
\]
Substituting these two estimates into \eqref{eq: Lip phib} gives
\[
    \lVert\widehat b_p(\h^1)-\widehat b_p(\h^2)\rVert_
    {C_{\ell_+,\lambda}^{2,\gamma}(\R^2_*,\L)}
    \lesssim
    \left(
        \eps^{2-\alpha}
        +\eps^{\beta-(1-\alpha)\tau}
    \right)M,
\]
because the remaining exponentially small term is smaller than either
displayed power. This proves \eqref{eq: compatible correction dependence}.
\end{proof}

\subsection{Proof of \Cref{prop: projected-h-error}}
\label{sec: projected-h-error}

\begin{proof}
Fix \(j\), and write
\[
    v_j=v_j(\h),
    \qquad
    b_p=b_p(\h),
    \qquad
    \psi=\psi(\h).
\]
All terms inside \(\mathfrak p_j\) below are read in the shifted interior
coordinates. Taking the \(H'\)-projection of the interior right-hand side in
\Cref{lem: inner-bdy decomp} gives
\begin{equation}\label{eq: expanded-projected-remainder}
\begin{split}
    R_j(\h)
    &=
    \eps\bigl(a_{j,\h}-\eta_{j,3}\bigr)h_j''
    +
    \mathfrak p_j\Bigg(
    \chi_{j,4}\eta_{j,4}\mathcal R_j(\phi_j(\h),\psi(\h))\\
    &\qquad
    -2\eps^2\chi_{j,4}\sum_{p\in\partial I_j}
    \nabla\zeta_{p,4}\cdot\nabla b_p
    -\eps^2\chi_{j,4}\sum_{p\in\partial I_j}
    \Delta\zeta_{p,4}b_p\\
    &\qquad
    -\chi_{j,4}\eta_{j,4}
    m_{j,h_j}^{\mathrm{bd}}[v_j]H'
    -\chi_{j,4}m_{j,h_j}^{\mathrm{int}}[v_j]H'
    \Bigg),
\end{split}
\end{equation}
Here \(a_{j,\h}\) is the coefficient from
\Cref{lem: approximation-remainder-size}, \(\phi_j(\h)\) is reconstructed by
\eqref{eq: decomp phi}, and the \(\h\)-dependence of the pulled-back cut-offs
is suppressed.

\emph{Support and noncommutator terms.}
The definition of \(\mathcal X_\eps\) gives
\(\spt h_j''\subset\spt\eta_{j,4}\). Thus the first term in
\eqref{eq: expanded-projected-remainder} has the required support. The
\(\mathcal R_j\)-term and the term containing
\(m_{j,h_j}^{\mathrm{bd}}[v_j]\)
carry the factor \(\eta_{j,4}\), while
\(m_{j,h_j}^{\mathrm{int}}[v_j]\) carries the factor \(h_j''\). They are
therefore also supported in \(\spt\eta_{j,4}\).
At each \(p\in\partial I_j\), use the inward coordinate \(x\geq0\) and
abbreviate \(h_j(p+xe_p)\) by \(h_j(x)\). The
\(\chi_{j,4}\)-weighted boundary commutators are supported where
\(\zeta_{p,4}\) varies. After multiplication by \(\chi_{j,4}\), the
derivatives of \(\zeta_{p,4}\) are horizontal, and their projected
\(x\)-support is contained in the interval \(I_\eps\) fixed in
\Cref{subs: cutoff}. On this support
\(\eta_{j,4}=1-\zeta_{p,4}\) after the other boundary
cutoff has vanished. Hence the projected \(x\)-support of these commutator
terms is contained in \(\spt\eta_{j,4}\). Since the cut-offs satisfy
\(\spt\eta_{j,4}\subset\{\eta_{j,3}=1\}\), this proves the support inclusion.

On \(\spt(\chi_{j,4}\eta_{j,4})\), cutoff nesting gives
\(\chi_{j,5}=\eta_{j,3}=1\). Hence the approximation part of
\(\chi_{j,4}\eta_{j,4}\mathcal R_j\) is
\(-\widetilde{\mathcal E}^{\mathrm{rem}}_{j,\h}\). Equation
\eqref{eq: projected-approximation-error-size} therefore gives, for every
\(m>0\),
\[
\begin{split}
    &\left\lVert
        \eps(a_{j,\h}-\eta_{j,3})h_j''
    \right\rVert_{C^{0,\gamma}(I_j)}\\
    &\qquad+
    \left\lVert
        \mathfrak p_j
        (\widetilde{\mathcal E}^{\mathrm{rem}}_{j,\h})
    \right\rVert_{C^{0,\gamma}(I_j)}
    \lesssim_m\eps^m.
\end{split}
\]

The remaining terms in \(\mathcal R_j\) contain either the nonlinear
remainder or the outer correction. The local assembly and fixed-shift bounds,
followed by the passage from the stretched H\"older seminorm to the ordinary
one in \(x\), give
\[
    \left\lVert
        \mathfrak p_j\bigl(
        \chi_{j,4}\eta_{j,4}\chi_{j,1}
        \mathcal N_{\omega_{\h}}(\psi+\phi_j(\h))
        \bigr)
    \right\rVert_{C^{0,\gamma}(I_j)}
    \lesssim\eps^{-\gamma}R_\eps^2.
\]
Since \(\psi\in\mathscr B_{\mathrm{out}}(\eps)\), the term containing
\([W''(\omega_{\h})-\kappa_W]\psi\) satisfies, for every \(m>0\),
\[
    \left\lVert
        \mathfrak p_j\bigl(
        \chi_{j,4}\eta_{j,4}\chi_{j,1}
        [W''(\omega_{\h})-\kappa_W]\psi
        \bigr)
    \right\rVert_{C^{0,\gamma}(I_j)}
    \lesssim_m\eps^m.
\]

Performing the \(t\)-integrals in the moment terms before taking the interval
norm gives
\[
\begin{aligned}
    \lVert m_{j,h_j}^{\mathrm{bd}}[v_j]\rVert_{C^{0,\gamma}(I_j)}
    &\lesssim_K
    \eps^{\beta-\gamma}R_\eps,\\
    \lVert m_{j,h_j}^{\mathrm{int}}[v_j]\rVert_{C^{0,\gamma}(I_j)}
    &\lesssim_K
    \eps^{\beta+1-\gamma}R_\eps.
\end{aligned}
\]
As \(R_\eps=\eps^{\beta+1-\alpha}\),
\[
\begin{aligned}
    \eps^{-\gamma}R_\eps^2
    &=\eps^{1+\beta}
      \eps^{\beta+1-2\alpha-\gamma},\\
    \eps^{\beta-\gamma}R_\eps
    &=\eps^{1+\beta}
      \eps^{\beta-\alpha-\gamma}.
\end{aligned}
\]
Both remaining exponents are positive. This proves the size estimate for all
noncommutator terms in \eqref{eq: expanded-projected-remainder}.

\emph{Boundary commutators.}
Fix a boundary point \(p\in\partial I_j\). By \Cref{subs: cutoff}, the
projected \(x\)-support of
\[
    \chi_{j,4}\nabla\zeta_{p,4}
    \quad\text{and}\quad
    \chi_{j,4}\Delta\zeta_{p,4}
\]
lies in \(I_\eps\). On this support, the derivatives of \(\zeta_{p,4}\) are
horizontal in the boundary coordinates. To compare their
segment projections with the tangential estimates from
\Cref{prop: invert boundary}, set
\[
\begin{split}
    P_{p,0}(x)&\coloneqq
    \frac1{A_H}\int_\R
    \chi_{j,4}b_p(x,h_j(x)+\eps t)H'(t)\,dt,\\
    P_{p,1}(x)&\coloneqq
    \frac1{A_H}\int_\R
    \chi_{j,4}\partial_xb_p(x,h_j(x)+\eps t)H'(t)\,dt,
\end{split}
\]
where \(\partial_xb_p\) is the boundary-horizontal derivative paired with
\(\nabla\zeta_{p,4}\) in the commutator.

Retaining the inward coordinate fixed above, write
\(\theta=\theta_{\h}(p)=\partial_{e_p}h_j(p)\),
\(\widehat b_p(X,s)=b_p(\eps R_{-\theta}(X,s))\), and recall that
\[
    \overline X_h(x)=
    \eps^{-1}\bigl(x\cos\theta+h_j(x)\sin\theta\bigr).
\]
If \(z=h_j(x)+\eps t\), then the unrotated boundary coordinates satisfy
\[
    X=\overline X_h(x)+t\sin\theta,
    \qquad
    s-t=t(\cos\theta-1)+r_h(x),
\]
where, with \(q(\theta)=\theta\cos\theta-\sin\theta\),
\[
\begin{split}
    r_h(x)
    &=\eps^{-1}\bigl(h_j(x)\cos\theta-x\sin\theta\bigr)\\
    &=\eps^{-1}xq(\theta)
      +\eps^{-1}\bigl(h_j(x)-\theta x\bigr)\cos\theta .
\end{split}
\]
On \(I_\eps\), Taylor's theorem and \(\h\in\mathcal X_\eps\) give
\[
    |h_j(x)-\theta x|\lesssim_K\eps^\beta x^{2+\gamma},
    \qquad
    |h_j'(x)-\theta|\lesssim_K\eps^\beta x^{1+\gamma}.
\]
Set
\[
    B_{p,0}(x)\coloneqq
    \frac1{A_H}\int_\R
    \widehat b_p(\overline X_h(x),s)H'(s)\,ds,
    \qquad
    B_{p,1}\coloneqq\partial_xB_{p,0}.
\]
These are the normalized tangential quantities from
\eqref{eq: decays tangential part}, composed with \(\overline X_h\). Set
\[
    E_{p,0}(\h)\coloneqq P_{p,0}-B_{p,0},
    \qquad
    E_{p,1}(\h)\coloneqq P_{p,1}-B_{p,1},
\]
and abbreviate these two quantities by \(E_{p,0}\) and \(E_{p,1}\) until
the shift comparison below. Then
\(q'(\theta)=-\theta\sin\theta=O(\theta^2)\), the mean value theorem,
\eqref{eq: size phi_b}, and the exponential decay of \(H'\) give, after
absorbing the cut-off tails,
\[
\begin{split}
    \lVert E_{p,0}\rVert_{L^\infty(I_\eps)}
    +\eps^\gamma[E_{p,0}]_{\gamma,I_\eps}
    &\lesssim_K\eps^\beta R_\eps,\\
    \lVert E_{p,1}\rVert_{L^\infty(I_\eps)}
    +\eps^\gamma[E_{p,1}]_{\gamma,I_\eps}
    &\lesssim_K\eps^{\beta-1}R_\eps.
\end{split}
\]
In particular,
\[
    \lVert E_{p,0}\rVert_{C^{0,\gamma}(I_\eps)}
    \lesssim_K\eps^{\beta-\gamma}R_\eps,
    \qquad
    \lVert E_{p,1}\rVert_{C^{0,\gamma}(I_\eps)}
    \lesssim_K\eps^{\beta-1-\gamma}R_\eps.
\]
Thus the projected Laplacian and gradient commutators are controlled by
\(\Delta\zeta_{p,4}(B_{p,0}+E_{p,0})\) and
\(\nabla\zeta_{p,4}(B_{p,1}+E_{p,1})\), respectively.

The tangential estimate \eqref{eq: decays tangential part} gives
\[
\begin{split}
    |b_{p,\infty}|
    &+\lVert B_{p,0}\rVert_{L^\infty(I_\eps)}
    +\eps^{\alpha\gamma}[B_{p,0}]_{\gamma,I_\eps}\\
    &+\eps^\alpha\lVert B_{p,1}\rVert_{L^\infty(I_\eps)}
    +\eps^{\alpha(1+\gamma)}[B_{p,1}]_{\gamma,I_\eps}
    \lesssim\eps^{2\beta-2+2\alpha}.
\end{split}
\]
The product rule and the cutoff derivative estimates therefore give
\[
\begin{split}
    &\lVert\eps^2\Delta\zeta_{p,4}B_{p,0}\rVert_{C^{0,\gamma}(I_\eps)}
    +\lVert\eps^2\nabla\zeta_{p,4}B_{p,1}\rVert_{C^{0,\gamma}(I_\eps)}\\
    &\qquad\lesssim
    \eps^{2-\alpha(2+\gamma)}
    \eps^{2\beta-2+2\alpha}
    =\eps^{2\beta-\alpha\gamma}.
\end{split}
\]
The corresponding calculation for the coordinate errors yields
\[
\begin{split}
    &\lVert\eps^2\Delta\zeta_{p,4}E_{p,0}\rVert_{C^{0,\gamma}(I_\eps)}
    +\lVert\eps^2\nabla\zeta_{p,4}E_{p,1}\rVert_{C^{0,\gamma}(I_\eps)}\\
    &\qquad\lesssim_K
    \eps^{1-\alpha+\beta-\gamma}R_\eps
    =\eps^{2\beta+2-2\alpha-\gamma}.
\end{split}
\]
Both exponents exceed \(1+\beta\), since
\[
    \beta-1-\alpha\gamma>0,
    \qquad
    \beta+1-2\alpha-\gamma>0.
\]
Summing over the two boundary points of \(I_j\) gives \eqref{eq: Rj size}.

\emph{Dependence on the shift.}
Set
\[
    M\coloneqq
    \lVert\h^1-\h^2\rVert_{C^{2,\gamma}(\boldsymbol I)},
    \qquad
    v_i\coloneqq v_j(\h^i),
    \qquad i=1,2.
\]
By \eqref{eq: compatible correction dependence},
\[
    \lVert v_i\rVert_{\mathcal C_{\eps,\lambda}^{2,\gamma}(I_j\times\R)}
    \lesssim R_\eps,
    \qquad
    \lVert v_1-v_2\rVert_{\mathcal C_{\eps,\lambda}^{2,\gamma}(I_j\times\R)}
    \lesssim \eps M.
\]
Equation \eqref{eq: projected-approximation-error-dependence} controls both
approximation terms in the difference of
\eqref{eq: expanded-projected-remainder}: for every \(m>0\), their combined
\(C^{0,\gamma}(I_j)\)-norm is \(O_m(\eps^mM)\).

For the moment terms, performing the \(t\)-integrals before passing to the
ordinary interval norm gives
\[
\begin{split}
    \lVert m_{j,h_j^1}^{\mathrm{bd}}[v_1]
    -m_{j,h_j^2}^{\mathrm{bd}}[v_2]\rVert_{C^{0,\gamma}(I_j)}
    &\lesssim_K
    \eps^{\beta-\gamma}\lVert v_1-v_2\rVert_
    {\mathcal C_{\eps,\lambda}^{2,\gamma}(I_j\times\R)}
    +\eps^{-\gamma}R_\eps M,\\
    \lVert m_{j,h_j^1}^{\mathrm{int}}[v_1]
    -m_{j,h_j^2}^{\mathrm{int}}[v_2]\rVert_{C^{0,\gamma}(I_j)}
    &\lesssim_K
    \eps^{\beta+1-\gamma}\lVert v_1-v_2\rVert_
    {\mathcal C_{\eps,\lambda}^{2,\gamma}(I_j\times\R)}
    +\eps^{1-\gamma}R_\eps M.
\end{split}
\]
Using
\(\lVert v_1-v_2\rVert_{
\mathcal C_{\eps,\lambda}^{2,\gamma}(I_j\times\R)}\lesssim\eps M\),
their total contribution to
\(R_j(\h^1)-R_j(\h^2)\) is
\[
    O_K\bigl(\eps^{\beta+1-\alpha-\gamma}M\bigr).
\]

For the nonlinear part of \(\mathcal R_j\), the weighted product and
composition estimates, \eqref{eq: compatible correction dependence}, and the
reconstruction formula give
\[
\begin{split}
    &\eps^{-\gamma}R_\eps
    \left(
    \lVert v_1-v_2\rVert_
    {\mathcal C_{\eps,\lambda}^{2,\gamma}(I_j\times\R)}
    +\sum_{p\in\partial I_j}
    \lVert\widehat b_p(\h^1)-\widehat b_p(\h^2)\rVert_
    {C_{\ell_+,\lambda}^{2,\gamma}(\R^2_*,\L)}
    \right)\\
    &\qquad\lesssim_K
    \eps^{-\gamma}R_\eps
    \left(
    \eps+\eps^{2-\alpha}
    +\eps^{\beta-(1-\alpha)\tau}
    \right)M
    \lesssim_K\eps^{\beta+1-\alpha-\gamma}M.
\end{split}
\]
Here \(\beta-(1-\alpha)\tau>0\) by
\eqref{eq: boundary-decay-exponent-range}. Variations of the pulled-back
cut-offs and coefficients are smaller, while all terms containing the outer
correction or its difference are exponentially small. Consequently, every
noncommutator term in the difference is bounded by
\[
    C_K\eps^{\beta+1-\alpha-\gamma}M
    =o\bigl(\eps^{\beta-\alpha\gamma}M\bigr),
\]
where the exponent gap is \((1-\alpha)(1-\gamma)>0\).

It remains to compare the boundary commutators. For \(i=1,2\), let
\(B_{p,0}^i\), \(B_{p,1}^i\), and \(b_{p,\infty}^i\) be the tangential
quantities above for \(\h^i\), and set
\(\theta_i\coloneqq\theta_{\h^i}(p)\). Substituting the
compatible-correction bounds into \eqref{eq: tangential dependence phib}
gives
\[
\begin{split}
    |b_{p,\infty}^1-b_{p,\infty}^2|
    &+\lVert B_{p,0}^1-B_{p,0}^2\rVert_{L^\infty(I_\eps)}
    +\eps^{\alpha\gamma}
      [B_{p,0}^1-B_{p,0}^2]_{\gamma,I_\eps}\\
    &+\eps^\alpha
      \lVert B_{p,1}^1-B_{p,1}^2\rVert_{L^\infty(I_\eps)}
    +\eps^{\alpha(1+\gamma)}
      [B_{p,1}^1-B_{p,1}^2]_{\gamma,I_\eps}\\
    &\lesssim_K
    \left(
        \eps^{1+(1-\alpha)(\tau-1)}
        +\eps^{\beta-2+2\alpha}
    \right)M\\
    &\lesssim_K
    \left(
        \eps^{2-\alpha}
        +\eps^{\beta-2+2\alpha}
    \right)M.
\end{split}
\]
The outer term omitted from the right-hand side is exponentially small. The
same cutoff product estimate used for the size bound now gives
\[
\begin{split}
    &\lVert\eps^2\Delta\zeta_{p,4}(B_{p,0}^1-B_{p,0}^2)\rVert_
    {C^{0,\gamma}(I_\eps)}\\
    &\quad+
    \lVert\eps^2\nabla\zeta_{p,4}(B_{p,1}^1-B_{p,1}^2)\rVert_
    {C^{0,\gamma}(I_\eps)}\\
    &\qquad\lesssim_K
    \left(
        \eps^{4-\alpha(3+\gamma)}
        +\eps^{\beta-\alpha\gamma}
    \right)M.
\end{split}
\]

The coordinate identities above, the Taylor estimates for
\((h_j^1-h_j^2)-(\theta_1-\theta_2)x\), and
\eqref{eq: compatible correction dependence} give
\[
\begin{split}
    \lVert E_{p,0}(\h^1)-E_{p,0}(\h^2)\rVert_{L^\infty(I_\eps)}
    +\eps^\gamma
    [E_{p,0}(\h^1)-E_{p,0}(\h^2)]_{\gamma,I_\eps}
    &\lesssim_K R_\eps M,\\
    \lVert E_{p,1}(\h^1)-E_{p,1}(\h^2)\rVert_{L^\infty(I_\eps)}
    +\eps^\gamma
    [E_{p,1}(\h^1)-E_{p,1}(\h^2)]_{\gamma,I_\eps}
    &\lesssim_K\eps^{-1}R_\eps M.
\end{split}
\]
In particular,
\[
\begin{aligned}
    \lVert E_{p,0}(\h^1)-E_{p,0}(\h^2)\rVert_{C^{0,\gamma}(I_\eps)}
    &\lesssim_K\eps^{-\gamma}R_\eps M,\\
    \lVert E_{p,1}(\h^1)-E_{p,1}(\h^2)\rVert_{C^{0,\gamma}(I_\eps)}
    &\lesssim_K\eps^{-1-\gamma}R_\eps M.
\end{aligned}
\]
Their contribution to the commutator difference is therefore
\[
    O_K\bigl(\eps^{1-\alpha-\gamma}R_\eps M\bigr)
    =O_K\bigl(\eps^{\beta+2-2\alpha-\gamma}M\bigr)
    =o\bigl(\eps^{\beta-\alpha\gamma}M\bigr),
\]
because the exponent gap is \((1-\alpha)(2-\gamma)>0\). Combining the
noncommutator, tangential, and coordinate estimates proves
\eqref{eq: Rj lip}. Its two exponents satisfy
\[
    \beta-\alpha\gamma>1,
    \qquad
    4-\alpha(3+\gamma)>1,
\]
by \(\beta\geq2\), \(\alpha<1/2\), and \(\gamma<1\).

The constants in the intermediate size bounds may depend on the fixed
\(K\). Every such term carries one of the strictly positive excess powers of
\(\eps\) displayed above, so this dependence is absorbed by decreasing
\(\eps_*(K)\). Thus the constant in \eqref{eq: Rj size} depends only on the
fixed data. For the sharp difference estimate we retain the dependence on the
fixed ball multiplier in the constant \(C_K\).
\end{proof}

\noindent{\bf Acknowledgments.}
MB was supported by the European Research Council under Grant
Agreement No 948029. MD has been supported by the Royal Society Research
Professorship grant RP-R1-180114 and by the ERC/UKRI Horizon Europe grant
ASYMEVOL, EP/Z000394/1.

\bibliographystyle{alpha}
\bibliography{Bib}

\end{document}